\documentclass[11pt]{article}
\usepackage[margin=.88in]{geometry}
\usepackage[utf8]{inputenc}
\usepackage{amsfonts}
\usepackage{dsfont}
\usepackage{amsgen,amsmath,amstext,amsbsy,amsopn,amssymb,subcaption}
\usepackage[dvips]{graphicx}
\usepackage{comment}

\usepackage[shortlabels]{enumitem}
\usepackage{natbib}
\usepackage{booktabs}
\usepackage{amsthm}
\usepackage{hyperref}
\usepackage{blkarray}
\usepackage{algorithm}
\usepackage{algorithmic}
\usepackage{url}
\usepackage{longtable}
\usepackage{stmaryrd}
\usepackage{mwe}

\SetSymbolFont{stmry}{bold}{U}{stmry}{m}{n}
\DeclareMathAlphabet\mathbfcal{OMS}{cmsy}{b}{n}

\makeatletter
\newcommand*{\rom}[1]{\expandafter\@slowromancap\romannumeral #1@}
\makeatother

\usepackage[dvipsnames]{xcolor}

\newcommand{\dmax}{ \overline{d} } 
\newcommand{\dmin}{\underline{d} }
\newcommand{\rmax}{ \overline{r} }
\newcommand{\rmin}{ \underline{r} } 
\newcommand{\sigmax}{\overline{\sigma}_{\mathsf{exp} }}
\newcommand{\sigmin}{\underline{\sigma}}

\newcommand{\bcX}{{\mathbfcal{X}}}
\newcommand{\bcA}{{\mathbfcal{A}}}
\newcommand{\bcC}{{\mathbfcal{C}}}
\newcommand{\bcD}{{\mathbfcal{D}}}

\newcommand{\bcI}{{\mathbfcal{I}}}
\newcommand{\bcJ}{{\mathbfcal{J}}}
\newcommand{\bcS}{{\mathbfcal{S}}}

\newcommand{\bcZ}{{\mathbfcal{Z}}}
\newcommand{\bcT}{{\mathbfcal{T}}}

\newcommand{\bcE}{{\mathbfcal{E}}}
\newcommand{\bcR}{{\mathbfcal{R}}}
\newcommand{\bcG}{{\mathbfcal{G}}}
\newcommand{\bcW}{{\mathbfcal{W}}}

\newcommand{\bA}{\mathbf{A}}
\newcommand{\bB}{\mathbf{B}}
\newcommand{\bC}{\mathbf{C}}
\newcommand{\bE}{\mathbf{E}}
\newcommand{\bT}{\mathbf{T}}
\newcommand{\bF}{\mathbf{F}}

\newcommand{\bH}{\mathbf{H}}
\newcommand{\bR}{\mathbf{R}}
\newcommand{\bX}{\mathbf{X}}

\newcommand{\bO}{\mathbf{O}}
\newcommand{\be}{\mathbf{e}}

\newcommand{\bU}{\mathbf{U}}
\newcommand{\bV}{\mathbf{V}}
\newcommand{\bW}{\mathbf{W}}

\newcommand{\bI}{\mathbf{I}}
\newcommand{\br}{\mathbf{r}}

\newcommand{\bx}{\mathbf{x}}
\newcommand{\by}{\mathbf{y}}
\newcommand{\bz}{\mathbf{z}}

\newcommand{\bdelta}{\boldsymbol{\delta}}
\newcommand{\Perm}{\mathbf{Pm}}

\newcommand{\bLambda}{\mathbf{\Lambda}}
\newcommand{\bomega}{{\boldsymbol{\omega}}}

\newcommand{\cD}{\mathcal{D}}

\newcommand{\cN}{\mathcal{N}}

\newcommand{\cS}{\mathcal{S}}

\newcommand{\cP}{\mathcal{P}}

\newcommand{\cE}{\mathcal{E}}
\newcommand{\cM}{\mathcal{M}}
\newcommand{\cQ}{\mathcal{Q}}

\newcommand{\frE}{\mathfrak{E}}
\newcommand{\frF}{\mathfrak{F}}
\newcommand{\frG}{\mathfrak{G}}
\newcommand{\frH}{\mathfrak{H}}

\def\bbB{\mathbb{B}}
\def\R{\mathbb{R}}
\def\O{\mathbb{O}}
\def\TT{\mathbb{T}}

\def\bbS{\mathbb{S}}
\newcommand{\bbP}{\mathbb{P}}
\newcommand{\bbH}{\mathbb{H}}

\newcommand{\E}{\mathbb{E}}

\newcommand{\argmin}{\mathop{\rm arg\min}}

\newcommand{\tF}{{\rm F}}
\newcommand{\SVD}{{\rm SVD}}
\newcommand{\HOSVD}{{\rm HOSVD}}
\newcommand{\rank}{{\rm rank}}

\newcommand{\Id}{{\rm Id}}
\newcommand{\Inco}{{\rm Inco}}
\newcommand{\Vect}{{\rm Vec}}

\newcommand{\Var}{{\rm{Var}}}

\newcommand{\idc}{\mathbb{I}}

\newcommand{\gap}{{\mathsf{gap} }}
\newcommand{\init}{{\mathsf{init}}}
\newcommand{\net}{{\mathsf{net}}}
\newcommand{\ubs}{{\mathsf{ubs}}}
\newcommand{\rn}{{\mathsf{rn}}}
\newcommand{\rem}{{\mathsf{rem}}}
\newcommand{\sfm}{{\mathsf{m}}}

\newcommand{\test}{{\mathsf{test}}}
\newcommand{\SNR}{{\mathsf{SNR}}}
\newcommand{\sfh}{{\mathsf{h}}}
\newcommand{\opt}{{\mathsf{opt}}}

\newcommand{\main}{{\mathsf{main}}}
\newcommand{\linf}{{\ell_\infty}}

\newcommand{\abs}[1]{\left\lvert#1\right\rvert}
\newcommand{\norm}[1]{\left\lVert#1\right\rVert}

\newcommand{\commentWT}[1]{{\color{blue}[#1]}}

\newtheorem{Theorem}{Theorem}
\newtheorem{Assumption}{Assumption}

\newtheorem{Lemma}{Lemma}
\newtheorem{Remark}{Remark}
\newtheorem{Corollary}{Corollary}

\newtheorem{Proposition}{Proposition}

\title{Statistical Inference in Tensor Completion: Optimal Uncertainty Quantification and Statistical-to-Computational Gaps}
\author{Wanteng Ma$^1$ and   Dong Xia$^2$ \\
~ \\
$^1$Department of Statistics and Data Science, University of Pennsylvania \\
$^2$Department of Mathematics, HKUST
}
\date{\today\\
(First version: October 15, 2024)}
\begin{document}

\maketitle

\footnotetext[1]{Email: wanteng@wharton.upenn.edu}
\footnotetext[2]{Dong Xia's research was partially supported by Hong Kong RGC GRF 16303224.}

\begin{abstract}
This paper presents a simple yet efficient method for statistical inference of tensor linear forms using incomplete and noisy observations. Under the Tucker low-rank tensor model and the missing-at-random assumption, we utilize an appropriate initial estimate along with a debiasing technique followed by a one-step power iteration to construct an asymptotically normal test statistic. This method is suitable for various statistical inference tasks, including constructing confidence intervals, inference under heteroskedastic and sub-exponential noise, and simultaneous testing. We demonstrate that the estimator achieves the Cramér-Rao lower bound on Riemannian manifolds, indicating its optimality in uncertainty quantification. We comprehensively examine the statistical-to-computational gaps and investigate the impact of initialization on the minimal conditions regarding sample size and signal-to-noise ratio required for accurate inference. Our findings show that with independent initialization, statistically optimal sample sizes and signal-to-noise ratios are sufficient for accurate inference. Conversely, if only dependent initialization is available, computationally optimal sample sizes and signal-to-noise ratio conditions still guarantee asymptotic normality without the need for data-splitting. We present the phase transition between computational and statistical limits. Numerical simulation results align with the theoretical findings. 
\end{abstract}
\begin{sloppypar}

\section{Introduction}\label{sec:intro}

\subsection{Background on tensor completion and its inference}  
Tensor data has attracted tremendous attention in the big data era due to its flexible structure and wide-ranging applications across various fields. Numerous tensor-based approaches have been introduced, demonstrating their effectiveness in addressing real-world problems, such as medical image processing \citep{semerci2014tensor}, spatio-temporal data analysis \citep{bahadori2014fast,chen2019bayesian}, recommender system design \citep{bi2018multilayer,zhang2021dynamic}, and time series analysis \citep{rogers2013multilinear,chen2022factor}. These statistical approaches commonly assume that the underlying data model is characterized by an $m$-th order array, i.e., a tensor $\bcT\in\R^{d_1\times\cdots\times d_m}$ of order $m$. The tensor $\bcT$ can be converted into an ultra-long vector with dimension $d^{\ast}:=\prod_{i=1}^m d_i$ if it possesses no additional structures. 
However, a fascinating advantage of the low-rank tensor model is its capability to capture multi-way structures, which not only significantly reduces model complexity but also plays critical roles in real-world applications. Here, a tensor is considered low-rank if it can be decomposed into the sum of a few rank-one tensors.   Notable low-rank tensor decompositions include the  Canonical Polyadic (CP) decomposition \citep{carroll1970analysis},   multilinear (Tucker) rank tensor decomposition \citep{tucker1966some}, and tensor-train decomposition \citep{oseledets2011tensor}, among others. 

Missing data is a common issue in many real-world applications. Tensor completion refers to the problem of reconstructing the tensor $\bcT$ by observing only a small subset of its entries, which may contain noise. Several tensor completion methods \citep{liu2012tensor,barak2016noisy,yuan2016tensor,xia2021statistically} have been developed to impute the missing values. Nevertheless, most current tensor completion techniques provide only recovery guarantees, while inference on underlying parameters and the quantification of uncertainties for these methods remain largely underexplored. From a pragmatic perspective, such statistical inferences in tensor models are crucial, as they allow for the assessment of confidence and risk in the reconstruction, thereby aiding further decision-making processes.

In this paper, we focus on the inference of linear forms of a tensor, i.e., $\langle \bcT, \bcI \rangle$, given noisy and incomplete observations of the entries of $\bcT$. Here, $\langle\cdot,\cdot\rangle$ denotes the inner product, and $\bcI$ is a general indexing tensor that can take on various values. Linear forms conveniently represent the linear combinations of tensor entries, making them adaptable to a wide range of applications. The following three examples illustrate the flexibility and significance of inferring linear forms in tensor data:
\begin{enumerate}
    \item The monthly trade flow data \citep{lyu2023latent,cai2023generalized} is a fourth-order tensor with the dimensions: $\text{Country1} \times \text{Country2} \times \text{Commodity} \times \text{Month}$. This tensor records the monthly trade volumes of various commodities among different countries. Detecting changes in trading volumes for a specific commodity can be formulated as testing whether $\bcT(i_1,i_2,i_3,i_4+1) - \bcT(i_1,i_2,i_3,i_4) > 0$ for two consecutive months. 
    \item A multi-dimensional recommender system \citep{bi2018multilayer} can be represented by a third-order tensor with dimensions:  $\text{User} \times \text{Item} \times \text{Context}$. Given the group information of the items, determining whether it is worthwhile to recommend a subgroup of items $\mathcal{G}$ to a user $i_1$ under context $i_3$ can be formulated as testing whether $\sum_{j \in \mathcal{G}} \bcT(i_1, j, i_3) > C$ for a certain threshold $C>0$. 
    \item In the medical image processing of 3D MRI data \citep{nandpuru2014mri,tandel2020multiclass}, many Haralick texture features used for brain cancer classification (such as dissimilarity, inverse, and contrast features) are linear combinations of voxels within a third-order tensor $\bcT$. For example, the dissimilarity features are weighted sums of entries within slices. Quantifying the uncertainties of these texture features involves providing confidence intervals for the linear forms of $\bcT$.
\end{enumerate}


  The statistical inference of linear forms for noisy matrix completion was studied in \cite{xia2021statistical}, with entrywise inference as a special case investigated by \cite{chen2019inference, farias2022uncertainty}. The problem becomes significantly more challenging in noisy tensor completion for two main reasons. First, the estimation procedure in noisy tensor completion is more complex, as the denoising process by a low-rank retraction must be examined more carefully. Second, the multi-way structures cause the low-rank factors to become entangled in a complicated manner, making the technical analysis more challenging. To date, reliable approaches that ensure the inference of general linear forms for noisy tensor completion, along with the necessary statistical conditions, remain largely unknown.  

Apart from statistical challenges, the potential computational intractability inherent in tensor problems is another critical issue. Due to the highly non-convex nature of the low-rank tensor parameter space, it is known that the conditions necessary for achieving computationally efficient tensor recovery are significantly more stringent than the minimum statistical requirements. This phenomenon, referred to as the ``statistical-to-computational gap'', has been widely observed in various tensor problems, including tensor PCA \citep{richard2014statistical, zhang2018tensor}, tensor regression \citep{han2022optimal, xia2022inference}, tensor completion \citep{barak2016noisy, xia2021statistically, cai2022provable}, and tensor-on-tensor regression \citep{luo2022tensor}, among others. This gap motivates us to investigate whether it exists in inference tasks and, consequently, to determine the exact phase transitions and their impacts on inference methods.


\subsection{Our contributions}

Our main contributions are summarized as follows:

First, we propose a simple yet effective method to construct a test statistic with a standard normal limit for any linear form of a Tucker low-rank tensor from noisy observations of a small subset of its entries. The method consists of two parts: debiasing using a warm initialization and a single-step power iteration, both of which are computationally fast. The variance of the test statistic matches the intrinsic Cramér–Rao lower bound on Riemannian manifolds from information geometry, indicating optimal uncertainty quantification. The primary challenge in our analysis lies in characterizing the randomness of the low-rank retraction. To overcome this technical hurdle, we carefully investigate the singular spaces generated by the power iteration through a fine-grained decomposition technique to unravel the randomness of perturbations and derive a sharp $\ell_{2,\infty}$ norm perturbation bound, enabling us to control the remainder error in developing the central limit theorem. 

Second, we extend our method to various inferential tasks including the confidence interval construction, simultaneous inference of two (or more) linear forms, and inference with heteroskedastic and sub-exponential noises. We establish the coverage probability for the confidence intervals constructed from our test statistics. Furthermore, we study the intricate dependency between testing different linear forms, enabling the study of simultaneous inference problems based on the limiting joint distribution. We also adapt our approach to handle heteroskedastic and sub-exponential noise with a minor adjustment in variance quantification.

Finally, we examine the statistical and computational gaps of our tensor inference framework across varying signal-to-noise ratios (SNRs) and sample sizes, taking initialization into account. We demonstrate that valid inferences can be performed using only statistically optimal SNR and sample size conditions when disregarding computational efficiency. However, to achieve computationally efficient inferences, computationally optimal SNR and sample size conditions are still required. Notably, under these computationally optimal conditions, data splitting is unnecessary for the initialization process. Specifically, when $m=3$, a leave-one-out type initialization provided by polynomial-time algorithms, such as gradient descent, can be employed to ensure the asymptotic normality of our test statistics. For $m \geq 4$, the conditions required to overcome the dependency between the data used for initialization and for debiasing are much weaker than the computationally optimal conditions, meaning that any warm initialization with an entrywise error guarantee is suitable for our purpose.

\subsection{Related literature}
Tensor completion, as a natural generalization of matrix completion, has been extensively studied in recent years. Many optimization algorithms have been proposed and proved theoretically reliable for the low-rank tensor recovery, including alternating minimization \citep{liu2012tensor,jain2014provable}, sum of squares hierarchy \citep{barak2016noisy,potechin2017exact}, power iterations \citep{uschmajew2012local,xia2021statistically}, gradient-based iterations \citep{cai2019nonconvex,tong2022scaling,wang2023implicit}, convex relaxation \citep{gandy2011tensor,yuan2016tensor}, among others. Unfortunately, most of the existing theories only focus on the error bounds of the tensor recovery but fail to provide uncertainty quantification and confidence intervals of the estimates. 

The pioneering work of \cite{cai2022uncertainty} first studied the inference of tensor factors and entries in the noisy tensor completion model. Their work is largely limited to the symmetric CP low-rank model, and the inference target can only be a single entry since their main technique, the leave-one-out argument, is not directly applicable to the inference of linear forms. Besides tensor completion, other studies in different tensor models are also noteworthy. For example, \cite{huang2022power,xia2022inference} investigate the statistical inference of power iteration procedures under tensor PCA and tensor regression. \cite{lyu2019tensor} explores multiple testing in a tensor graphical model where the observations form a tensor, but the targets of inference remain within a matrix structure. Recent work by \cite{wen2023online} studies the inference of linear forms in online tensor regression, but the conditions used in their online debiased approach are suboptimal for offline inference. More importantly, they do not address a key issue in inference -- namely, the usual dependence of initialization on the data used for bias-correction steps -- which is circumvented by the online nature of their method.

Our work is closely related to inference problems in matrix completion, as low-rank matrices can be viewed as second-order tensors. Several studies \citep{chen2019inference, xia2021statistical, farias2022uncertainty} have investigated the inference of quantities such as linear forms, entries, and factors in matrix completion. For further reading, see \cite{bai2021matrix, chernozhukov2023inference, cahan2023factor, choi2024matrix}, among others. These works commonly utilize low-rank projection through singular value decomposition (SVD). While the properties of SVD are well understood in the matrix case, generalizing them to higher-order tensors remains challenging due to the intricate Tucker low-rank structures and the increased dimensionality resulting from tensor unfolding. 
 

Our proposed debiasing-based method shares the same spirit as the classic debiasing approaches in Lasso regression \citep{zhang2014confidence, javanmard2014confidence, sara2014asymptotically}. Although this concept has been extended to matrix inference problems (e.g., matrix regression \citep{cai2016geometric, xia2019confidence} and matrix completion \citep{carpentier2018adaptive, chen2019inference, xia2021statistical}), the complex low-rank structure of higher-order tensors introduces unique challenges related to the bias and variance trade-off that are distinct from those encountered in matrix cases. Furthermore, debiasing approaches often require sample splitting to ensure independence between initialization and bias-correction steps. The debiasing methods proposed in \citep{carpentier2018adaptive, xia2019confidence, xia2021statistical} rely on such sample splitting, which may lead to potential power loss. In the context of tensor completion, it remains largely unknown whether sample splitting can be avoided.


\section{Noisy Tensor Completion with Tucker Decomposition}
\subsection{Preliminaries and notations}

 Suppose we have a $m$-th order tensor $\bcT\in \R^{d_1\times\cdots\times d_m}$.  Let $\cM_j(\bcT)$ be the mode-$j$ \textit{unfolding} (or \textit{matricization}) of $\bcT$ that rearrange all the  mode-$j$ fibers of $\bcT$ columnwisely into a matrix with size $\cM_j(\bcT)\in \R^{d_j\times d_{-j}}$, where $d_{-j}:=d^{\ast}/d_j$. For example, the mode-1 unfolding for $m=3$ is: $\left[\cM_1(\bcT)\right]_{i_1,\left(i_2-1\right) d_3+i_3}=[\bcT]_{i_1, i_2, i_3}$ for $\forall i_j \in\left[d_j\right]$. The multilinear rank of a tensor is $\br:= \left(\rank\left(\mathcal{M}_1(\mathcal{T})\right), \cdots, \rank\left(\mathcal{M}_m(\mathcal{T})\right)\right)$. Denoting $r_j=\mathcal{M}_j(\mathcal{T})$, $\forall j\in[m]$, we assume the low-rankness: $r_j\ll d_j$ for each mode $j\in[m]$. It is clear that any tensor with multilinear rank $\br$ admits a Tucker decomposition:
\begin{equation*}
\bcT=\bcC\cdot\left(\bU_1,\cdots,\bU_m\right):=\bcC\times_1\bU_1\times_2\cdots\times_m \bU_m,
\end{equation*}
where $\bcC\in \R^{r_1\times\cdots\times r_m}$ is the core tensor and the orthogonal matrices $\{\bU_j\in \O^{d_j\times r_j}\}_{j=1}^m$ represent the singular subspaces. For a more detailed introduction to the Tucker decomposition and tensor marginal product $\times_j$, please refer to \cite{kolda2009tensor}.

We formulate the noisy tensor completion problem in a trace regression model. Suppose that the $n$ observations $\left\{ (\bcX_i, Y_i) \right\}_{i=1}^{n}$ satisfy
\begin{equation}\label{eq:trace-reg}
    Y_i = \left\langle \bcT,\bcX_i \right\rangle +\xi_i,
\end{equation}
where we assume that each $\bcX_i$ is uniformly distributed from the canonical basis $\bbH:=\left\{\be_{1,k_1}\circ\cdots\circ\be_{m,k_m}: \be_{j,k_j}\subset \R^{d_j}, k_j\in[d_j] \right\}$ in $\R^{d_1\times\cdots\times d_m}$. Assume that noise is centered i.i.d with a variance $\E \xi_i^2 =\sigma^2$ and sub-Gaussian tail $\mathbb{E} \exp\left(C \xi_i^2 / \sigma^2\right) \leq 2$ for some constant $C>0$.  We first focus on the case of homogeneous variances and will study heteroskedastic and sub-exponential noise in Section~\ref{sec:sub-exponential}. Our goal is to construct a test statistic with an asymptotically normal distribution that can infer the value of a linear form $ \left\langle\bcT,\bcI \right\rangle$ for a given tensor $\bcI$. 


Some notations will be used throughout the paper. Let $d^*:=d_1 d_2\cdots d_m$ be the total number of entries and denote $d_{-j}: = d^*/d_j$. The largest dimension is $\dmax: = \max_{j\in[m]}\{d_j\}$, and the smallest dimension $\dmin := \min_{j\in[m]}\{d_j\}$. Define $r^*$, $r_{-j}$, $\rmax$, and $\rmin$ similarly. Given the low-rank structure of $\bcT$, the smallest singular value is given by $\lambda_{\min}: =\min_{j\in[m]}\lambda_{r_j}(\cM_j(\bcT))$, with the largest singular value $\lambda_{\max}$ correspondingly defined. The condition number of the tensor $\bcT$ is given by  $\kappa(\bcT): = \lambda_{\max} /\lambda_{\min}$, where we assume that $\kappa(\bcT)\le\kappa_0$ throughout the paper. Without loss of generality, we assume $n\le \dmax^{4m}$ for simplicity. This is a mild assumption since the most interesting scenario typically involves $n\ll d^{\ast}$. 

Let $\norm{\cdot}_{\ell_p}$ denote the $\ell_p$ norm of a vector or vectorized tensor, i.e., $\norm{\bcA}_{\ell_p} = \norm{\Vect(\bcA)}_{p}$. Specifically,  $\norm{\bcA}_{\ell_2}=\norm{\bcA}_{\tF}$ represents the Frobenius norm of $\bcA$, and we will use $\norm{\cdot}_{\tF}$ instead; $\norm{\bcA}_{\linf}=\displaystyle\max_{\omega\in [d_1]\times \cdots \times[d_m] } \abs{\left[\bcA\right]_{\omega} }$ denotes the largest absolute value among the entries of tensor $\bcA$.  We use $\left[\bcA\right]_{\omega}$ to represent the $\omega$-th entry of $\bcA$. With a slight abuse of notation, let bold face $\bomega$ denote the corresponding indexing tensor such that $\langle\bcT,\bomega\rangle = \left[\bcA\right]_{\omega}$. Moreover, the $\norm{\cdot}_{p}$-norm is understood as the matrix operator norm induced by the $\ell_p$ vector norm.
For a matrix $\bA$,  the largest
row-wise $\ell_2$ norm is defined as $\norm{\bA}_{2,\infty}=\max_i\norm{\bA_{i\cdot}}_2$.  Given any column-wise orthogonal matrix $\bU\in \O^{d\times r}$, the orthogonal projection matrix is represented by $\cP_{\bU}=\bU \bU^\top$, with its orthogonal complement $\bU_{\perp}\in \O^{d\times (d-r)}$ and the corresponding projection $\cP_{\bU}^\perp=\bU_{\perp} \bU^\top_{\perp}$. The Kronecker product between matrices $\bA$ and $\bB$ is denoted by $\bA\otimes \bB$, and we use the symbol ``$\circ$'' to represent the vector outer product. We write the Hadamard (entrywise) product between two matrices $\bA$ and $\bB$ as $\bA\odot\bB$. 




\subsection{Tensor manifold and tangent space}  
The set of tensors with bounded multilinear rank $\br$ forms a smooth Riemannian manifold $\cM_{\br}$ embedded in $\R^{d_1\times\cdots\times d_m}$, where $\cM_{\br}:=\left\{\bcW \in\R^{d_1\times\cdots\times d_m}, \rank(\cM_j(\bcW_j))\le r_j, \ \forall j\in [m]  \right\}$. The degree of freedom of $\cM_{\br}$ is given by $\mathsf{Dof}:=r^*+\sum_{j=1}^m r_j d_j -r_j^2$, which is much smaller than the ambient dimension $d^*$.  We can view low-rank tensor estimation as manifold learning on $\cM_{\br}$. 
Unlike the ``flat'' Euclidean space, the manifold $\cM_{\br}$ is curved, making it challenging to trace the statistical properties of estimators restricted to $\cM_{\br}$. We shall endow $\cM_{\br}$ with Euclidean (Frobenius) metric defined by the inner product $\langle\cdot,\cdot\rangle$. It is well-known that the property of tangent space plays a critical role in manifold learning. 
The tangent space of $\cM_{\br}$ at the point $\bcT$, denoted by $\TT$, can be explicitly written as  \citep{koch2010dynamical,uschmajew2013geometry}
\begin{equation}\label{eq:true-TT}
    \TT:=\big\{\bcD\times_{j=1}^{m} \bU_{ j}+\sum\nolimits_{j=1}^m  \bcC\times_{k\neq i} \bU_{k}\times_{j}\bW_{j}: \bcD\in\R^{\br}, \bW_j\in\R^{d_j\times r_j}, \bW_j^{\top}\bU_{j}={\bf 0} \big\}.
\end{equation}
Moreover, for any tensor  $\bcI\in \R^{d_1\times\cdots\times d_m}$, the projection of $\bcI$ onto this tangent space is 
\begin{equation}\label{eq:true-TT-proj}
    \cP_{\TT}(\bcI):=\argmin_{\bcX\in \TT} \|\bcI-\bcX\|_{\rm F}^2 = \bcI \cdot \left( \cP_{{\bU}_{1}},\dots,\cP_{{\bU}_{m}}\right) + \sum_{j=1}^{m} \bcC \cdot\left(\bU_1,\dots, \bW_j, \dots, \bU_m\right), 
\end{equation}
where each $\bW_j$ is given by:
\begin{equation*}
\bW_j= \cP_{{\bU}_{j}}^{\perp} \mathcal{M}_j\left(\bcI\right)\left(\otimes_{k \neq j} {\bU}_{ k}\right) \mathcal{M}_j^{\dagger}\left({\bcC }\right),
\end{equation*}
and the notation $\bA^{\dagger}$ stands for the pseudo-inverse of matrix $\bA$. 
The equations \eqref{eq:true-TT} and \eqref{eq:true-TT-proj} are critical for characterizing our inference task, as the tangent space provides good first-order approximations of local perturbations on the low-rank tensor manifold \citep{absil2008optimization}. This approximation makes it possible to precisely characterize the randomness of low-rank variables around $\bcT$. For an introduction to the tangent space of low-rank tensor manifolds, see \cite{uschmajew2013geometry, kressner2014low}, among others.

\subsection{Assumptions}
The following assumptions are typical in tensor completion and matrix inference problems. 
\begin{Assumption}[Incoherence]\label{asm:incoherence}
There exists a constant $\mu>0$, usually called the incoherence constant, such that:
    \begin{equation}\label{eq:incoherence}
    \left\|\bU_i\right\|_{2, \infty} \leq \sqrt{\frac{\mu  r_i }{d_i}}, \quad \forall i\in [m].
\end{equation}
\end{Assumption}

\begin{Assumption}[Alignment condition]\label{asm:alignment}
There exists a constant $\alpha_I>0$, usually called the alignment parameter, such that 
\begin{equation*}
    \norm{ \cP_{\TT}(\bcI) }_\tF \ge \alpha_I \cdot \dmax^{\frac{1}{2}}\left(d^*\right)^{-\frac{1}{2}}\norm{ \bcI }_\tF
\end{equation*}   
\end{Assumption}
The incoherence condition is standard in matrix and tensor completion problems; without it, the problems may become ill-posed. Assumption \ref{asm:alignment} generalizes the alignment requirement from matrix completion \citep{chen2019inference, xia2021statistical, farias2022uncertainty}, and its order matches that of the previous entry-wise tensor inference study \citep{cai2022uncertainty} in the case $m=3$. This condition is much weaker than that in the recent work \citep{wen2023online}, which has an order of $\sqrt{d}$ for $m=3$. In an incoherent tensor, the alignment condition typically regulates $\alpha_{I}$ at the $O(1)$ level, since when taking $\bcI=\bomega$ as a sparse tensor with only one non-zero entry, we already have $ \norm{ \cP_{\TT}(\bomega) }_\tF\lesssim \dmax^{\frac{1}{2}}\left(d^*\right)^{-\frac{1}{2}}\norm{\bomega}_{\tF}$, given $m,\mu,r^*=O(1)$. In our subsequent discussion, we allow $\alpha_I\to 0$ as long as the SNR is sufficiently large.  Unlike existing works, which usually confine $m$ to $2$ or $3$, our theory accommodates any mode $m$ as well as unbalanced dimensions and ranks, provided that $m\ll \dmax $ and $ r_{-j}\lesssim d_j$.

\section{Cramér–Rao Lower Bound for Linear Form Inference}
Before developing an inference method for $\langle\bcT,\bcI\rangle$, it is important to understand the best possible accuracy, in terms of variance, that any unbiased estimator can achieve when estimating it. In this section, we develop the Cramér–Rao (CR) type lower bound on the estimation of $\langle\bcT,\bcI\rangle$, using the intrinsic CR lower bound commonly studied in information geometry \citep{smith2005covariance}. For simplicity, we assume that the rank $r_j=1$, for all $j\in [m]$. As suggested in \cite{xia2021statistically}, an estimate of $\bcT$ is said to be statistically optimal w.r.t. Frobenius norm if it falls within the rate-optimal ball 
$$
\bbB_{\opt}:=\left\{\widetilde{\bcT}: \norm{\widetilde{\bcT}-\bcT}_{\tF}^2\le C\sigma^2\cdot \frac{d^*\dmax\log\dmax}{n}  \right\},
$$
with high probability. The following theorem presents the optimal uncertainty quantification for all rate-optimal estimators. 

\begin{Theorem}[Lower bound on the variance of unbiased estimators]\label{thm:opt-uct-qtf} Suppose Assumptions \ref{asm:incoherence} and \ref{asm:alignment} hold,  the tensor $\bcT$ is rank one, and assume that the noise follows $\xi_i \sim \cN(0, \sigma^2)$. Let $\widetilde{m}: = \max\big\{m, \log(n \vee \frac{\lambda_{\max}}{\sigma}) / \log\dmax\big\}$. Let $\Bar{\bcT} \in \cM_{\br}$ be a rate-optimal estimator satisfying $\bbP(\Bar{\bcT} \in \bbB_{\opt}) \ge 1 - \dmax^{-C_{\gap}\widetilde{m}}$ that is unbiased within $\bbB_{\opt}$, i.e., $\E\Bar{\bcT}\idc\{\Bar{\bcT} \in \bbB_{\opt}\} = \bcT$, for some absolute constant $C_{\gap}>0$. Let $g_{\bcI}(\Bar{\bcT}) = \langle\Bar{\bcT}, \bcI\rangle$ be an estimator of the linear form $g_{\bcI}(\bcT) = \langle\bcT, \bcI\rangle$. If the SNR satisfies
\begin{equation}\label{eq:opt-var-snr}
    \varepsilon_{\SNR} := C\frac{ \sigma }{\lambda_{\min}}\sqrt{\left(\frac{\norm{\bcI}_{\ell_1}}{\alpha_I\norm{\bcI}_{\tF}}\vee 1 \right)\frac{ \mu^{\frac{3m}{2}}\dmax d^*\log\dmax  }{ n}} \longrightarrow 0,
\end{equation}
for some absolute constant $C>0$,  then 
\begin{equation*}
     \E\left(g_{\bcI}(\Bar{\bcT}) - g_{\bcI}({\bcT})\right)^2 \ge \left(1- \varepsilon_{\SNR}^2\right)\frac{\sigma^2 d^*}{n} \norm{\cP_\TT (\bcI)}_{\tF}^2.
\end{equation*}
\end{Theorem}

We note that such a lower bound can be extended to general low ranks $\br\geq 1$ by further investigating the curvature of the low-rank tensor manifold. Theorem \ref{thm:opt-uct-qtf} reveals an intriguing finding: the inverse Fisher information of the estimator $g_{\bcI}(\bcT)$ is asymptotically $(\sigma^2 d^*/n) \norm{\cP_{\TT} (\bcI)}_{\tF}^2$ (as SNR goes to infinity), which is proportional to the norm of the projection of $\bcI$ onto the tangent space $\TT$. This universal form holds for any tensor $\bcI$ satisfying Assumption~\ref{asm:alignment} and remains valid for higher-order tensors. However, this also raises the question: 
\begin{center}
    \textit{Is the Cramér–Rao lower bound attainable? \\
    If so, can we provide a statistical estimator that achieves the optimal accuracy?}
\end{center}

To address these questions and gain insight into low-rank manifold inference, we examine the linear function $g_{\bcI}(\bcT)=\langle\bcT,\bcI\rangle$ defined on the low-rank tensor manifold $\cM_{\br}$. Clearly, $g_{\bcI}$ has the ordinary gradient $\nabla g_{\bcI}(\bcT) = \bcI$ at $\bcT$, and the tangential gradient $\cP_{\TT}(\nabla g_{\bcI}(\bcT))$ on $\cM_{\br}$. Let us consider a small perturbation around $\bcT$ in $\cM_{\br}$: $\widehat{\bcT} =\bcT+\bcE $. The first-order approximation for small enough $\bcE$  \citep{smith1994optimization,absil2008optimization,tu2017differential} gives:
\begin{equation}\label{eq:riem-taylor-approx}
    \begin{aligned}
        g_{\bcI}(\widehat{\bcT}) - g_{\bcI}(\bcT) & = \langle\cP_{\TT}(\nabla g_{\bcI}(\bcT)),\bcE \rangle + \mathsf{remainder \ term} \\
        & \approx \langle\cP_{\TT}(\bcI),\bcE \rangle.
    \end{aligned}
\end{equation}
Using the first-order approximation in \eqref{eq:riem-taylor-approx}, if the initializer is sufficiently close to $\bcT$, the following debiasing method produces an estimator centered at $\bcT$ with random perturbations: 
\begin{equation}\label{eq:ubs-approx-decomp}
    \begin{aligned}
        \widehat \bcT_{\ubs} & =  \widehat \bcT_{\init } +\frac{d^*}{n} \sum_{i=1}^{n}\left(Y_i - \left\langle \widehat\bcT_{\init }, \bcX_i \right\rangle\right)\cdot \bcX_i \\
        & = \bcT + \underbrace{\widehat \bcT_{\init}- \bcT - \frac{d^*}{n} \sum_{i=1}^{n} \left\langle\widehat \bcT_{\init}- \bcT ,\bcX_i \right\rangle\bcX_i}_{\bcE_{\init}, \text{ initialization error} } + \underbrace{\frac{d^*}{n} \sum_{i=1}^{n} \xi_i\bcX_i}_{ \bcE_{\rn}, \text{ i.i.d. noises}} \\
        & \approx \bcT + \bcE_{\rn}.
    \end{aligned}
\end{equation}
The initialization error term, $\bcE_{\init}$, is controlled provided that the entry-wise error $\norm{\widehat\bcT_{\init}- \bcT }_{\linf}$ is sufficiently small. The noise term, $\bcE_{\rn}$, consists of a sum of i.i.d. noise and is the main contributor to variance. If $\widehat \bcT_{\ubs}$ lies in $\cM_{\br}$, we can substitute equation \eqref{eq:ubs-approx-decomp} into \eqref{eq:riem-taylor-approx} by setting  $\widehat \bcT=\widehat \bcT_{\ubs}$, leading to the following asymptotic normality:
\begin{equation}\label{eq:linear-approx-clt}
     g_{\bcI}(\widehat{\bcT}) - g_{\bcI}(\bcT)\stackrel{\eqref{eq:riem-taylor-approx},\eqref{eq:ubs-approx-decomp}}{\approx}  \langle\cP_{\TT}(\bcI),\bcE_{\rn} \rangle = \frac{d^*}{n} \sum_{i=1}^{n} \xi_i\langle\cP_{\TT}(\bcI), \bcX_i\rangle \xrightarrow{\text{CLT}} \cN\Big(0,\frac{d^*}{n}\sigma^2\norm{\cP_{\TT}(\bcI)}_{\tF}^2 \Big).
\end{equation}
However, this approach fails because $\widehat \bcT_{\ubs}$ varies within $\R^{d_1\times\cdots\times d_m}$ with high variance contributed by $\bcE_{\rn}$, making the approximation \eqref{eq:riem-taylor-approx} invalid. Nonetheless, \eqref{eq:linear-approx-clt} remains insightful, as its variance achieves the  Cramér–Rao lower bound in Theorem~\ref{thm:opt-uct-qtf}. To leverage \eqref{eq:linear-approx-clt} for inference, the following conditions must be met:
\begin{enumerate}[(i)]
    \item Good initialization: ensure the initialization error $\bcE_{\init}$ in \eqref{eq:ubs-approx-decomp} is negligible;
    \item Appropriate retraction: project $\widehat \bcT_{\ubs}$ back onto $\cM_{\br}$ to reduce the variance contributed by $\bcE_{\rn}$.
\end{enumerate}
We propose a debiasing method combined with a one-step power iteration to satisfy requirements (i) and (ii). In the next section, we detail an initialization condition that ensures (i) is met. Additionally, we show that the one-step power iteration effectively addresses (ii) by reducing the variance of i.i.d. perturbations and enabling their statistical characterization through a careful study of the spectral decomposition \citep{xia2019confidence}. 

\begin{Remark}
    
Equation \eqref{eq:linear-approx-clt} employs a first-order approximation similar to the classic delta method \citep{kelley1928crossroads,rao1973linear}. However, unlike the delta method, our approach must handle noise around the manifold by retracting perturbed data back onto it. The sensitivity of low-rank retractions has been extensively studied in optimization \citep{absil2012projection, vandereycken2013low, sun2019escaping, boumal2023introduction}. Despite these efforts, only approximate bounds are available, and an exact representation of perturbations after retraction remains elusive. This gap makes statistical inference on low-rank manifolds significantly more challenging. 

\end{Remark}
\section{Methodology: Inference by Debiasing and Power Iteration}\label{sec:inference-NTC}
\subsection{Tensor completion methods and warm initialization}
Our debiasing and one-step power iteration method requires a warm initialization $\widehat \bcT_{\init}$ with multilinear rank $\br$  that satisfies
\begin{equation}\label{eq:tensor-init}
\norm{\widehat \bcT_{\init} -   \bcT }_{\linf } \le C_1 \sigma \sqrt{\frac{\dmax \log \dmax }{n}},
\end{equation}
with probability at least $1-\dmax^{-3m}$. Here $C_1>0$ is a large number that may depend on $\mu$, $\br$, $\kappa_0$ and $m$. Such initialization can be attained using computationally efficient methods like vanilla gradient descent \citep{cai2019nonconvex}, Riemannian gradient descent \citep{wang2023implicit}, or online Riemannian gradient descent \citep{cai2023online}, provided that the signal-to-noise ratio and sample size are sufficiently large. Note that a minimax optimal estimator $\widehat \bcT_{\init}$ achieving the Frobenius-norm error rate $\norm{\widehat \bcT_{\init} -\bcT}_{\tF}\lesssim_p \sigma \sqrt{(\dmax d^*/n)\log\dmax}$ (\cite{xia2021statistically}) is not  always suitable for statistical inference. For instance, small perturbations around $\bcT$ can attain this error rate but fail to provide distributional guarantees for the tensor structure.

The following SNR and sample size conditions are usually required for the aforementioned computationally efficient methods:
\begin{equation}\label{eq:SNR-Comp}
    \frac{\lambda_{\min} }{\sigma} \gg \sqrt{\frac{(d^* )^{\frac{3}{2}} }{n}}\quad {\rm and}\quad  n \gg   (d^*)^{\frac{1}{2}}.
\end{equation}
These conditions are common in various tensor models \citep{zhang2018tensor, xia2022inference, han2022optimal, cai2022provable} and are believed necessary for polynomial-time algorithms \citep{barak2016noisy}.  However, the conditions in \eqref{eq:SNR-Comp} are overly stringent when focusing on the statistical perspective with unlimited computational resources. 
For example, as suggested in \cite{zhang2018tensor, xia2022inference} and the recent study \cite{cai2023online}, with sufficient oracle information, minimax optimal tensor completion only requires
\begin{equation}\label{eq:SNR-Stat}
\frac{\lambda_{\min}}{\sigma} \gg \sqrt{\frac{d^* \dmax}{n}} \quad \text{and} \quad n \gg \dmax.
\end{equation}
The statistical-to-computational gap between \eqref{eq:SNR-Comp} and \eqref{eq:SNR-Stat} motivates us to investigate whether \eqref{eq:SNR-Stat} is sufficient for statistical inference on linear forms. Focusing on the inference task, we assume that we have obtained an initial estimate  $\widehat \bcT_{\init}$ satisfying \eqref{eq:tensor-init} with high probability. We further extend our framework by allowing dependency between $\widehat \bcT_{\init}$ and the observed data to determine the necessary SNR and sample size conditions for valid inference.

\subsection{Inference of linear form with independent initialization}
We detail our debiasing and one-step power iteration method in Algorithm \ref{alg:debias-powerit}, assuming that a good initial estimator$\widehat \bcT_{\init }=\widehat \bcC_0\cdot \left( \widehat \bU_{1,0},\cdots,\widehat \bU_{m,0}\right)$ is available, which provides estimates of the singular subspaces $\widehat{\bU}_{j,0}$ and the core tensor $\widehat\bcC_0$. The debiasing technique \citep{zhang2014confidence, javanmard2014confidence, sara2014asymptotically} serves to re-randomize the initial estimator:
\begin{equation}\label{eq:debiasing}
    \widehat \bcT_{\ubs} = \widehat \bcT_{\init } +\frac{d^*}{n}\sum_{i=1}^{n}\left(Y_i - \left\langle \widehat\bcT_{\init }, \bcX_i \right\rangle\right)\cdot \bcX_i
\end{equation}

We then employ the one-step power iteration to retract the unbiased $\widehat \bcT_{\ubs}$ onto the low-rank manifold $\cM_{\br}$. This step significantly reduces the variance of the perturbation term $\bcE_{\rn}$ as indicated in \eqref{eq:ubs-approx-decomp} but introduces a new negligible bias. This produces our final estimator $\widehat \bcT$, and we infer the linear form $\left\langle  \bcT,\bcI  \right\rangle$ using $\big\langle  \widehat{\bcT},\bcI  \big\rangle$. The following theorem shows that when the initialization $\widehat \bcT_{\init }$ and the observations are independent, the proposed test statistic achieves asymptotic normality with a variance reaching the Cramér–Rao lower bound.

\begin{algorithm}
\caption{Debiasing and One-Step Power Iteration}
\label{alg:debias-powerit}
\begin{algorithmic}
\REQUIRE $\left\{ (\bcX_i, Y_i) \right\}_{i=1}^{n}$ and an initial estimator $\widehat \bcT_{\init }=\widehat{\bcC}_0 \cdot \left( \widehat \bU_{1,0},\cdots,\widehat \bU_{m,0}\right)$
\STATE{Debias the initialization by $\widehat \bcT_{\ubs} =  \widehat \bcT_{\init } +\frac{d^*}{n} \sum_{i=1}^{n}\left(Y_i - \left\langle \widehat\bcT_{\init }, \bcX_i \right\rangle\right)\cdot \bcX_i$ }
\FOR{$j = 1,\dots, m$}
	\STATE{Calculate $\widehat{\bU}_{j,1}=\SVD_{r_j}\left(\mathcal{M}_j\left(\widehat \bcT_{\ubs}\times_1 \widehat{\mathbf{U}}_{1,0}^\top \cdots \times_{j-1} \widehat{\mathbf{U}}_{j-1,0}^\top \times_{j+1} \widehat{\mathbf{U}}_{j+1,0}^\top \cdots \times_{m} \widehat{\mathbf{U}}_{m,0}^\top\right)\right) $}
\ENDFOR
\STATE{Return the final estimator $\widehat{\bcT}=\widehat \bcT_{\ubs} \times_1\cP_{\widehat{\bU}_{1,1} }\cdots \times_m\cP_{\widehat{\bU}_{m,1} } $}

\end{algorithmic}
\end{algorithm}


\begin{Theorem}[Asymptotic normality with independent initialization]\label{thm:lf-inference-popvar} Under Assumptions \ref{asm:incoherence} and \ref{asm:alignment}, suppose the initialization $\widehat \bcT_{\init }$ is independent of the observed data $\left\{ (\bcX_i, Y_i) \right\}_{i=1}^{n}$ and satisfies \eqref{eq:tensor-init} with probability at least $1-\dmax^{-3m}$. Moreover, assume that the sample size $n\ge C_{\gap} C_1^2  m^2 (2\mu)^{m-1} r^* \dmax \log^2\dmax $ and SNR condition:
\begin{equation*}
    \frac{\lambda_{\min}}{\sigma} \ge C_{\gap} C_1^2 \frac{\kappa_0\norm{\bcI}_{\ell_1}}{\alpha_I\norm{\bcI}_\tF } \sqrt{\frac{ m^4 (2\mu)^{3 m} (r^*)^{3 }\dmax d^*  \log^2 \dmax }{\rmin^2 n } },
\end{equation*}
for a large numerical constant $C_{\gap}>0$. Then the test statistic $\langle \widehat \bcT, \bcI\rangle$ using the estimator $\widehat \bcT$ returned by Algorithm \ref{alg:debias-powerit} satisfies
\begin{equation*}
\begin{aligned}
        &\max_{t\in\R}\abs{\bbP\left( \frac{\left\langle  \widehat{\bcT} - \bcT,\bcI  \right\rangle }{ \sigma\norm{\cP_{ \TT }(\bcI)}_\tF \sqrt{{d^*}/{n}} } \le t\right)- \Phi(t)}  \\
        &\le C C_1 \sqrt{\frac{ m^2 (2\mu)^{m-1} r^* \dmax \log^2\dmax }{ \rmax \rmin n}}  + C \frac{ C_1^2\sigma  }{ \lambda_{\min}} \frac{\norm{\bcI}_{\ell_1}}{\alpha_I\norm{\bcI}_\tF }   \sqrt{\frac{ m^4 (2\mu)^{3 m} (r^*)^{3 }\dmax d^*  \log^2 \dmax }{\rmin^2 n } },
\end{aligned}
\end{equation*}
where $C>0$ is an absolute constant. 
\end{Theorem}

Theorem \ref{thm:lf-inference-popvar} shows that $\big\langle \widehat{\bcT}, \bcI \big\rangle$ follows a limiting normal distribution with mean $\left\langle \bcT, \bcI \right\rangle$ and variance $\frac{\sigma^2 d^{\ast}}{n} \norm{\cP_{\TT} (\bcI)}_{\tF}^2$ . This variance matches the Cramér–Rao lower bound established in Theorem \ref{thm:opt-uct-qtf}, indicating optimality in uncertainty quantification. Additionally, the ratio $\norm{\bcI}_{\ell_1}/\norm{\bcI}_{\tF}$ is bounded by the sparsity level of $\bcI$. Therefore, for a sparse $\bcI$, the SNR and sample size conditions required in Theorem \ref{thm:lf-inference-popvar} nearly match the statistically optimal conditions in \eqref{eq:SNR-Stat} (up to factors of $r^{\ast}$, $m$, $\mu$, $\kappa_0$, $\alpha_I$, and logarithmic terms). This demonstrates that, given a sufficiently accurate initial estimate, statistical inference can be performed under only statistically optimal conditions.

The distributional characterization of $\big\langle \widehat{\bcT}, \bcI \big\rangle$ relies on a fine-grained analysis of the tensor power iteration procedure. Specifically, we employ a spectral representation tool from \cite{xia2019confidence, xia2021normal} to achieve a second-order accurate characterization of matrix SVD. This decomposition allows us to precisely characterize the behavior of the leading term during power iteration and provides a sharp $\ell_{2,\infty}$-norm bound to control higher-order terms. Furthermore, Theorem~\ref{thm:lf-inference-popvar} shows that while a $\kappa_0$-dependent SNR condition is necessary to validate the spectral analysis, the variance in the proposed test statistic is, in fact, $\kappa_0$-free (although the constant $C_1$ appearing in the Berry-Esseen bound may still be $\kappa_0$-dependent).

An empirical estimate of the variance term $\sigma\norm{\cP_{ \TT }(\bcI)}_\tF$ is needed to facilitate statistical inference in practice. By leveraging the warm initialization  $\widehat \bcT_{\init }=\widehat \bcC_0\cdot \left( \widehat \bU_{1,0},\cdots,\widehat \bU_{m,0}\right) $, we provide the following estimate of the projection onto the tangent space of $\widehat \bcT$, denoted by $\cP_{\widehat \TT }(\bcI)$:
\begin{equation}\label{eq:emp-PT}
\begin{gathered}
        \cP_{\widehat \TT }(\bcI) = \bcI \cdot \left( \cP_{\widehat{\bU}_{1,0}},\dots,\cP_{\widehat{\bU}_{m,0}}\right) + \sum_{j=1}^{m} \widehat{\bcC}_0 \cdot\left(\widehat{\bU}_{1,0},\dots, \widehat{\bW}_{j}, \dots, \widehat{\bU}_{m,0}\right), \\
        \text{ with   } \widehat \bW_j=\cP_{\widehat {\bU}_{j,0}}^{\perp} \mathcal{M}_j\left(\bcI\right)\left(\otimes_{k \neq j} \widehat{\bU}_{ k,0}\right) \mathcal{M}_j^{\dagger}(\widehat {\bcC }_0).
\end{gathered}
\end{equation}
Equation \eqref{eq:emp-PT} is derived by replacing the tangent space at $\bcT$ in \eqref{eq:true-TT-proj} with the tangent space at $\bcT_{\init}$. Additionally, we estimate the noise variance by 
\begin{equation}\label{eq:emp-sigma}
\begin{gathered}
        \widehat{\sigma}^2: = \frac{1}{n}\sum_{i=1}^n\left(Y_i-\left\langle \widehat \bcT_{\init},\bcX_i \right\rangle\right)^2. \\
\end{gathered}
\end{equation}

\begin{Theorem}[Inference with independent initialization]\label{thm:lf-inference-empvar} Under Assumptions \ref{asm:incoherence} and \ref{asm:alignment}, suppose the initialization $\widehat \bcT_{\init }$ is independent of the observed data $\left\{ (\bcX_i, Y_i) \right\}_{i=1}^{n}$ and satisfies \eqref{eq:tensor-init} with probability at least $1-\dmax^{-3m}$. Moreover, assume that the sample size $n\ge C_{\gap} C_1^2  m^2 (2\mu)^{m-1} r^* \dmax \log^2\dmax $ and SNR condition:
\begin{equation*}
    \frac{\lambda_{\min}}{\sigma} \ge C_{\gap} C_1^2 \frac{\kappa_0\norm{\bcI}_{\ell_1}}{\alpha_I\norm{\bcI}_\tF } \sqrt{\frac{ m^5 (2\mu)^{3 m} (r^*)^{3 }\dmax d^*  \log^2 \dmax }{\rmin^2 n } }
\end{equation*}
for a large numerical constant $ C_{\gap}>0$,  Then the studentized statistic
    \begin{equation}\label{eq:hWtest-def}
        \widehat W_{\test}({\bcI}):=\frac{ \widehat \bcT (\bcI) - \bcT(\bcI)  }{ \widehat{\sigma} \norm{\cP_{\widehat \TT }(\bcI)  }_{\tF} \sqrt{d^*/n} } \longrightarrow \cN (0,1),\quad \text{ as } n\to\infty. 
    \end{equation}
Moreover, there exists an absolute constant $C>0$ such that 
    \begin{equation*}
\begin{aligned}
             &\max_{t\in\R}\abs{\bbP\left( \widehat W_{\test}(\bcI)\le t\right)- \Phi(t)} \\
         &\le  C C_1 \sqrt{\frac{m^2 \mu^{m} \rmin r^* \dmax \log^2 \dmax }{n}} + C\frac{\norm{\bcI}_{\ell_1}}{\norm{\bcI}_{\tF}}\frac{C_1^2  \kappa_0 \sigma  }{\alpha_I \lambda_{\min}}\sqrt{\frac{m^5 (2\mu)^{3m} (r^*)^3 d^*\dmax \log^2 \dmax }{\rmin^2 n}}.
\end{aligned}
    \end{equation*}
\end{Theorem}

Although tensor linear form inference has been explored in \cite{huang2022power, xia2022inference, wen2023online} under tensor PCA and tensor regression models, their distributional characterizations rely heavily on the rotation invariance of Gaussian noise/measurements, which does not apply to tensor completion. 
Our method combines debiasing with a one-step power iteration, closely related to projected gradient descent approaches in tensor completion \citep{kressner2014low, steinlechner2016riemannian, chen2019non, cai2023generalized}. Specifically, the debiasing step in \eqref{eq:debiasing} can be viewed as a single gradient descent step using the least-squares loss function. Unlike the traditional projected gradient descent methods that use HOSVD for retraction with small step sizes, we employ a one-step power iteration with a step size of $d^{\ast}/n$. This step size is more effective for bias correction (as also seen in matrix and tensor completion \citep{ma2018implicit, ding2020leave, wang2023implicit}), and the power iteration better leverages singular subspace information from the initial estimate to further reduce error from random sampling.



\begin{Remark}
    Note that the debiasing step in Algorithm~\ref{alg:debias-powerit} corrects error caused by incomplete observations. When all observations are available, debiasing becomes unnecessary since $\widehat{\bcT}_{\ubs} = \sum_{i=1}^{n} Y_i \bcX_i =: \widehat{\bcT}_{\mathsf{obv}}$ in Algorithm \ref{alg:debias-powerit} simply equals the observation tensor. In this case, the algorithm reduces to a one-step power iteration, improving upon the previous two-step procedure in fully observed tensor PCA \citep{xia2022inference}. 
\end{Remark}

Theorem \ref{thm:lf-inference-empvar} provides a practical procedure for constructing confidence intervals for tensor linear forms.

\begin{Corollary}[Confidence interval]\label{coro:CIs}
	Under the same conditions as Theorem \ref{thm:lf-inference-empvar}, the two-sided $1-\alpha$ confidence interval of $\langle \bcT, \bcI\rangle$ is
 \begin{equation*}
     \widehat{\mathrm{CI}}_\alpha(\bcI)=\left[\langle  \widehat{\bcT} ,\bcI  \rangle- z_{\alpha/2} \widehat{\sigma} \norm{\cP_{\widehat \TT }(\bcI)  }_{\tF} \sqrt{\frac{d^*}{n}} , \  \langle  \widehat{\bcT} ,\bcI  \rangle+ z_{\alpha/2} \widehat{\sigma} \norm{\cP_{\widehat \TT }(\bcI)  }_{\tF} \sqrt{\frac{d^*}{n}}\right].
 \end{equation*}
The coverage of the above confidence interval is guaranteed by
 \begin{equation*}
     \abs{\bbP\left(\left\langle\bcT,\bcI\right\rangle\in \widehat{\mathrm{CI}}_\alpha(\bcI) \right)-(1-\alpha)} \le  CC_1 \sqrt{\frac{m^2 \mu^{m} \rmin r^* \dmax \log^2 \dmax }{n}} + C\frac{\norm{\bcI}_{\ell_1}}{\norm{\bcI}_{\tF}}\frac{C_1^2  \kappa_0 \sigma  }{\alpha_I \lambda_{\min}}\sqrt{\frac{m^5 (2\mu)^{3m} (r^*)^3 d^*\dmax \log^2 \dmax }{\rmin^2 n}}. 
 \end{equation*}
\end{Corollary}
Corollary \ref{coro:CIs} indicates that this confidence interval is sufficiently precise for many statistical applications. Theorem \ref{thm:opt-uct-qtf} demonstrates that its length attains the Cramér–Rao lower bound asymptotically, making it unimprovable in general. 

Our theoretical results benefit from the warm initialization in \eqref{eq:tensor-init}, which places the starting point in a region where the $\ell_{2,\infty}$-norm perturbations of the singular subspaces are well-controlled. It also ensures that the initial spectral estimates $\widehat{\bU}_{1,0}, \dots, \widehat{\bU}_{m,0}$ are incoherent. However, these theories require strong independence between the initial estimator and the data used for debiasing or re-randomization. Although sample splitting techniques (e.g., double-sample-splitting \citep{chernozhukov2018double, xia2021statistical}) can achieve this independence, they may reduce statistical power. In the following section, we demonstrate that our method remains valid for linear form inference, even when the initialization is derived from the same observations used for debiasing.

\subsection{Inference of linear form with dependent initialization}

We extend our framework to the case of dependent initialization, meaning that it is obtained from an estimation procedure based on the same observations $\left\{ (\bcX_i, Y_i) \right\}_{i=1}^{n}$ used for debiasing. We introduce a special type of initialization, referred to as the leave-one-out initialization. This approach is inspired by the gradient-descent-based tensor completion methods, which were analyzed using the sophisticated leave-one-out techniques \citep{javanmard2014confidence, ma2018implicit, cai2019nonconvex}.





\begin{Assumption}[Leave-one-out initialization]\label{asm:lol-init}
    We say $\widehat \bcT_{\init }=\widehat \bcC_0\cdot \big( \widehat \bU_{1,0},\cdots,\widehat \bU_{m,0}\big)$ is a leave-one-out initialization if for all $k\in[m]$ and $l\in[d_k]$ , there exists an estimator $\widehat \bcT_{\init }^{(k,l)}=\widehat \bcC_0^{(k,l)}\cdot \big( \widehat \bU_{1,0}^{(k,l)},\cdots,\widehat \bU_{m,0}^{(k,l)}\big)$ that is independent of all the  observations from slice $l$ on mode $k$ and a number $C_2>0$ such that the following bounds hold with probability at least $1-\dmax^{-3m}$,
\begin{equation}\label{eq:asp-loo-init}
    \norm{\widehat \bcT_{\init}^{(k,l)} -   \bcT }_{\linf } \le C_2 \sigma \sqrt{\frac{\dmax \log \dmax }{n}} \quad  {\rm and}\quad   \norm{ \cP_{\widehat{\bU}_{j,0}^{(k,l)}} - \cP_{\widehat{\bU}_{j,0}} }_{\tF} \le C_2\frac{\sigma}{\lambda_{\min}}\sqrt{\frac{\mu r_k}{d_k}} \sqrt{   \frac{ d^* \dmax \log \dmax }{ n }}, 
\end{equation}
for all $j\in[m]$. 
\end{Assumption} 

The leave-one-out initialization can be achieved using iterative methods such as gradient descent \citep{cai2019nonconvex, cai2022uncertainty} and Riemannian gradient descent \citep{wang2023implicit}. These algorithms compute $\widehat{\bcT}_{\init}$ using all observations, while the leave-one-out series $\widehat{\bcT}_{\init}^{(k,l)}$ are \emph{virtually constructed}, for the purpose of theoretical analysis, by excluding all observations from slice $l \in [d_k]$ on mode $k \in [m]$. The next theorem shows that, if the initial estimator $\widehat{\bcT}_{\init}$ satisfies \eqref{eq:asp-loo-init}, we can still infer the tensor linear form under nearly statistically optimal sample size and SNR conditions. Therefore, sample splitting becomes unnecessary.

\begin{Theorem}[Inference with a leave-one-out initialization]\label{thm:clt-lol-init}

Denote the dimension imbalance ratio $\alpha_{d}=\dmax/\dmin$. Under Assumptions \ref{asm:incoherence}, \ref{asm:alignment}, and \ref{asm:lol-init}, suppose the initialization $\widehat \bcT_{\init }=\widehat \bcC_0\cdot \big( \widehat \bU_{1,0},\cdots,\widehat \bU_{m,0}\big) $ satisfies \eqref{eq:tensor-init} with probability at least $1-\dmax^{-3m}$. Moreover, assume that  the sample size $n\ge C_{\gap} C_2^2 \kappa_0^2 \frac{\norm{\bcI}_{\ell_1}^2}{\alpha_I^2\norm{\bcI}_{\tF}^2} m^3 (2\mu)^{3m-1} (r^*)^3 \dmax \log^2\dmax $ and SNR condition:
\begin{equation*}
    \frac{\lambda_{\min}}{\sigma} \ge C_{\gap} C_2^2 \frac{\kappa_0\norm{\bcI}_{\ell_1}}{\alpha_I\norm{\bcI}_\tF } \sqrt{\frac{ \alpha_{d} m^5 (2\mu)^{3 m} (r^*)^{3 }\dmax d^*  \log^2 \dmax }{ n } }
\end{equation*}
for a large numerical constant $C_{\gap}>0$. Then the test statistic defined in \eqref{eq:hWtest-def}  using $\widehat \bcT$ returned by Algorithm \ref{alg:debias-powerit} satisfies:
\begin{equation*}
\begin{aligned}
        &\max_{t\in\R}\abs{\bbP\left( \widehat W_{\test}(\bcI)\le t\right)- \Phi(t)}  \\
        &\le  C C_2 \kappa_0  \frac{\norm{\bcI}_{\ell_1}}{\alpha_I \norm{\bcI}_{\tF}} \sqrt{\frac{m^3 (2\mu)^{3m-1} (r^*)^3 \dmax \log^2\dmax  }{ n}} + C \frac{ C_2^2\sigma \kappa_0  \norm{\bcI}_{\ell_1} }{\alpha_I \lambda_{\min}  \norm{\bcI}_\tF }\sqrt{\frac{\alpha_d m^5 (2\mu)^{3 m} (r^*)^{3 }\dmax d^*  \log^2 \dmax }{ n } },
\end{aligned}
\end{equation*}
where $C>0$ is an absolute constant. 
\end{Theorem}

The rationale behind Theorem \ref{thm:clt-lol-init} is analogous to that of the debiased Lasso \citep{javanmard2014confidence, javanmard2018debiasing}, as we decouple the dependence between initialization and bias correction through a row-wise leave-one-out analysis, which avoids the requirement of sample splitting. However, the analysis in Theorem \ref{thm:clt-lol-init} is more intricate due to multiple modes and intertwined tensor factors. Furthermore, Theorem \ref{thm:clt-lol-init} extends the previous leave-one-out-based tensor inference in \cite{cai2022uncertainty}, which focused solely on entrywise inference, by enabling the inference of general linear forms under weaker conditions. The gradient descent approach in \cite{cai2022uncertainty} is primarily limited to CP tensor decomposition. In contrast, our debiasing method is more flexible, allowing us to handle  Tucker decompositions using power iteration. It is important to note that although Theorem \ref{thm:clt-lol-init} only requires statistically optimal conditions, obtaining an initialization that satisfies Assumption \ref{asm:lol-init} remains computationally demanding and necessitates stronger conditions to ensure the convergence of the iterative methods discussed above. We will discuss this gap in more detail in Section \ref{sec:s-c-gap}.

Finally, we consider the case where the initial $\widehat{\bcT}_{\init}$ is obtained by any method that satisfies \eqref{eq:tensor-init}, allowing it to depend arbitrarily on the observations $\left\{ (\bcX_i, Y_i) \right\}_{i=1}^{n}$. We can still establish asymptotic normality under stronger sample size and SNR conditions.

\begin{Theorem}[Inference with an arbitrarily dependent initialization]\label{thm:clt-dep-init}
Under Assumptions \ref{asm:incoherence} and \ref{asm:alignment}, suppose that the sample size and SNR satisfies 
\begin{equation*}
 n\ge C_{\gap} C_1^2 m^3(2\mu)^m \rmax^2 r^*\dmax^2\log^2\dmax, \quad   \frac{\lambda_{\min}}{ \sigma} \ge  C_{\gap}\kappa_0 \left(\frac{C_1 \norm{\bcI}_{\ell_1}}{\alpha_I\norm{\bcI}_{\tF}} \bigvee \sqrt{\dmax}\right)C_1\sqrt{\frac{  m^5 (2\mu)^{3m} (r^*)^3   d^* \dmax \log^2 \dmax }{ n}}, 
\end{equation*}
for a numerical constant $C_{\gap}>0$.  
Let $\widehat W_{\test}(\bcI)$ be the test statistic defined in \eqref{eq:hWtest-def}  using $\widehat \bcT$ returned by Algorithm \ref{alg:debias-powerit} with an initialization $\widehat{\bcT}_{\init}$  that is arbitrarily dependent on the observations $\left\{ (\bcX_i, Y_i) \right\}_{i=1}^{n}$. If $\widehat{\bcT}_{\init}$ satisfies \eqref{eq:tensor-init} with probability at least $1-\dmax^{-3m}$, then there exists an absolute constant $C>0$ such that 
    \begin{equation*}
\begin{aligned}
             &\max_{t\in\R}\abs{\bbP\left( \widehat W_{\test}(\bcI)\le t\right)- \Phi(t)} \le  C C_1 \sqrt{\frac{m^3(2\mu)^m \rmax^2 r^*\dmax^2\log^2\dmax }{n}} \\
             &\quad \quad + C C_1\frac{\sigma}{\lambda_{\min}}\sqrt{\frac{  m^4 \mu^{m-1} (r^*)^2   d^* \dmax^2 \log^2 \dmax }{\rmin^3 n}}   + C\frac{\norm{\bcI}_{\ell_1}}{\norm{\bcI}_{\tF}}\frac{C_1^2  \kappa_0 \sigma  }{\alpha_I \lambda_{\min}}\sqrt{\frac{m^5 (2\mu)^{3m} (r^*)^3 d^*\dmax \log^2 \dmax }{\rmin^2 n}}.
\end{aligned}
    \end{equation*}
\end{Theorem}

Theorem \ref{thm:clt-dep-init} shows that even without knowing the dependence between $\widehat{\bcT}_{\init}$ and the observations used for debiasing, the required SNR and sample size need only to be $O(\dmax^{1/2})$ times larger than the statistically optimal condition \eqref{eq:SNR-Stat}, rather than the computationally efficient condition \eqref{eq:SNR-Comp}.

\section{Correlation of test statistics for multiple linear forms}\label{sec:joint}
Simultaneous inference for multiple linear forms is useful in applications like constructing confidence intervals for several missing entries in tensor $\bcT$. To achieve this, it is essential to characterize the correlations among their test statistics. The following theorem establishes the joint distribution for two such test statistics. 

\begin{Theorem}\label{thm:clt-two-test}
Let $\bcI_1$ and $\bcI_2$ be two tensors satisfying Assumption~\ref{asm:alignment}. Define 
 \begin{equation*}
     \rho(\bcI_1,\bcI_2) = \frac{\left\langle\cP_{ \TT }(\bcI_1),\cP_{ \TT }(\bcI_2)\right\rangle}{\norm{\cP_{ \TT }(\bcI_1)}_{\tF} \norm{\cP_{ \TT }(\bcI_2)}_{\tF} },
 \end{equation*}
 and denote $\Phi_{\rho}(t_1,t_2)$ as the c.d.f. of the $2$-dimensional normal distribution with unit variance and covariance $ \rho(\bcI_1,\bcI_2) $. Suppose that the conditions in Theorem \ref{thm:lf-inference-empvar} hold. Then the joint c.d.f. of $ ( \widehat W_{\test}({\bcI_1}), \widehat W_{\test}({\bcI_2}))^{\top}$ converges to $\Phi_{\rho}$ with the rate:
\begin{equation*}
\begin{aligned}
        &\max_{t_1,t_2\in\R}\abs{\bbP\left( \widehat W_{\test}(\bcI_1)\le t_1, \widehat W_{\test}(\bcI_2)\le t_2\right)- \Phi_{\rho}(t_1,t_2)} \\
    &\le C C_1 \sqrt{\frac{m^2 \mu^{m} \rmin r^* \dmax \log^2 \dmax }{n}} + C\left(\frac{\norm{\bcI_1}_{\ell_1}}{\norm{\bcI_1}_{\tF}} \vee \frac{\norm{\bcI_2}_{\ell_1}}{\norm{\bcI_2}_{\tF}}\right)\frac{C_1^2  \kappa_0 \sigma  }{\alpha_I \lambda_{\min}}\sqrt{\frac{m^5 (2\mu)^{3m} (r^*)^3 d^*\dmax \log^2 \dmax }{\rmin^2 n}},
\end{aligned}
\end{equation*}
 for a numerical constant $C>0$.  
\end{Theorem}

Theorem~\ref{thm:clt-two-test} establishes the joint asymptotic normality of two test statistics:
 \begin{equation*}
  \left[\begin{array}{c}\widehat W_{\test}(\bcI_1) \\ \widehat W_{\test}(\bcI_2)\end{array} \right] \longrightarrow \cN\left(\mathbf{0},\left[\begin{array}{cc}
         1 & \rho(\bcI_1,\bcI_2) \\
         \rho(\bcI_1,\bcI_2) &  1
     \end{array}\right] \right).
 \end{equation*}
This result can be readily extended to multiple test statistics. Notably, the correlation between two test statistics equals the cosine similarity between $\cP_{\TT}(\bcI_1)$ and $\cP_{\TT}(\bcI_2)$, consistent with previous findings in \cite{ma2023multiple}. Furthermore, this correlation structure benefits from the incoherence and alignment conditions, as supported by the following facts. 

\begin{Proposition}\label{prop:corr-ub}
Under Assumptions \ref{asm:incoherence} and \ref{asm:alignment}, the correlation $ \rho(\bcI_1,\bcI_2)$ can be bounded by
\begin{equation}
    \abs{\rho(\bcI_1,\bcI_2)} \le \frac{2\mu^m r^* \norm{\bcI_1}_{\ell_1}\norm{\bcI_2}_{\ell_1}}{\dmax\alpha_I^2\norm{\bcI_1}_{\tF}\norm{\bcI_2}_{\tF} } + \frac{d^*\sum_{j=1}^{m}\abs{ \left\langle \cM_j(\bcI_1)^\top \cM_j(\bcI_2), \cP_{\bH_j}\right\rangle}}{\dmax\alpha_I^2\norm{\bcI_1}_{\tF}\norm{\bcI_2}_{\tF} },
\end{equation}
where each $\cP_{\bH_j}$ is the projection into the incoherent singular subspace $\cP_{\bU_m}\otimes\cdots \cP_{\bU_{j+1}}\otimes \cP_{\bU_{j-1}}\cdots\otimes \cP_{\bU_{1}} $ in the space $\R^{d_{-j}\times d_{-j}}$. 
\end{Proposition}

When $\cM_j(\bcI_1)^\top \cM_j(\bcI_2) = \mathbf{0}$ for each $j \in [m]$ (for example, if $\bcI_1$ and $\bcI_2$ correspond to two entries $\omega_1$ and $\omega_2$ with no overlapping indices in any mode), we have $\abs{\rho(\bcI_1,\bcI_2)} \lesssim 1/\dmax$, assuming that $\bcI_1$ and $\bcI_2$ are sparse and that $\mu$, $m$, $r^*$, and $1/\alpha_I$ are $O(1)$. This result may be of independent interest for exploring other simultaneous inference tasks, such as multiple testing, within our inference framework for tensor linear forms.

\section{Inference with Heteroskedastic and Sub-Exponential Noises}\label{sec:sub-exponential}
The theorems established in Sections~\ref{sec:inference-NTC} and \ref{sec:joint} assume i.i.d. and sub-Gaussian noises, which can be restrictive in many scenarios. 
In this section, we extend these results to accommodate sub-exponential and heteroskedastic noises. Specifically, we assume that the independent noise $\xi_i$'s are heteroskedastic in that their distributions may depend on the sampling positions $\bcX_i$, and that they have $\sigmax$ sub-exponential tails:
\begin{equation}
    \E \big[\exp(\abs{\xi_i} / \sigmax)\mid \left[\bcX_i\right]_{\omega} = 1\big] \leq 2, \quad \forall \omega\in [d_1]\times\cdots\times [d_m].
\end{equation}
The heteroscedasticity is characterized by the standard deviation tensor of the noise, denoted by $\bcS$, whose entries are defined by
\begin{equation*}
    [\bcS]_{\omega}: = \sqrt{\E\left[\xi_i^2\mid \left[\bcX_i\right]_{\omega} = 1  \right]}, \ \forall \omega\in [d_1]\times \cdots \times[d_m].
\end{equation*}

By the property of sub-exponential norm, we have $[\bcS]_{\omega} \leq 4 \sigmax$ for all $\omega$. Due to the large number of free parameters in $\bcS$, there is no consistent estimator for a general $\bcS$. Fortunately, by assuming that the variances of heterogeneous noises are at the same level, that is, there exists $\sigmin$ such that $\sigmin \leq [\bcS]_{\omega} \leq 4 \sigmax$ for all $\omega \in [d_1] \times \cdots \times [d_m]$, we show that our method still provides valid inference for the linear form $\left\langle \bcT, \bcI \right\rangle$. Define the level of heteroscedasticity by $\kappa_{\sigma}: = \sigmax / \sigmin$. For simplicity, we only present the inference result under independent initialization. It can be easily extended to the cases with dependent initialization.


\begin{Theorem}\label{thm:hetero-exp-clt}
 Under Assumptions \ref{asm:incoherence} and \ref{asm:alignment}, suppose that  the initialization $\widehat \bcT_{\init }=\widehat \bcC_0\cdot \left( \widehat \bU_{1,0},\cdots,\widehat \bU_{m,0}\right) $ is independent of the observed data $\left\{ (\bcX_i, Y_i) \right\}_{i=1}^{n}$ satisfying  
 \begin{equation}\label{eq:tensor-init-subexp}
     \norm{\widehat \bcT_{\init} -   \bcT }_{\linf } \le C_1 \sigmax \sqrt{\frac{\dmax \log \dmax }{n}},
 \end{equation}
 with probability at least $1-\dmax^{-3m}$, and there exist a numerical constant $ C_{\gap}>0$ such that the sample size and SNR conditions hold: 
 \begin{equation*}
     n\ge C_{\gap} C_1^2\kappa_{\sigma}^6 m^6\mu^{m} r^*\dmax\log^5\dmax, \quad \frac{\lambda_{\min}}{\sigmax} \ge C_{\gap}\frac{ C_1^2\kappa_{\sigma}^2  \norm{\bcI}_{\ell_1} }{\alpha_I  \norm{\bcI}_\tF }\sqrt{\frac{ m^5 (2\mu)^{3 m} (r^*)^{3 }\dmax d^*  \log^2 \dmax }{\rmin^2 n } }.
 \end{equation*}
Then the test statistics $ W_{\test}^{\sfh}({\bcI})$ and $\widehat W_{\test}^{\sfh}(\bcI)$ using population variance  and estimated variance, respectively, converge to $\cN (0,1)$ in distribution as $n\to\infty$: 
\begin{equation*}
    W_{\test}^{\sfh}({\bcI}):= \frac{ \left\langle \widehat\bcT, \bcI\right\rangle - \left\langle\bcT,\bcI\right\rangle  }{ \norm{\cP_{\TT }(\bcI)\odot\bcS  }_{\tF} \sqrt{d^*/n} } \stackrel{{\rm d.}}{\longrightarrow} \cN (0,1) , \quad    \widehat W_{\test}^{\sfh}({\bcI}):=  \frac{ \left\langle \widehat\bcT, \bcI\right\rangle - \left\langle\bcT,\bcI\right\rangle  }{ \widehat s(\bcI) \sqrt{d^*/n} } \stackrel{{\rm d.}}{\longrightarrow} \cN (0,1),
\end{equation*}
where the estimated variance $\widehat s^2(\bcI)$ is given by
\begin{equation}\label{eq:emp-var-subexp}
    \widehat s^2(\bcI) = \frac{d^*}{n}\sum_{i=1}^n\left[\left(Y_i-\left\langle \bcT_{\init},\bcX_i \right\rangle\right)\left\langle\cP_{\widehat \TT }(\bcI), \bcX_i\right\rangle \right]^2.
\end{equation}
Moreover, there exists a numerical constant $C>0$ such that 
\begin{equation*}
\begin{aligned}
        &\max_{t\in\R}\abs{\bbP\left( \widehat W^{\sfh}_{\test}(\bcI)\le t\right)- \Phi(t)}  \\
        &\le   C C_1\kappa_{\sigma}^3\sqrt{\frac{m^6\mu^{m} r^*\dmax\log^5\dmax }{n }}  + C \frac{ C_1^2\kappa_{\sigma}^2\sigmax  \norm{\bcI}_{\ell_1} }{\alpha_I \lambda_{\min}  \norm{\bcI}_\tF }\sqrt{\frac{ m^5 (2\mu)^{3 m} (r^*)^{3 }\dmax d^*  \log^2 \dmax }{\rmin^2 n } }.
\end{aligned}
\end{equation*}
\end{Theorem}

The difference between \eqref{eq:tensor-init-subexp} and the initialization condition \eqref{eq:tensor-init} is that we replace $\sigma$ with the largest sub-exponential norm, $\sigmax$, to accommodate heteroskedastic noises. The required SNR and sample size conditions are slightly stronger than those in Theorem \ref{thm:lf-inference-empvar} by a factor of $\operatorname{poly}(\kappa_{\sigma} m \log \dmax)$. Theorem~\ref{thm:hetero-exp-clt} shows that 
heteroskedasticity primarily affects variance characterization by replacing $\sigma\norm{\cP_{\TT }(\bcI)}_{\tF}$ with $\norm{\cP_{\TT }(\bcI)\odot\bcS}_{\tF}$. In the special case of homogeneous noise, the tensor $\bcS$ has identical entries, i.e., $\bcS = \sigma \cdot \mathbf{1}_{d_1} \circ \mathbf{1}_{d_2} \circ \cdots \circ \mathbf{1}_{d_m}$. Then the variance in Theorem \ref{thm:hetero-exp-clt} reduces to that in Theorem \ref{thm:lf-inference-popvar} for i.i.d. noises. Additionally, our theory demonstrates that \eqref{eq:emp-var-subexp} reliably estimates $\norm{\cP_{\TT }(\bcI)\odot\bcS}_{\tF}$ under heteroskedastic noise.

Theorem~\ref{thm:hetero-exp-clt} has broad applications. Beyond the heterogeneous Gaussian setting \citep{cai2022uncertainty,farias2022uncertainty}, it extends to various practical scenarios. We outline several applications of Theorem~\ref{thm:hetero-exp-clt} below: 
\begin{enumerate}
    \item Binary tensor inference. Binary tensors are commonly used for network analysis and link prediction \citep{jing2021community,lyu2023latent,wang2020learning,han2022optimal}. For a latent low-rank tensor $\bcT$ with  $0\le [\bcT]_{\omega}\le 1$, it is assumed that the observations  satisfy
    $$\left[\left([ \bcT]_{\omega}+\xi_i\right)\mid\left[\bcX_i\right]_{\omega} = 1\right]   \sim \operatorname{Bernoulli}([ \bcT]_{\omega}), \quad \forall \omega\in [d_1]\times\cdots\times [d_m],  $$ 
    with each observation $Y_i = \langle\bcT,\bcX_i \rangle+\xi_i$ being binary. 
In this model, we have    $[\bcS]_{\omega} = \sqrt{[\bcT]_{\omega}(1-[\bcT]_{\omega})}$, with the  standard deviation for  inferring $\langle \bcT, \bcI\rangle$ given by $\norm{\cP_{\TT }(\bcI)\odot\sqrt{\bcT}\odot\sqrt{\bcJ-\bcT}  }_{\tF} \sqrt{d^*/n}$. Here, $\bcJ$ means the all-one tensor and $\sqrt{\cdot}$ is the entrywise square root operator.

    \item Poisson tensor inference. Poisson distribution is widely used for studying count-type data \citep{shi2014production,amjad2017censored}.   In the Poisson noise model, it is assumed that the  observations satisfy
     $$\left[\left([ \bcT]_{\omega}+\xi_i\right)\mid\left[\bcX_i\right]_{\omega} = 1\right]   \sim \operatorname{Poisson}([ \bcT]_{\omega}), \quad \forall \omega\in [d_1]\times\cdots\times [d_m],  $$ 
     with each observation $Y_i = \langle\bcT,\bcX_i \rangle+\xi_i$ following a Poisson distribution. In this case, we have
    $[\bcS]_{\omega} = \sqrt{[\bcT]_{\omega}}$, with the  standard deviation for  inferring $\langle \bcT, \bcI\rangle$ given by $\norm{\cP_{\TT }(\bcI)\odot\sqrt{\bcT} }_{\tF} \sqrt{d^*/n}$.  Our results improve the existing ones in Poisson matrix completion \citep{cao2015poisson,mcrae2021low,farias2022uncertainty} and are applicable for tensor linear form inference.
    \item Exponential tensor inference. Suppose that the noise is sub-exponential satisfying 
    $$\left[\left([ \bcT]_{\omega}+\xi_i\right)\mid\left[\bcX_i\right]_{\omega} = 1\right]   \sim \operatorname{exp}(1/[ \bcT]_{\omega}), \quad \forall \omega\in [d_1]\times\cdots\times [d_m],  $$ 
    where each observation $Y_i=\langle\bcT,\bcX_i \rangle+\xi_i$ follows an exponential distribution with conditional mean $[ \bcT]_{\omega}$. In this case, we have  $[\bcS]_{\omega} = {[\bcT]_{\omega}}$, with the  standard deviation for  inferring $\langle \bcT, \bcI\rangle$ given by $\norm{\cP_{\TT }(\bcI)\odot {\bcT} }_{\tF} \sqrt{d^*/n}$. Our results make it possible to study the statistical inference of matrix and tensor completion under exponential noises 
    \citep{lafond2015low,qiu2022noisy}.
\end{enumerate}

\section{Statistical-to-Computational Gaps}\label{sec:s-c-gap}
\subsection{Initialization without computational constraints}

Theorems \ref{thm:lf-inference-popvar} and \ref{thm:lf-inference-empvar} demonstrate that inference is possible requiring only statistically optimal SNR and sample size conditions. However, it remains unclear whether these minimal conditions can reliably initialize \eqref{eq:tensor-init}, ensuring the feasibility of the entire inference procedure. Consequently, a natural question is which method can achieve the entrywise estimation error rate \eqref{eq:tensor-init} if unlimited computational resources are available.  Fortunately, we can address the problem by combining the maximum likelihood estimator and the recent result from online tensor learning \citep{cai2023online}. 

A major challenge for iterative methods \citep{chen2019inference,wang2023implicit} in establishing entrywise estimation error bounds under weak conditions is the lack of a warm initialization. When the SNR and sample size conditions are minimal, existing polynomial-time algorithms can not provide a warm initialization. The following theorem demonstrates that the solution to a constrained least squares problem, inspired by the \emph{maximum likelihood estimator} (MLE) under i.i.d. Gaussian noise, provides a warm initialization under a minimal sample size condition. Note that solving the MLE is computationally NP-Hard in general \citep{zhang2018tensor}. 
Given $\bW\in\O^{d\times r}$, define $\Inco(\bW):=d\norm{\bW}_{2,\infty}^2/r$ as its incoherence parameter.

\begin{Theorem}[Least squares initialization]\label{thm:constrained-ls}
Let $\widetilde{\bcT}$ be the solution to the following constrained least square minimization problem:
\begin{equation}\label{eq:const-LS}
\begin{aligned}
    & \min_{\bcW=\bcD\times_{j=1}^m\bW_j} h_n(\bcW)=\frac{d^*}{n}\sum_{i=1}^n\left(Y_i-\left\langle\bcW,\bcX_i \right\rangle \right)^2 \\
    & \quad \quad \text{s.t.}\  \bW_j\in\O^{d_j\times r_j}, \Inco(\bW_j)\le \mu, \bcD\in \R^{\br}. \
   \end{aligned}
\end{equation}
If $n\ge C_{\gap} \mu^m m^2\rmax r^*\dmax\log \dmax$ for a large enough numerical constant $C_{\gap}>0$, then there exists a numerical constant $C>0$ such that  
    \begin{equation}\label{eq:const-LS-rate}
  \norm{\widetilde{\bcT}-\bcT}_{\tF}\le C (\sigma\vee\norm{\bcT}_{\linf})\sqrt{\frac{\mu^{2m} m^2\rmax (r^*)^2 d^*\dmax \log \dmax}{n}},
\end{equation}
with probability at least $1-6\dmax^{-3m}$. 
\end{Theorem}

Note that Theorem \ref{thm:constrained-ls} holds provided that $n$ is on the order $\Omega(\dmax\log \dmax)$, regardless of the SNR condition. 
Furthermore, the rate in \eqref{eq:const-LS-rate} matches the estimation error in \cite{xia2021statistically}, suggesting that when $\rmax,m,\mu=O(1)$, this rate is nearly statistically optimal (up to logarithmic factors).
Additionally, the rate in \eqref{eq:const-LS-rate} implies that if the sample size satisfies  $n \gtrsim \mu^{3m} (r^* )^3\rmax \kappa_0^2 m^2 \dmax\log\dmax$, then $\norm{\widetilde{\bcT}-\bcT}_{\tF}\lesssim\lambda_{\min}$  under statistically optimal SNR conditions. This is because incoherence implies $\norm{\bcT}_{\linf}\le \sqrt{\mu^m r^*/d^*}\kappa_0\lambda_{\min}$. Therefore, according to \cite{cai2023online}, the $\mu$-incoherent tensor $\widetilde{\bcT}$ can serve as an initializer for online tensor gradient descent. After performing online tensor gradient descent with an additional $O(n)$ iterations, we achieve $\norm{\widehat \bcT_{\init} -   \bcT }_{\linf } \le C_1 \sigma \sqrt{(\dmax/n) \log \dmax} $ \citep{cai2023online}. This indicates that proper initialization can be obtained under statistically optimal SNR and sample size conditions $n\gtrsim \dmax\log \dmax$ through constrained least squares combined with online tensor gradient descent. The detailed procedure is presented in Algorithm \ref{alg:oracle-init}, and the online Riemannian gradient descent is described in the supplementary document.

\begin{algorithm}
\caption{Oracle Initialization by Least Squares Estimator}
\label{alg:oracle-init}
\begin{algorithmic}
\REQUIRE{$\left\{ (\bcX_i, Y_i) \right\}_{i=1}^{n}$}\
\STATE{Split the observations evenly into two halves: $\cD_1 =\left\{ (\bcX_i, Y_i) \right\}_{i=1}^{n_{0}} $, $\cD_2 =\left\{ (\bcX_i, Y_i) \right\}_{i=n_0+1}^{n} $;}
\STATE{Solve the least square \eqref{eq:const-LS} using only $\cD_1$ and obtain the solution $\widetilde{\bcT}$;
}
\STATE{Use $\widetilde{\bcT}$ as the start point to run the online Riemannian gradient descent \citep{cai2023online} on $\cD_2$;}
\STATE{Return $\widehat{\bcT}_{\mathsf{oRGrad}}$, which is the outcome of online Riemannian gradient descent.}
\end{algorithmic}
\end{algorithm}

\begin{Proposition}\label{prop:oracle-init-linf} Under Assumption \ref{asm:incoherence}, suppose that the SNR and sample size satisfy
$$
n\ge C_m \kappa_0^{4 m+2} \mu^{3 m}\rmax \left(r^*\right)^3 {\dmax }\log\dmax, \quad \frac{\lambda_{\min } }{\sigma} \geq C_m  \kappa_0^{2 m-4}\sqrt{ \frac{\mu^{2 m} \rmax (r^*)^2 d^* \dmax \log^3\dmax}{  n}   }  ,
$$
for a large constant  $C_m>0$ depending on $m$ only. Let $\widehat{\bcT}_{\mathsf{oRGrad}}$ be the output of Algorithm~\ref{alg:oracle-init} with the step size in online Riemannian gradient descent set as $\eta = c_{0,m} \log\dmax/n$, then  

\begin{equation*}
    \norm{\widehat{\bcT}_{\mathsf{oRGrad}}-\bcT}_{\linf} \le C_{1,m} \kappa_0^{m+3} \sigma \sqrt{\frac{ \mu^m r^*\rmax \dmax \log \dmax}{n}}, 
\end{equation*}
with probability at least $1-7\dmax^{-3m}$, where $c_{0,m},  C_{1,m}>0$ are constants depending on $m$ only. 
\end{Proposition}

Proposition~\ref{prop:oracle-init-linf} directly follows from Theorem 5 in \cite{cai2023online} and Theorem~\ref{thm:constrained-ls}. To establish it, we simply verify that the initial assumptions of Theorem 5 in \cite{cai2023online} are met under the given SNR and sample size conditions, using $\widetilde{\bcT}$ from Theorem~\ref{thm:constrained-ls}. We can refer to Algorithm~\ref{alg:oracle-init} as the ``oracle initialization" because solving the least squares problem \eqref{eq:const-LS} is computationally intensive and generally infeasible under only the statistically optimal conditions stated in Proposition~\ref{prop:oracle-init-linf}.

\begin{Remark}
The least squares approach in Theorem \ref{thm:constrained-ls} can be seen as a generalization of the maximum likelihood estimation (MLE) used in tensor PCA \citep{zhang2018tensor,yeredor2019maximum,jagannath2020statistical}. Unlike the PCA problem with full observations, the MLE for tensor completion is less studied because of missing values and computational infeasibility. Our Theorem \ref{thm:constrained-ls} fills in this gap by demonstrating that the solution of constrained least squares itself also possesses favorable statistical properties under minimal conditions. 
\end{Remark}

\subsection{Discussion on statistical and computational gaps}\label{sec:s-c-gap+sparse}

With Proposition~\ref{prop:oracle-init-linf}, we can now integrate the initialization and debiasing approaches to examine the statistical-to-computational gaps in the inferential task. To simplify the discussion, we assume that the dimensions are balanced (i.e., $d_1\asymp\cdots \asymp d_m \asymp d$) and that our indexing tensor $\bcI$ is $O(1)$ sparse. Additionally, we assume that $\mu$, $m$, $r^*$, $\kappa_0$, $1/\alpha_I = O(1)$, and we omit the $\log\dmax$ term for clarity.

We summarize the statistical-to-computational gaps in Figure~\ref{fig:s-c-gap-m=3-4} by varying the SNR and sample size conditions. Specifically, we consider the cases where $m=3$ and $m\geq 4$, respectively, and categorize the entire region into distinct phases.


\begin{figure}[H]
\centering
\begin{subfigure}[b]{0.5\textwidth}
         \centering
         \includegraphics[width=\textwidth]{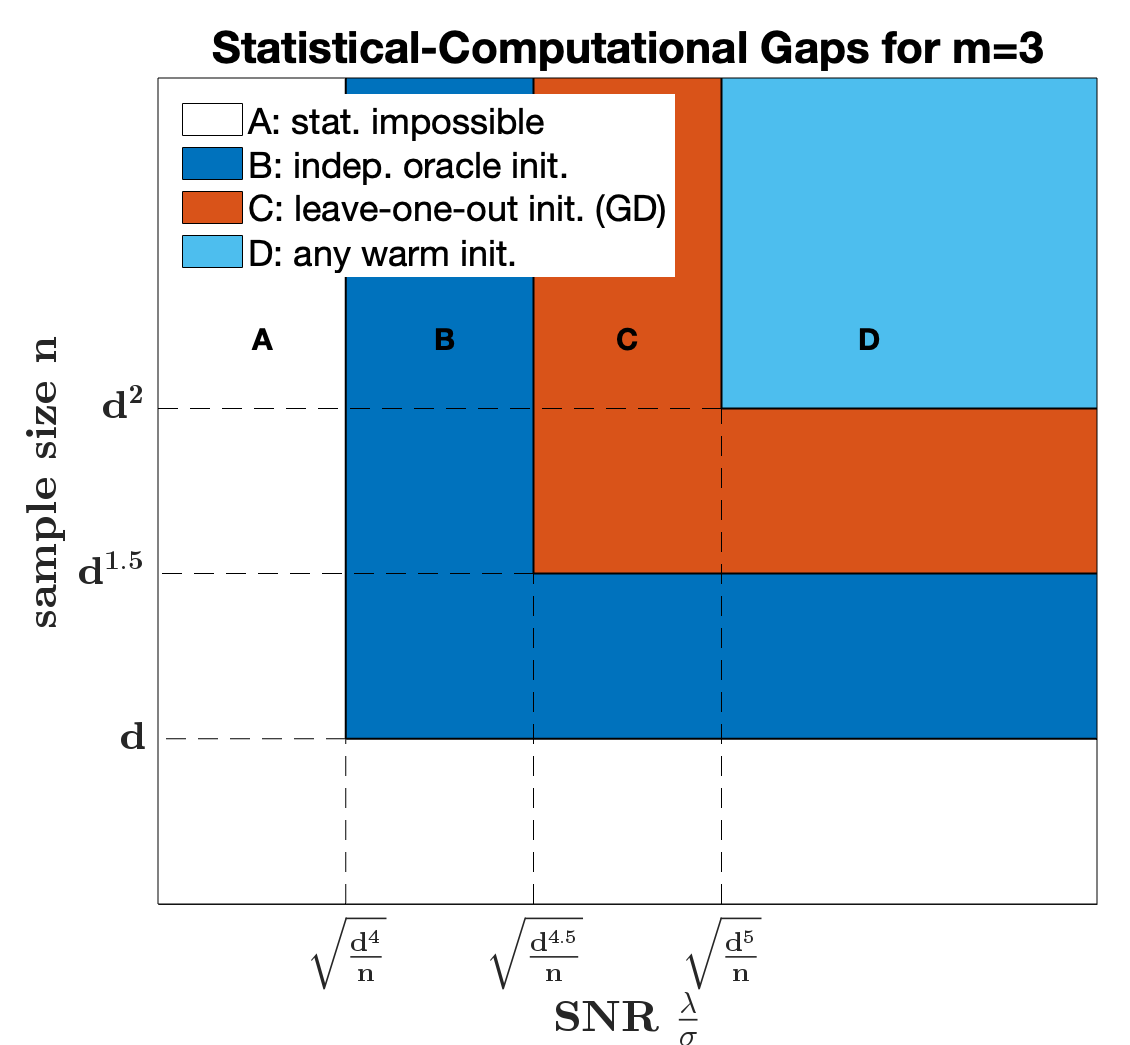}
         \caption{$m=3$ }
         \label{fig:phase-m=3}
     \end{subfigure}%
     \begin{subfigure}[b]{0.5\textwidth}
         \centering
         \includegraphics[width=\textwidth]{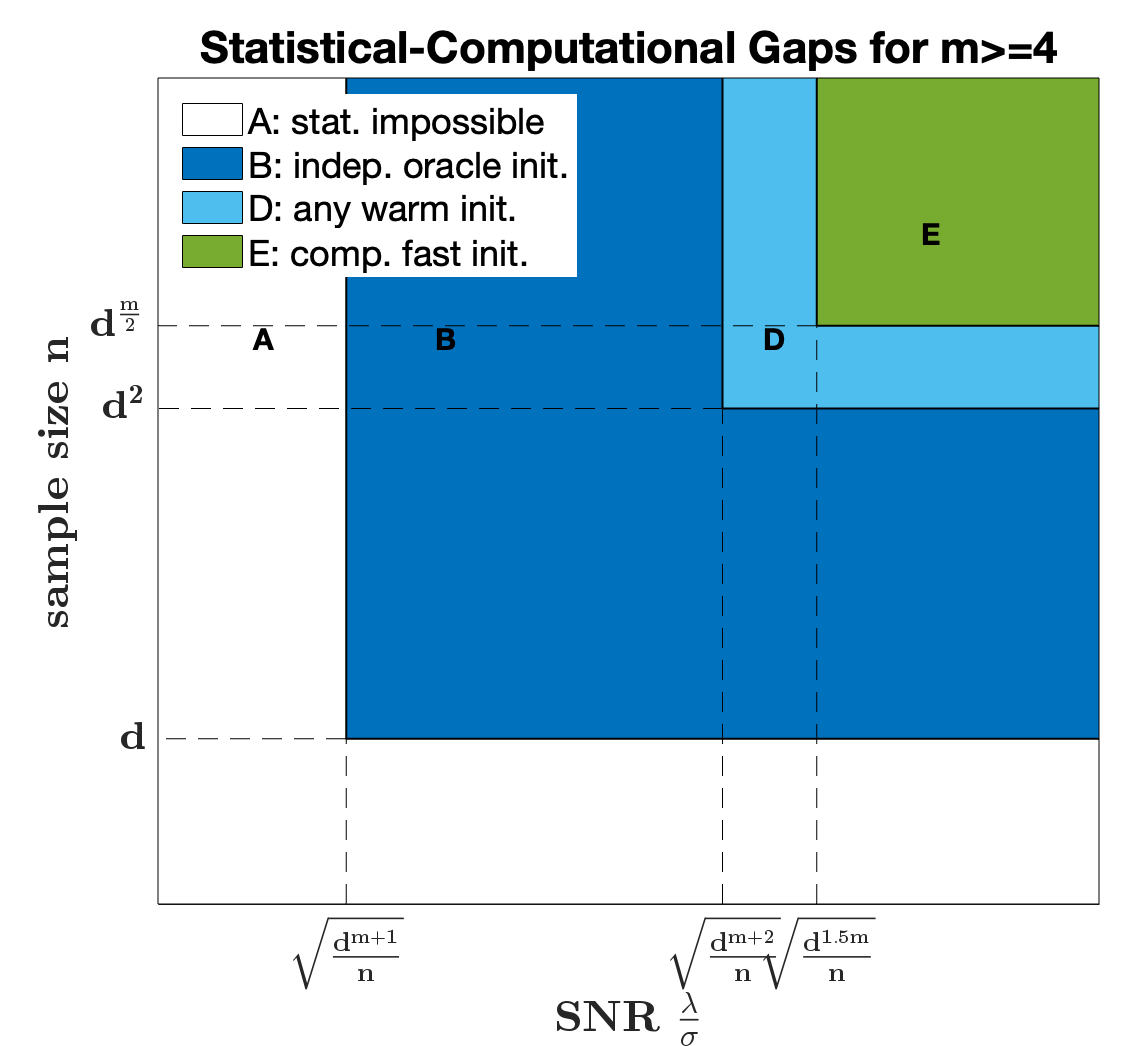}
         \caption{$m\ge 4$ }
         \label{fig:phase-m>=4}
     \end{subfigure}
     \hfill
\caption{Phase transitions and statistical-to-computational gaps for inference on tensor linear forms. Region A: statistically impossible region; Region B: inference can be achieved by independent oracle initialization in Algorithm \ref{alg:oracle-init} with data-splitting; Region C: inference can be achieved by leave-one-out-type initialization \citep{cai2019nonconvex,wang2023implicit} without data-splitting;  Region D: inference can be achieved by any warm initialization with guaranteed entywise error bound, with no data-splitting;  Region E: inference can be achieved by computationally fast algorithms with no data-splitting.
}
\label{fig:s-c-gap-m=3-4}
\end{figure}

We now elaborate on these different regions. Region A, identified in \cite{zhang2018tensor,xia2021statistically}, is where statistical inference is impossible because the SNR is below the statistically optimal threshold or when $n$ is less than the degrees of freedom. When both the SNR and sample size satisfy the statistically optimal conditions in \eqref{eq:SNR-Stat}, initialization and debiasing are addressed in Proposition \ref{prop:oracle-init-linf} and Theorem \ref{thm:lf-inference-empvar}, defining Region B. 

Next, we consider the cases of $m=3$:
\begin{enumerate}[leftmargin=1cm]
    \item For $m=3$, the computational optimal condition \eqref{eq:SNR-Comp} is $n \gg d^{1.5}$ and $\lambda_{\min}/\sigma \gg \sqrt{d^{4.5}/n}$, defining region C. In this scenario, \cite{cai2019nonconvex, wang2023implicit} provide a leave-one-out-type initialization via gradient descent, satisfying the assumptions of Theorem \ref{thm:clt-lol-init} and eliminating the need for data splitting. 
    \item When the SNR and sample size conditions are $\sqrt{d}$
  times larger than the statistically optimal thresholds, Theorem \ref{thm:clt-dep-init} shows that any algorithm providing a warm initialization \eqref{eq:tensor-init} can be used for dependent initialization, leading to Region D.
\end{enumerate}

For $m \ge 4$, the conditions for Region D are less stringent than the computational optimal condition \eqref{eq:SNR-Comp}, allowing Region D to extend beyond Region C. Additionally, since dependence is not an issue in Region D, leave-one-out analysis becomes unnecessary, and any computationally efficient algorithm with a warm initialization \eqref{eq:tensor-init} is suitable. Thus, we designate this computationally feasible area as Region E.

\begin{Remark}[Sparsity level of linear form]
    The discussion in Figure \ref{fig:s-c-gap-m=3-4} clearly delineates the statistical-to-computational gaps for tensor linear form inference under $O(1)$ sparsity of $\bcI$. In fact, Theorem  \ref{thm:clt-lol-init} and \ref{thm:clt-dep-init} also provide the feasible sparsity conditions of $\bcI$.  It turns out that the additional SNR and sample size required under computationally optimal conditions can be beneficial for inference, as they allow for a higher sparsity level 
 $s_0=O\left(d^{\frac{m}{2}-1 }\right)$ of $\bcI$, where $s_0$ is the number of non-zero entries in $\bcI$. We elaborate on the benefit of such sparsity levels below: 
\begin{enumerate}
    \item When $m= 3$, we allow for $s_0=O(\sqrt{d})$, which enables us to select $\bcI$ as the groupwise combination \citep{bi2018multilayer} with a moderate group size.
    \item When $m= 4$, $s_0=O(d)$ allows us to infer fiber-based linear transformations. For example, we can estimate the average treatment effect among different groups of trials: $\frac{1}{d_2}\left(\sum_{j\in[d_2]}\bcT(2,j,i_3,i_4)- \bcT(1,j,i_3,i_4)\right) $, where the first argument represents the treatment group label. 
    \item When $m= 6$, with $s_0=O(d^2)$, we can perform statistical inference on slice-based transformations. This includes comparisons between panel data \citep{badinger2013estimation} and inferring 2-D features in medical image processing \citep{nandpuru2014mri, tandel2020multiclass}.
\end{enumerate}

Conversely, in the matrix case ($m = 2$), there is no similar advantage in sparsity because there is no statistical-to-computational gap. Consequently, the sparsity level is limited to
$O(1)$ under computationally optimal SNR and sample size conditions \citep{chen2019inference, xia2021statistical}.
\end{Remark}

\section{Simulation Experiments}

We conduct several simulations to corroborate the theories above. Specifically, we examine the inference of linear forms under both independent and dependent initialization, as well as the coverage rates of confidence intervals generated by our test statistics.

\paragraph{Inference with independent initialization.} In each setting, we randomly generate an incoherent Tucker low-rank tensor $\bcT$ with $m=3$, $d_1=d_2=d_3=d=100$, $\br=(2,2,2)$. The incoherent singular subspaces are constructed from the SVD of random Gaussian matrices. In this section, we denote the sampling rate $p = n/d^*$ to make the description consistent with the matrix/tensor completion literature. 
In the independent initialization case, we assume that we have the oracle information on $\bcT$ such that the initializer $\widehat\bcT_{\init}$ is constructed independent of the observations $\{(Y_i,\bcX_i)\}_{i=1}^n$. Following the theorems in Section \ref{sec:inference-NTC}, 
We set the sample size $n$ at the  level $O(d\log^2 d)$ with the corresponding sampling rate  $p=n/d^3 \asymp mr^2d^{-2}\log^2 d \approx 0.02$. For the tensor signals, we take $\lambda_{\min} = 10 d^{\gamma}$ for different values of $\gamma$ showing above to adjust the SNR, and set $\xi_i\sim \cN(0,1)$ with fixed $\sigma =1$.  

The results of 1000 independent Monte Carlo trials on different SNR settings are presented in Figure \ref{fig:clt-indep}, where we take $\bcI$ as a sparse 0-1 tensor such that $\abs{\operatorname{supp}(\bcI)} = 2$, so that $\langle\bcT,\bcI \rangle$ represents the sum of two entries in $\bcT$. Notice that the bottom right panel shows the results under heteroskedastic noises, where we allow the standard deviation of each $\xi_i$ change according to the position of sampling. We set the minimum and maximum noise level: $[\min_{\omega} [\bcS]_{\omega},\max_{\omega} [\bcS]_{\omega}] = [0.75,1.25]$. In this case, we slightly enlarge the size by choosing $p=0.03$, and let $\lambda_{\min} = 20 d^{\frac{1}{2}}$ to overcome the heteroscedasticity.  

From Figure \ref{fig:clt-indep}, it is clear that Theorem \ref{thm:lf-inference-popvar}, \ref{thm:lf-inference-empvar}, and Theorem \ref{thm:hetero-exp-clt} are justified under only moderate SNR $\lambda/\sigma \gg \sqrt{d^* d/n }$ and sample size $n\gg d$ conditions. Also, it is obvious that with the growth of $\lambda_{\min}$, the distance between the empirical distribution and $\cN(0,1)$ becomes smaller.
\begin{figure}[H] 
\centering
\includegraphics[width=1\textwidth]{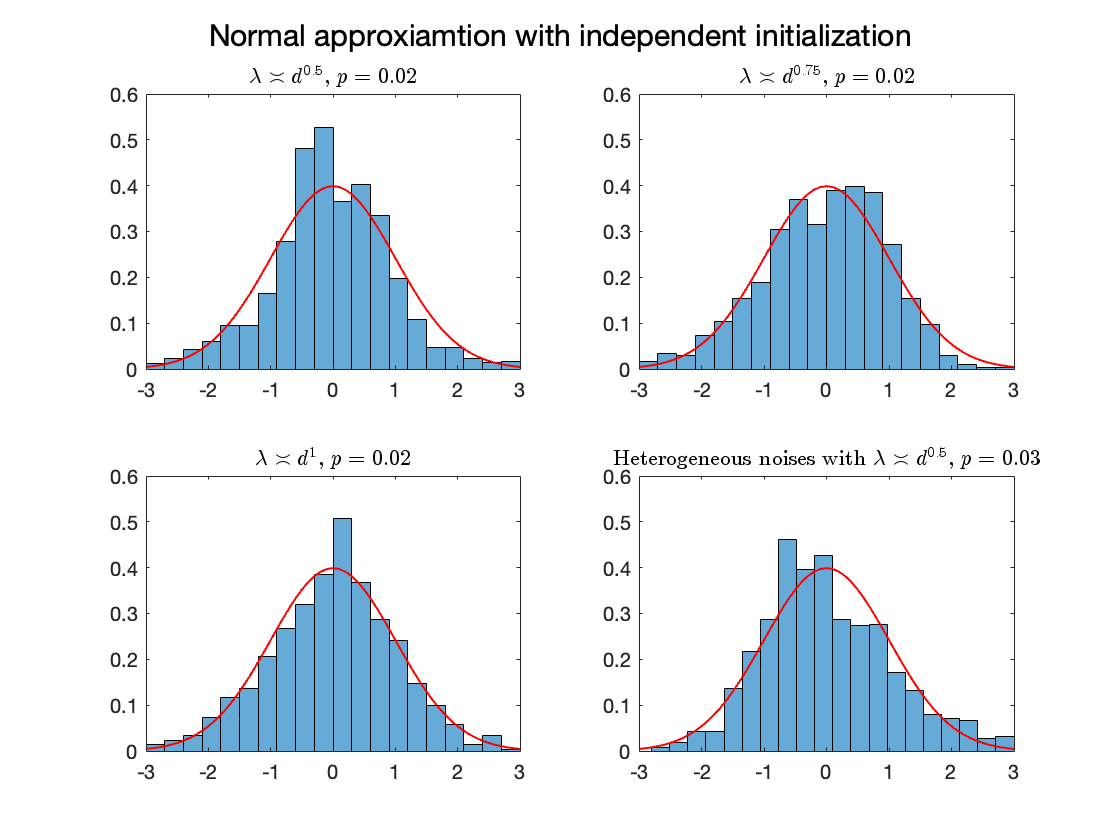}
\caption{Histogram of normal approximation over 1000 independent trails with $\gamma=\{0.5,0.75,1,0.5\}$. The first 3 panels show results under homogeneous Guassian noises while the last panel shows results under heterogeneous Guassian noises.
}
\label{fig:clt-indep}
\end{figure}

\paragraph{Inference with dependent initialization.} With the tensor generating mechanism the same as above, we now focus on the case when the initialization is yielded directly from the observations $\{(Y_i,\bcX_i)\}_{i=1}^n$.  We still take $\lambda_{\min} = 10 d^{\gamma}$ for different values of $\gamma$ showing above and set $\xi_i\sim \cN(0,1)$ with fixed $\sigma =1$. In the case of dependent initialization, we need a good recovery of $\bcT$ that is computationally tractable.  To this end, 
we set sample size $n$ at the computationally optimal level $O(d^{3/2 }\log^2 d)$ with the corresponding sampling rate $p= n/d^3\asymp r d^{-3 / 2} \log ^2 d\approx 0.04$, and the SNR at the level $\lambda_{\min}/\sigma \ge \sqrt{d^*d^{1.5}/n}\gtrsim d^{0.75}$, which suggest that $\gamma\ge 0.75$.

Since the computationally efficient conditions are stronger than the statistically optimal conditions, Section \ref{sec:s-c-gap+sparse} suggests that we can take $\bcI$ as a sparse 0-1 tensor with the sparsity level as large as $\abs{\operatorname{supp}(\bcI)} = \sqrt{d}=10$. In this case, our $\langle\bcT,\bcI \rangle$ represents the sum of 10 entries in $\bcT$. We adopt the offline Riemannian gradient descent approach in  \cite{wang2023implicit} with the off-diagonal initialization \citep{chen2021spectral,zhang2022heteroskedastic} that deletes the diagonal elements in the singular subspace computing to offer a start point of gradient descent. The results of 1000 independent trials are presented in Figure \ref{fig:clt-dep}.

\begin{figure}[H] 
\centering
\includegraphics[width=1\textwidth]{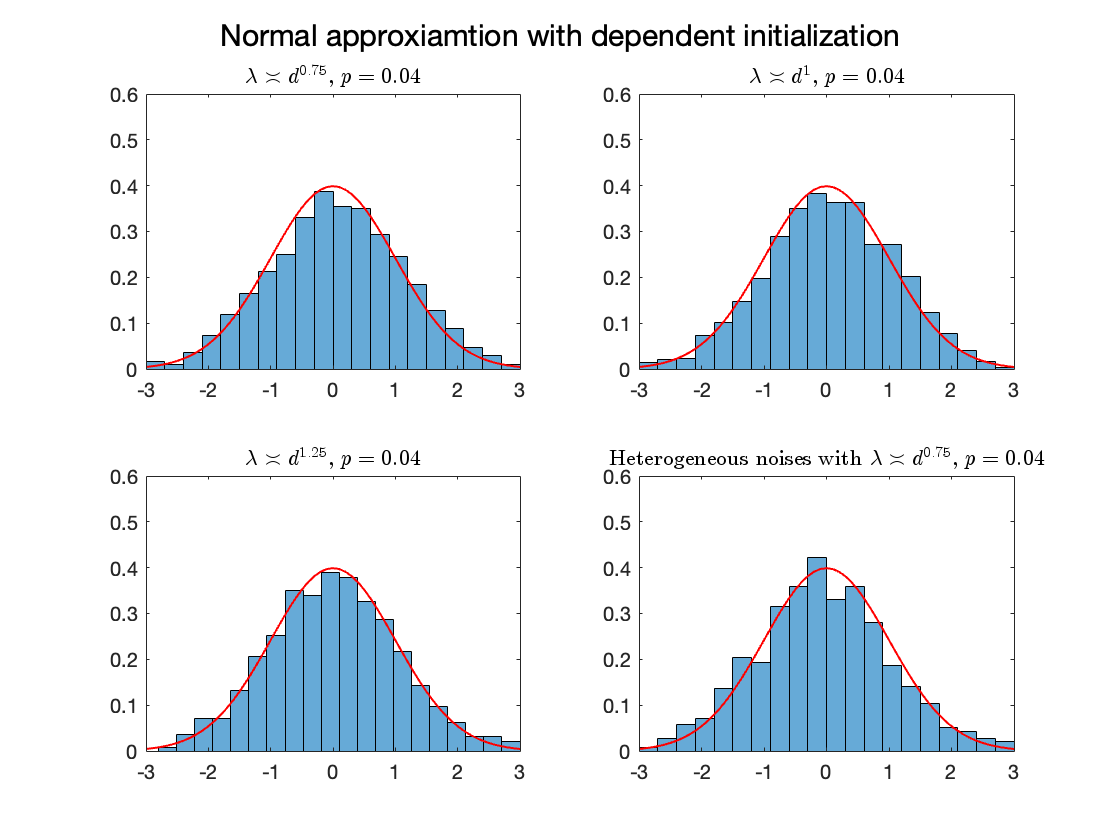}
\caption{Histogram of normal approximation over 1000 independent trails with $\lambda_{\min} = 10 d^{\gamma}$ and $\gamma=\{0.75,1,1.25,0.75\}$. $\bcI$ is sparse 0-1 tensor with $\abs{\operatorname{supp}(\bcI)} =10$. The last panel shows the results under heteroskedastic noises, where we set: $[\min_{\omega} [\bcS]_{\omega},\max_{\omega} [\bcS]_{\omega}] = [0.75,1.25]$.
}
\label{fig:clt-dep}
\end{figure}
Figure \ref{fig:clt-dep} clearly demonstrates the feasibility of our method under dependent initialization without data splitting. The computationally optimal conditions allow for inference of linear forms with extra sparsity level even under heteroskedastic noises. 

\paragraph{Average coverage of confidence intervals.} We now examine our method in a broader area of settings by considering different noise types and ranks. Still, the tensor generating mechanism is similar to above, but now we vary the hypothesized rank $\br=(r,r,r)$ for different $r$. We fix $d=100$, $p = 0.04$, and $\lambda_{\min} =10d^{1.25}$, and change the noise type into Gaussian noises, Poisson noises, and exponential noises, respectively. Additionally, we assume all the noises are heterogeneous, with noise scales $[\bcS]_{\omega}$ ranging from  $(0,2)$.

Now, for each latent $\bcT$ with a group of observations, we use the RGD to obtain an initialization and run our Algorithm \ref{alg:debias-powerit} to obtain $\widehat\bcT$. To consider the average coverage of linear form CIs, we select a group of linear forms $\{\bcI_{(i,j,k)}\}_{(i,j,k)\in \cQ}$, where  $\langle\bcT,\bcI_{(i,j,k)}\rangle = \bcT(1,1,1)+\bcT(1,1,2)-\bcT(i,j,k)$ and $\cQ$ are sampled in $[100]\times[100]\times[100]$ with cardinality $\abs{\cQ} = 100$. The average coverage in one trial is thus defined by

\begin{equation*}
    \mathsf{AvgCov} = \frac{1}{\abs{\cQ}}\sum_{(i,j,k)\in\abs{\cQ} } \idc\left\{\langle\bcT,\bcI_{(i,j,k)} \rangle \in  \widehat{\mathrm{CI}}_{\alpha}(\bcI_{(i,j,k)} ) \right\}, \quad \alpha = 0.9, \  0.95,
\end{equation*}
where $\widehat{\mathrm{CI}}_{\alpha}(\bcI )$ is given in Corollary \ref{coro:CIs}.
We now run the experiments by 100 independent trails and record the averaged $\mathsf{AvgCov} $, together with their error bars showing the prediction intervals in Figure \ref{fig:avg-coverage}.

\begin{figure}[H] 
\centering
\includegraphics[width=1\textwidth]{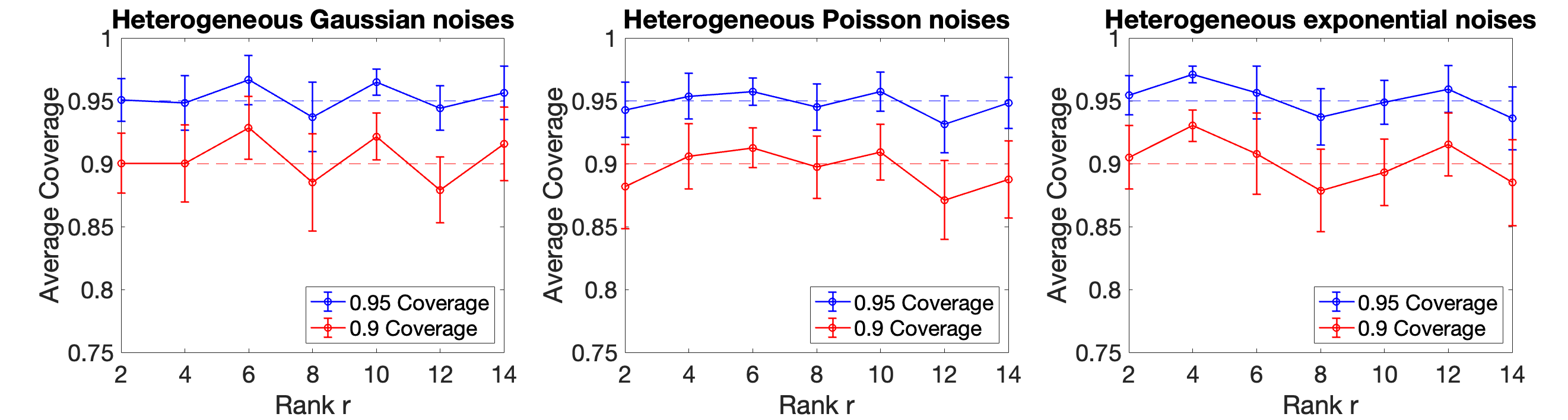}
\caption{$\mathsf{AvgCov}$ of 0.95 and 0.9 level CIs over 100 independent trails. The  error bars represent $ \mathsf{AvgCov} \pm z_{0.1} \widehat{\sigma}$,
showing the $80\%$ confidence intervals of the $\mathsf{AvgCov}$ }
\label{fig:avg-coverage}
\end{figure}

The results show that our coverage rates nicely sit around 0.95 and 0.9 given different levels of CIs when $r$ ranges from $2$ to $14$ and the distribution of $\xi_i$ changes to heterogeneous sub-exponential noises. This evidences the accurate uncertainty quantification of our debiasing + one-step power iteration approach.



\section{More Discussions}\label{sec:discuss}
We present a simple yet efficient method for performing statistical inference on tensor linear forms with incomplete observations. Our approach combines a warm, achievable initialization with straightforward debiasing and a single power iteration step. This general framework can be flexibly extended to various tensor models. 

Our debiasing and one-step power iteration framework naturally extends to tensor PCA, which can be viewed as a special case of tensor completion with full observations and $n=d^*$. Under full observations, the initialization error $\bcE_{\init}$ in \eqref{eq:ubs-approx-decomp} vanishes, eliminating the need for debiasing. The bounds developed for $\bcE_{\rn}$	
 remain valid by substituting $n$ with $d^*$. Consequently, our one-step power iteration yields results comparable to \cite{xia2022inference}. Unlike \cite{xia2022inference}, which is limited to rank-one linear forms and requires two power iterations, our method accommodates general ranks 
$\br$ with only one power iteration. Additionally, our variance characterization is sharper, achieving the Cramér–Rao lower bound.

Furthermore, our method extends to general tensor regression models with various measurements. For example, with Gaussian measurements where each $\bcX_i$ has independent standard Gaussian entries, the debiasing scheme becomes: 
\begin{equation*}
        \widehat \bcT_{\ubs} = \widehat \bcT_{\init } +\frac{1}{n}\sum_{i=1}^{n}\left(Y_i - \left\langle \widehat\bcT_{\init }, \bcX_i \right\rangle\right)\cdot \bcX_i,
\end{equation*}
enabling a similar linear form inference approach. Under Gaussian measurements, a lower bound analogous to Theorem \ref{thm:opt-uct-qtf} is $\Var(g_{\bcI}(\Bar{\bcT}_{\mathsf{Gauss}})) \ge \left(1- \varepsilon_{\SNR}\right)\frac{\sigma^2}{n} \norm{\cP_\TT (\bcI)}_{\tF}^2$, where $\varepsilon_{\SNR}$ vanishes as the SNR tends to infinity. 
 
Our work primarily focuses on inferring linear forms under uniform sampling. Future research directions include: (i) adapting the debiasing approach: extending our method to non-uniform missing patterns, such as heterogeneous missingness or missing at random scenarios \citep{choi2024matrix}. (ii) confidence regions for singular subspaces: investigating confidence regions for singular subspaces based on incomplete tensor observations. This can potentially be achieved by combining \cite{xia2019confidence, xia2022inference} with our debiasing and one-step power iteration approach.

\end{sloppypar}
\newpage
\bibliography{reference}
\bibliographystyle{apalike}

\appendix
\newpage

\begin{center}
{\bf\LARGE Supplement to ``Statistical Inference in Tensor Completion: Optimal Uncertainty
Quantification and statistical-to-computational Gaps''}
\end{center}
\smallskip

\section{Supplementary Tensor Estimation Procedures}
\subsection{Online Riemannian Gradient Descent}

\begin{algorithm}[H]
\caption{Online Riemannian Gradient Descent \citep{cai2023online}}
\label{alg:online-RGD}
\begin{algorithmic}
\REQUIRE $\left\{ (\bcX_t, Y_t) \right\}_{t=1}^{n}$, initial $\widehat \bcT_{0}$, step size $\eta>0$, multilinear rank $\br=\left(r_1,r_2,\dots,r_m\right)$
\FOR{$t = 1,\dots, n$}
	\STATE{Calculate online vanilla gradient $\bcG_t = \left(\langle\widehat \bcT_{t-1},\bcX_t\rangle - Y_t\right)\cdot \bcX_t$}
 \STATE{Calculate Riemannian GD $\widehat \bcT_{t}^{+} = \widehat \bcT_{t-1} - \eta \cdot \cP_{\TT_{t-1}}(\bcG_t) $}
 \STATE{Retraction $\widehat \bcT_{t}=\HOSVD_{\br}\left(\widehat \bcT_{t}^{+} \right)$}
\ENDFOR
\RETURN {$\widehat{\bcT}_n$}
\end{algorithmic}
\end{algorithm}
Here, then tangent space projection $\cP_{\TT_t}(\cdot)$ is align with \eqref{eq:true-TT-proj} and \eqref{eq:emp-PT} by changing the $\bcT$ in \eqref{eq:true-TT-proj} with $\widehat \bcT_{t}$. the $\HOSVD_{\br}$ means the higher order singular value decomposition which is given in Algorithm \ref{alg:hosvd}. In our version of online RGD, our definition of $\bcX_t$ is different from that in  \cite{cai2023online}, where their sensing tensors are $\sqrt{d^*}$ times larger than ours, meaning that the bound in \cite{cai2023online} is actually $1/\sqrt{d^*}$ times smaller than the true $\norm{\widehat{\bcT}_t-\bcT}_{\linf}$ in our setting.
\begin{algorithm}[H]
\caption{Higher Order Singular Value Decomposition (HOSVD)}
\label{alg:hosvd}
\begin{algorithmic}
\REQUIRE $\bcA\in\R^{d_1\times\cdots \times d_m} $, multilinear rank $\br=\left(r_1,r_2,\dots,r_m\right)$
\FOR{$j = 1,\dots, m$}
	\STATE{Calculate $\widehat{\bU}_{j}=\SVD_{r_j}\left(\mathcal{M}_j\left(\bcA\right)\right) $}
\ENDFOR
\STATE{Calculate core tensor $\widehat{\bcC}=\bcA \times_1 \widehat{\bU}_{1}^\top \cdots \times_m \widehat{\bU}_{m}^\top $}
\STATE{Calculate $\widehat{\bcA}=\widehat{\bcC} \times_1 \widehat{\bU}_{1} \cdots \times_m \widehat{\bU}_{m} $}
\RETURN {$\widehat{\bcA}$, core tensor $\widehat{\bcC}$, singular subspaces $\left(\widehat{\bU}_{1},\dots, \widehat{\bU}_{m} \right)$}
\end{algorithmic}
\end{algorithm}

\subsection{Offline RGD and Spectral Initialization}
In our data experiments, we initialize our debasing approach by offline RGD. The detailed implementation of offline RGD is described in the following Algorithm \ref{alg:diag-dele} and \ref{alg:offline-RGD}, where the return of Algorithm \ref{alg:diag-dele} serve as a good starting point of RGD iterations in Algorithm \ref{alg:offline-RGD}. 

\begin{algorithm}[H]
\caption{Initialization with Diagonal Deletion \citep{chen2021spectral}}
\label{alg:diag-dele}
\begin{algorithmic}
\REQUIRE $\left\{ (\bcX_i, Y_i) \right\}_{i=1}^{n}$, multilinear rank $\br=\left(r_1,r_2,\dots,r_m\right)$
\STATE{Observation tensor $\widehat{\bcT}_{\mathsf{obv} } = \sum_{i=1}^{n} Y_i\bcX_i$ }
\FOR{$j = 1,\dots, m$}
	\STATE{Calculate $\widehat{\bU}_{j}=\SVD_{r_j}\left(\cP_{\mathsf{off}\mbox{-}\mathsf{diag}}\left(\mathcal{M}_j(\widehat{\bcT}_{\mathsf{obv} })\mathcal{M}_j^\top(\widehat{\bcT}_{\mathsf{obv} })\right)\right) $}
\ENDFOR
\STATE{Low-rank projection: $\widehat{\bcT}_{\mathsf{off}\mbox{-}\mathsf{diag}}= \frac{d^*}{n}\cdot \widehat \bcT_{\mathsf{obv}} \times_1\cP_{\widehat{\bU}_{1} }\cdots \times_m\cP_{\widehat{\bU}_{m} } $}
\RETURN {$\widehat{\bcT}_{\mathsf{off}\mbox{-}\mathsf{diag}}$}
\end{algorithmic}
\end{algorithm}
Here, the off-diagonal projection $\cP_{\mathsf{off}\mbox{-}\mathsf{diag}}(\cdot)$ means that we substitute all the diagonal elements with $0$. As indicated in \cite{wang2023implicit}, the return of Algorithm \ref{alg:diag-dele} is suitable for offline RGD in Algorithm \ref{alg:offline-RGD}, with a verified leave-one-out condition for the base case for mathematical induction.
\begin{algorithm}[H]
\caption{Offline Riemannian Gradient Descent \citep{wang2023implicit}}
\label{alg:offline-RGD}
\begin{algorithmic}
\REQUIRE $\left\{ (\bcX_i, Y_i) \right\}_{i=1}^{n}$, initial $\widehat \bcT_{0}$, multilinear rank $\br=\left(r_1,r_2,\dots,r_m\right)$
\STATE{Observation tensor $\widehat{\bcT}_{\mathsf{obv} } = \sum_{i=1}^{n} Y_i\bcX_i$ }
\FOR{$t = 1,\dots, T$}
	\STATE{Calculate offline vanilla gradient $\bcG_t = \sum_{i=1}^n\langle\widehat \bcT_{t-1},\bcX_i\rangle\bcX_i - \widehat{\bcT}_{\mathsf{obv} }$}
 \STATE{Calculate Riemannian GD $\widehat \bcT_{t}^{+} = \widehat \bcT_{t-1} - \frac{d^*}{n}\cdot \cP_{\TT_{t-1}}(\bcG_t) $}
 \STATE{Retraction $\widehat \bcT_{t}=\HOSVD_{\br}\left(\widehat \bcT_{t}^{+} \right)$}
\ENDFOR

\RETURN {$\widehat{\bcT}_T$}
\end{algorithmic}
\end{algorithm}

\section{Proofs of Main Results}
\subsection{Preliminaries for the proofs}
\noindent \textbf{On the Tucker decomposition.} 
Without loss of generality, we shall assume that the core tensor satisfies $\cM_j(\bcC) \cM_j(\bcC)^\top = \bLambda_j^2$ for a diagonal matrix $\bLambda_j\in\R^{r_j\times r_j}$, which means that the mode-$j$ unfolding of the core tensor can be decomposed as $\cM_j(\bcC) = \bLambda_j \bV_j^\top$ for a certain $\bV_j\in \O^{r_{-j} \times r_j}$. This can always be met because the low-rank tensor $\bcT=\bcC\times_1\bU_1\times_2\cdots\times_m \bU_m$ is invariant up to any rotations: $$\bcC\times_1\bU_1\times_m \bU_m=\left(\bcC\times_{j=1}^m\bO_j\right)\times_1\bU_1\bO_1^\top\cdots\times_m \bU_m\bO_m^\top.$$  

\vspace{5pt}
\noindent \textbf{On the order of Kronecker products in tensor unfolding.} To ease the notation, in our proof, we shall write the mode-$j$ unfolding of tensor $\bcT$ as 
$$\bT_j = \cM_j(\bcC\times_{j=1}^m \bU_j ) = \bU_j\cM_j(\bcC)\left(\otimes_{k\neq j}\bU_{k}^\top\right). $$ 
Here the notation $\otimes_{k\neq j}\bU_{k}^\top$ means the Kronecker products follow the right order from $m$ to $1$, but not the ascending order. Also, when we are interested in the mode-$k$ singular subspace under the  mode-$j$ unfolding, we may write the unfolding as 
$$\bT_j = \bU_j\cM_j(\bcC)\otimes_{k}\bU_k^\top \left(\otimes_{l\neq j,k}\bU_{l}^\top\right). $$ Readers should notice that we write $\otimes_{k}\bU_k^\top$ in front of others just for a clearer expression. In our proof, we often write $\otimes_{k}\bU_k^\top$ with subscript $k$ indicating the position of $\bU_k^\top$ in the Kronecker product, which is consistent with \cite{kolda2009tensor}. 

\vspace{5pt}
\noindent \textbf{Initialization in the inference process.} In our assumptions, we assume the initialization error on the $\linf$ norm. However, the following lemma shows that such an initialization implies the incoherent structure given that $\bcT$ is incoherent, which is critical for our proofs of main results:
\begin{Lemma}[Perturbation bound for initialization]\label{lemma:l2inf-init}
    Suppose $\bcT$ is an incoherent and multilinear rank-$\br$ tensor. $\widehat{\bcT}_0$ is a (not necessarily low rank) estimate of $\bcT$ with $\linf$-norm error bounded by
    \begin{equation*}
        \norm{\widehat{\bcT}_0- \bcT}_{\linf} \le C_1 \sigma \sqrt{\frac{\dmax \log \dmax }{n}}
    \end{equation*}  
for a $C_1>0$, and with multilinear rank-$\br$ higher-order SVD: 
    $$\left(\widehat{\bU}_{1,0},\dots, \widehat{\bU}_{m,0} \right)=\HOSVD_{\br}\left(\widehat{\bcT}_0 \right). $$  
If there also exists a large number $C_\gap>16C_1 $ such that the SNR condition satisfies 
\begin{equation}
    \frac{\lambda_{\min} }{\sigma} \ge C_{\mathsf{gap} }\kappa_0 \sqrt{\frac{d^* \dmax \log \dmax }{n} },
\end{equation}
then we have the $\ell_{2,\infty}$ norm bound for each $j\in[m]$ of the singular subspaces: 
    \begin{equation*}
   \norm{ \cP_{\widehat{\bU}_{j,0}} - \cP_{\bU_j} }_{2,\infty} \le \frac{8 C_1\sigma  }{\lambda_{\min}}\sqrt{\frac{\mu r_j}{d_j} }\sqrt{\frac{d^*\dmax \log \dmax }{n}} \le \frac{8 C_1 }{ C_\gap}\sqrt{\frac{\mu r_j}{d_j} }.
\end{equation*}
Consequently, the $\linf$ norm bound:
    \begin{equation*}
   \norm{ \cP_{\widehat{\bU}_{j,0}} - \cP_{\bU_j} }_{\linf} \le \frac{8 C_1\sigma  }{\lambda_{\min}}{\frac{\mu r_j}{d_j} }\sqrt{\frac{d^*\dmax \log \dmax }{n}}.
\end{equation*}
And also, the spectral norm bound for each $j\in[m]$ is given by: 
    \begin{equation*}
   \norm{ \cP_{\widehat{\bU}_{j,0}} - \cP_{\bU_j} }_{2} \le \frac{8 C_1 \sigma  }{\lambda_{\min}}\sqrt{\frac{d^*\dmax \log \dmax }{n}}.
\end{equation*}


\end{Lemma}
\begin{proof}
    Here we consider $j=1$ for demonstration. Denote $\widehat{\bcT}_0 = {\bcT}+ \widehat{\bcE}_0$, and we write $\bT_1 =\cM_1(\bcT)$, $\widehat{\bE}_{1,0}= \cM_1(\widehat{\bcE}_0)$ for the mode-$1$ tensor unfolding. We use the second-order technique to analyze the perturbation of $\cP_{\widehat{\bU}_{j,0}}$ because of the large dimension of the tensor unfolding $\cM_j({\bcT})\in\R^{d_j\times d^{-}_j }$. 
    Since $\widehat{\bU}_{1,0}$ is the SVD of the following matrix: 
    \begin{equation*}
       \begin{aligned}
            \cM_1( \widehat{\bcT}_0) \cdot \cM_1( \widehat{\bcT}_0)^\top & = \bT_1\bT_1^\top +\widehat{\bE}_{1,0}\bT_1^\top + \bT_1 \widehat{\bE}_{1,0}^\top + \widehat{\bE}_{1,0}\widehat{\bE}_{1,0}^\top \\
            & := \bT_1\bT_1^\top + \widehat{\Delta}_{1,0}.
       \end{aligned}
    \end{equation*}
We first validate the condition to apply the spectral representation: the size of the error $\widehat{\Delta}_{1,0}$ is controlled by 
\begin{equation*}
    \norm{\widehat{\Delta}_{1,0}}_2\le  2\lambda_{\max} C_1 \sigma \sqrt{\frac{\dmax d^* \log \dmax }{n}} + C_1^2 \sigma^2 {\frac{\dmax d^* \log \dmax }{n}} \le \frac{1}{2}\lambda_{\min}^2\cdot 6\kappa_0\frac{\sigma}{\lambda_{\min} } \sqrt{\frac{\dmax d^* \log \dmax }{n}} \le \frac{1}{2}\lambda_{\min}^2.
\end{equation*}
        Thus, we can expand the orthogonal projection of the singular subspaces $\cP_{\widehat{\bU}_{1,0}}$ by representation formula consisting of perturbation error $\widehat{\Delta}_{1,0}$ \citep{xia2021normal}:
    \begin{equation}\label{eq:init-expansion}
        \cP_{\widehat{\bU}_{1,0}} - \cP_{\bU_1} = \sum_{k\ge 1} \cS_{\widehat{\bU}_{1,0}}^{(k)},
    \end{equation}
where
\begin{equation}\label{eq:init-expansion-S}
\cS_{\widehat{\bU}_{1,0}}^{(k)}=\sum_{\mathbf{s}: s_1+\cdots+s_{k+1}=k}(-1)^{1+\tau(\mathbf{s})} \cdot \mathfrak{P}_1^{-s_1} \widehat{\Delta}_{1,0} \mathfrak{P}_1^{-s_2} \widehat{\Delta}_{1,0} \cdots \widehat{\Delta}_{1,0} \mathfrak{P}_1^{-s_{k+1}},
\end{equation}
and $s_1\ge 0, \cdots, s_{k+1}\ge 0$ are non-negative integers with $
\tau(\mathbf{s})=\sum_{j=1}^{k+1} \mathbb{I}\left(s_j>0\right)$. In the series, the power of the operator $\mathfrak{P}_1$ is given by:
\begin{equation*}
    \mathfrak{P}_1^{-s}= \left\{ \begin{array}{lcl} \bU_1\bLambda_1^{-2s}\bU_1^\top
         & \mbox{for} & s>0 \\ 
         \cP_{\bU_1}^\perp  & \mbox{for} & s=0
                \end{array}\right..
\end{equation*}
For any $k_1\in[d_1]$, we can decomposed the error $\norm{\be_{k_1}^{\top}\left( \cP_{\widehat{\bU}_{1,0}} - \cP_{\bU_1}\right)}_2$ using \eqref{eq:init-expansion}:
\begin{equation*}
    \begin{aligned}
    \norm{\be_{k_1}^{\top}\left( \cP_{\widehat{\bU}_{1,0}} - \cP_{\bU_1}\right)}_2 \le \sum_{k\ge 1} \norm{\be_{k_1}^\top \cS_{\widehat{\bU}_{1,0}}^{(k)} }_2.
    \end{aligned}
\end{equation*}
To bound each $\be_{k_1}^\top \cS_{\widehat{\bU}_{1,0}}^{(k)}$, it suffices to bounding the all the terms in \eqref{eq:init-expansion-S} multiplied by $\be_{k_1}$. To this end, we shall consider all the possible situations of $\widehat{\Delta}_{1,0}$ in the series. Since $\widehat{\Delta}_{1,0}=\widehat{\bE}_{1,0}\bT_1^\top + \bT_1 \widehat{\bE}_{1,0}^\top + \widehat{\bE}_{1,0}\widehat{\bE}_{1,0}^\top$, we have the following guarantees on $\widehat{\Delta}_{1,0}$:
\begin{equation}\label{eq:init-delta10-spec}
\begin{gathered}
        \max\left\{ \norm{\cP_{\bU_1}^\perp\widehat{\Delta}_{1,0}\bU_1\bLambda_1^{-2}}_2,\norm{\bLambda_1^{-2}\bU_1^\top\widehat{\Delta}_{1,0}\cP_{\bU_1}^\perp}_2 \right\} \le \frac{\norm{\widehat{\bE}_{1,0}}_2}{\lambda_{\min}} + \frac{\norm{\widehat{\bE}_{1,0}}_2^2}{\lambda_{\min}^2} \le 2 C_1\frac{\sigma}{\lambda_{\min}} \sqrt{\frac{d^*\dmax \log \dmax }{n}} \\
         \norm{\bLambda_1^{-2}\bU_1^\top\widehat{\Delta}_{1,0}\bU_1\bLambda_1^{-2}}_2 \le 2 C_1\frac{\sigma}{\lambda_{\min}} \sqrt{\frac{d^*\dmax \log \dmax }{n}} \cdot \frac{1}{\lambda_{\min}^2} \\ 
         \norm{\cP_{\bU_1}^\perp\widehat{\Delta}_{1,0}\cP_{\bU_1}^\perp}_2 \le C_1\frac{\sigma}{\lambda_{\min}} \sqrt{\frac{d^*\dmax \log \dmax }{n}} \cdot \lambda_{\min}^2
\end{gathered}
\end{equation}
where we use the $\linf$-norm error bounded of $\widehat{\bE}_{1,0}$, and assume that $C_\gap \ge 2 C_1$. We also have the $\ell_{2,\infty}$-norm-related bound:
\begin{equation}\label{eq:init-delta10-2inf}
\begin{gathered}
            \norm{\cP_{\bU_1}^\perp\widehat{\Delta}_{1,0}\cP_{\bU_1}^\perp}_{2,\infty}  \le \norm{\cP_{\bU_1}\widehat{\bE}_{1,0}}_{2,\infty}\norm{\widehat{\bE}_{1,0}}_2 +\norm{\widehat{\bE}_{1,0}}_{2,\infty} \norm{\widehat{\bE}_{1,0}}_{2} \\
             \le  2 C_1 \sqrt{\mu r_1}\frac{\sigma}{\lambda_{\min}} \sqrt{\frac{d^*\dmax \log \dmax }{d_1 n}}\cdot \lambda_{\min}^2 \\
            \norm{\cP_{\bU_1}^\perp\widehat{\Delta}_{1,0}\bU_1\bLambda_1^{-2}}_{2,\infty}  \le \norm{\cP_{\bU_1}^\top\widehat{\bE}_{1,0}\bT_1^\top \cP_{\bU_1}\bLambda_1^{-2} }_{2,\infty} +\norm{\cP_{\bU_1}^\top\widehat{\bE}_{1,0}\widehat{\bE}_{1,0}^\top \cP_{\bU_1}\bLambda_1^{-2} }_{2,\infty} \\
            \le 2 C_1 \sqrt{\mu r_1}\frac{\sigma}{\lambda_{\min}} \sqrt{\frac{d^*\dmax \log \dmax }{d_1 n}}
\end{gathered}   
\end{equation}
Now, we consider two cases: $s_1=0$ or $s_1\ge1$ for each order $k$ in $\cS_{\widehat{\bU}_{1,0}}^{(k)}$.

\noindent \textbf{(i)} If $s_1=0$: each term in  \eqref{eq:init-expansion-S} multiplied by $\be_{k_1}$ is given by
\begin{equation}\label{eq:l2inf-decomp1}
    \begin{aligned}
    \norm{\be_{k_1}^{\top}\mathfrak{P}_1^{-s_1} \widehat{\Delta}_{1,0}  \cdots \widehat{\Delta}_{1,0} \mathfrak{P}_1^{-s_{k+1}}}_2 & = \norm{\be_{k_1}^{\top} \cP_{\bU_1}^\perp \widehat{\Delta}_{1,0} \cdots \widehat{\Delta}_{1,0} \mathfrak{P}_1^{-s_{k+1}}}_2 \\
    \end{aligned}
\end{equation}
Combining \eqref{eq:init-delta10-spec} and \eqref{eq:init-delta10-2inf}, we conclude that no matter $s_2=0$ or not, we always have
\begin{equation*}
    \norm{\be_{k_1}^{\top}\mathfrak{P}_1^{-s_1} \widehat{\Delta}_{1,0}  \cdots \widehat{\Delta}_{1,0} \mathfrak{P}_1^{-s_{k+1}}}_2 \le \sqrt{\frac{\mu r_1}{d_1} } \left(\frac{2 C_1\sigma  }{\lambda_{\min}}\sqrt{\frac{d^*\dmax \log \dmax }{n}}\right)^{k},
\end{equation*}
because the order of $\widehat{\Delta}_{1,0}$ always matches the order of $\bLambda_1^{-2}$ in each term.

\noindent \textbf{(ii)} If $s_1\ge 1$: the term in  \eqref{eq:init-expansion-S} multiplied by $\be_{k_1}$ can be controlled by 
\begin{equation}\label{eq:l2inf-decomp2}
    \begin{aligned}
     &\norm{\be_{k_1}^{\top}\mathfrak{P}_1^{-s_1} \widehat{\Delta}_{1,0}  \cdots \widehat{\Delta}_{1,0} \mathfrak{P}_1^{-s_{k+1}}}_2
    = \norm{\be_{k_1}^{\top}\bU_1 \bLambda_1^{-2 s_1 }\bU_1^\top \widehat{\Delta}_{1,0}  \cdots \widehat{\Delta}_{1,0} \mathfrak{P}_1^{-s_{k+1}}}_2  \\
    &\le \norm{\cP_{\bU_1} }_{2,\infty} \norm{\bLambda_1^{-2 s_1 }\bU_1^\top \widehat{\Delta}_{1,0}  \cdots \widehat{\Delta}_{1,0} \mathfrak{P}_1^{-s_{k+1}} }_2 \le \sqrt{\frac{\mu r_1}{d_1} }\left(\frac{2 C_1 \sigma  }{\lambda_{\min}}\sqrt{\frac{d^*\dmax \log \dmax }{n}}\right)^{k}.
    \end{aligned}
\end{equation}
Thus, we know that $\norm{\be_{k_1}^\top \cS_{\widehat{\bU}_{1,0}}^{(k)}}_2$ is controlled by
\begin{equation*}
    \norm{\be_{k_1}^\top \cS_{\widehat{\bU}_{1,0}}^{(k)} }_2 \le \sum_{\mathbf{s}: s_1+\cdots+s_{k+1}=k} \sqrt{\frac{\mu r_1}{d_1} } \left(\frac{2 C_1\sigma  }{\lambda_{\min}}\sqrt{\frac{d^*\dmax \log \dmax }{n}}\right)^{k} \le \sqrt{\frac{\mu r_1}{d_1} } \left(\frac{8 C_1\sigma  }{\lambda_{\min}}\sqrt{\frac{d^*\dmax \log \dmax }{n}}\right)^{k} 
\end{equation*}
Given that $C_\gap> 16C_1$, we can sum up all the series for $k$ from $1$ to $\infty$ can derive the following error bound
\begin{equation*}
    \begin{aligned}
    \norm{\be_{k_1}^{\top}\left( \cP_{\widehat{\bU}_{1,0}} - \cP_{\bU_1}\right)}_2 \le \sum_{k\ge 1} \norm{\be_{k_1}^\top \cS_{\widehat{\bU}_{1,0}}^{(k)} }_2 \le  \frac{8 C_1 \sigma  }{\lambda_{\min}}\sqrt{\frac{\mu r_1}{d_1} }\sqrt{\frac{d^*\dmax \log \dmax }{n}},
    \end{aligned}
\end{equation*}
for all $k_1\in[d_1]$. Thus, we finish the proof of the $\ell_{2,\infty}$ bound. For the $\ell_{\infty}$ norm bound, we only need to notice that the right-hand side of each $\cS_{\widehat{\bU}_{1,0}}^{(k)}$ is also $\mu$-incoherent, which can be shown following \eqref{eq:init-delta10-2inf}. We can thus have the   $\ell_{\infty}$ norm bound.

The spectral norm bound is derived following the same fashion. Notice that, from \eqref{eq:init-delta10-spec} we have
\begin{equation*}
    \norm{\mathfrak{P}_1^{-s_1} \widehat{\Delta}_{1,0}  \cdots \widehat{\Delta}_{1,0} \mathfrak{P}_1^{-s_{k+1}}}_2 \le \left(\frac{2 C_1\sigma  }{\lambda_{\min}}\sqrt{\frac{d^*\dmax \log \dmax }{n}}\right)^{k}.
\end{equation*}
Summing up all the possible series with $\mathbf{s}: s_1+\cdots+s_{k+1}=k$, for $k$ from $1$ to $\infty$ we have the desired bound.
\end{proof}

\subsection{Proof of Theorem \ref{thm:lf-inference-popvar}}\label{apx:sec:clt-indp}
According to Algorithm \ref{alg:debias-powerit}, we have 
\begin{equation}
    \begin{aligned}
        \widehat \bcT_{\ubs} & =  \widehat \bcT_{\init } +\frac{d^*}{n} \sum_{i=1}^{n}\left(Y_i - \left\langle \widehat\bcT_{\init }, \bcX_i \right\rangle\right)\cdot \bcX_i \\
        & = \bcT + \widehat \bcT_{\init}- \bcT - \frac{d^*}{n} \sum_{i=1}^{n} \left\langle\widehat \bcT_{\init}- \bcT ,\bcX_i \right\rangle\bcX_i + \frac{d^*}{n} \sum_{i=1}^{n} \xi_i\bcX_i \\
        &  = \bcT +  \underbrace{\bcT^{\vartriangle}\ - \frac{d^*}{n} \sum_{i=1}^{n}\left\langle\bcT^{\vartriangle} ,\bcX_i \right\rangle\bcX_i}_{\bcE_{\init} } + 
        \underbrace{\frac{d^*}{n} \sum_{i=1}^{n} \xi_i\bcX_i }_{\bcE_{\rn} } ,
    \end{aligned}
\end{equation}
where we denote $\bcT^{\vartriangle} = \widehat \bcT_{\init}- \bcT$ as the error tensor introduced by initialization, and decouple the error of the debiased estimator $\widehat \bcT_{\ubs}$ into two parts: (i) $\bcE_{\init}$, which represents the initialization error; and (ii) $\bcE_{\rn} $,  which represents the effect of random sampling and noises. At a high level, our proof is going to show that, the asymptotic normality of the linear form is mainly contributed by the random error $\bcE_{\rn}$ while the initialization error $\bcE_{\init}$ is generally negligible.  To this end, we need a fine-grained analysis of the one-step power iteration, especially the singular subspaces from the SVD.  

\subsubsection{Analysis of perturbation in SVD}

For each mode $j$, the singular subspace is calculated by 
\begin{equation}\label{eq:svd-power-itr-modej}
\begin{aligned}
        \widehat{\bU}_{j,1} & = \SVD_{r_j}\left(\mathcal{M}_j\left(\widehat \bcT_{\ubs}\times_1 \widehat{\mathbf{U}}_{1,0}^\top \cdots \times_{j-1} \widehat{\mathbf{U}}_{j-1,0}^\top \times_{j+1} \widehat{\mathbf{U}}_{i+1,0}^\top \cdots \times_{m} \widehat{\mathbf{U}}_{m,0}^\top\right)\right)  \\
        & = \SVD_{r_j} \left( \mathcal{M}_j(\widehat \bcT_{\ubs}) \left(\widehat{\mathbf{U}}_{m,0} \otimes \cdots \otimes \widehat{\mathbf{U}}_{j+1,0} \otimes \widehat{\mathbf{U}}_{j-1,0} \otimes \cdots \otimes \widehat{\mathbf{U}}_{1,0}\right)  \right).
\end{aligned}
\end{equation}
Without loss of generality, in the following proof, we shall consider $j=1$ for demonstration. Still, we use the second-order argument by examining the SVD of the symmetrized matrix in \eqref{eq:svd-power-itr-modej}:
\begin{equation}\label{eq:svd-U1-decomp}
    \begin{aligned}
        \widehat{\bU}_{1,1} & =\SVD_{r_1} \left( \mathcal{M}_1(\widehat \bcT_{\ubs}) \left(\widehat{\mathbf{U}}_{m,0} \otimes \cdots \otimes \widehat{\mathbf{U}}_{2,0} \right)  \right)  \\
        =  &\SVD_{r_1} \left( \left( \bT_{1} + \bE_{1,\init}+\bE_{1,\rn} \right) \left( \otimes_{j\neq 1}\cP_{\widehat{\mathbf{U}}_{j,0}} \right)\left( \bT_{1} + \bE_{1,\init}+\bE_{1,\rn} \right)^\top  \right) \\
        = & \SVD_{r_1} \left( \bT_{1}\bT_{1}^\top + \underbrace{\bT_{1} \left(\otimes_{j\neq 1}\cP_{\widehat{\mathbf{U}}_{j,0}} -\otimes_{j\neq 1}\cP_{{\mathbf{U}}_{j}}  \right)\bT_{1}^\top}_{ \mathfrak{B}_1 }  \right. \\
        & 
        + \underbrace{\left(\bE_{1,\init}+\bE_{1,\rn} \right)\left( \otimes_{j\neq 1}\cP_{\widehat{\mathbf{U}}_{j,0}} \right)\bT_{1}^\top}_{\mathfrak{B}_2} 
        +  
        \underbrace{\bT_{1}\left( \otimes_{j\neq 1}\cP_{\widehat{\mathbf{U}}_{j,0}} \right)\left(\bE_{1,\init}+\bE_{1,\rn} \right)^\top}_{\mathfrak{B}_2^\top} \\
        & + \left. \underbrace{\left(\bE_{1,\init}+\bE_{1,\rn} \right) \left( \otimes_{j\neq 1}\cP_{\widehat{\mathbf{U}}_{j,0}} \right)\left(\bE_{1,\init}+\bE_{1,\rn} \right)^\top}_{\mathfrak{B}_3} \right) .  \\   
    \end{aligned}
\end{equation}
Here $\bE_{1,\init}=\cM_1(\bcE_{\init})$ and $\bE_{1,\rn}=\cM_1(\bcE_{\rn})$. To ensure that the perturbation is within a fine region so that we can further give the exact expression of the SVD error, we shall check that whether $\norm{\mathfrak{B}_1 +\mathfrak{B}_2 +\mathfrak{B}_2^\top +\mathfrak{B}_3 }_2\le \frac{1}{2}\lambda_{\min}^2$ holds. We analyze this error term by term:

\noindent \textbf{(i) Term} $\mathfrak{B}_1$ \textbf{in} \eqref{eq:svd-U1-decomp}. For $\mathfrak{B}_1$, we have
\begin{equation}
    \begin{aligned}
        \mathfrak{B}_1=&\bU_1\cM_{1}(\bcC)\left(\otimes_{j\neq 1}\bU_j^\top\right)\left(\otimes_{j\neq 1}\cP_{\widehat{\mathbf{U}}_{j,0}} -\otimes_{j\neq 1}\cP_{{\mathbf{U}}_{j}}  \right)\left(\otimes_{j\neq 1}\bU_j\right) \cM_{1}(\bcC)^\top \bU_1^\top \\
        = &\bU_1\bLambda_1 \bV_1^\top\left[\otimes_{j\neq 1} \left( \bI_{r_1}+ \bU_j^\top(\cP_{\widehat{\mathbf{U}}_{j,0}}-\cP_{{\mathbf{U}}_{j}})\bU_j \right) \right]  \bV_1 \bLambda_1 \bU_1^\top,
    \end{aligned}
\end{equation}
and thus,
\begin{equation}\label{eq:B1-up1}
\begin{aligned}
        \norm{\mathfrak{B}_1}_2 & \le \kappa_0^2 \lambda_{\min}^2 \norm{\otimes_{j\neq 1} \left( \bI_{r_1}+ \bU_j^\top(\cP_{\widehat{\mathbf{U}}_{j,0}}-\cP_{{\mathbf{U}}_{j}})\bU_j \right) }_2 \\
        & \le \kappa_0^2 \lambda_{\min}^2 \left(\sum_{j=1}^{m}\norm{\bU_j^\top(\cP_{\widehat{\mathbf{U}}_{j,0}}-\cP_{{\mathbf{U}}_{j}})\bU_j}_2 + 
        \underbrace{ 2^m \max_{j}\norm{\bU_j^\top(\cP_{\widehat{\mathbf{U}}_{j,0}}-\cP_{{\mathbf{U}}_{j}})\bU_j}_2^2}_{\text{higher order term}}
        \right)
\end{aligned}
\end{equation}
by the binomial expansion. Here, the higher order term bound is obtained by considering all the terms with more than one $\bU_j^\top(\cP_{\widehat{\mathbf{U}}_{j,0}}-\cP_{{\mathbf{U}}_{j}})\bU_j$ and use the fact that $\norm{ \cP_{\widehat{\bU}_{j,0}} - \cP_{\bU_j} }_{2}\le 1$. Now, we derive the bound of $\norm{\bU_j^\top(\cP_{\widehat{\mathbf{U}}_{j,0}}-\cP_{{\mathbf{U}}_{j}})\bU_j}_2$: According to the expansion of the initial singular subspaces in \eqref{eq:init-expansion-S}, when $k=1$, the first order term in $\cP_{\widehat{\bU}_{j,0}} - \cP_{\bU_j}$ is 
\begin{equation*}
   \cS_{\widehat{\bU}_{1,0}}^{(1)} = \bU_j\bLambda_j^{-2}\bU_j^\top \widehat{\Delta}_{j,0}\cP_{\bU_j}^\perp + \cP_{\bU_j}^\perp\widehat{\Delta}_{j,0}\bU_j\bLambda_j^{-2}\bU_j^\top,
\end{equation*}
which will vanish since $\bU_j^\top \cS_{\widehat{\bU}_{1,0}}^{(1)}\bU_j=\mathbf{0}$. Following a similar argument as in the proof of Lemma \ref{lemma:l2inf-init}, we have 
\begin{equation}
    \norm{\bU_j^\top(\cP_{\widehat{\mathbf{U}}_{j,0}}-\cP_{{\mathbf{U}}_{j}})\bU_j}_2 \le  \left(\frac{8 C_1 \sigma  }{\lambda_{\min}}\sqrt{\frac{d^*\dmax \log \dmax }{n}}\right)^2.
\end{equation}
By assuming the SNR condition, it is clear that when $\frac{8\cdot 2^m C_1 \sigma  }{\lambda_{\min}}\sqrt{\frac{d^*\dmax \log \dmax }{n}}\le 1$, the term $\mathfrak{B}_1$ in \eqref{eq:B1-up1} has
\begin{equation}\label{eq:B1-bound}
    \norm{\mathfrak{B}_1}_2 \le \lambda_{\min}^2 \left(\frac{8 C_1 \kappa_0 \sigma\sqrt{2m}  }{\lambda_{\min}}\sqrt{\frac{d^*\dmax \log \dmax }{n}}\right)^2.
\end{equation}

\noindent \textbf{(ii) Term} $\mathfrak{B}_2$ \textbf{in} \eqref{eq:svd-U1-decomp}. To check $\mathfrak{B}_2$, we only need to control the term 
$\bE_{1,\init}\left( \otimes_{j\neq 1} \widehat{\mathbf{U}}_{j,0} \right)$ and $\bE_{1,\rn}\left( \otimes_{j\neq 1} \widehat{\mathbf{U}}_{j,0} \right) $. We consider the case when the initialization is independent of the sampling, i.e., $\bcT^{\vartriangle}$ and $\{\widehat{\mathbf{U}}_{j,0}\}_{j=1}^{m} $ are independent of $\left\{(\bcX_i,\xi_i)\right\}_{i=1}^n$. In this case, 
\begin{equation*}
    \begin{aligned}
        \bE_{1,\init}\left(\otimes_{j\neq 1} \widehat{\mathbf{U}}_{j,0} \right)= \frac{d^*}{n} \sum_{i=1}^{n}\left\langle \bT^{\vartriangle}_1, \cM_1(\bcX_i)\right\rangle \cM_1(\bcX_i)\left(\otimes_{j\neq 1} \widehat{\mathbf{U}}_{j,0} \right)-\bT^{\vartriangle}_1\left(\otimes_{j\neq 1} \widehat{\mathbf{U}}_{j,0} \right).
    \end{aligned}
\end{equation*}
where $\bT^{\vartriangle}_1=\cM_1(\bcT^{\vartriangle})$. By Lemma \ref{lemma:l2inf-init}, we know that under the statistically optimal SNR condition, $\{\widehat{\mathbf{U}}_{j,0}\}$ are also incoherent with the incoherence parameter bounded by $2\mu$. Thus, it is clear that the following inequalities hold:
\begin{equation}\label{eq:E-init-Ber-condition}
    \begin{gathered}
        \norm{\left\langle \bT^{\vartriangle}_1, \cM_1(\bcX_i)\right\rangle \cM_1(\bcX_i)\left(\otimes_{j\neq 1} \widehat{\mathbf{U}}_{j,0} \right)}_2\le C_1 \sigma \sqrt{\frac{\dmax \log \dmax }{n}}\sqrt{\frac{\mu^{m-1} r_{-1}}{d_{-1}}} \\
        \norm{\E \left\langle \bT^{\vartriangle}_1, \cM_1(\bcX_i)\right\rangle^2 \cM_1(\bcX_i)\left(\otimes_{j\neq 1} \cP_{\widehat{\mathbf{U}}_{j,0}} \right)\cM_1(\bcX_i)^\top }_2 \le C_1^2 \sigma^2 \frac{\dmax \log \dmax }{n} \frac{(2\mu)^{m-1} r_{-1}}{d_{}^*}  \\
        \norm{\E \left\langle \bT^{\vartriangle}_1, \cM_1(\bcX_i)\right\rangle^2 \left(\otimes_{j\neq 1} \cP_{\widehat{\mathbf{U}}_{j,0}} \right)\cM_1(\bcX_i)^\top\cM_1(\bcX_i)\left(\otimes_{j\neq 1} \cP_{\widehat{\mathbf{U}}_{j,0}} \right) }_2 \le C_1^2 \sigma^2 \frac{\dmax \log \dmax }{n} \frac{1}{d_{-1}}.
    \end{gathered}
\end{equation}
According to the matrix Bernstein inequality \citep{koltchinskii2011remark,tropp2012user}, we have the following bound of $\bE_{1,\init}\left(\otimes_{j\neq 1} \widehat{\mathbf{U}}_{j,0} \right)$:
\begin{equation}\label{eq:E-init-Ber}
    \begin{aligned}
        \norm{\bE_{1,\init}\left(\otimes_{j\neq 1} \widehat{\mathbf{U}}_{j,0} \right)}_2 &\le C C_1 \sigma \left(\frac{d^*}{n} \sqrt{\frac{\dmax \log \dmax }{n}}\sqrt{\frac{\mu^{m-1} r_{-1}}{d_{-1}}}\cdot m\log \dmax  +  \sqrt{\frac{ m (2\mu)^{m-1} d_1 \dmax d^*\log^2 \dmax }{n^2}} \right) \\
        & \le C C_1 \sigma \sqrt{\frac{ m r^* (2\mu)^{m-1} d_1 \dmax d^*\log^2 \dmax }{\rmin r_1 n^2}} 
    \end{aligned}
\end{equation}
with probability at least $1-\dmax^{-3 m}$, given that $n\ge m \rmax\dmax \log \dmax$.  We now use the same technique to control $\bE_{1,\rn}\left( \otimes_{j\neq 1} \widehat{\mathbf{U}}_{j,0} \right) $. Here we treat $r^*\gtrsim \rmax^2$ in general. Notice that, $\bE_{1,\rn}\left( \otimes_{j\neq 1} \widehat{\mathbf{U}}_{j,0} \right) $ has the  expression
\begin{equation*}
    \bE_{1,\rn}\left( \otimes_{j\neq 1} \widehat{\mathbf{U}}_{j,0} \right) =  \frac{d^*}{n} \sum_{i=1}^{n} \xi_i \cM_1(\bcX_i)\left( \otimes_{j\neq 1} \widehat{\mathbf{U}}_{j,0} \right),
\end{equation*}
with the following inequalities
\begin{equation}\label{eq:E-rn-Ber-condition}
    \begin{gathered}
       \norm{ \norm{\xi_i \cM_1(\bcX_i)\left( \otimes_{j\neq 1} \widehat{\mathbf{U}}_{j,0} \right)}_2 }_{\psi_2} \le C \sigma \sqrt{\frac{\mu^{m-1} r_{-1}}{d_{-1}}}\\
       \norm{\E \xi_i^2\cM_1(\bcX_i)\left(\otimes_{j\neq 1} \cP_{\widehat{\mathbf{U}}_{j,0}} \right)\cM_1(\bcX_i)^\top   }_2 \le \sigma^2 \frac{(2\mu)^{m-1} r_{-1}}{d_{}^*}   \\
       \norm{\E \xi_i^2 \left(\otimes_{j\neq 1} \cP_{\widehat{\mathbf{U}}_{j,0}} \right)\cM_1(\bcX_i)^\top\cM_1(\bcX_i)\left(\otimes_{j\neq 1} \cP_{\widehat{\mathbf{U}}_{j,0}} \right)}_2 \le  \sigma^2  \frac{1}{d_{-1}}.
    \end{gathered}
\end{equation}
Using the matrix Bernstein inequality similar as \eqref{eq:E-init-Ber}, we have 
\begin{equation}\label{eq:E-rn-Ber}
    \begin{aligned}
        \norm{\bE_{1,\rn}\left(\otimes_{j\neq 1} \widehat{\mathbf{U}}_{j,0} \right)}_2 &\le C \sigma \left(\frac{d^*}{n}\sqrt{\frac{\mu^{m-1} r_{-1}}{d_{-1}}}\cdot m\log \dmax  +  \sqrt{\frac{ m (2\mu)^{m-1} d_1  d^*\log \dmax }{n}} \right) \\
        & \le C \sigma \sqrt{\frac{ m (2\mu)^{m-1} r^* d_1  d^*\log \dmax }{\rmin r_1 n}} ,
    \end{aligned}
\end{equation}
with probability at least $1-\dmax^{-3 m}$, given that $n\ge m\rmax\dmax \log \dmax$. Combining \eqref{eq:E-init-Ber} and \eqref{eq:E-rn-Ber}, we know that with probability at least $1-2\dmax^{-3 m}$, the term $\mathfrak{B}_2$ can be controlled by
\begin{equation}\label{eq:B2-bound}
    \norm{\mathfrak{B}_2}_2\le C \kappa_0 \sigma \sqrt{\frac{ m (2\mu)^{m-1}r^* d_1  d^*\log \dmax }{\rmin r_1 n}} \cdot\lambda_{\min}
\end{equation}

\noindent \textbf{(iii) Term} $\mathfrak{B}_3$ \textbf{in} \eqref{eq:svd-U1-decomp}. Since the third term $\mathfrak{B}_3$ can be controlled by
\begin{equation*}
   \norm{\mathfrak{B}_3}_2\le \norm{\left(\bE_{1,\init}+\bE_{1,\rn} \right) \left(\otimes_{j\neq 1} \widehat{\mathbf{U}}_{j,0}  \right)}_2^2,
\end{equation*}
we can again use the bound \eqref{eq:E-init-Ber} and \eqref{eq:E-rn-Ber} to derive the following inequality:
\begin{equation}\label{eq:B3-bound}
    \norm{\mathfrak{B}_3}_2\le C \sigma^2 \frac{ m (2\mu)^{m-1}r^* d_1  d^*\log \dmax }{\rmin r_1 n}.
\end{equation}
By \eqref{eq:B1-bound},\eqref{eq:B2-bound}, and \eqref{eq:B3-bound}, we know that when the SNR condition is large enough, i.e.,
\begin{equation*}
     \frac{\lambda_{\min} }{\sigma} \ge C_{\mathsf{gap} } C_1\kappa_0\sqrt{\frac{m (2\mu)^{m-1} r^* d^* \dmax \log \dmax }{n} },
\end{equation*}
we have the error perturbation bound $\norm{\mathfrak{B}_1 +\mathfrak{B}_2 +\mathfrak{B}_2^\top +\mathfrak{B}_3 }_2\le \frac{1}{2}\lambda_{\min}^2$.
\subsubsection{Representation of SVD operator}
As long as the perturbation error in \eqref{eq:svd-U1-decomp} can be controlled by $\frac{1}{2}\lambda_{\min}^2$, we are able to give the explicit expression of $\cP_{\widehat{\bU}_{1,1}}$ using the representation formula. This argument can be extended to each $j\in [m]$, but here we still consider $j=1$ for demonstration. Defien $\Delta_{1,1} = \mathfrak{B}_1 +\mathfrak{B}_2 +\mathfrak{B}_2^\top +\mathfrak{B}_3$.  Again, we expand $\cP_{\widehat{\bU}_{1,0}}$ by representation formula in \cite{xia2021normal}:
    \begin{equation}\label{eq:U1-expansion}
        \cP_{\widehat{\bU}_{1,1}} - \cP_{\bU_1} = \sum_{k\ge 1} \cS_{\widehat{\bU}_{1,1}}^{(k)},
    \end{equation}
where
\begin{equation}\label{eq:U1-expansion-S}
\cS_{\widehat{\bU}_{1,1}}^{(k)}=\sum_{\mathbf{s}: s_1+\cdots+s_{k+1}=k}(-1)^{1+\tau(\mathbf{s})} \cdot \mathfrak{P}_1^{-s_1} {\Delta}_{1,1} \mathfrak{P}_1^{-s_2} {\Delta}_{1,1} \cdots {\Delta}_{1,1} \mathfrak{P}_1^{-s_{k+1}},
\end{equation}
and $s_1\ge 0, \cdots, s_{k+1}\ge 0$ are non-negative integers with $
\tau(\mathbf{s})=\sum_{j=1}^{k+1} \mathbb{I}\left(s_j>0\right)$. According to the expansion, when $k=1$, the first order term in $\cP_{\widehat{\bU}_{1,1}} - \cP_{\bU_1}$ is 
\begin{equation*}
   \cS_{\widehat{\bU}_{1,1}}^{(1)} = \bU_1\bLambda_1^{-2}\bU_1^\top {\Delta}_{1,1}\cP_{\bU_1}^\perp + \cP_{\bU_1}^\perp {\Delta}_{1,1}\bU_1\bLambda_1^{-2}\bU_1^\top,
\end{equation*}
By analyzing $\{\mathfrak{B}_i\}_{i=1}^3$ in \eqref{eq:svd-U1-decomp}, we know that $\norm{\Delta_{1,1}}_2\cdot \norm{\Lambda^{-2}}_2\le  C\frac{ C_1 \kappa_0 \sigma\sqrt{2m}  }{\lambda_{\min}}\sqrt{\frac{(2\mu)^{m-1} d^*\dmax \log \dmax }{n}}\le \frac{1}{8}$ given the SNR condition. Therefore, It is clear that all the higher order terms of $\cS_{\widehat{\bU}_{1,1}}^{(k)}$ for $k\ge 2$ can be dominated by  $\cS_{\widehat{\bU}_{1,1}}^{(1)}$. In the following discussion, we will show that for any $j\in [m]$, the higher-order series $\sum_{k\ge 2}\cS_{\widehat{\bU}_{j,1}}^{(k)}$ will be negligible when compared to $ \cS_{\widehat{\bU}_{j,1}}^{(1)}$.

To analyze the behavior of test statistic, we need the perturbation bound of $\widehat{\bU}_{j,1}$. Actually, we can show that after one-step power iteration, the singular subspaces still preserve the statistically optimal estimation error rates. We describe the corresponding error rates in the following Lemma \ref{lemma:l2inf-powerit}:
\begin{Lemma}[Perturbation bound after power iteration]\label{lemma:l2inf-powerit} 
Under the assmuptions of Theorem \ref{thm:lf-inference-popvar}, after the power iteration, the following perturbation bounds
\begin{equation*}
    \begin{gathered}
         \norm{\cP_{\widehat{\bU}_{j,1}} - \cP_{\bU_j}}_{2,\infty} \le  \frac{C C_1\sigma  }{\lambda_{\min}}\sqrt{\frac{ m (2\mu)^{m-1} r_{-j} d_j  d^*\log \dmax }{n}}\cdot\sqrt{\frac{\mu r_j}{d_j}} \\
         \norm{\cP_{\widehat{\bU}_{j,1}} - \cP_{\bU_j}}_{2} \le \frac{C C_1\sigma  }{\lambda_{\min}}\sqrt{\frac{ m (2\mu)^{m-1} r^* d_j  d^*\log \dmax }{ \rmin r_j n}} \\
         \norm{\cP_{\widehat{\bU}_{j,1}} - \cP_{\bU_j}}_{\linf}\le \frac{C C_1\sigma  }{\lambda_{\min}}\sqrt{\frac{ m (2\mu)^{m-1} r_{-j} d_j  d^*\log \dmax }{n}}\cdot\frac{\mu r_j}{d_j},
    \end{gathered}
\end{equation*}
hold uniformly for all $j\in[m]$ with probability at least $1-3m^2 \dmax^{3m}$.
\end{Lemma}
Moreover, a closer look at (a part of) $\cS_{\widehat{\bU}_{j,1}}^{(1)}$ for $j=1$ gives that 
\begin{equation}\label{eq:U1-order1-expansion}
\begin{aligned}
        &\cP_{\bU_1}^\perp {\Delta}_{1,1}\bU_1\bLambda_1^{-2}\bU_1^\top =  \cP_{\bU_1}^\perp (\mathfrak{B}_2 + \mathfrak{B}_3 )\bU_1\bLambda_1^{-2}\bU_1^\top  \\
       = & \cP_{\bU_1}^\perp(\bE_{1,\rn} +\bE_{1,\init})\left( \otimes_{j\neq 1}\cP_{\widehat{\mathbf{U}}_{j,0}} \right)\bT_{1}^\top \bU_1\bLambda_1^{-2}\bU_1^\top  +\cP_{\bU_1}^\perp \mathfrak{B}_3 \bU_1\bLambda_1^{-2}\bU_1^\top  \\
         = & \underbrace{\cP_{\bU_1}^\perp\bE_{1,\rn} \left( \otimes_{j\neq 1}\bU_{j} \right)\bV_1 \bLambda_1^{-1}\bU_1^\top}_{\text{main term}} \\
        &\quad + \underbrace{\cP_{\bU_1}^\perp\bE_{1,\rn} \left( \otimes_{j\neq 1}\cP_{\widehat{\mathbf{U}}_{j,0}} -\otimes_{j\neq 1}\cP_{{\mathbf{U}}_{j}} \right)\left( \otimes_{j\neq 1}\bU_{j} \right)\bV_1 \bLambda_1^{-1}\bU_1^\top + \cP_{\bU_1}^\perp\left( \mathfrak{B}_3 +\bE_{1,\init}\left( \otimes_{j\neq 1}\cP_{\widehat{\mathbf{U}}_{j,0}} \right)\bT_{1}^\top \right)\bU_1\bLambda_1^{-2}\bU_1^\top}_{\text{remainder term}}.
\end{aligned}
\end{equation}
 It can be observed that the main term in \eqref{eq:U1-order1-expansion} is the sum of i.i.d. random variables, and it will converge to the asymptotic normal distribution. The remainder term in \eqref{eq:U1-order1-expansion}, as we will show later, will be dominated by the main term and thus vanish. We point out that it is the main term in \eqref{eq:U1-order1-expansion} that contributes to the vast majority of the asymptotic theory in the test statistic $\widehat{\bcT}$. The higher-order series $\sum_{k\ge 2}\cS_{\widehat{\bU}_{j,1}}^{(k)}$ and the remainder term in \eqref{eq:U1-order1-expansion} will all vanish by proper normalization in the inference process. For brevity, we denote the (full symmetric) main term and the remainder term in \eqref{eq:U1-order1-expansion} as $\cS_{\widehat{\bU}_{j,1}}^{(\main)}$ and $\cS_{\widehat{\bU}_{j,1}}^{(\rem)}$ for each mode $j$.

\subsubsection{Test statistic from low-rank projection}
With the representation of the SVD operators, we are able to give the exact decomposition of the test statistic that leads to the asymptotic theory. Recall that the test statistic is calculated by $\widehat{\bcT}=\widehat \bcT_{\ubs} \times_1\cP_{\widehat{\bU}_{1,1} }\cdots \times_m\cP_{\widehat{\bU}_{m,1} } $. Based on this, it can be shown that

\begin{equation}\label{eq:hatT-decomp-Erem}
    \widehat{\bcT} - \bcT = \left(\widehat{\bcT}_\ubs- \bcT\right)\times_{j=1}^m \cP_{\bU_{j} } + \sum_{j=1}^m \bcT\times_{k\neq j}\cP_{\bU_{k} } \times_{j}\left( \cP_{\widehat{\bU}_{j,1}} - \cP_{\bU_j}\right) + \bcE_{\rem},
\end{equation}
where $\bcE_{\rem}$ is the higher order remainder term which contains at most $2^{m+1}$ single terms, where each single term contains at least 2 differences in the tensor production (a difference can be either $\left(\widehat{\bcT}_\ubs- \bcT\right)$ or $\left( \cP_{\widehat{\bU}_{j,1}} - \cP_{\bU_j}\right)$). We will show that $\bcE_{\rem}$ will vanish and have little influence on our uncertainty quantification. 
\begin{equation}\label{eq:test-decomp}
\begin{aligned}
        \left\langle  \widehat{\bcT} - \bcT,\bcI  \right\rangle 
        &  = \left\langle\left(\widehat{\bcT}_\ubs- \bcT\right)\times_{j=1}^m \cP_{\bU_{j} },  \bcI \right\rangle + \sum_{j=1}^m \left\langle \bcT\times_{k\neq j}\cP_{\bU_{k} } \times_{j}\left( \cP_{\widehat{\bU}_{j,1}} - \cP_{\bU_j}\right), \bcI \right\rangle \\
        &  \quad +  \left\langle \bcE_{\rem},\bcI  \right\rangle \\
        & = \underbrace{\left\langle\bcE_{\rn}\times_{j=1}^m \cP_{\bU_{j} },  \bcI \right\rangle + \sum_{j=1}^m \left\langle \bcT\times_{k\neq j}\cP_{\bU_{k} } \times_{j}\cS_{\widehat{\bU}_{j,1}}^{(\main)}, \bcI \right\rangle}_{\frE_1} + \underbrace{\left\langle\bcE_{\init}\times_{j=1}^m \cP_{\bU_{j} },  \bcI \right\rangle}_{\frE_2}  \\
          &   \quad + \underbrace{\sum_{j=1}^m \left\langle \bcT\times_{k\neq j}\cP_{\bU_{k} } \times_{j}\cS_{\widehat{\bU}_{j,1}}^{(\rem)}, \bcI \right\rangle }_{\frE_3}
          +\underbrace{\sum_{j=1}^m \left\langle \bcT\times_{k\neq j}\cP_{\bU_{k} } \times_{j}\left(\sum_{l\ge 2}\cS_{\widehat{\bU}_{j,1}}^{(l)}\right), \bcI \right\rangle}_{\frE_4} \\
          &   \quad  +  \underbrace{\left\langle \bcE_{\rem},\bcI  \right\rangle}_{\frE_5}.
\end{aligned}
\end{equation}
Here, $\cS_{\widehat{\bU}_{j,1}}^{(\main)}$ and $\cS_{\widehat{\bU}_{j,1}}^{(\rem)}$ represent the (symmetric) main term and the remainder term in \eqref{eq:U1-order1-expansion} for each mode $j$. Notice that we only need to consider $\cS_{\widehat{\bU}_{j,1}}^{(\main)}$ and $\cS_{\widehat{\bU}_{j,1}}^{(\rem)}$ as a part of $\cS_{\widehat{\bU}_{j,1}}^{(1)}$ because the other part of $\cS_{\widehat{\bU}_{j,1}}^{(1)}$, i.e., $  \bU_j\bLambda_j^{-2}\bU_j^\top{\Delta}_{j,1}\cP_{\bU_j}^{\perp}$, will be canceled when multiplied by $\bcT$ in the tensor production. We are going to show that $\frE_1$ in \eqref{eq:test-decomp} is the sum of i.i.d. random variables and it will converge to the asymptotic normal distribution, and $\frE_2$ to $\frE_5$  in \eqref{eq:test-decomp} will all be dominated by the scale of $\frE_1$ after the normalization. We now carefully discuss these terms.

\noindent \textbf{(i) $\frE_1$ in} \eqref{eq:test-decomp}. For each $j$, we can write $\left\langle \bcT\times_{k\neq j}\cP_{\bU_{k} } \times_{j}\cS_{\widehat{\bU}_{j,1}}^{(\main)}, \bcI \right\rangle$ as 
\begin{equation}\label{eq:frE-1-mode-j}
\begin{aligned}
        &\left\langle \bcT\times_{k\neq j}\cP_{\bU_{k} } \times_{j}\cS_{\widehat{\bU}_{j,1}}^{(\main)}, \bcI \right\rangle = \left\langle \cP_{\bU_j}^\perp\bE_{j,\rn} \left( \otimes_{k\neq j}\bU_{k} \right)\bV_j \bLambda_j^{-1}\bU_j^\top \bT_j \left( \otimes_{k\neq j}\cP_{\bU_{k}} \right) , \cM_j(\bcI) \right\rangle \\
   =&  \left\langle \cP_{\bU_j}^\perp\bE_{j,\rn} \left( \otimes_{k\neq j}\bU_{k} \right)\bV_j \bV_j^\top \left( \otimes_{k\neq j}\bU_{k}^\top \right) , \cM_j(\bcI) \right\rangle = \left\langle \bE_{j,\rn}  , \cP_{\bU_j}^\perp \cM_j(\bcI) \left( \otimes_{k\neq j}\bU_{k} \right)\cP_{\bV_j} \left( \otimes_{k\neq j}\bU_{k}^\top \right)\right\rangle \\
   =& \frac{d^*}{n} \sum_{i=1}^{n} \xi_i \left\langle \cM_j(\bcX_i), \cP_{\bU_j}^\perp \cM_j(\bcI) \left( \otimes_{k\neq j}\bU_{k} \right)\cP_{\bV_j} \left( \otimes_{k\neq j}\bU_{k}^\top \right)\right\rangle.
\end{aligned}
\end{equation}
Notice that, since we regulate $\bcC$ as $\cM_j(\bcC) = \bLambda_j \bV_j^\top$, we can check that the mode-$j$ unfolding of the component in the tangent space projection, $\bcC \cdot\left(\bU_1,\dots, \bW_j, \dots, \bU_m\right)$, is
\begin{equation*}
\begin{aligned}
        \cM_j\left(\bcC \cdot\left(\bU_1,\dots, \bW_j, \dots, \bU_m\right) \right) &= \bW_j \bLambda_j \bV_j^\top \left(\otimes_{k\neq j}\bU_{k}\right) \\
    &= \cP_{{\bU}_{j}}^{\perp} \mathcal{M}_j\left(\bcI\right)\left(\otimes_{k \neq j} {\bU}_{ k}\right) \mathcal{M}_j^{\dagger}\left({\bcC }\right)\bLambda_j \bV_j^\top \left(\otimes_{k\neq j}\bU_{k}\right),
\end{aligned}
\end{equation*}
where we use the definition of $\bW_j$:
\begin{equation*}
\bW_j=\cP_{{\bU}_{j}}^{\perp} \mathcal{M}_j\left(\bcI\right)\left(\otimes_{k \neq j} {\bU}_{ k}\right) \mathcal{M}_j^{\dagger}\left({\bcC }\right).
\end{equation*}
According to the SVD of $\cM_j(\bcC)$, the pseudo-inverse of $\cM_j(\bcC)$ is given by $ \mathcal{M}_j^{\dagger}\left({\bcC }\right) = \bV_j\bLambda_j^{-1}$. This shows that the mode-$j$ unfolding of $\bcC \cdot\left(\bU_1,\dots, \bW_j, \dots, \bU_m\right)$ is exactly
\begin{equation}\label{eq:mode-j-tangent}
    \cM_j\left(\bcC \cdot\left(\bU_1,\dots, \bW_j, \dots, \bU_m\right) \right) = \cP_{\bU_j}^\perp \cM_j(\bcI) \left( \otimes_{k\neq j}\bU_{k} \right)\cP_{\bV_j} \left( \otimes_{k\neq j}\bU_{k}^\top \right)
\end{equation}
Bring \eqref{eq:mode-j-tangent} into \eqref{eq:frE-1-mode-j}, we have
\begin{equation}\label{eq:frE-1-p1}
    \left\langle \bcT\times_{k\neq j}\cP_{\bU_{k} } \times_{j}\cS_{\widehat{\bU}_{j,1}}^{(\main)}, \bcI \right\rangle = \frac{d^*}{n} \sum_{i=1}^{n} \xi_i \left\langle \bcX_i, \bcC \cdot\left(\bU_1,\dots, \bW_j, \dots, \bU_m\right) \right\rangle,
\end{equation}
The first term in $\frE_1$ is also given by
\begin{equation}\label{eq:frE-1-p2}
    \left\langle\bcE_{\rn}\times_{j=1}^m \cP_{\bU_{j} },  \bcI \right\rangle = \frac{d^*}{n} \sum_{i=1}^{n} \xi_i \left\langle \bcX_i, \bcI \times_{j=1}^m \cP_{\bU_{j} }\right\rangle.
\end{equation}
Combining \eqref{eq:frE-1-p1} and \eqref{eq:frE-1-p2}, we know that $\frE_1$ can be written as
\begin{equation}\label{eq:frE-1-true}
    \frE_1= \frac{d^*}{n} \sum_{i=1}^{n} \xi_i \left\langle \bcX_i, \cP_{\TT}(\bcI) \right\rangle,
\end{equation}
which is the sum of i.i.d. random variables. To use the Berry-Esseen theorem \citep{stein1972bound,raivc2019multivariate}, we compute the second and third moments of $ \frE_1$:
\begin{equation*}
    \begin{gathered}
        \E\abs{\frE_1}^2 =\sigma^2\norm{\cP_{\TT}(\bcI)}_\tF^2\frac{d^*}{n},
    \end{gathered}
\end{equation*}
and
\begin{equation}\label{eq:variance-BE}
\begin{aligned}
    & \E\abs{\xi_i \left\langle \bcX_i, \cP_{\TT}(\bcI) \right\rangle}^3 \le C \sigma^3 \E \abs{\left\langle \bcX_i, \cP_{\TT}(\bcI) \right\rangle}^2 \cdot \max\abs{\left\langle \bcX_i, \cP_{\TT}(\bcI) \right\rangle} \\
    & \le C \sigma^3 \frac{\norm{\cP_{\TT}(\bcI)}_\tF^2}{d^*}\max\left(\abs{\left\langle \bcX_i, \bcI \times_{j=1}^m \cP_{\bU_{j} }\right\rangle} 
    + \sum_{j=1}^{m}\abs{\left\langle \cM_j(\bcX_i), \cP_{\bU_j}^\perp \cM_j(\bcI) \left( \otimes_{k\neq j}\bU_{k} \right)\cP_{\bV_j} \left( \otimes_{k\neq j}\bU_{k}^\top \right)\right\rangle}\right) \\
    & \le C \sigma^3 \frac{\norm{\cP_{\TT}(\bcI)}_\tF^2}{d^*}\max \left(  \abs{\left\langle \bcX_i \times_{j=1}^m \cP_{\bU_{j} }, \bcI \times_{j=1}^m \cP_{\bU_{j} }\right\rangle} \right. \\
    & \quad +\left. \sum_{j=1}^{m}\abs{\left\langle \cM_j(\bcX_i)\times_{k\neq j}^m \cP_{\bU_{k}}, \cP_{\bU_j}^\perp \cM_j(\bcI) \left( \otimes_{k\neq j}\bU_{k} \right)\cP_{\bV_j} \left( \otimes_{k\neq j}\bU_{k}^\top \right)\right\rangle}  \right) \\
    & \le C \sigma^3 \frac{\norm{\cP_{\TT}(\bcI)}_\tF^2}{d^*}\max \left( \sqrt{\norm{\bcX_i \times_{j=1}^m \cP_{\bU_{j} }}_{\tF}^2+ \sum_{j=1}^m \norm{\cM_j(\bcX_i)\times_{k\neq j}^m \cP_{\bU_{k}}}_{\tF}^2} \right.\\
    & \quad \left. \cdot\sqrt{\norm{\bcI \times_{j=1}^m \cP_{\bU_{j} }}_{\tF}^2+ \sum_{j=1}^m \norm{\cP_{\bU_j}^\perp \cM_j(\bcI) \left( \otimes_{k\neq j}\bU_{k} \right)\cP_{\bV_j} \left( \otimes_{k\neq j}\bU_{k}^\top \right)}_{\tF}^2 }\right) \\
    & = C \sigma^3 \frac{\norm{\cP_{\TT}(\bcI)}_\tF^3}{d^*} \max \left( \sqrt{\norm{\bcX_i \times_{j=1}^m \cP_{\bU_{j} }}_{\tF}^2+ \sum_{j=1}^m \norm{\cM_j(\bcX_i)\times_{k\neq j}^m \cP_{\bU_{k}}}_{\tF}^2} \right).
\end{aligned}
\end{equation}
Here, in the third inequality, we use the Cauchy–Schwarz inequality. Since $\bU_{j}$ are coherent for each $j$. we have the max bound:
\begin{equation}\label{eq:variance-BE-max}
    \max \left( \sqrt{\norm{\bcX_i \times_{j=1}^m \cP_{\bU_{j} }}_{\tF}^2+ \sum_{j=1}^m \norm{\cM_j(\bcX_i)\times_{k\neq j}^m \cP_{\bU_{k}}}_{\tF}^2} \right) \le \sqrt{\frac{2m r^* \mu^{m-1} \dmax}{ \rmin d^*}}.
\end{equation}
We then have 
\begin{equation}\label{eq:variance-BE-res}
     \E\abs{\xi_i \left\langle \bcX_i, \cP_{\TT}(\bcI) \right\rangle}^3 \le  C \sigma^3 \sqrt{m r^* \mu^{m-1} \dmax} \frac{\norm{\cP_{\TT}(\bcI)}_\tF^3}{\sqrt{\rmin }\left(d^*\right)^{\frac{3}{2}}} 
\end{equation}
Thus, we have the asymptotic normality:
\begin{equation}\label{eq:core-clt}
    \max_{t\in\R}\abs{\bbP\left(\frac{\frE_1}{\sigma\norm{\cP_{\TT}(\bcI)}_\tF\sqrt{d^*/n} } \le t\right)- \Phi(t)} \le  C \sqrt{\frac{m r^* \mu^{m-1} \dmax}{\rmin n}}
\end{equation}
We now turn to focus on the terms from $\frE_2$ to $\frE_5$:

\noindent \textbf{(ii) $\left\{\frE_i\right\}_{i=2}^{5}$ in} \eqref{eq:test-decomp}.

\begin{Lemma}[Controlling remainder terms]\label{lemma:oracle-init-frE-2-5} Under the conditions in Theorem \ref{thm:lf-inference-popvar}, the terms from $\frE_2$ to $\frE_5$ can be controlled by
\begin{equation*}
    \begin{gathered}
              \abs{\frE_2}\le C C_1 \sigma \sqrt{\frac{\dmax \log \dmax }{n}}\sqrt{\frac{m \mu^{m} \rmin r^*  d^* \log \dmax }{ n }} \norm{\cP_{\TT}(\bcI)}_\tF\\
       \abs{\frE_3}\le C \frac{C_1^2 \sigma^2}{\lambda_{\min}} \norm{\bcI}_{\ell_1} \frac{ m^2 (2\mu)^{\frac{3}{2}m-1} (r^*)^{\frac{3}{2}}\dmax \sqrt{d^*} \log \dmax }{\rmin n}+ C C_1 \sigma \sqrt{\frac{ m^2 (2\mu)^{m-1} r^* \dmax d^* \log^2\dmax }{ \rmax \rmin n^2}}\norm{\cP_{\TT}(\bcI)}_{\tF}\\
        \abs{\frE_4}\le C\norm{\bcI}_{\ell_1} \frac{ C_1^2\sigma^2  }{\lambda_{\min}}\frac{ m^2 (2\mu)^{\frac{3}{2}m} (r^*)^{\frac{3}{2}} \dmax  \sqrt{d^*}\log \dmax }{\rmin n}\\
        \abs{\frE_5}\le C\norm{\bcI}_{\ell_1}\frac{ C_1^2\sigma^2  }{\lambda_{\min}}\frac{ m^2 (2\mu)^{\frac{3}{2}m} (r^*)^{\frac{3}{2}} \dmax  \sqrt{d^*}\log \dmax }{\rmin n} \\  
    \end{gathered}
\end{equation*}
uniformly with probability at least $1-\dmax^{-2m}$.
\end{Lemma}
The proof of Lemma \ref{lemma:oracle-init-frE-2-5} is one of the most technically challenging parts of this paper. We defer it to the Section \ref{apx:sec:proof-all-rem-ind}.

We can then invoke the decomposition in \eqref{eq:test-decomp} and write our test statistic $\widehat{\bcT}$ under population variance as 
\begin{equation}\label{eq:test-stat-popvar}
    W_{\test}(\bcI)=\frac{\left\langle  \widehat{\bcT} - \bcT,\bcI  \right\rangle }{ \sigma\norm{\cP_{ \TT }(\bcI)}_\tF \sqrt{{d^*}/{n}} } = \frac{\frE_1}{\sigma\norm{\cP_{\TT}(\bcI)}_\tF\sqrt{d^*/n} }  + \frac{\sum_{i=2}^{5}\frE_i}{\sigma\norm{\cP_{\TT}(\bcI)}_\tF\sqrt{d^*/n} }.  
\end{equation}
The first term in \eqref{eq:test-stat-popvar} converges to the asymptotic normal random variable following \eqref{eq:core-clt}, and all the error terms from $\frE_2$ to $\frE_5$ are controlled by Lemma \ref{lemma:oracle-init-frE-2-5} with
\begin{equation}\label{eq:test-stat-popvar-rem}
    \abs{\frac{\sum_{i=2}^{5}\frE_i}{\sigma\norm{\cP_{\TT}(\bcI)}_\tF\sqrt{d^*/n} }}\le C C_1 \sqrt{\frac{ m^3 (2\mu)^{m-1} r^* \dmax \log^2\dmax }{ \rmax \rmin n}}+ C \frac{ C_1^2\sigma  \norm{\bcI}_{\ell_1}}{\lambda_{\min}  \norm{\cP_{\TT}(\bcI)}_\tF }\frac{ m^2 (2\mu)^{\frac{3}{2}m} (r^*)^{\frac{3}{2}} \dmax \log \dmax }{\rmin \sqrt{n}}.
\end{equation}
Given the alignment condition in Assumption \ref{asm:alignment}, we have the following bound from $\frE_2$ to $\frE_5$:
\begin{equation}\label{eq:frE2-5-homo}
    \begin{aligned}
       \abs{\frac{\sum_{i=2}^{5}\frE_i}{\sigma\norm{\cP_{\TT}(\bcI)}_\tF\sqrt{d^*/n} }}\le C C_1 \sqrt{\frac{ m^3 (2\mu)^{m-1} r^* \dmax \log^2\dmax }{ \rmax \rmin n}} + C \frac{ C_1^2\sigma  \norm{\bcI}_{\ell_1} }{\alpha_I \lambda_{\min}  \norm{\bcI}_\tF }\sqrt{\frac{ m^4 (2\mu)^{3 m} (r^*)^{3 }\dmax d^*  \log^2 \dmax }{\rmin^2 n } }
    \end{aligned}
\end{equation}
with probability at least $1- \dmax^{-2m}$. Combining this result with \eqref{eq:core-clt}, we can derive that
\begin{equation*}
\begin{aligned}
        &\max_{t\in\R}\abs{\bbP\left( \frac{\left\langle  \widehat{\bcT} - \bcT,\bcI  \right\rangle }{ \sigma\norm{\cP_{ \TT }(\bcI)}_\tF \sqrt{{d^*}/{n}} } \le t\right)- \Phi(t)}  \\
        &\le C C_1 \sqrt{\frac{ m^3 (2\mu)^{m-1} r^* \dmax \log^2\dmax }{ \rmax \rmin n}} + C \frac{ C_1^2\sigma  \norm{\bcI}_{\ell_1}}{\alpha_I \lambda_{\min}  \norm{\bcI}_\tF }\sqrt{\frac{ m^4 (2\mu)^{3 m} (r^*)^{3 }\dmax d^*  \log^2 \dmax }{\rmin^2 n } } +C \dmax^{-2m}.
\end{aligned}
\end{equation*}
By assuming that $n=O(\dmax^{4m})$, we can prove the desired bound.



\subsection{Proof of Theorem \ref{thm:lf-inference-empvar}}\label{apx:sec:clt-indp-var}
Based on the proof of Theorem \ref{thm:lf-inference-popvar}, we  now use the empirical estimation 
\begin{equation*}
\begin{gathered}
        \cP_{\widehat \TT }(\bcI) = \bcI \cdot \left( \cP_{\widehat{\bU}_{1,0}},\dots,\cP_{\widehat{\bU}_{m,0}}\right) + \sum_{j=1}^{m} \widehat{\bcC}_0 \cdot\left(\widehat{\bU}_{1,0},\dots, \widehat{\bW}_{j}, \dots, \widehat{\bU}_{m,0}\right), \text{ with } \\
        \widehat \bW_j=\cP_{\widehat {\bU}_{j,0}}^{\perp} \mathcal{M}_j\left(\bcI\right)\left(\otimes_{k \neq j} \widehat{\bU}_{ k,0}\right) \mathcal{M}_j^{\dagger}(\widehat {\bcC }_0),
\end{gathered}
\end{equation*}
and $\widehat{\sigma}^2 = \frac{1}{n}\sum_{i=1}^n\left(Y_i-\left\langle \bcT_{\init},\bcX_i \right\rangle\right)^2$ to approximate $ \cP_{ \TT }(\bcI) $ and $\sigma^2$. We first check the difference $\abs{\norm{\cP_{\widehat \TT }(\bcI)}_\tF -\norm{\cP_{\TT }(\bcI)}_\tF}$. Although the difference between projectors $\cP_{\widehat \TT}$ and $\cP_{\TT}$  has been partially studied in, e.g., \cite{lubich2013dynamical} or in \cite{cai2020provable,cai2023generalized}, the techniques developed in previous literature are not applicable in our case since our $\bcI$ here can be sparse in many cases.   Instead, we establish a new technique to handle the difference between projectors, which is given in Lemma \ref{lemma:diff-tangent-proj}. According to Lemma \ref{lemma:diff-tangent-proj}, we have  
    \begin{equation}\label{eq:diff-tangent-proj}
        \norm{\cP_{\widehat \TT }(\bcI) - \cP_{\TT }(\bcI) }_\tF \le \norm{\bcI}_{\ell_1}\frac{C C_1  m^2\kappa_0 \sigma  }{\lambda_{\min}}\sqrt{\frac{(2\mu)^m r^*\dmax^2 \log \dmax }{n}}.  
    \end{equation}
    Given the SNR condition that 
    \begin{equation*}
         \frac{\lambda_{\min}}{\sigma}\ge C_{\gap} \frac{\norm{\bcI}_{\ell_1}}{\norm{\bcI}_{\tF}}\frac{C_1  m^2\kappa_0  }{\alpha_I }\sqrt{\frac{(2\mu)^m r^*d^*\dmax \log \dmax }{n}},
    \end{equation*}
    we have
    \begin{equation}\label{eq:PT-diff-var}
        \frac{\norm{\cP_{\widehat \TT }(\bcI) - \cP_{\TT }(\bcI) }_\tF}{ \norm{\cP_{\TT }(\bcI) }_\tF } \le C\frac{\norm{\bcI}_{\ell_1}}{\norm{\bcI}_{\tF}}\frac{C_1  m^2\kappa_0 \sigma  }{\alpha_I \lambda_{\min}}\sqrt{\frac{(2\mu)^m r^*d^*\dmax \log \dmax }{n}}\le\frac{1}{2}.
    \end{equation}
We now focus on the variance estimation:
\begin{equation}\label{eq:sigma-est}
\begin{aligned}
            \widehat{\sigma}^2& = \frac{1}{n}\sum_{i=1}^n\left(Y_i-\left\langle \bcT_{\init},\bcX_i \right\rangle\right)^2 =  \frac{1}{n}\sum_{i=1}^n\left( \xi_i -\left\langle\bcT^{\vartriangle},\bcX_i \right\rangle\right)^2\\
        & =\frac{1}{n} \sum_{i=1}^n \xi_i^2 -2 \frac{1}{n} \sum_{i=1}^n \xi_i \left\langle\bcT^{\vartriangle},\bcX_i \right\rangle + \frac{1}{n}\sum_{i=1}^n\left\langle\bcT^{\vartriangle},\bcX_i \right\rangle^2,
\end{aligned}
\end{equation}
where  $\bcT^{\vartriangle} = \widehat \bcT_{\init}- \bcT$. We deal with the terms in \eqref{eq:sigma-est} respectively by comparing them with $\sigma^2$. Using the concentration of sub-exponential random variables on $\frac{1}{n} \sum_{i=1}^n \xi_i^2$, we have  
\begin{equation*}
   \abs{ \frac{1}{n} \sum_{i=1}^n \xi_i^2 - \sigma^2} \le C\sigma^2\frac{ m\log\dmax}{n} + C\sigma^2\sqrt{\frac{m\log\dmax}{n}}\le  C\sigma^2\sqrt{\frac{m\log\dmax}{n}},
\end{equation*}
with probability at least $1-\dmax^{-3m}$,
since $\norm{\xi_i^2}_{\psi_1}\le \norm{\xi_i}_{\psi_2}^2\le C \sigma^2$. The sub-Gaussian tail indicates that the inequality $\frac{1}{n} \sum_{i=1}^n \abs{\xi_i}\le C\sigma\left(1+\sqrt{\frac{m\log\dmax}{n}}\right)$ holds with probability at least $1-\dmax^{-3m}$. Given that $n\ge C_1^2\vee(\rmax m) \dmax \log\dmax$, we have 
\begin{equation*}
    \abs{\frac{1}{n} \sum_{i=1}^n \xi_i \left\langle\bcT^{\vartriangle},\bcX_i \right\rangle }\le  \frac{1}{n} \sum_{i=1}^n \abs{\xi_i}\norm{ \widehat\bcT_{\init} - \bcT}_{\linf} \le C C_1\sigma^2\sqrt{\frac{\dmax \log\dmax }{ n }}.
\end{equation*}
The last term $\frac{1}{n}\sum_{i=1}^n\left\langle\bcT^{\vartriangle},\bcX_i \right\rangle^2$ can be controlled by the simple bound:
\begin{equation*}
    \frac{1}{n}\sum_{i=1}^n\left\langle\bcT^{\vartriangle},\bcX_i \right\rangle^2 \le C_1^2 \sigma^2 \frac{\dmax \log\dmax }{ n }.
\end{equation*}
Thus, we conclude that the variance estimation gives the error: 
\begin{equation}\label{eq:sigma-est-ind}
    \abs{\frac{ \widehat{\sigma}^2}{\sigma^2}-1}\le C\big(\sqrt{\frac{m\log\dmax}{n}}+C_1\sqrt{\frac{\dmax\log\dmax}{n}} + C_1^2 \frac{\dmax \log\dmax }{ n }\big)\le CC_1 \sqrt{\frac{\dmax\log\dmax}{n}}
\end{equation}
Combining the difference between $\norm{\cP_{\widehat \TT }(\bcI)}_\tF$ and $\norm{\cP_{ \TT }(\bcI)}_\tF$ studied in \eqref{eq:diff-tangent-proj}, and the variance estimation in \eqref{eq:sigma-est-ind}, we have 

\begin{equation}\label{eq:homovar-vardiff}
\begin{aligned}
       \abs{ \frac{\widehat\sigma \norm{\cP_{\widehat \TT }(\bcI)}_\tF}{\sigma \norm{\cP_{ \TT }(\bcI)}_\tF} - 1 } & \le 2\left(\abs{ \frac{\widehat\sigma }{\sigma} - 1 }+ \abs{  \frac{\norm{\cP_{\widehat \TT }(\bcI)}_\tF}{\norm{\cP_{ \TT }(\bcI)}_\tF} - 1 }\right) \\
       &\le 2 \left(\abs{ \frac{\widehat\sigma^2 }{\sigma^2} - 1 }+ \abs{  \frac{\norm{\cP_{\widehat \TT }(\bcI) - \cP_{\TT }(\bcI) }_\tF}{ \norm{\cP_{\TT }(\bcI) }_\tF } }\right) \\
       &\le CC_1 \sqrt{\frac{\dmax\log\dmax}{n}} + C\frac{\norm{\bcI}_{\ell_1}}{\norm{\bcI}_{\tF}}\frac{C_1  m^2\kappa_0 \sigma  }{\alpha_I \lambda_{\min}}\sqrt{\frac{(2\mu)^m r^*d^*\dmax \log \dmax }{n}}
\end{aligned}
\end{equation}
Notice that \eqref{eq:homovar-vardiff} holds with probability $1-2\dmax^{-3m}$ no matter whether $\widehat{\TT}$ and $\bcT^{\vartriangle}$ depend on the observations $\{(\bcX_i,\xi_i)\}_{i=1}^n$ or not.

Moreover, We know that $ W_{\test}(\bcI)$ can be decomposed by terms related $\frE_i$ for $i=1,\dots,5$, and when $2\le i \le 5$, the terms have been bounded by \eqref{eq:frE2-5-homo}. For $\frE_1$, the concentration inequality also gives 
\begin{equation*}
    \abs{\frac{\frE_1}{\sigma\norm{\cP_{\TT}(\bcI)}_\tF\sqrt{d^*/n} } } \le C \sqrt{\frac{m \mu^{m-1}r^*\dmax}{\rmin }} \frac{m\log\dmax}{\sqrt{n}} + C\sqrt{m\log\dmax} \le  C\sqrt{m\log\dmax}
\end{equation*}
under the sample size condition of Theorem \ref{thm:lf-inference-popvar} with probability at least $1-\dmax^{-3m}$. Here we use the fact that 
$$\norm{\xi_i \left\langle \bcX_i, \cP_{\TT}(\bcI) \right\rangle}_{\psi_2}\le C\sigma \sqrt{\frac{m\mu^{m-1}r^*\dmax}{\rmin d^*}}\norm{ \cP_{\TT}(\bcI)}_{\tF}.$$
Thus, we can control $ W_{\test}(\bcI)$ collectively by:

\begin{equation}\label{eq:bound-Wtest}
\begin{aligned}
        \abs{W_{\test}(\bcI)}& \le C\sqrt{m\log\dmax} + C C_1 \sqrt{\frac{ m^3 (2\mu)^{m-1} r^* \dmax \log^2\dmax }{ \rmax \rmin n}} + C \frac{ C_1^2\sigma  \norm{\bcI}_{\ell_1} }{\alpha_I \lambda_{\min}  \norm{\bcI}_\tF }\sqrt{\frac{ m^4 (2\mu)^{3 m} (r^*)^{3 }\dmax d^*  \log^2 \dmax }{\rmin^2 n } } \\
   & \le C\sqrt{m\log\dmax}
\end{aligned}
\end{equation}
We can also check that the following decomposition is valid:
\begin{equation}\label{eq:hatW-decomp}
\begin{aligned}
        & \widehat W_{\test}(\bcI)=\frac{\left\langle  \widehat{\bcT} - \bcT,\bcI  \right\rangle }{ \widehat\sigma \norm{\cP_{\widehat \TT }(\bcI)}_\tF \sqrt{{d^*}/{n}} } \\ &=\frac{\frE_1}{\sigma\norm{\cP_{\TT}(\bcI)}_\tF\sqrt{d^*/n} } + \frac{\sum_{i=2}^5\frE_i}{\sigma\norm{\cP_{\TT}(\bcI)}_\tF\sqrt{d^*/n} } + W_{\test}(\bcI)\left( \frac{\sigma \norm{\cP_{\widehat \TT }(\bcI)}_\tF}{\widehat\sigma \norm{\cP_{\widehat \TT }(\bcI)}_\tF}-1\right).
\end{aligned}
\end{equation}
Thus, plugging in \eqref{eq:frE2-5-homo}, \eqref{eq:homovar-vardiff}, and  \eqref{eq:bound-Wtest}, we have
\begin{equation*}
\begin{aligned}
        &\max_{t\in\R}\abs{\bbP\left( \widehat W_{\test}(\bcI)\le t\right)- \Phi(t)}  \\
        &\le  C C_1 \sqrt{\frac{m \mu^{m} \rmin r^* \dmax \log^2 \dmax }{n}}+ C \frac{ C_1^2\sigma  }{ \lambda_{\min}} \frac{\norm{\bcI}_{\ell_1}}{\alpha_I\norm{\bcI}_\tF }   \sqrt{\frac{ m^4 (2\mu)^{3 m} (r^*)^{3 }\dmax d^*  \log^2 \dmax }{\rmin^2 n } } \\
        &\quad \quad + CC_1 \sqrt{\frac{m\dmax\log^2\dmax}{n}} + C\frac{\norm{\bcI}_{\ell_1}}{\norm{\bcI}_{\tF}}\frac{C_1  m^{2.5}\kappa_0 \sigma  }{\alpha_I \lambda_{\min}}\sqrt{\frac{(2\mu)^m r^*d^*\dmax \log^2 \dmax }{n}} \\
        &\le  C C_1 \sqrt{\frac{m \mu^{m} \rmin r^* \dmax \log^2 \dmax }{n}} + C\frac{\norm{\bcI}_{\ell_1}}{\norm{\bcI}_{\tF}}\frac{C_1^2  \kappa_0 \sigma  }{\alpha_I \lambda_{\min}}\sqrt{\frac{m^5 (2\mu)^{3m} (r^*)^3 d^*\dmax \log^2 \dmax }{\rmin^2 n}} 
\end{aligned}
\end{equation*}

\subsection{Proof of Corollary \ref{coro:CIs}}
\begin{proof}
    The proof is quite straightforward. It is clear that 
    \begin{equation*}
        \bbP\left(\left\langle\bcT,\bcI\right\rangle\in \widehat{\mathrm{CI}}_\alpha(\bcI) \right) = \bbP\left( -z_{\frac{\alpha}{2}}\le \widehat W_{\test}(\bcI) \le z_{\frac{\alpha}{2}}\right) = \bbP\left( \widehat W_{\test}(\bcI)\le z_{\frac{\alpha}{2}}\right) - \bbP\left( \widehat W_{\test}(\bcI)\le -z_{\frac{\alpha}{2}}\right).
    \end{equation*}
    Thus, as has been revealed in Theorem \ref{thm:lf-inference-empvar}, we have
    \begin{equation*}
        \begin{aligned}
           & \abs{\bbP\left(\left\langle\bcT,\bcI\right\rangle\in \widehat{\mathrm{CI}}_\alpha(\bcI) \right) - (1-\alpha)} \\
           & \le  \abs{ \bbP\left( \widehat W_{\test}(\bcI)\le z_{\frac{\alpha}{2}}\right) - (1-\frac{\alpha}{2}) } +\abs{ \bbP\left( \widehat W_{\test}(\bcI)\le-  z_{\frac{\alpha}{2}}\right) -\frac{\alpha}{2} } \\
           & \le 2 \max_{t\in\R}\abs{\bbP\left( \widehat W_{\test}(\bcI)\le t\right)- \Phi(t)} \\
           & \le C C_1 \sqrt{\frac{m \mu^{m} \rmin r^* \dmax \log^2 \dmax }{n}} + C\frac{\norm{\bcI}_{\ell_1}}{\norm{\bcI}_{\tF}}\frac{C_1^2  \kappa_0 \sigma  }{\alpha_I \lambda_{\min}}\sqrt{\frac{m^5 (2\mu)^{3m} (r^*)^3 d^*\dmax \log^2 \dmax }{\rmin^2 n}}.
        \end{aligned}
    \end{equation*}
\end{proof}

\subsection{Proof of Theorem \ref{thm:clt-lol-init}, \ref{thm:clt-dep-init} with dependent initialization}\label{apx:sec:proof-clt-dep-init}
\begin{proof}
Since the requirement of Theorem \ref{thm:clt-dep-init} is weaker than Theorem \ref{thm:clt-lol-init}, we shall first prove Theorem \ref{thm:clt-dep-init} and then transfer it to Theorem \ref{thm:clt-lol-init}. The proof of Theorem \ref{thm:clt-dep-init} follows essentially the same ideas as the proof of Theorem \ref{thm:lf-inference-popvar} in Section \ref{apx:sec:clt-indp}. However, the difference is that, when the initialization depends on the data $\left\{(\bcX_i,\xi_i)\right\}_{i=1}^n$, the uniform error bounds on $\bE_{j,\init}$, $\bE_{j,\rn}$, and $\cP_{\widehat{\bU}_{j,1}} - \cP_{\bU_j}$, as well as the remainder terms need to be carefully handled. To construct the uniform bound, a refined analysis of the covering numbers on the initialization $\widehat{\bcT}_{\init}$ and the initial singular subspaces $\{\cP_{\widehat{\bU}_{j,0}} \}_{j\in[m]}$ is presented.

\subsubsection{Covering numbers of low-rank tensors}
We know that $\widehat \bcT_{\init}$ is multilinear rank-$\br$ tensor and $\norm{\widehat \bcT_{\init}-\bcT}_{\linf}$ meets the initialization bound \eqref{eq:tensor-init}. We can analyze the metric entropy of rank-$\br$ tensors with bounded $\linf$-norm perturbation:
\begin{Lemma}[Covering number of perturbed rank-$\br$ tensors]\label{lemma:cover-init}
    Assume $\bcT$ is multilinear rank-$\br$ with incoherence parameter $\mu$ and minimum eigenvalue $\lambda_{\min}$. Denote the set of multilinear rank-$\br$ tensor in $\R^{d_1\times\cdots \times d_m }$ with bounded $\linf$-norm perturbation as $\mathbb{T}_{\init}(c_0)=\left\{ \widehat{\bcT}_0: \ \norm{\widehat{\bcT}_0- \bcT}_{\linf} \le c_0 , \rank(\widehat{\bcT}_0)=\br\right\}, $ for any $c_0\le \frac{\lambda_{\min}}{16\sqrt{d^*}}$. Then, the covering number of $\mathbb{T}_{\init}(c_0)$ can be controlled by 
    \begin{equation*}
        \cN(\mathbb{T}_{\init}(c_0),\norm{\cdot}_{\linf},\varepsilon) \le \left(1+ C 2^{2m+5}\kappa_0\sqrt{(2\mu)^m r^*\rmax^2} \frac{c_0}{\varepsilon}\right)^{r^*+ \sum_{j=1}^{m}d_j r_j }.
    \end{equation*}
\end{Lemma}
With the help of Lemma \ref{lemma:cover-init}, we can let $c_0$ be $C_1 \sigma \sqrt{\frac{\dmax \log \dmax }{n}} $, and select $\varepsilon= C_1 \sigma \sqrt{\frac{\log \dmax }{\dmax^{C_{\varepsilon} m } n}}$ for a large constant $C_{\varepsilon}\ge 1$ such that $n\le \dmax^{C_{\varepsilon} m}$. Without loss of generality, we also treat $C_{\varepsilon}\le C$. Then, it is clear that under the initialization condition, there exists an $\varepsilon$-net $\operatorname{Net}(\varepsilon)$ such that $\abs{\operatorname{Net}(\varepsilon)}\le \dmax^{C_{\varepsilon}  m\rmax\dmax }$ and for any initialization $\widehat \bcT_{\init}$, there always exist an element $\widehat \bcT_{\net}\in \operatorname{Net}(\varepsilon)$ such that $\widehat \bcT_{\net} = \widehat\bcC_{\net}\times_{j=1}^{m}\widehat{\mathbf{U}}_{j,\net}$, and 
\begin{equation*}
    \norm{\widehat{\bcT}_\init- \widehat \bcT_{\net}}_{\linf} \le C_1 \sigma \sqrt{\frac{\log \dmax }{\dmax^{C_{\varepsilon} m } n}}.
\end{equation*}
Moreover, according to the proof of Lemma \ref{lemma:l2inf-init}, we also have 
    \begin{equation*}
   \norm{ \cP_{\widehat{\bU}_{j,0}} - \cP_{\widehat{\bU}_{j,\net}} }_{2,\infty} \le \frac{8 C_1\sigma  }{\lambda_{\min}}\sqrt{\frac{\mu r_j}{d_j} }\sqrt{\frac{d^* \log \dmax }{\dmax^{C_{\varepsilon} m } n}}.
\end{equation*}
The $\varepsilon$ here is small enough for our purpose and we only need to consider the uniform bound of all elements in $\operatorname{Net}(\varepsilon)$ instead of all the initializations satisfying the condition \eqref{eq:tensor-init}.

\subsubsection{Controlling errors under dependent initialization}
    \noindent We now control the errors under dependence using the $\varepsilon$-net constructed above. 
    \begin{Lemma}\label{lemma:Einit-Ern-dep-l2} Under dependent initialization, the following bounds hold with probability at least $1-m2^{m+3}\dmax^{-3m}$:
\begin{equation}\label{eq:E-rn-Ber-dep}
    \begin{aligned}
        \norm{ \bE_{1,\rn}\left(\otimes_{j\neq 1} \widehat{\mathbf{U}}_{j,0} \right) }_2 &  \le C \sigma\left(1+\frac{C_1\sigma}{\lambda_{\min}}\sqrt{\frac{m^2 d^*\dmax\log\dmax}{n}} + \frac{ C_1^2 \sigma^2  }{\lambda_{\min}^2}\frac{d^*\dmax^{\frac{3}{2}} \log \dmax }{n} \right) \sqrt{\frac{  m \mu^{m-1} r^* d_1  d^*\log \dmax }{\rmin r_1 n}},
    \end{aligned}
\end{equation}
\begin{equation}\label{eq:dep-E-init-uniform}
    \norm{\bE_{1,\init}\left(\otimes_{j\neq 1} \widehat{\mathbf{U}}_{j,0} \right) }_2 \le C C_1 \sigma \sqrt{\frac{ m (2\mu)^{m-1} \rmax r^* \dmax d^*\log \dmax }{r_1 n}\cdot \frac{d_1\dmax \log \dmax}{n}\wedge 1}.
\end{equation}

    \end{Lemma}
    \begin{proof}
          For $\bE_{1,\rn}$, we always have the following splitting:
\begin{equation}\label{eq:dep-B2-rn-decomp}
\begin{aligned}
         \norm{\bE_{1,\rn}\left( \otimes_{j\neq 1} \widehat{\bU}_{j,0} \right)}_2 & \le \frac{d^*}{n} \left( \norm{\sum_{i=1}^{n} \xi_i \cM_1(\bcX_i)\left( \otimes_{j\neq 1} {\bU}_{j} \right)}_2 \right.
         \\
         & + \left. \underbrace{\sum_{j\neq 1} \norm{\sum_{i=1}^{n} \xi_i \cM_1(\bcX_i) \otimes_j \left( \cP_{\widehat{\bU}_{j,0}} - \cP_{\bU_{j}}\right)\left(  \otimes_{k\neq j, 1} {\bU}_{k} \right) }_2}_{\text{first-order term}}\right.\\
         &
         \left. + \underbrace{\sum_{\cS\subseteq [m]/\{1\}, \abs{\cS}\ge2}\norm{\sum_{i=1}^{n} \xi_i \cM_1(\bcX_i) \otimes_{j\in \cS } \left( \cP_{\widehat{\bU}_{j,0}} - \cP_{\bU_{j}}\right)\left(  \otimes_{k\in \cS^c } {\bU}_{k} \right)}_2}_{\text{higher-order term}} \right)  \\
\end{aligned}
\end{equation}
Here, the inequality holds because for any matrix $\bA$ and singular subspace $\bU$, we always have $\norm{A\bU}_2 =\norm{A\bU \bU^\top}_2$. We now look at the first term of the RHS in \eqref{eq:dep-B2-rn-decomp}.   Using the Bernstein inequality, with the same argument as \eqref{eq:E-rn-Ber}, for the sum of i.i.d random variable term in \eqref{eq:dep-B2-rn-decomp}, we have 
\begin{equation}
    \begin{aligned}
        \norm{\bE_{1,\rn}\left( \otimes_{j\neq 1} {\mathbf{U}}_{j} \right)}_2 & \le C \sigma \sqrt{\frac{  m \mu^{m-1} r^* d_1  d^*\log \dmax }{\rmin r_1 n}} ,
    \end{aligned}
\end{equation}
with probability at least $1-\dmax^{-3m }$ given that $n\ge m \rmax \dmax \log \dmax$. 


For the first-order term, we can write it as:
\begin{equation*}
    \begin{aligned}
       &\frac{d^*}{n}\norm{\sum_{i=1}^{n} \xi_i \cM_1(\bcX_i) \otimes_j \left( \cP_{\widehat{\bU}_{j,0}} - \cP_{\bU_{j}}\right)\left(  \otimes_{k\neq j, 1} {\bU}_{k} \right) }_2  & = \frac{d^*}{n}\left\| \left[\sum_{i=1}^{n} \xi_i \cM_1(\bcX_i) \otimes_j \bI_{d_j}\left(  \otimes_{k\neq j, 1} {\bU}_{k} \right)\right] \right.\\
      &  &\quad \quad\left.\cdot\otimes_j\left( \cP_{\widehat{\bU}_{j,0}} - \cP_{\bU_{j}}\right)\left(  \otimes_{k\neq j, 1} {\bI}_{r_k} \right)\right\|_2 \\
        & \le  \frac{8C_1\sigma}{\lambda_{\min}}\sqrt{\frac{d^*\dmax\log\dmax}{n}} \frac{d^*}{n} \norm{\sum_{i=1}^{n} \xi_i \cM_1(\bcX_i) \otimes_j \bI_{d_j}\left(  \otimes_{k\neq j, 1} {\bU}_{k} \right)}_2,
    \end{aligned}
\end{equation*}
according to the initialization error in Lemma \ref{lemma:l2inf-init}. Using the matrix Bernstein inequality, we have
\begin{equation*}
\begin{aligned}
        \frac{d^*}{n} \norm{\sum_{i=1}^{n} \xi_i \cM_1(\bcX_i) \otimes_j \bI_{d_j}\left(  \otimes_{k\neq j, 1} {\bU}_{k} \right)}_2 & \le C \sigma \frac{m\log\dmax}{n}\sqrt{ \mu^{m-2} r_{-1}/\rmin d_1 d^* \dmax }  + C\sqrt{\frac{m \mu^{m-2} r_{-1} d^* \dmax \log\dmax}{ \rmin n}} \\
    & \le C\sqrt{\frac{m \mu^{m-2} r_{-1} d^* \dmax \log\dmax}{ \rmin n}}.
\end{aligned}
\end{equation*}  
Thus, we control the first-order term by 

\begin{equation}\label{eq:E-rn-init-order-1-dep}
    \begin{aligned}
        \frac{d^*}{n}\norm{\sum_{i=1}^{n} \xi_i \cM_1(\bcX_i) \otimes_j \left( \cP_{\widehat{\bU}_{j,0}} - \cP_{\bU_{j}}\right)\left(  \otimes_{k\neq j, 1} {\bU}_{k} \right) }_2 & \le  \frac{C C_1\sigma}{\lambda_{\min}}\sqrt{\frac{d^*\dmax\log\dmax}{n}} \sqrt{\frac{m \mu^{m-2} r_{-1} d^* \dmax \log\dmax}{ \rmin n}} \\
        & \le \frac{C C_1\sigma}{\lambda_{\min}} \sqrt{\frac{m \mu^{m-2} r_{-1} (d^*)^2 \dmax^2 \log^2\dmax}{ \rmin n^2}} 
    \end{aligned}
\end{equation}
We now use the $\varepsilon$-net argument to deal with all the higher-order terms. We can decouple each one of the higher-order terms in \eqref{eq:dep-B2-rn-decomp} as 

\begin{equation*}
    \begin{aligned}
        &\frac{d^*}{n} \norm{\sum_{i=1}^{n} \xi_i \cM_1(\bcX_i) \otimes_{j\in \cS } \left( \cP_{\widehat{\bU}_{j,0}} - \cP_{\bU_{j}}\right)\left(  \otimes_{k\in \cS^c } {\bU}_{k} \right) }_2  \le \frac{d^*}{n} \norm{\sum_{i=1}^{n} \xi_i \cM_1(\bcX_i) \otimes_{j\in \cS} \left( \cP_{\widehat{\bU}_{j,\net}} - \cP_{\bU_{j}}\right)\left(  \otimes_{k\in \cS^c} {\bU}_{k} \right) }_2\\
        & + \frac{d^*}{n} \norm{\sum_{i=1}^{n} \xi_i \cM_1(\bcX_i) \underbrace{\left[\otimes_{j\in \cS } \left( \cP_{\widehat{\bU}_{j,0}} - \cP_{\widehat\bU_{j}}\right)-\otimes_{j\in \cS } \left( \cP_{\widehat{\bU}_{j,0}} - \cP_{\widehat\bU_{j,\net}}\right)\right]}_{\Delta_{\net} }\left(  \otimes_{k\in \cS^c} {\bU}_{k} \right) }_2 .
    \end{aligned}
\end{equation*}
The construction of $\varepsilon$-net indicate that we can always select a $\widehat\bU_{j,\net}$ in the net such that $\norm{\widehat\bU_{j,\net} -\widehat\bU_{j,0}}_2$ is sufficiently small. If we suppose $\norm{\widehat\bU_{j,\net} -\widehat\bU_{j,0}}_2\le \varepsilon$, then we always have $\norm{\Delta_{\net}}_2\le C m \varepsilon \cdot \frac{C_1\sigma  }{\lambda_{\min}} \sqrt{\frac{d^*\dmax \log \dmax }{n}}$ for any small $\varepsilon$.  
Given that $C_{\varepsilon}$ is large enough, we will have
\begin{equation}\label{eq:dep-B2-rn-order2-diff}
\begin{aligned}
        &\frac{d^*}{n} \norm{\sum_{i=1}^{n} \xi_i \cM_1(\bcX_i) \Delta_{\net}\left(  \otimes_{k\neq j, 1} {\bU}_{k} \right) }_2  \\
        &\le \frac{d^*}{n} \sum_{i=1}^n \norm{\xi_i  \be_{1,i} }_2   \norm{\Delta_{\net} }_2  \\
        & \le C \frac{C_1^2\sigma^2  }{\lambda_{\min}^2}\frac{m\mu^{m-1} r_{-1}d^* \dmax \log^2 \dmax }{n^2},
\end{aligned}
\end{equation}
which is sufficiently small and negligible. On the other hand, the term containing $ \cP_{\widehat{\bU}_{j,\net}} - \cP_{\bU_{j}}$ is still the sum of i.i.d. random matrices and can be handled with matrix Bernstein inequality. Therefore, we have
\begin{equation}\label{eq:dep-B2-rn-order2-net}
\begin{aligned}
        & \frac{d^*}{n} \norm{\sum_{i=1}^{n} \xi_i \cM_1(\bcX_i) \otimes_{j\in \cS} \left( \cP_{\widehat{\bU}_{j,\net}} - \cP_{\bU_{j}}\right)\left(  \otimes_{k\in \cS^c} {\bU}_{k} \right) }_2 \\
        &\le C \sigma \left(\frac{C C_1\sigma  }{\lambda_{\min}}\sqrt{\frac{d^*\dmax \log \dmax }{n}} \right)^{\abs{\cS}}\left(\frac{d^*}{n}\sqrt{\frac{\mu^{m-1} r_{-1}}{d_{-1}}}\cdot m\rmax \dmax \log \dmax  +  \sqrt{\frac{ m \mu^{m-1} \rmax d_1 \dmax d^*\log \dmax }{n}} \right) \\
        & \le C \sigma \left(\frac{ C C_1\sigma  }{\lambda_{\min}}\sqrt{\frac{d^*\dmax \log \dmax }{n}} \right)^{\abs{\cS}}\sqrt{\frac{ m \mu^{m-1} \rmax r_{-1}  d_1 \dmax d^*\log \dmax }{n}},
\end{aligned}
\end{equation}
for each element in the $\varepsilon$-net with probability at least $1-\dmax^{ C_{\varepsilon}m \rmax\dmax -3 m}$. Notice that here $\abs{\cS}\ge 2$. This suggests that such a bound holds uniformly for all the elements in the $\varepsilon$-net with probability at least $1-\dmax^{-3 m}$. Combining \eqref{eq:E-rn-init-order-1-dep}, \eqref{eq:dep-B2-rn-order2-diff}, and
\eqref{eq:dep-B2-rn-order2-net}, it is clear that under the SNR condition:
\begin{equation*}
     \frac{\lambda_{\min} }{\sigma} \ge C_{\mathsf{gap} } C_1\kappa_0\sqrt{  \frac{ m^2  \rmax r^* (2\mu)^{m-1} d^* \dmax \log \dmax }{\rmin n} },
\end{equation*}
we can  control the \eqref{eq:dep-B2-rn-decomp} with probability at least $1-\dmax^{-3m}$ with the bound:
\begin{equation*}
    \begin{aligned}
  \norm{ \bE_{1,\rn}\left(\otimes_{j\neq 1} \widehat{\mathbf{U}}_{j,0} \right) }_2 &  \le C \sigma\left(1+\frac{C_1\sigma}{\lambda_{\min}}\sqrt{\frac{m^2 d^*\dmax\log\dmax}{n}} + \frac{ C_1^2 \sigma^2  }{\lambda_{\min}^2}\frac{d^*\dmax^{\frac{3}{2}} \log \dmax }{n} \right) \sqrt{\frac{  m \mu^{m-1} r^* d_1  d^*\log \dmax }{\rmin r_1 n}}.
    \end{aligned}
\end{equation*}

\noindent We now focus on $\bE_{1,\init}$. We shall use the the $\varepsilon$-net argument to study it. To this end, we define $\bE_{1,\init}^{\net}$ following the definition of $\bE_{1,\init}$ in \eqref{eq:svd-U1-decomp} by substituting $\widehat\bcT_{\init}$ with $\widehat \bcT_{\net}$. Then, it is clear that
\begin{equation*}
    \norm{\bE_{1,\init}^{\net}\left(\otimes_{j\neq 1} \widehat{\mathbf{U}}_{j,0} \right) -\bE_{1,\net}\left(\otimes_{j\neq 1} \widehat{\mathbf{U}}_{j,\net} \right) }_2 \le 2^m C_1 \sigma \sqrt{\frac{ (2\mu)^{m-1} r_{-1} d^* d_1\log\dmax }{\dmax^{C_{\varepsilon m \rmax\dmax }} n}},
\end{equation*}
which is of a very small order. For each element in the $\operatorname{Net}(\varepsilon)$, we have the bound
\begin{equation}\label{eq:dep-E-init-Ber}
    \begin{aligned}
        \norm{\bE_{1,\init}^{\net}\left(\otimes_{j\neq 1} \widehat{\mathbf{U}}_{j,\net} \right)}_2 &\le C C_1 \sigma \left(\frac{d^*}{n} \sqrt{\frac{\dmax \log \dmax }{n}}\sqrt{\frac{\mu^{m-1} r_{-1}}{d_{-1}}}\cdot m \rmax\dmax \log \dmax  +  \sqrt{\frac{ m (2\mu)^{m-1} \rmax\dmax d_1 \dmax d^*\log^2 \dmax }{n^2}} \right) \\
        & \le C C_1 \sigma \sqrt{\frac{ m (2\mu)^{m-1}\rmax r^* \dmax d^*\log \dmax }{r_1 n}\cdot \frac{d_1\dmax \log \dmax}{n}}
    \end{aligned}
\end{equation}
that holds with probability at least $1-\dmax^{-C_{\varepsilon} m\rmax\dmax-3m}$ following the argument in \eqref{eq:E-init-Ber}. Then, we have the uniform bound of all the elements in the net with probability at least $1-\dmax^{-3m}$ with the order in \eqref{eq:dep-E-init-Ber}:

\begin{equation*}
    \norm{\bE_{1,\init}\left(\otimes_{j\neq 1} \widehat{\mathbf{U}}_{j,0} \right) }_2 \le C C_1 \sigma \sqrt{\frac{ m (2\mu)^{m-1} \rmax r^* \dmax d^*\log \dmax }{r_1 n}\cdot \frac{d_1\dmax \log \dmax}{n} } .
\end{equation*}
We now develop another phase. By writing $\cM_1(\bcX_i)$ as $\be_{1,i}\be_{-1,i}^\top$, we can use the dilation trick \citep{tropp2015introduction} to deal with the $\bE_{1,\init}\left(\otimes_{j\neq 1} \widehat{\mathbf{U}}_{j,0} \right) $, which gives:
\begin{equation*}
\begin{aligned} 
        & \frac{d^*}{n}\norm{\sum_{i=1}^n\langle\bcT^{\vartriangle} ,\bcX_i\rangle\cM_1(\bcX_i)\left(\otimes_{j\neq 1} \widehat{\mathbf{U}}_{j,0} \right) }_2  = \frac{d^*}{n}\norm{\sum_{i=1}^n\bcT^{\vartriangle}_{\bomega_i} \be_{1,i}\be_{-1,i}^\top\left(\otimes_{j\neq 1} \widehat{\mathbf{U}}_{j,0} \right) }_2 \\ 
        & =  \frac{d^*}{n}\norm{\sum_{i=1}^n\left[ \begin{array}{cc}
      & \bcT^{\vartriangle}_{\bomega_i} \be_{1,i}\be_{-1,i}^\top\left(\otimes_{j\neq 1} \widehat{\mathbf{U}}_{j,0}\right)\\ \bcT^{\vartriangle}_{\bomega_i} \left(\otimes_{j\neq 1} \widehat{\mathbf{U}}_{j,0}^\top  \right)\be_{-1,i}\be_{1,i}^\top &  \end{array} \right]
  }_2   \\
  & \le \frac{d^*}{n}\norm{\sum_{i=1}^n\abs{\bcT^{\vartriangle}_{\bomega_i} }\left[ \begin{array}{cc}
      K\be_{1,i}\be_{1,i}^\top & \\  &  \frac{1}{K}\left(\otimes_{j\neq 1} \widehat{\mathbf{U}}_{j,0}^\top  \right)\be_{-1,i} \be_{-1,i}^\top\left(\otimes_{j\neq 1} \widehat{\mathbf{U}}_{j,0} \right) \end{array} \right]
  }_2,  \\ 
  & \le  C_1\sqrt{\frac{\dmax \log\dmax}{n}}\cdot d^*\norm{\left[ \begin{array}{cc}
      \frac{\sum_{i=1}^n}{n} K\be_{1,i}\be_{1,i}^\top & \\  &  \frac{1}{K} \frac{(2\mu)^{m-1} r_{-1}}{d_{-1}}\bI_{d_{-1}} \end{array} \right]
  }_2
\end{aligned}
\end{equation*}
for any $K$ with probability at least $1-\dmax^{-3m}$, where the first inequality holds because $\left[ \begin{array}{cc}
     & \bx\by^\top \\ \by\bx^\top &  \end{array} \right] \preceq \left[ \begin{array}{cc} \bx\bx^\top
     & \\  &  \by\by^\top \end{array} \right]$ for any vectors $\bx$, $\by$. By basic matrix concentration, it is clear that with probability at least $1-\dmax^{-3m}$, we have $\norm{\frac{\sum_{i=1}^n}{n}\be_{1,i}\be_{1,i}^\top}_2\le\frac{1}{d_1}+C\frac{m\log\dmax}{n}+ C\sqrt{\frac{m\log\dmax}{n d_1}}\le \frac{2}{d_1}$. By choosing $K=\sqrt{\frac{d_1 (2\mu)^{m-1} r_{-1} }{d_{-1}}}$, we know that 

\begin{equation*}
    \begin{aligned}
         &\norm{\bE_{1,\init}\left(\otimes_{j\neq 1} \widehat{\mathbf{U}}_{j,0} \right) }_2 \le  \frac{d^*}{n}\norm{\sum_{i=1}^n\langle\bcT^{\vartriangle} ,\bcX_i\rangle\cM_1(\bcX_i)\left(\otimes_{j\neq 1} \widehat{\mathbf{U}}_{j,0} \right) }_2  + \norm{\cM_1(\bcT^{\vartriangle} )}_2 \\
         & \le  C C_1\left(\sqrt{\frac{(2\mu)^{m-1} r_{-1} d^*\dmax \log\dmax}{n}} +  \sqrt{\frac{\dmax d^*\log\dmax}{n}} \right)\le  C C_1\sqrt{\frac{(2\mu)^{m-1} r_{-1} d^*\dmax \log\dmax}{n}}.
    \end{aligned}
\end{equation*}
     The two-phase bound on $\norm{\bE_{1,\init}\left(\otimes_{j\neq 1} \widehat{\mathbf{U}}_{j,0} \right) }_2$ can thus be obtained.
    \end{proof}


The $\ell_{2,\infty}$ bound of $\bE_{1,\rn}\left( \otimes_{j\neq 1} \widehat{\bU}_{j,0} \right)$ and $\bE_{1,\init}\left( \otimes_{j\neq 1} \widehat{\bU}_{j,0} \right)$ can be treated accordingly by the $\varepsilon$-net argument used in \eqref{eq:E-rn-Ber-dep} and \eqref{eq:dep-E-init-uniform}.  

Such rates can be directly used in the control of SVD representations like term $\mathfrak{B}_2$, $\mathfrak{B}_3$ in \eqref{eq:svd-U1-decomp}, or the bound of $\cP_{\widehat{\bU}_{j,1}} - \cP_{\bU_j}$, etc.

\begin{Lemma}\label{lemma:E-init-rn-2inf-dep} Under the same assumptions as Theorem \ref{thm:clt-dep-init}, when the initialization depends on the data $\left\{(\bcX_i,\xi_i)\right\}_{i=1}^n$,  we have the following inequalities with probability at least $1-m 2^{m+1}\dmax^{-3m}$:
\begin{equation*}
    \begin{gathered}
            \norm{\bE_{1,\rn}\left(\otimes_{j\neq 1} \widehat{\mathbf{U}}_{j,0} \right)}_{2,\infty} \le C \sigma \sqrt{\frac{ m (2\mu)^{m-1} r_{-1} d^*\log \dmax }{n}}+ C \sigma \frac{ C_1\sigma  }{\lambda_{\min}}\sqrt{\frac{m^2 d^*\dmax \log \dmax }{ n}} \sqrt{\frac{ m \mu^{m-1}  \rmax  r^* \dmax d^*\log \dmax }{r_1 n}}\\
        \norm{\bE_{1,\init}\left(\otimes_{j\neq 1} \widehat{\mathbf{U}}_{j,0} \right)}_{2,\infty} \le C C_1 \sigma \sqrt{\frac{1}{d_1}} \sqrt{\frac{ m (2\mu)^{m-1}\rmax r^* \dmax d^*\log \dmax }{r_1 n}\cdot \frac{d_1 \dmax  \log \dmax}{n}\wedge 1}.\\
    \end{gathered}
\end{equation*}
\end{Lemma}
Lemma \ref{lemma:E-init-rn-2inf-dep} serves as counterpart of Lemma \ref{lemma:E-init-rn-2inf} under the dependent initialization. Its proof is much more involved than Lemma \ref{lemma:E-init-rn-2inf}.

\begin{proof}
    
The proof extends the discussion of Section \ref{apx:sec:proof-Einit-rn}. Still, we consider controlling $\norm{\be_{1,k}^\top \bE_{1,\rn}\left(\otimes_{j\neq 1} \widehat{\mathbf{U}}_{j,0} \right) }_{2}$ and $\norm{\be_{1,k}^\top\bE_{1,\init}\left(\otimes_{j\neq 1} \widehat{\mathbf{U}}_{j,0} \right) }_{2}$ for any $\be_{1,k}$ with $k\in[d_1]$ uniformly. We first consider $\norm{\be_{1,k}^\top \bE_{1,\rn}\left(\otimes_{j\neq 1} \widehat{\mathbf{U}}_{j,0} \right) }_{2}$. Following \eqref{eq:dep-B2-rn-decomp}, we do the splitting
\begin{equation}\label{eq:dep-l2inf-B2-rn-decomp}
\begin{aligned}
         \norm{\be_{1,k}^\top \bE_{1,\rn}\left( \otimes_{j\neq 1} \widehat{\bU}_{j,0} \right)}_2 & \le \frac{d^*}{n} \left( \norm{\sum_{i=1}^{n} \be_{1,k}^\top \xi_i \cM_1(\bcX_i)\left( \otimes_{j\neq 1} {\bU}_{j} \right)}_2 \right.
         \\
         & + \left. \underbrace{\sum_{j\neq 1} \norm{\sum_{i=1}^{n} \xi_i \be_{1,k}^\top  \cM_1(\bcX_i) \otimes_j \left( \cP_{\widehat{\bU}_{j,0}} - \cP_{\bU_{j}}\right)\left(  \otimes_{k'\neq j, 1} {\bU}_{k'} \right) }_2}_{\text{first-order term}}\right.\\
         &
         \left. + \underbrace{ \sum_{\cS\subseteq [m]/\{1\}, \abs{\cS}\ge2}\norm{\sum_{i=1}^{n} \xi_i\be_{1,k}^\top \cM_1(\bcX_i) \otimes_{j\in \cS } \left( \cP_{\widehat{\bU}_{j,0}} - \cP_{\bU_{j}}\right)\left(  \otimes_{k'\in \cS^c } {\bU}_{k'} \right)}_2}_{\text{higher-order term}} \right).  \\
\end{aligned}
\end{equation}
In \eqref{eq:dep-l2inf-B2-rn-decomp}, the first term is the sum of i.i.d. variables and can be controlled by \eqref{eq:E-rn-Ber-2inf}, which provides the first phase in the error bound. In the first-order and higher-order terms, we use a refined analysis of the concentration with Bernoulli random variables on the $\varepsilon$-net. We decompose the unfolding of each $\bcX_i$ as $\cM_1(\bcX_i)=\be_{1,k_i}\be_{-1, l_i}^\top$. Then it is clear that  $\be_{1,k_i}$ and $\be_{-1,l_i}$ are uniformly distributed over the canonical bases in $\R^{d_1}$ and $\R^{d_{-1}}$, and they are independent. Then, we define Bernoulli random variable $\idc_{k,i}=\be_{1,k}^\top \be_{1,k_i}=\idc\{k_i = k\}$. Thus, $\idc_{k,i}\sim \operatorname{Ber}(\frac{1}{d_1})$. Also, let the sum of these Bernoulli random variables be $\chi_{k}=\sum_{i=1}^n \idc_{k,i}$. To study the first-order term in  \eqref{eq:dep-l2inf-B2-rn-decomp},  it amounts to bounding each element in the net by choosing small enough $\varepsilon$:
\begin{equation*}
    \frac{d^*}{n}\norm{\sum_{}\xi_i \be_{1,k}^\top \cM_1(\bcX_i) \otimes_j \left( \cP_{\widehat{\bU}_{j,\net}} - \cP_{\bU_{j}} \right)\left(  \otimes_{k\neq j, 1} {\bU}_{k} \right)}_2.
\end{equation*}
Since $\idc_{k,i}$ and $\chi_{k}$ are independent of $\xi_i$ and $\be_{-1, l_i}$, conditional on each realization of $\{\idc_{k,i} \}_{i=1}^n$, we have
\begin{equation*}
\begin{aligned}
        &\frac{d^*}{n}\norm{\sum_{i=1}^n\xi_i \be_{1,k}^\top \cM_1(\bcX_i) \otimes_j \left( \cP_{\widehat{\bU}_{j,\net}} - \cP_{\bU_{j}} \right)\left(  \otimes_{k\neq j, 1} {\bU}_{k} \right)}_2 \mid\{\idc_{k,i} \}_{i=1}^n \\
    & = \frac{\chi_{k}}{n}\cdot \frac{d^*}{\chi_{k}} \norm{\sum_{\idc_{k,i}=1}\xi_i  \be_{-1, l_i}^\top \otimes_j \left( \cP_{\widehat{\bU}_{j,\net}} - \cP_{\bU_{j}} \right)\left(  \otimes_{k\neq j, 1} {\bU}_{k} \right)}_2\mid\{\idc_{k,i} \}_{i=1}^n. 
\end{aligned}
\end{equation*}
Now, conditional on $\{\idc_{k,i} \}_{i=1}^n$, we have 
\begin{equation*}
    \begin{gathered}
        \norm{\be_{-1, l_i}^\top \otimes_j \left( \cP_{\widehat{\bU}_{j,\net}} - \cP_{\bU_{j}} \right)\left(  \otimes_{k\neq j, 1} {\bU}_{k} \right)}_2\le C \frac{ C_1\sigma  }{\lambda_{\min}}\sqrt{\frac{d^*\dmax \log \dmax }{n}}\cdot \sqrt{\frac{\mu^{m-1} r_{-1}}{d_{-1}}} \\
    \end{gathered}
\end{equation*}
By the concentration inequality with sub-Gaussian series in \cite{tropp2012user},  This leads to the following bound:
\begin{equation}\label{eq:dep-l2inf-B2-rn-order1-net}
\begin{aligned}
        &\frac{d^*}{\chi_{k}} \norm{\sum_{\idc_{k,i}=1}\xi_i  \be_{-1, l_i}^\top \otimes_j \left( \cP_{\widehat{\bU}_{j,\net}} - \cP_{\bU_{j}} \right)\left(  \otimes_{k\neq j, 1} {\bU}_{k} \right)}_2\mid\{\idc_{k,i} \}_{i=1}^n \\
        &\le C \sigma \frac{ C_1\sigma  }{\lambda_{\min}}\sqrt{\frac{d^*\dmax \log \dmax }{n}}   \sqrt{\frac{ m \mu^{m-1} \rmax r^*  d_1 \dmax d^*\log \dmax }{r_1 \chi_{k}}}\\
\end{aligned}
\end{equation}
for each element in the $\varepsilon$-net with probability at least $1-\dmax^{ C_{\varepsilon}m \rmax\dmax -3 m}$. Thus, \eqref{eq:dep-l2inf-B2-rn-order1-net} holds uniformly for all elements in the $\varepsilon$-net with probability at least $1-\dmax^{ -3 m}$. By the basic concentration of Bernoulli random variables, it is clear that 
\begin{equation}\label{eq:idc-concentration}
    \bbP\left(\frac{\chi_{k}}{n}\ge \frac{1}{d_1} +\frac{Cm\log\dmax }{n} +\sqrt{\frac{Cm\log\dmax}{d_1 n}}\right) \le \dmax^{-3m} .
\end{equation}
Given that $n\ge C m\dmax\log\dmax$, we have $\chi_{k} \le 2 n/d_1 $ with probability at least $1-\dmax^{-3m}$. Combining $\chi_{k} \le 2 n/d_1 $ with \eqref{eq:dep-l2inf-B2-rn-order1-net}, we conclude that with probability at least $1-2\dmax^{-3m}$, we have 
\begin{equation}\label{eq:dep-l2inf-E-rn-Ber}
    \begin{aligned}
        &\frac{d^*}{n}\norm{\sum_{i=1}^n\xi_i \be_{1,k}^\top \cM_1(\bcX_i) \otimes_j \left( \cP_{\widehat{\bU}_{j,\net}} - \cP_{\bU_{j}} \right)\left(  \otimes_{k\neq j, 1} {\bU}_{k} \right)}_2  \\
        & = \frac{\chi_{k}}{n}\cdot \frac{d^*}{\chi_{k}} \norm{\sum_{\idc_{k,i}=1}\idc_{k,i}\xi_i  \be_{-1, l_i}^\top \otimes_j \left( \cP_{\widehat{\bU}_{j,\net}} - \cP_{\bU_{j}} \right)\left(  \otimes_{k\neq j, 1} {\bU}_{k} \right)}_2 \\
        &\le C \sigma \frac{ C_1\sigma  }{\lambda_{\min}}\sqrt{\frac{d^*\dmax \log \dmax }{n}} \sqrt{\frac{ m \mu^{m-1} \rmax r^*  d_1 \dmax d^*\log \dmax }{r_1 n}}\sqrt{\frac{\chi_{k}}{n} } \le C \sigma \frac{ C_1\sigma  }{\lambda_{\min}}\sqrt{\frac{d^*\dmax \log \dmax }{ n}} \sqrt{\frac{ m \mu^{m-1}  \rmax  r^* \dmax d^*\log \dmax }{r_1 n}}.
    \end{aligned}
\end{equation}
Notice that, this bound uses the size of each i.i.d. matrix, and it is clear that such an argument can be naturally extended to higher-order terms. Thus, the total higher-order term is also controlled by \eqref{eq:dep-l2inf-E-rn-Ber}. This gives the second phase of the error bound.

\noindent Then, we focus on $ \norm{\be_{1,k}^\top\bE_{1,\init}\left(\otimes_{j\neq 1} \widehat{\mathbf{U}}_{j,0} \right)}_2$ by studying the uniform bound of each element, i.e., $\norm{\be_{1,k}^\top\bE_{1,\init}^{\net}\left(\otimes_{j\neq 1} \widehat{\mathbf{U}}_{j,\net} \right)}_2$ in the $\varepsilon$-net. Using the argument above again, we can check that the following inequality: 
\begin{equation*}
    \begin{aligned}
        \norm{\be_{1,k}^\top\bE_{1,\init}^{\net}\left(\otimes_{j\neq 1} \widehat{\mathbf{U}}_{j,\net} \right)}_2 &\le C C_1 \sigma \sqrt{\frac{ m (2\mu)^{m-1} \rmax r^* \dmax d_1 \dmax d^*\log^2 \dmax }{r_1 n^2}}\sqrt{\frac{\chi_{k}}{n} } \\
        & \le C C_1 \sigma \sqrt{\frac{ m (2\mu)^{m-1} \rmax r^* \dmax d^*\log \dmax }{r_1 n}\cdot \frac{\dmax \log \dmax}{n}}
    \end{aligned}
\end{equation*}
holds uniformly with probability at least $1-2\dmax^{-3m}$. Thus, we conclude that
\begin{equation}\label{eq:dep-l2inf-E-init-Ber}
        \norm{\be_{1,k}^\top\bE_{1,\init}\left(\otimes_{j\neq 1} \widehat{\mathbf{U}}_{j,0} \right) }_2 \le C C_1 \sigma \sqrt{\frac{ m (2\mu)^{m-1}  \rmax r^* \dmax d^*\log \dmax }{r_1 n}\cdot \frac{\dmax \log \dmax}{n}}.
\end{equation}

Applying the results in \eqref{eq:dep-l2inf-E-rn-Ber} and \eqref{eq:dep-l2inf-E-init-Ber} to all the $k\in[d_1]$, we can prove the first phase. For the second phase, we use $\idc_{k,i}=\be_{1,k}^\top \be_{1,k_i}=\idc\{k_i = k\}$ and apply the simple size bound:
\begin{equation*}
    \begin{aligned}
        &\frac{d^*}{n}\norm{\sum_{i=1}^n\left[\bcT^{\vartriangle}\right]_{\omega_i} \be_{1,k}^\top\be_{1,i}\be_{-1,i}^\top\left(\otimes_{j\neq 1} \widehat{\mathbf{U}}_{j,0} \right) }_2  =  \frac{d^*}{n} \norm{\sum_{i=1}^{n} \left[\bcT^{\vartriangle}\right]_{\omega_i}\idc_{k,i} \be_{-1, l_i}^\top\left(\otimes_{j\neq 1} \widehat{\mathbf{U}}_{j,0} \right)  }_2 \\
        &\le C_1\sqrt{\frac{\dmax\log\dmax}{n}}\sqrt{\frac{(2\mu)^{m-1}r_{-1}}{d_{-1}}}\frac{d^*}{n}\sum_{i=1}^n \idc_{k,i}.
    \end{aligned}
\end{equation*}
Using the concentration in \eqref{eq:idc-concentration}, we have the following inequality with the same probability:
\begin{equation*}
\begin{aligned}
         \norm{\be_{1,k}^\top\bE_{1,\init}\left(\otimes_{j\neq 1} \widehat{\mathbf{U}}_{j,0} \right) }_2 & \le \frac{d^*}{n}\norm{\sum_{i=1}^n\left[\bcT^{\vartriangle}\right]_{\omega_i} \be_{1,k}^\top\be_{1,i}\be_{-1,i}^\top\left(\otimes_{j\neq 1} \widehat{\mathbf{U}}_{j,0} \right) }_2 + \norm{\be_{1,k}^\top\cM_1(\bcT^{\vartriangle})}_2 \\
     & \le  C C_1\sqrt{\frac{\dmax d^*\log\dmax}{n}}\sqrt{\frac{(2\mu)^{m-1}r_{-1}}{d_1 } }.
\end{aligned}
\end{equation*}
Thus, we finish the proof of the two-phase error bound.
\end{proof}

We now give a detailed study on the singular subspace perturbation using results in Lemma \ref{lemma:Einit-Ern-dep-l2}, \ref{lemma:E-init-rn-2inf-dep}:

\begin{Lemma}[Perturbation bound with dependent initialization]\label{lemma:l2inf-powerit-dep} 
Under the assumptions of Theorem \ref{thm:clt-dep-init}, after the power iteration, the following perturbation bounds:
    \begin{equation*}
\begin{aligned}
        \norm{\cP_{\widehat{\bU}_{j,1}} - \cP_{\bU_j}}_{2} &\le \frac{C C_1 \sigma}{\lambda_{\min} }\sqrt{\frac{  m (2\mu)^{m-1}\rmax r^* d_j  d^*\log \dmax }{ r_j n}} \\
        \norm{\cP_{\widehat{\bU}_{j,1}} - \cP_{\bU_j}}_{2,\infty} & \le  CC_1\sqrt{\frac{\mu r_j }{d_j}} \left(\frac{\sigma}{\lambda_{\min} }\sqrt{\frac{  m (2\mu)^{m-1}\rmax r^* d_j d^*\log \dmax }{ r_j n}} + \frac{C_1 \sigma^2}{\lambda_{\min}^2 }\frac{  m^{\frac{3}{2}} (2\mu)^{\frac{1}{2}m}\rmax^{\frac{1}{2}} (r^*)^{\frac{1}{2}} d^*\dmax^{\frac{3}{2}} \log \dmax }{ \sqrt{r_j} n} \right) \\
         \norm{\cP_{\widehat{\bU}_{j,1}} - \cP_{\bU_j}}_{\linf} & \le CC_1\frac{\mu r_j}{d_j} \cdot \left(\frac{\sigma}{\lambda_{\min} }\sqrt{\frac{  m (2\mu)^{m-1}\rmax r^* d_j  d^*\log \dmax }{ r_j n}} + \frac{C_1\sigma^2}{\lambda_{\min}^2 }\frac{  m^{\frac{3}{2}} (2\mu)^{\frac{1}{2}m}\rmax^{\frac{1}{2}} (r^*)^{\frac{1}{2}} d^*\dmax^{\frac{3}{2}} \log \dmax }{ \sqrt{r_j} n}\right). \\
    \end{aligned}
\end{equation*}
hold uniformly for all $j\in[m]$ with probability at least $1-3m^{2} 2^{m+3} \dmax^{-3m}$.
\end{Lemma}
The proof of this Lemma is deferred to Section \ref{apx:sec:proof-l2inf-powerit-dep}.  With the help of Lemma \ref{lemma:Einit-Ern-dep-l2}, \ref{lemma:E-init-rn-2inf-dep}, \ref{lemma:l2inf-powerit-dep}, and results developed above in Lemma \ref{lemma:oracle-init-frE-2-5}, we can derive the following error control on the remainder terms:
\begin{Lemma}[Controlling remainder terms with dependent initialization]\label{lemma:dep-init-frE-2-5} Under the conditions in Theorem \ref{thm:clt-dep-init}, the terms from $\frE_2$ to $\frE_5$ can be controlled by
\begin{equation*}
    \begin{aligned}
              \abs{\frE_2} &\le   C C_1 \sigma \left(\norm{\cP_{\TT}(\bcI)}_{\tF}\sqrt{d^*}\wedge \norm{\bcI}_{\ell_1}\right) \sqrt{\frac{ m (2\mu)^{m} \rmax^2 r^*  \dmax \log \dmax }{ n}\cdot \frac{\dmax \log \dmax}{n} \wedge \frac{1}{\dmin} } \\
       \abs{\frE_3} & \le C C_1\norm{ \cP_{\TT}(\bcI) }_{\tF} \cdot \frac{\sigma^2}{\lambda_{\min}}\sqrt{\frac{m^2 d^*\dmax\log\dmax}{n}}\sqrt{\frac{  m^2 \mu^{m-1} (r^*)^2 \dmax  d^*\log \dmax }{\rmin^3 n}} \\
       & +C C_1\norm{\bcI}_{\ell_1} \frac{\sigma^2}{\lambda_{\min} } \frac{  m^2 (2\mu)^{\frac{3}{2}m}\rmax (r^*)^{\frac{3}{2}} \dmax\sqrt{d^*}  \log \dmax }{ \rmin n} \cdot\left(1  + \frac{C_1 \sigma}{\lambda_{\min} }\frac{  m \rmax^{\frac{1}{2}} (r^*)^{\frac{1}{2}}  \dmax \sqrt{d^*} \sqrt{\log \dmax} }{ \sqrt{n}} \right)  \\
       & + C C_1\norm{ \cP_{\TT}(\bcI) }_{\tF} \cdot \sigma \sqrt{\frac{ m^3 (2\mu)^{m-1} r^* \dmax^2 d^* \log^2\dmax }{  r_j n^2}}\\
        \abs{\frE_4}& \le C C_1 \norm{\bcI}_{\ell_1} \frac{\sigma^2}{\lambda_{\min} } \frac{  m^2 (2\mu)^{\frac{3}{2}m}\rmax (r^*)^{\frac{3}{2}} \dmax\sqrt{d^*}  \log \dmax }{ \rmin n} \cdot\left(1  + \frac{C_1 \sigma}{\lambda_{\min} }\frac{  m \rmax^{\frac{1}{2}} (r^*)^{\frac{1}{2}}  \dmax \sqrt{d^*} \sqrt{\log \dmax} }{ \sqrt{n}} \right) \\
        \abs{\frE_5} & \le   C C_1^2 \norm{\bcI}_{\ell_1} \frac{\sigma^2 }{\lambda_{\min} }\frac{  m^2 (2\mu)^{\frac{3}{2}m-1}\rmax (r^*)^{\frac{3}{2}} \dmax  \sqrt{d^*} \log \dmax }{ \rmin n} \left( 1 + \frac{C_1^2 \sigma^2 }{\lambda_{\min}^2 }\frac{m d^* \dmax^2 \log\dmax}{n} \right) \\
    \end{aligned}
\end{equation*}
uniformly with probability at least $1-\dmax^{-2m}$.
\end{Lemma}
This lemma essentially follows Lemma \ref{lemma:oracle-init-frE-2-5} and its proof is also similar to that described Section \ref{apx:sec:proof-all-rem-ind}. We only need to revisit the proof in Section \ref{apx:sec:proof-all-rem-ind} and plug in new error rates under the dependent initialization conditions. We defer the proof of this lemma to Section \ref{apx:sec:proof-lemma:dep-init-frE-2-5}.

According to \eqref{eq:core-clt} and \eqref{eq:test-stat-popvar}, the term $\frE_1$ converge to asymptotic normal distribution, when the SNR satisfies 
\begin{equation*}
   \frac{\lambda_{\min}}{ \sigma} \ge  C_{\gap}\kappa_0 C_1\sqrt{\frac{  m^4 (2\mu)^{m-1} (r^*)^2   d^* \dmax^2 \log^2 \dmax }{ n}}, 
\end{equation*}
the remainder terms from $\frE_2$ to $\frE_5$ can be bounded by
\begin{equation}\label{eq:test-stat-popvar-rem-dep}
\begin{aligned}
        \abs{\frac{\sum_{i=2}^{5}\frE_i}{\sigma\norm{\cP_{\TT}(\bcI)}_\tF\sqrt{d^*/n} }} & \le  C C_1 \sqrt{\frac{m^3(2\mu)^m \rmax^2 r^*\dmax^2\log^2\dmax }{n}} + C C_1\frac{\sigma}{\lambda_{\min}}\sqrt{\frac{  m^4 \mu^{m-1} (r^*)^2   d^* \dmax^2 \log^2 \dmax }{\rmin^3 n}}  \\
        & + C \frac{ C_1^2\sigma  \norm{\bcI}_{\ell_1}}{\lambda_{\min}  \norm{\cP_{\TT}(\bcI)}_\tF }\frac{  m^2 (2\mu)^{\frac{3}{2}m}\rmax (r^*)^{\frac{3}{2}} \dmax \log \dmax }{ \rmin \sqrt{n}}.
\end{aligned}
\end{equation}
Given the alignment condition in Assumption \ref{asm:alignment}, we have the following bound from $\frE_2$ to $\frE_5$:
\begin{equation*}
    \begin{aligned}
       \abs{\frac{\sum_{i=2}^{5}\frE_i}{\sigma\norm{\cP_{\TT}(\bcI)}_\tF\sqrt{d^*/n} }} & \le  C C_1 \sqrt{\frac{m^3(2\mu)^m \rmax^2 r^*\dmax^2\log^2\dmax }{n}} + C C_1\frac{\sigma}{\lambda_{\min}}\sqrt{\frac{  m^4 \mu^{m-1} (r^*)^2   d^* \dmax^2 \log^2 \dmax }{\rmin^3 n}}  \\
        & + C \frac{ C_1^2\sigma  \norm{\bcI}_{\ell_1}}{\alpha_I \lambda_{\min}  \norm{\bcI}_\tF }\sqrt{\frac{ m^4 (2\mu)^{3 m} \rmax^2 (r^*)^{3 }\dmax d^*  \log^2 \dmax }{\rmin^2 n } }
    \end{aligned}
\end{equation*}
with probability at least $1- \dmax^{-2m}$. Associated with \eqref{eq:core-clt}, it is clear that 
\begin{equation*}
\begin{aligned}
        &\max_{t\in\R}\abs{\bbP\left( \frac{\left\langle  \widehat{\bcT} - \bcT,\bcI  \right\rangle }{ \sigma\norm{\cP_{ \TT }(\bcI)}_\tF \sqrt{{d^*}/{n}} } \le t\right)- \Phi(t)} \le  C C_1 \sqrt{\frac{m^3(2\mu)^m \rmax^2 r^*\dmax^2\log^2\dmax }{n}} \\
        & \quad \quad + C C_1\frac{\sigma}{\lambda_{\min}}\sqrt{\frac{  m^4 \mu^{m-1} (r^*)^2   d^* \dmax^2 \log^2 \dmax }{\rmin^3 n}} + C \frac{ C_1^2\sigma  \norm{\bcI}_{\ell_1}}{\alpha_I \lambda_{\min}  \norm{\bcI}_\tF }\sqrt{\frac{ m^4 (2\mu)^{3 m} \rmax^2 (r^*)^{3 }\dmax d^*  \log^2 \dmax }{\rmin^2 n } }
\end{aligned}
\end{equation*}
\end{proof}
\subsubsection{Improved bound with leave-one-out initialization to prove Theorem \ref{thm:clt-lol-init}}
Next, we assume the leave-one-out initialization in \eqref{eq:asp-loo-init} holds. Then, we can show that the improved $\ell_{2,\infty}$ norm bound (and consequently the $\ell_{2}$ bound) can be achieved 
\begin{equation}
    \norm{\widehat \bcT_{\init}^{(k,l)} -   \bcT }_{\linf } \le C_2 \sigma \sqrt{\frac{\dmax \log \dmax }{n}}, \quad    \norm{ \cP_{\widehat{\bU}_{j,0}^{(k,l)}} - \cP_{\widehat{\bU}_{j,0}} }_{\tF} \le C_2\frac{\sigma}{\lambda_{\min}}\sqrt{\frac{\mu r_k}{d_k}} \sqrt{   \frac{ d^* \dmax \log \dmax }{ n }}
\end{equation}


Moreover, we assume the sample size: $n\ge C_{\gap}C_2^2\kappa_0^2m^2(2\mu)^m\rmax r^*\dmax\log\dmax$
\begin{Lemma}\label{lemma:loo-l2inf-ipv}
    With the leave-one-out initialization in \eqref{eq:asp-loo-init}, the following improvement on the $\ell_{2,\infty}$ norm bound can be made:
\begin{subequations}\label{eq:loo-l2inf-ipv}
\begin{align}
    &\norm{\bE_{1,\rn} \left( \otimes_{j\neq 1}\cP_{\widehat{\mathbf{U}}_{j,0}} -\otimes_{j\neq 1}\cP_{{\mathbf{U}}_{j}} \right)}_{2,\infty} \le C \sigma \frac{ C_2\sigma  }{\lambda_{\min}}\sqrt{\frac{d^*\dmax \log \dmax }{n}} \sqrt{\frac{ m^3 \alpha_d (2\mu)^{m-1} r_{-1}  d^*\log \dmax }{n}} \label{eq:loo-l2inf-ipv-1}\\
    &  \norm{\bE_{1,\rn}\left(\otimes_{j\neq 1} \widehat{\mathbf{U}}_{j,0} \right)}_{2,\infty} \le C \sigma \sqrt{\frac{ m (2\mu)^{m-1} r_{-1} d^*\log \dmax }{n}}\label{eq:loo-l2inf-ipv-2}\\
    & \norm{\bE_{1,\init}\left(\otimes_{j\neq 1} \widehat{\mathbf{U}}_{j,0} \right)}_{2,\infty} \le CC_2 \kappa_0\sigma \sqrt{\frac{m^3 (2\mu)^{2m-1} (r^*)^2 d^*  \dmax \log^2\dmax  }{ n^2}} \label{eq:loo-l2inf-ipv-3}\\
    &\norm{\cP_{\widehat{\bU}_{1,1}} - \cP_{\bU_1}}_{2,\infty} \le  \frac{C C_2\sigma  }{\lambda_{\min}}\sqrt{\frac{ m (2\mu)^{m-1} r_{-1} d_1  d^*\log \dmax }{n}}\cdot\sqrt{\frac{\mu r_1}{d_1}}\label{eq:loo-l2inf-ipv-4} 
\end{align}
\end{subequations}
with probability at least $1-3 m 2^{m+3} \dmax^{-3m}$.
\end{Lemma}
\begin{proof}

    We begin with \eqref{eq:loo-l2inf-ipv-1}. For the $\ell_{2,\infty}$ norm, it is equivalent to study the $\ell_{2}$ norm when multiplied by $\be_{1,k_1}$ for any $k_1\in [d_1]$, i.e.,
    \begin{equation*}
\begin{aligned}
            &\norm{\be_{1,k_1}^\top\bE_{1,\rn} \left( \otimes_{j\neq 1}\cP_{\widehat{\mathbf{U}}_{j,0}} -\otimes_{j\neq 1}\cP_{{\mathbf{U}}_{j}} \right) }_2\\
            & \le \norm{\be_{1,k_1}^\top\bE_{1,\rn} \left( \otimes_{j\neq 1}\cP_{\widehat{\mathbf{U}}_{j,0}} -\otimes_{j\neq 1}\cP_{\widehat \bU_{j,0}^{(1,k_1)}} \right) }_2 + \underbrace{\norm{\be_{1,k_1}^\top\bE_{1,\rn} \left( \otimes_{j\neq 1}\cP_{\widehat \bU_{j,0}^{(1,k_1)}} -\otimes_{j\neq 1}\cP_{{\mathbf{U}}_{j}} \right) }_2}_{\text{independent term}}\\
\end{aligned}
    \end{equation*}
where we borrow the leave-out-out initialization $\widehat{\bcT}_{\init}^{(1,k_1)}=\widehat{\bcC}_0^{(1,k_1)}\times_{j=1}^m \widehat \bU_{j,0}^{(1,k_1)}$ which is independent of the observations from the $k_1$ slice on mode $1$. That is to say,  all the non-zero terms in $\sum_{i=1}^{n} \xi_i \be_{1,k_1}^\top  \cM_1(\bcX_i)$ is independent of the singular subspace $\cP_{\widehat \bU_{j,0}^{(1,k_1)}} $. Thus, the second term can be simply handled by the rate under independent initialization just following Lemma \ref{lemma:E-rn-Udiff}:
\begin{equation}\label{eq:loo-l2inf-ipv-1-base} 
                \norm{\be_{1,k_1}^\top\bE_{1,\rn} \left( \otimes_{j\neq 1}\cP_{\widehat \bU_{j,0}^{(1,k_1)}} -\otimes_{j\neq 1}\cP_{{\mathbf{U}}_{j}} \right) }_2 \le C \sigma \frac{ C_2\sigma  }{\lambda_{\min}}\sqrt{\frac{d^*\dmax \log \dmax }{n}} \sqrt{\frac{ m (2\mu)^{m-1} r_{-1}  d^*\log \dmax }{n}}.
\end{equation}
We now focus on the first term:
 \begin{equation*}
        \begin{aligned}
        & \norm{\be_{1,k_1}^\top\bE_{1,\rn} \left( \otimes_{j\neq 1}\cP_{\widehat{\mathbf{U}}_{j,0}} -\otimes_{j\neq 1}\cP_{\widehat \bU_{j,0}^{(1,k_1)}} \right) }_2 \\
        & \le { \sum_{\cS\subseteq [m]/\{1\}, \abs{\cS}\ge 1}\norm{\frac{ d^* }{n}  \sum_{i=1}^{n} \xi_i\be_{1,k}^\top \cM_1(\bcX_i) \otimes_{j\in \cS } \left( \cP_{\widehat{\bU}_{j,0}} - \cP_{\widehat \bU_{j,0}^{(1,k_1)}}\right)\left(  \otimes_{k'\in \cS^c } \widehat \bU_{k',0}^{(1,k_1)} \right)}_2},  \\
        \end{aligned}
    \end{equation*}
in which each term can be split by:
\begin{equation}\label{eq:loo-l2inf-ipv-1-add}
    \begin{aligned}
        & \norm{\frac{ d^* }{n} \sum_{i=1}^{n}\xi_i \be_{1,k_1}^\top  \cM_1(\bcX_i) \otimes_{j\in \cS} \left( \cP_{\widehat{\bU}_{j,0}} - \cP_{\widehat \bU_{j,0}^{(1,k_1)}}\right)\left(  \otimes_{k'\in \cS^c} \widehat \bU_{k',0}^{(1,k_1)} \right) }_2  \\
        & =\norm{\frac{ d^* }{n} \sum_{i=1}^{n}\xi_i \be_{1,k_1}^\top  \cM_1(\bcX_i) \left[\otimes_{j\in \cS} \bI_{d_j}   \otimes_{k'\in \cS^c } \widehat \bU_{k',0}^{(1,k_1)} \right] \left[\otimes_{j\in \cS} \left( \cP_{\widehat{\bU}_{j,0}} - \cP_{\widehat \bU_{j,0}^{(1,k_1)}}\right)\otimes_{k'\in \cS^c} \bI_{d_{k'}}\right] }_2 \\
        & \le \norm{\frac{ d^* }{n} \sum_{i=1}^{n}\xi_i \be_{1,k_1}^\top  \cM_1(\bcX_i) \left[\otimes_{j\in \cS} \bI_{d_j}   \otimes_{k'\in \cS^c} \widehat \bU_{k',0}^{(1,k_1)} \right]  }_2 \cdot \prod_{j\in \cS}\norm{ \cP_{\widehat{\bU}_{j,0}^{(1,k_1)}} - \cP_{\widehat{\bU}_{j,0}} }_{\tF} \\
        & \le C\sigma \sqrt{\frac{m\mu^{m-2} r^* d^* \prod_{j\in \cS} d_j\log\dmax  }{r_1 \prod_{j\in \cS}\mu r_j n}}\left(\frac{C_2 \sigma}{\lambda_{\min}}\sqrt{\frac{\mu r_1}{d_1}} \sqrt{ \frac{ d^* \dmax \log \dmax }{ n }}\right)^{\abs{\cS}} \\
        & \le C (\frac{1}{4})^{\abs{\cS}}  C_2\sigma\frac{ \sigma}{\lambda_{\min}} \sqrt{\frac{m \alpha_d \mu^{m-1} r^* (d^* )^2 \dmax \log\dmax  }{ n^2}},
    \end{aligned}
\end{equation}
where the first inequality is due to the fact $\norm{\otimes_{j\in \cS} \bA_j\otimes_{k'\cS^c} \bI_{d_{k'}}}_2 = \prod_{j\in \cS}\norm{\bA_j}_2\le\prod_{j\in \cS}\norm{\bA_j}_{\tF} $ for any $\bA_j$, and the second inequality is from the matrix concentration inequality:
\begin{equation*}
\begin{aligned}
        & \norm{\frac{ d^* }{n} \sum_{i=1}^{n}\xi_i \be_{1,k_1}^\top  \cM_1(\bcX_i) \left[\otimes_{j\in \cS} \bI_{d_j}   \otimes_{k'\in \cS^c} \widehat \bU_{k',0}^{(1,k_1)} \right]   }_2 \\
        & \le C \sigma \frac{m \log\dmax }{n} \sqrt{\frac{ \mu^{m-2} r^* d^* d_1 \prod_{j\in \cS} d_j }{r_1 \prod_{j\in \cS} \mu r_j}} + C\sigma \sqrt{\frac{m\mu^{m-2} r^* d^* \prod_{j\in \cS} d_j\log\dmax  }{r_1 \prod_{j\in \cS} \mu r_j n}},
\end{aligned}
\end{equation*}
which holds with probability at least $1-d^{-3m}$.  Summing up \eqref{eq:loo-l2inf-ipv-1-base}, \eqref{eq:loo-l2inf-ipv-1-add} we know that under the SNR
\begin{equation}\label{eq:apx-loo-SNR}
    \frac{C_2 \sigma}{\lambda_{\min}} \sqrt{ \frac{ \alpha_d \mu \rmax d^* \dmax \log \dmax }{ n }}\le 1/C_{\gap},
\end{equation}
the desired bound holds with probability at least $1-2^{m+1}d^{-3m}$.

\eqref{eq:loo-l2inf-ipv-2} immediately follows from \eqref{eq:loo-l2inf-ipv-1} by noticing that 
\begin{equation*}
    \norm{\bE_{1,\rn} \left( \otimes_{j\neq 1}\cP_{\widehat{\mathbf{U}}_{j,0}} \right)}_{2,\infty} \le \norm{\bE_{1,\rn} \left( \otimes_{j\neq 1}\cP_{\widehat{\mathbf{U}}_{j,0}} -\otimes_{j\neq 1}\cP_{{\mathbf{U}}_{j}} \right)}_{2,\infty} + \norm{\bE_{1,\rn} \left(\otimes_{j\neq 1}\cP_{{\mathbf{U}}_{j}} \right)}_{2,\infty},
\end{equation*}
where the latter has been bounded by  Lemma \ref{lemma:E-init-rn-2inf}.

We now consider \eqref{eq:loo-l2inf-ipv-3} by noticing that:
\begin{equation*}
\begin{aligned}
        & \norm{\be_{1,k_1}^\top\bE_{1,\rn} \left( \otimes_{j\neq 1}\cP_{\widehat{\mathbf{U}}_{j,0}}\right) }_2 \le \norm{ \frac{d^*}{n}\sum_{i=1}^n\langle \bcT^{\vartriangle} , \bcX_i\rangle \be_{1,k_1}^\top\cM_1(\bcX_i)\left( \otimes_{j\neq 1}\cP_{\widehat{\mathbf{U}}_{j,0}} \right) }_2 \\
\end{aligned}
\end{equation*}
where we can further control the first term by

\begin{equation}\label{eq:loo-Einit-decomp}
\begin{aligned}
         & \norm{\be_{1,k_1}^\top\bE_{1,\init} \left( \otimes_{j\neq 1}\cP_{\widehat{\mathbf{U}}_{j,0}}\right) }_2  = \norm{ \left[\frac{d^*}{n}\sum_{i=1}^n \langle \bcT^{\vartriangle} , \bcX_i\rangle \be_{1,k_1}^\top\cM_1(\bcX_i)-\be_{1,k_1}^\top\cM_1(\bcT^{\vartriangle})\right]\left( \otimes_{j\neq 1}\cP_{\widehat{\mathbf{U}}_{j,0}} \right) }_2  \\
        & \le  \norm{ \left[\frac{d^*}{n}\sum_{i=1}^n \langle \bcT^{\vartriangle}_{(1,k_1)} , \bcX_i\rangle \be_{1,k_1}^\top\cM_1(\bcX_i)-\be_{1,k_1}^\top\cM_1(\bcT^{\vartriangle}_{(1,k_1)})\right]\left( \otimes_{j\neq 1}\cP_{\widehat{\mathbf{U}}_{j,0}} \right) }_2 \\
        & \quad + \norm{ \left[\frac{d^*}{n}\sum_{i=1}^n \langle \bcT^{\vartriangle}-\bcT^{\vartriangle}_{(1,k_1)} , \bcX_i\rangle \be_{1,k_1}^\top\cM_1(\bcX_i)-\be_{1,k_1}^\top\cM_1(\bcT^{\vartriangle}-\bcT^{\vartriangle}_{(1,k_1)})\right]\left( \otimes_{j\neq 1}\cP_{\widehat{\mathbf{U}}_{j,0}} \right) }_2\\
       &  \le  \underbrace{\norm{ \left[\frac{d^*}{n}\sum_{i=1}^n \langle \bcT^{\vartriangle}_{(1,k_1)} , \bcX_i\rangle \be_{1,k_1}^\top\cM_1(\bcX_i)-\be_{1,k_1}^\top\cM_1(\bcT^{\vartriangle}_{(1,k_1)})\right]\left(\otimes_{j\neq 1}\cP_{\widehat \bU_{j,0}^{(1,k_1)}} \right) }_2}_{\frH_1} \\
       & \quad + \underbrace{\norm{ \left[\frac{d^*}{n}\sum_{i=1}^n \langle \bcT^{\vartriangle}_{(1,k_1)} , \bcX_i\rangle \be_{1,k_1}^\top\cM_1(\bcX_i)-\be_{1,k_1}^\top\cM_1(\bcT^{\vartriangle}_{(1,k_1)})\right]\left( \otimes_{j\neq 1}\cP_{\widehat{\mathbf{U}}_{j,0}}  -\otimes_{j\neq 1}\cP_{\widehat \bU_{j,0}^{(1,k_1)}} \right) }_2}_{\frH_2}  \\
        & \quad + \underbrace{\norm{ \left[\frac{d^*}{n}\sum_{i=1}^n \langle \widehat \bcT_{\init}-\widehat \bcT_{\init}^{(k,l)} , \bcX_i\rangle \be_{1,k_1}^\top\cM_1(\bcX_i)-\be_{1,k_1}^\top\cM_1(\widehat \bcT_{\init}-\widehat \bcT_{\init}^{(k,l)})\right]\left( \otimes_{j\neq 1}\cP_{\widehat{\mathbf{U}}_{j,0}} \right) }_2}_{\frH_3}\\
\end{aligned}
\end{equation}
In $\frH_1$ of \eqref{eq:loo-Einit-decomp}, the $\bcT^{\vartriangle}_{(1,k_1)}$ and singular subspace $\left(\otimes_{j\neq 1}\cP_{\widehat \bU_{j,0}^{(1,k_1)}} \right)$ are independent of the observations in the $k_1$ slice, i.e., the non-zero terms of $\be_{1,k_1}^\top\cM_1(\bcX_i)$. Thus, we can simply invoke Lemma \ref{lemma:E-init-rn-2inf} to bound $\frH_1$ by:

\begin{equation}\label{frH-1}
    \norm{ \left[\frac{d^*}{n}\sum_{i=1}^n \langle \bcT^{\vartriangle}_{(1,k_1)} , \bcX_i\rangle \be_{1,k_1}^\top\cM_1(\bcX_i)-\be_{1,k_1}^\top\cM_1(\bcT^{\vartriangle}_{(1,k_1)})\right]\left(\otimes_{j\neq 1}\cP_{\widehat \bU_{j,0}^{(1,k_1)}} \right) }_2\le C C_2 \sigma \sqrt{\frac{ m (2\mu)^{m-1} r_{-1} \dmax d^*\log^2 \dmax }{n^2}}
\end{equation}
In $\frH_2$ of \eqref{eq:loo-Einit-decomp}, we can follow the steps in \eqref{eq:loo-l2inf-ipv-1-base}  and \eqref{eq:loo-l2inf-ipv-1-add} and use the same argument to treat it. Clearly, this will lead to the following:
\begin{equation}\label{frH-2}
\begin{aligned}
    & \norm{ \frac{d^*}{n}\sum_{i=1}^n \langle \bcT^{\vartriangle}_{(1,k_1)} , \bcX_i\rangle \be_{1,k_1}^\top\cM_1(\bcX_i)\left( \otimes_{j\neq 1}\cP_{\widehat{\mathbf{U}}_{j,0}}  -\otimes_{j\neq 1}\cP_{\widehat \bU_{j,0}^{(1,k_1)}} \right) }_2  \\
    & \le    C \sigma \frac{ C_2\sigma  }{\lambda_{\min}}\sqrt{\frac{d^*\dmax \log \dmax }{n}} \sqrt{\frac{ m \alpha_d (2\mu)^{m-1} r_{-1}  d^* \dmax \log^2 \dmax }{n^2}} 
\end{aligned}
\end{equation}
where the bound holds under the SNR condition in \eqref{eq:apx-loo-SNR}.

In  $\frH_3$ of \eqref{eq:loo-Einit-decomp}, we can use the fact that for any $\bcT_1$, $\bcT_2$ and $\bcX$: 
\begin{equation}\label{eq:T1-T2-decomp}
\begin{aligned}
         \langle\bcT_1 -\bcT_2,\bcX\rangle & = \left[\Vect(\bcC_1)^\top\left(\otimes_{j\in[m]}\bU_{j,1}^\top\right)-\Vect(\bcC_2)^\top\left(\otimes_{j\in[m]}\bU_{j,2}^\top\right)\right]\Vect(\bcX) \\
        & = \sum_{\abs{\cS}\ge 1} \Vect(\bcC_2)^\top \otimes_{j\in\cS } \left(\bU_{j,1} -\bU_{j,2}\right)^\top  \otimes_{k'\in\cS^c} \bU_{k',2}^\top  \Vect(\bcX)  \\
        & \quad + \sum_{\cS} \Vect(\bcC_1-\bcC_2)^\top \otimes_{j\in\cS } \left(\bU_{j,1} -\bU_{j,2}\right)^\top  \otimes_{k'\in\cS^c} \bU_{k',2}^\top \Vect(\bcX), \\
\end{aligned}
\end{equation}
where in \eqref{eq:T1-T2-decomp}, the first term is the summation for all the $\cS\subseteq[m]$ with $\abs{\cS}\ge 1$, while in the second term $\cS$ can take all the sets in $[m]$ including the empty set. Letting $\widehat \bcT_{\init},\widehat \bcT_{\init}^{(1,k_1)} $ be the corresponding $\bcT_1$ and $\bcT_2$ in \eqref{eq:T1-T2-decomp}, we can rewrite $\frH_3$ in \eqref{eq:loo-Einit-decomp} as follows (for simplicity, we ignore the mean term $\be_{1,k_1}^\top\cM_1(\widehat \bcT_{\init}-\widehat \bcT_{\init}^{(k,l)})$ in the expression):
\begin{subequations}\label{eq:loo-init-kl-diff}
    \begin{align}
         & \norm{ \frac{d^*}{n}\sum_{i=1}^n \langle \widehat \bcT_{\init}-\widehat \bcT_{\init}^{(1,k_1)} , \bcX_i\rangle \be_{1,k_1}^\top\cM_1(\bcX_i)\left( \otimes_{j\neq 1}\cP_{\widehat{\mathbf{U}}_{j,0}} \right) }_2 \notag\\
         & \le  \sum_{\abs{\cS}\ge 1, \cS'}\left\| \frac{d^*}{n}\sum_{i=1}^n \Vect(\widehat{\bcC}_0^{(1,k_1)})^\top\otimes_{j\in \cS} \left( \widehat{\bU}_{j,0} - \widehat \bU_{j,0}^{(1,k_1)}\right)^\top\left(  \otimes_{k'\in \cS^c} \widehat \bU_{k',0}^{(1,k_1)} \right)^\top  \Vect(\bcX_i) \be_{1,k_1}^\top\cM_1(\bcX_i) \right. \notag\\
         & \quad \quad \quad \quad \left.\otimes_{j\in \cS'} \left( \cP_{\widehat{\bU}_{j,0}} - \cP_{\widehat \bU_{j,0}^{(1,k_1)}}\right)\left(  \otimes_{k'\in \cS^{'c}} \widehat \bU_{k',0}^{(1,k_1)} \right) \right\|_2 \notag\\
         & +  \sum_{\cS, \cS'}\left\| \frac{d^*}{n}\sum_{i=1}^n \Vect(\widehat{\bcC}_0-\widehat{\bcC}_0^{(1,k_1)})^\top\otimes_{j\in \cS} \left( \widehat{\bU}_{j,0} - \widehat \bU_{j,0}^{(1,k_1)}\right)^\top\left(  \otimes_{k'\in \cS^c} \widehat \bU_{k',0}^{(1,k_1)} \right)^\top  \Vect(\bcX_i) \be_{1,k_1}^\top\cM_1(\bcX_i) \right. \notag\\
         & \quad \quad \quad \quad \left.\otimes_{j\in \cS'} \left( \cP_{\widehat{\bU}_{j,0}} - \cP_{\widehat \bU_{j,0}^{(1,k_1)}}\right)\left(  \otimes_{k'\in \cS^{'c}} \widehat \bU_{k',0}^{(1,k_1)} \right) \right\|_2 \notag\\
         & = \sum_{\abs{\cS}\ge 1, \cS'}\left\| \Vect(\widehat{\bcC}_0^{(1,k_1)})^\top\left[\otimes_{j\in \cS} \left( \widehat{\bU}_{j,0} - \widehat \bU_{j,0}^{(1,k_1)}\right)^\top\otimes_{k'\in \cS^c} \bI_{r_{k'}}\right] \right. \label{eq:loo-init-kl-diff-1}\\
          & \quad \quad \quad \quad \left.  \cdot \frac{d^*}{n}\sum_{i=1}^n \left(\otimes_{j\in \cS}\bI_{d_j}  \otimes_{k'\in \cS^c} \widehat \bU_{k',0}^{(1,k_1)} \right)^\top  \Vect(\bcX_i) \be_{1,k_1}^\top\cM_1(\bcX_i)
        \otimes_{j\in \cS'} \bI_{d_j}\left(  \otimes_{k'\in \cS^{'c}} \widehat \bU_{k',0}^{(1,k_1)} \right) \right. \notag\\
        & \quad \quad \quad \quad\left.\cdot \otimes_{j\in \cS'} \left( \cP_{\widehat{\bU}_{j,0}} - \cP_{\widehat \bU_{j,0}^{(1,k_1)}}\right)\left(  \otimes_{k'\in \cS^{'c}} \bI_{r_{k'}} \right) \right\|_2 \notag\\
        &  + \sum_{\cS, \cS'}\left\|  \Vect(\widehat{\bcC}_0-\widehat{\bcC}_0^{(1,k_1)})^\top\left[\otimes_{j\in \cS} \left( \widehat{\bU}_{j,0} - \widehat \bU_{j,0}^{(1,k_1)}\right)^\top\otimes_{k'\in \cS^c} \bI_{r_{k'}}\right] \right. \label{eq:loo-init-kl-diff-2}\\
          & \quad \quad \quad \quad \left.  \cdot \frac{d^*}{n}\sum_{i=1}^n\left(\otimes_{j\in \cS}\bI_{d_j}  \otimes_{k'\in \cS^c} \widehat \bU_{k',0}^{(1,k_1)} \right)^\top  \Vect(\bcX_i) \be_{1,k_1}^\top\cM_1(\bcX_i)
        \otimes_{j\in \cS'} \bI_{d_j}\left(  \otimes_{k'\in \cS^{'c}} \widehat \bU_{k',0}^{(1,k_1)} \right) \right. \notag\\
        & \quad \quad \quad \quad\left.\cdot \otimes_{j\in \cS'} \left( \cP_{\widehat{\bU}_{j,0}} - \cP_{\widehat \bU_{j,0}^{(1,k_1)}}\right)\left(  \otimes_{k'\in \cS^{'c}} \bI_{r_{k'}} \right) \right\|_2. \notag
    \end{align}
\end{subequations}
Clearly, the mean term can also be decomposed like this. Here $\cS$ can be taken as all the subsets of $[m]$, and $\cS'$ can be taken as all the subsets of $[m]/\{1\}$. In \eqref{eq:loo-init-kl-diff-1} and \eqref{eq:loo-init-kl-diff-2}, we actually only need the concentration of the middle term, where $ \Vect(\bcX_i) \be_{1,k_1}^\top\cM_1(\bcX_i)$ and singular subspace $\widehat \bU_{k',0}^{(1,k_1)}$ are independent for all the $k'\in[m]$. To study the concentration, we denote the interested middle term by:
\begin{equation*}
    \begin{aligned}
        \frH_m(\cS,\cS')=\left(\otimes_{j\in \cS}\bI_{d_j}  \otimes_{k'\in \cS^c} \widehat \bU_{k',0}^{(1,k_1)} \right)^\top  \Vect(\bcX_i) \be_{1,k_1}^\top\cM_1(\bcX_i)
        \otimes_{j\in \cS'} \bI_{d_j}\left(  \otimes_{k'\in \cS^{'c}} \widehat \bU_{k',0}^{(1,k_1)} \right).
    \end{aligned}
\end{equation*}
Conditional on the leave-one-out initialization, we can check that:
\begin{equation*}
    \begin{aligned}
        &\norm{\frH_m(\cS,\cS')}_{\psi_2}\le \sqrt{\frac{(2\mu)^{2m-1} r^* r_{-1}}{d^* d_{-1}} }\cdot\sqrt{\frac{\prod_{j\in \cS} d_j \prod_{j\in \cS'} d_j }{\prod_{j\in \cS} 2\mu r_j \prod_{j\in \cS'} 2\mu r_j }} \\
        &\max\left\{\norm{\E\frH_m(\cS,\cS')\frH_m^\top(\cS,\cS')}_2 ,\norm{\E\frH_m^\top(\cS,\cS')\frH_m(\cS,\cS')}_2 \right\} \le \frac{(2\mu)^{2m-1} r^* r_{-1}}{(d^*)^2 }\frac{\prod_{j\in \cS} d_j \prod_{j\in \cS'} d_j }{\prod_{j\in \cS} 2\mu r_j \prod_{j\in \cS'} 2\mu r_j }.
    \end{aligned}
\end{equation*}
Therefore, for each $\cS,\cS'$, we have 
\begin{equation}\label{loo-middle}
    \begin{aligned}
        & \norm{\frac{d^*}{n}\sum_{i=1}^n \left(\otimes_{j\in \cS}\bI_{d_j}  \otimes_{k'\in \cS^c} \widehat \bU_{k',0}^{(1,k_1)} \right)^\top  \Vect(\bcX_i) \be_{1,k_1}^\top\cM_1(\bcX_i)
        \otimes_{j\in \cS'} \bI_{d_j}\left(  \otimes_{k'\in \cS^{'c}} \widehat \bU_{k',0}^{(1,k_1)} \right)-\Perm_{1,k_1}^\top }_2 \\
        & \le \sqrt{\frac{m(2\mu)^{2m-1} r^* r_{-1}\log\dmax}{n } }\cdot\sqrt{\frac{\prod_{j\in \cS} d_j \prod_{j\in \cS'} d_j }{\prod_{j\in \cS} 2\mu r_j \prod_{j\in \cS'} 2\mu r_j }} 
    \end{aligned}
\end{equation}
with probability at least $1-\dmax^{-4m}$. Here $\Perm_{1,k_1}\in\R^{d_{-1}\times d^*}$ represents the mean of $ \frac{d^*}{n}\frH_m(\cS,\cS')^\top$, and it is actually a 0-1 matrix that projects a vectorized tensor of dimension $d^*$ to its vectorized $k_1$-th slice (dimension $d_{-1}$) in mode-1 unfolding.

We first handle \eqref{eq:loo-init-kl-diff-1}. Here we point out that under the leave-one-out initialization condition, there always exists global rotations $\{\widehat\bO_j\in\O^{r_j\times r_j}\}_{j=1}^{m}$ acting on  $\widehat\bU_{j,0}^{(1,k_1)}$ for $j\in[m]$ such that each $\norm{\bU_{j,0}^{(1,k_1)}\widehat\bO_j -\widehat\bU_{j,0}}_{\tF}$ share the same order as $\norm{ \cP_{\widehat{\bU}_{j,0}^{(k,l)}} - \cP_{\widehat{\bU}_{j,0}} }_{\tF}$. For brevity, we assume that $\widehat\bU_{j,0}^{(1,k_1)}$ are the corresponding $\bU_{j,0}^{(1,k_1)}\widehat\bO_j$ after such rotations.  In \eqref{eq:loo-init-kl-diff-1}, we have $\norm{\Vect(\widehat{\bcC}_0^{(1,k_1)})}_2=\norm{\widehat{\bcT}_{\init}^{(1,k_1)}}_{\tF}\le C \sqrt{\rmin}\lambda_{\max}$, where we use the initialization condition of leave-one-out series. Applying the tricks in \eqref{eq:loo-l2inf-ipv-1-add}, we derive that:

\begin{equation}\label{eq:loo-init-kl-diff-1-res}
    \begin{aligned}
        & \left\| \Vect(\widehat{\bcC}_0^{(1,k_1)})^\top\left[\otimes_{j\in \cS} \left( \widehat{\bU}_{j,0} - \widehat \bU_{j,0}^{(1,k_1)}\right)^\top\otimes_{k'\in \cS^c} \bI_{r_{k'}}\right] \right. \\
          & \quad   \cdot\left[ \frac{d^*}{n}\sum_{i=1}^n \left(\otimes_{j\in \cS}\bI_{d_j}  \otimes_{k'\in \cS^c} \widehat \bU_{k',0}^{(1,k_1)} \right)^\top  \Vect(\bcX_i) \be_{1,k_1}^\top\cM_1(\bcX_i)
        \otimes_{j\in \cS'} \bI_{d_j}\left(  \otimes_{k'\in \cS^{'c}} \widehat \bU_{k',0}^{(1,k_1)} \right)-\Perm_{1,k_1}^\top \right] \\
        & \quad \left.\cdot \otimes_{j\in \cS'} \left( \cP_{\widehat{\bU}_{j,0}} - \cP_{\widehat \bU_{j,0}^{(1,k_1)}}\right)\left(  \otimes_{k'\in \cS^{'c}} \bI_{r_{k'}} \right) \right\|_2 \\
        & \le C \sqrt{\rmin}\lambda_{\max}\cdot \sqrt{\frac{m(2\mu)^{2m-1} r^* r_{-1}\log\dmax}{n } }\cdot\sqrt{\frac{\prod_{j\in \cS} d_j \prod_{j\in \cS'} d_j }{\prod_{j\in \cS} 2\mu r_j \prod_{j\in \cS'} 2\mu r_j }} \cdot \left(\frac{C_2 \sigma}{\lambda_{\min}}\sqrt{\frac{\mu r_1}{d_1}} \sqrt{ \frac{ d^* \dmax \log \dmax }{ n }}\right)^{\abs{\cS}+\abs{\cS'}} \\
        & \le CC_2 \sqrt{\rmin}\lambda_{\max} \cdot (\frac{1}{4})^{\abs{\cS}+\abs{\cS'}} \frac{ \sigma}{\lambda_{\min}} \sqrt{\frac{m \alpha_d (2\mu)^{2m-1} (r^*)^2 d^*  \dmax \log^2\dmax  }{ \rmin n^2}}\\
        &\le  CC_2 \cdot (\frac{1}{4})^{\abs{\cS}+\abs{\cS'}}  \kappa_0\sigma \sqrt{\frac{m \alpha_d (2\mu)^{2m-1} (r^*)^2 d^*  \dmax \log^2\dmax  }{ n^2}},
    \end{aligned}
\end{equation}
under the SNR condition in \eqref{eq:apx-loo-SNR} for
each $\cS,\cS'$ where $\abs{\cS}\ge 1$. 

We now focus on \eqref{eq:loo-init-kl-diff-2}. Under the global rotation described above, we actually have:
\begin{equation}
    \begin{aligned}
        & \norm{\widehat{\bcC}_0-\widehat{\bcC}_0^{(1,k_1)}}_{\tF} = \norm{\widehat{\bcT}_{\init}\times_{j=1}^m \widehat{\bU}_{j,0}^\top - \widehat{\bcT}_{\init}^{(1,k_1)}\times_{j=1}^m \widehat{\bU}_{j,0}^{(1,k_1)\top}  }_{\tF} \\
        &= \norm{ \left(\widehat{\bcT}_{\init}-\widehat{\bcT}_{\init}^{(1,k_1)}\right)\times_{j=1}^m \widehat{\bU}_{j,0}^{\top}  }_{\tF} + \norm{\widehat{\bcT}_{\init}^{(1,k_1)}\times_1\left(\widehat{\bU}_{1,0}-\widehat{\bU}_{1,0}^{(1,k_1)}\right)^{\top}\times_{j=2}^{m} \widehat{\bU}_{j,0}^\top}_{\tF} \\
        & \quad + \norm{\widehat{\bcT}_{\init}^{(1,k_1)}\times_1\widehat{\bU}_{1,0}^{(1,k_1)\top}\times_2\left(\widehat{\bU}_{2,0}-\widehat{\bU}_{2,0}^{(1,k_1)}\right)^{\top}\times_{j=3}^{m} \widehat{\bU}_{j,0}^\top}_{\tF} \\
        & \cdots \\
        & \quad + \norm{\widehat{\bcT}_{\init}^{(1,k_1)}\times_{j=1}^{m-1}\widehat{\bU}_{j,0}^{(1,k_1)\top}\times_{m}\left(\widehat{\bU}_{m,0}-\widehat{\bU}_{m,0}^{(1,k_1)}\right)^{\top}}_{\tF} \\
        &\le C_2\sigma \sqrt{\frac{d^*\dmax\log\dmax}{n}} + C C_2 m  \sqrt{\rmin}\lambda_{\max} \frac{ \sigma}{\lambda_{\min}}\sqrt{\frac{\mu r_1}{d_1}} \sqrt{ \frac{ d^* \dmax \log \dmax }{ n }} \\
        &\le C C_2\kappa_0\sigma \sqrt{\frac{d^*\dmax\log\dmax}{n}},
    \end{aligned}
\end{equation}
where in the first inequality we use the leave-one-out initialization condition and $\ell_{\infty}$ norm bound between $\widehat{\bcT}_{\init},\widehat{\bcT}_{\init}^{(1,k_1)}$, and in the second inequality we use the fact that $d_1\gg m^2 \rmax^2\mu $. With this bound on $\widehat{\bcC}_0-\widehat{\bcC}_0^{(1,k_1)}$, we can repeat the analysis in
\eqref{eq:loo-init-kl-diff-1-res} and derive the following bound:
\begin{equation}\label{eq:loo-init-kl-diff-2-res}
    \begin{aligned}
         &  \left\|  \Vect(\widehat{\bcC}_0-\widehat{\bcC}_0^{(1,k_1)})^\top\left[\otimes_{j\in \cS} \left( \widehat{\bU}_{j,0} - \widehat \bU_{j,0}^{(1,k_1)}\right)^\top\otimes_{k'\in \cS^c} \bI_{r_{k'}}\right] \right. \\
         & \quad   \cdot\left[ \frac{d^*}{n}\sum_{i=1}^n \left(\otimes_{j\in \cS}\bI_{d_j}  \otimes_{k'\in \cS^c} \widehat \bU_{k',0}^{(1,k_1)} \right)^\top  \Vect(\bcX_i) \be_{1,k_1}^\top\cM_1(\bcX_i)
        \otimes_{j\in \cS'} \bI_{d_j}\left(  \otimes_{k'\in \cS^{'c}} \widehat \bU_{k',0}^{(1,k_1)} \right)-\Perm_{1,k_1}^\top \right] \\
        & \quad\left.\cdot \otimes_{j\in \cS'} \left( \cP_{\widehat{\bU}_{j,0}} - \cP_{\widehat \bU_{j,0}^{(1,k_1)}}\right)\left(  \otimes_{k'\in \cS^{'c}} \bI_{r_{k'}} \right) \right\|_2 \\
       & \le  C C_2\kappa_0\sigma \sqrt{\frac{d^*\dmax\log\dmax}{n}}\cdot \sqrt{\frac{m(2\mu)^{2m-1} r^* r_{-1}\log\dmax}{n } }\cdot\sqrt{\frac{\prod_{j\in \cS} d_j \prod_{j\in \cS'} d_j }{\prod_{j\in \cS} 2\mu r_j \prod_{j\in \cS'} 2\mu r_j }} \cdot \left(\frac{C_2 \sigma}{\lambda_{\min}}\sqrt{\frac{\mu r_1}{d_1}} \sqrt{ \frac{ d^* \dmax \log \dmax }{ n }}\right)^{\abs{\cS}+\abs{\cS'}} \\
        & \le CC_2 \kappa_0\sigma \sqrt{\frac{d^*\dmax\log\dmax}{n}}\cdot (\frac{1}{4})^{\abs{\cS}+\abs{\cS'}}  \sqrt{\frac{m(2\mu)^{2m-1} r^* r_{-1}\log\dmax}{n } }\\
        &\le  CC_2 \cdot (\frac{1}{4})^{\abs{\cS}+\abs{\cS'}}  \kappa_0\sigma \sqrt{\frac{m (2\mu)^{2m-1} (r^*)^2 d^*  \dmax \log^2\dmax  }{ n^2}}.
    \end{aligned}
\end{equation}
Here \eqref{eq:loo-init-kl-diff-1-res} and \eqref{eq:loo-init-kl-diff-2-res} holds uniformly for all suitable $\cS,\cS'\subseteq[m]$ under the leave-one-out initialization with probability at least $1-2^{m+1}\dmax^{-4m}\ge 1-\dmax^{-3m}$. Summing up all the $\cS,\cS'$, we have the desired bound for \eqref{eq:loo-l2inf-ipv-3}.

We turn to handle \eqref{eq:loo-l2inf-ipv-4} with additional sample size condition:
$n\ge C_{\gap}C_2^2\kappa_0^2m^2(2\mu)^m\rmax r^*\dmax\log\dmax$. Under this condition, clearly the error $\norm{\bE_{1,\rn}\left(\otimes_{j\neq 1} \widehat{\mathbf{U}}_{j,0} \right)}_{2,\infty}$ will dominate the initialization error $\norm{\bE_{1,\init}\left(\otimes_{j\neq 1} \widehat{\mathbf{U}}_{j,0} \right)}_{2,\infty}$, we thus have 
\begin{equation}\label{eq:loo-El2linf-new}
    \begin{aligned}
        & \norm{(\bE_{1,\rn}+\bE_{1,\init})\left(\otimes_{j\neq 1} \widehat{\mathbf{U}}_{j,0} \right)}_{2,\infty} \le C \sigma \sqrt{\frac{ m (2\mu)^{m-1} r_{-1} d^*\log \dmax }{n}} \\
        & \norm{(\bE_{1,\rn}+\bE_{1,\init})\left(\otimes_{j\neq 1} \widehat{\mathbf{U}}_{j,0} \right)}_{2} \le C \sigma \sqrt{\frac{ m (2\mu)^{m-1} r_{-1} d_1 d^*\log \dmax }{n}}.
    \end{aligned}
\end{equation}
Plugging in \eqref{eq:loo-El2linf-new} into the analysis of Lemma \ref{lemma:l2inf-powerit} by substituting the error rates, we can easily derive the claimed bound of \eqref{eq:loo-l2inf-ipv-4}, As a by-product, we also have the following norm bound:
\begin{equation}\label{eq:loo-U11-l2linf-new}
    \begin{aligned}
         &\norm{\cP_{\widehat{\bU}_{1,1}} - \cP_{\bU_1}}_{2} \le \frac{C C_2\sigma  }{\lambda_{\min}}\sqrt{\frac{ m (2\mu)^{m-1} r^* d_1  d^*\log \dmax }{  r_1 n}} \\
         &\norm{\cP_{\widehat{\bU}_{1,1}} - \cP_{\bU_1}}_{\linf}\le \frac{C C_2\sigma  }{\lambda_{\min}}\sqrt{\frac{ m (2\mu)^{m-1} r_{-1} d_1 d^*\log \dmax }{n}}\cdot\frac{\mu r_1}{d_1}.
    \end{aligned}
\end{equation}

    
\end{proof}

With the new $\ell_{2,\infty}$ norm bound in Lemma \ref{lemma:loo-l2inf-ipv} and the corresponding $\ell_{2}$ norm bound shown in its proof, we are able to prove the following lemma on the remainder terms:

\begin{Lemma}[Remainder terms with leave-one-out initialization]\label{lemma:loo-init-frE-2-5} Under the conditions in Theorem , the terms from $\frE_3$ to $\frE_5$ can be controlled by
\begin{equation*}
    \begin{gathered}
       \abs{\frE_3}\le C \frac{C_2^2 \sigma^2}{\lambda_{\min}} \norm{\bcI}_{\ell_1}  \frac{ (r^*)^{\frac{3}{2}} m^{\frac{5}{2}} (2\mu)^{\frac{3}{2}m-1} \dmax  \sqrt{\alpha_d d^*}\log \dmax }{ n} + C C_2 \kappa_0\sigma  \norm{\bcI}_{\ell_1} \sqrt{\frac{m^3 (2\mu)^{3m-1} (r^*)^3 \dmax^2 \log^2\dmax  }{ n^2}} \\
        \abs{\frE_4}\le C\norm{\bcI}_{\ell_1} \frac{ C_2^2\sigma^2  }{\lambda_{\min}}\frac{ m^2 (2\mu)^{\frac{3}{2}m} (r^*)^{\frac{3}{2}} \dmax  \sqrt{d^*}\log \dmax }{ n}\\
        \abs{\frE_5}\le C\norm{\bcI}_{\ell_1} \frac{ C_2^2\sigma^2  }{\lambda_{\min}}\frac{ m^2 (2\mu)^{\frac{3}{2}m} (r^*)^{\frac{3}{2}} \dmax  \sqrt{d^*}\log \dmax }{ n} \\  
    \end{gathered}
\end{equation*}
uniformly with probability at least $1-\dmax^{-2m}$.
\end{Lemma}
\begin{proof}
    We only need to go back to the proof of Lemma \ref{lemma:oracle-init-frE-2-5} (Section \ref{apx:sec:proof-all-rem-ind}) and plug in the new error rates developed in Lemma \ref{lemma:loo-l2inf-ipv} to update the bounds. 
    
    \noindent \textbf{(i) $\frE_3$ in} \eqref{eq:test-decomp}. We begin with $\abs{\frE_3}$ by controlling $\frF_1$, $\frF_2$ and $\frF_3$ in \eqref{eq:frE-3-decomp}.  For $\frF_1$, we have the following bound with probability $1-\dmax^{-3m}$:
\begin{equation*}
    \abs{\frF_1} \le C \sigma  \norm{\bcI}_{\ell_1}\frac{ C_2\sigma  }{\lambda_{\min}}\sqrt{\frac{d^*\dmax \log \dmax }{n}} \sqrt{\frac{ \alpha_d m^3 (2\mu)^{2(m-1)} r_{-j} r^*  d_j\log \dmax }{ n}}.
\end{equation*}
This is a direct consequence of \eqref{eq:frF-1} and Lemma \ref{lemma:loo-l2inf-ipv}. Moreover, we use \eqref{eq:frF-2} to control $\frF_2$ at the level:
\begin{equation*}
\begin{aligned}
 \abs{\frF_2} & \le \norm{\bcI}_{\ell_1}\sqrt{\frac{\mu^{m-1} r_{-j}}{d_{-j}}} \left( \sqrt{\frac{\mu r_j}{d_j}}\norm{\mathfrak{B}_{j,3} \bU_j\bLambda_j^{-1}\bV_j^\top }_{2} +\norm{\mathfrak{B}_{j,3} \bU_j\bLambda_j^{-1}\bV_j^\top }_{2,\infty}\right) \\
    & \le C \frac{C_2^2 \sigma^2}{\lambda_{\min}} \norm{\bcI}_{\ell_1}\sqrt{\frac{\mu^{m} r^* }{d^* }}   \frac{ \sqrt{r^* r_{-j}} m (2\mu)^{m-1} d_j  d^*\log \dmax }{ n}.
\end{aligned}
\end{equation*}
For $\frF_3$, the way we deal with it is different from Section \ref{apx:sec:proof-all-rem-ind}. We consider the entrywise argument:

\begin{equation*}
 \begin{aligned}
      \abs{\frF_3} & = \abs{  \left\langle \cP_{\bU_j}^\perp \bE_{j,\init}\left( \otimes_{k\neq j}\cP_{\widehat{\mathbf{U}}_{k,0}} \right) \left( \otimes_{k\neq j}\bU_{k} \right)\cP_{\bV_j}\left(\otimes_{k\neq j}\bU_k^\top\right) , \sum_{\omega}\left[\bcI\right]_{\omega}\cM_j(\bomega)\otimes_{k\neq j}\cP_{{\mathbf{U}}_{k}}\right\rangle} \\
  & \le   \sum_{\omega} \abs{\left[\bcI\right]_{\omega}  \left\langle \cP_{\bU_j}^\perp \bE_{j,\init}\left( \otimes_{k\neq j}\cP_{\widehat{\mathbf{U}}_{k,0}} \right) \left( \otimes_{k\neq j}\bU_{k} \right)\cP_{\bV_j}\left(\otimes_{k\neq j}\bU_k^\top\right) , \cM_j(\bomega)\otimes_{k\neq j}\cP_{{\mathbf{U}}_{k}}\right\rangle} \\
  & \le  \norm{\bcI}_{\ell_1} \sqrt{\frac{\mu^{m-1} r_{-j}}{d_{-j}}} \left( \sqrt{\frac{\mu r_j}{d_j}}\norm{\bE_{j,\init}\left( \otimes_{k\neq j}\cP_{\widehat{\mathbf{U}}_{k,0}} \right) }_{2} +\norm{\bE_{j,\init}\left( \otimes_{k\neq j}\cP_{\widehat{\mathbf{U}}_{k,0}} \right) }_{2,\infty}\right) \\
  & \le  C C_2 \kappa_0\sigma  \norm{\bcI}_{\ell_1}\sqrt{\frac{\mu^{m} r^* }{d^* }}   \sqrt{\frac{m^3 (2\mu)^{2m-1} (r^*)^2 d^* d_j \dmax \log^2\dmax  }{ n^2}}.
 \end{aligned}
\end{equation*}
Combining the results on $\frF_1$, $\frF_2$, $\frF_3$, we have:
\begin{equation*}
    \begin{aligned}
         \abs{\frE_3}\le C \frac{C_2^2 \sigma^2}{\lambda_{\min}} \norm{\bcI}_{\ell_1}  \frac{ (r^*)^{\frac{3}{2}} m^{\frac{5}{2}} (2\mu)^{\frac{3}{2}m-1} \dmax  \sqrt{\alpha_d d^*}\log \dmax }{ n} + C C_2 \kappa_0\sigma  \norm{\bcI}_{\ell_1} \sqrt{\frac{m^3 (2\mu)^{3m-1} (r^*)^3 d_j \dmax \log^2\dmax  }{ n^2}}.
    \end{aligned}
\end{equation*}
 \noindent \textbf{(ii) $\frE_4$ in} \eqref{eq:test-decomp}. According to \eqref{eq:frE-4} and the new $\ell_{2,\infty}$ of Lemma \ref{lemma:loo-l2inf-ipv}, we have

\begin{equation}\label{eq:loo-frE-4}
\begin{aligned}
         \abs{\frE_4} & \le \sum_{j=1}^m \norm{\bcI}_{\ell_1}\sqrt{\frac{\mu^{m-1} r_{-j}}{d_{-j}}}\norm{\left(\sum_{l\ge 2}\cS_{\widehat{\bU}_{j,1}}^{(l)}\right)\bT_j}_{2,\infty} \\
         & \le C\norm{\bcI}_{\ell_1}\sqrt{\frac{\mu^{m} r^* }{d^* }} \frac{ C_2^2\sigma^2  }{\lambda_{\min}}\frac{ m^2 (2\mu)^{m-1} r^* \dmax  d^*\log \dmax }{ n}.
\end{aligned}
\end{equation}
Thus, we have 
\begin{equation*}
    \abs{\frE_4}\le C\norm{\bcI}_{\ell_1}\sqrt{\frac{\mu^{m} r^* }{d^* }} \frac{ C_2^2\sigma^2  }{\lambda_{\min}}\frac{ m^2 (2\mu)^{m-1} r^* \dmax  d^*\log \dmax }{ n}.
\end{equation*}

  \noindent \textbf{(iii) $\frE_5$ in} \eqref{eq:test-decomp}.
We need to repeat the analyses of two different types of errors. For case (1), \eqref{eq:Erem-noTdiff} and \eqref{eq:Erem-noTdiff-res} 

\begin{equation}\label{eq:loo-Erem-noTdiff-res}
    \begin{aligned}
        \abs{\left\langle\bcT\times_{j\in \cS } \left( \cP_{\widehat{\bU}_{j,1}} - \cP_{\bU_j}\right)\times_{k\in \cS^c } \cP_{{\bU}_{k}}, \bcI \right\rangle} \le  \frac{C C_1^2\sigma^2  }{\lambda_{\min}}\norm{\bcI}_{\ell_1}\frac{ m (2\mu)^{m-1} r_{-j} \dmax  d^*\log \dmax }{n} \cdot\sqrt{\frac{\mu^m r^*}{d^*}}.
    \end{aligned}
\end{equation}
For case (2), following the steps in 
\eqref{eq:metric-ety-Uj1}, \eqref{eq:Erem-Tdiff}, \eqref{eq:Erem-Tdiff-ind-conc}, \eqref{eq:Erem-Tdiff-res}, and substituting the $\ell_{2,\infty}$, $\ell_{\infty}$ bound of $\cP_{\widehat{\bU}_{j,1}} - \cP_{\bU_j}$ with the new order presented in Lemma \eqref{lemma:loo-l2inf-ipv}, we can conclude that the error of the second type can be controlled by:
\begin{equation}\label{eq:loo-Erem-Tdiff-res-dep}
    \begin{aligned}
        & \abs{\left\langle \left(\bcE_{\rn} + \bcE_{\init}\right)\times_{j\in \cS } \left( \cP_{\widehat{\bU}_{j,1}} - \cP_{\bU_j}\right)\times_{k\in \cS^c } \cP_{{\bU}_{k}} , \bcI \right\rangle} \\
        & \le C (\frac{1}{4})^{\abs{\cS}}\norm{\bcI}_{\ell_1} \frac{C_2\sigma^2  }{\lambda_{\min}}\sqrt{\frac{ m^2 (2\mu)^{3m-1} (r^*)^3 \dmax^2  d^*\log^2 \dmax }{n^2}},
    \end{aligned}
\end{equation}
for $\abs{\cS}=1$. Here the term related to $\bcE_{\rn}$ can be handled by \eqref{eq:Erem-Tdiff-ind-conc}, and the term  related to $\bcE_{\init}$ can be simply controlled by the size bound of each term in the summation:
\begin{equation*}
\begin{aligned}
        & \abs{\langle \bcT^{\vartriangle}, \bcX_i\rangle\left\langle \left(\bcX_i\right)\times_{j\in \cS } \left( \cP_{\widehat{\bU}_{j,1}} - \cP_{\bU_j}\right)\times_{k\in \cS^c } \cP_{{\bU}_{k}} , \bomega \right\rangle }\le C C_1\sigma\sqrt{\frac{\dmax\log\dmax}{n}} \prod_{j\in\cS}\norm{\cP_{\widehat{\bU}_{j,1}} - \cP_{\bU_j}}_{\ell_{\infty}}\prod_{j\in\cS^c}\frac{\mu r_j }{d_j} \\
    & \le CC_1\frac{\mu^{m} r^* }{d^*} \cdot \sigma\sqrt{\frac{\dmax\log\dmax}{n}} \left(\frac{C_2 \sigma}{\lambda_{\min} }\sqrt{\frac{  m (2\mu)^{m-1}\rmax r^* \dmax  d^*\log \dmax }{ n}}\right)^{\abs{\cS}} \\
    & \le  C (\frac{1}{4})^{\abs{\cS}}\frac{C_2^2\sigma^2  }{\lambda_{\min}}\sqrt{\frac{ m^2 (2\mu)^{3m-1} (r^*)^3 \dmax^2  \log^2 \dmax }{d^* n^2}}.
\end{aligned}
\end{equation*}
Summing up all the \eqref{eq:loo-Erem-noTdiff-res} and \eqref{eq:loo-Erem-Tdiff-res-dep} for every $\abs{\cS}\ge 1$, we can prove the claim. 
\end{proof}
Using the bound of $\frE_2$ in  Lemma \ref{lemma:dep-init-frE-2-5} and bounds for $\frE_3$ to $\frE_5$ in Lemma \ref{lemma:loo-init-frE-2-5}, we can control the remainder terms in \eqref{eq:test-stat-popvar}  as
\begin{equation}\label{eq:loo-test-stat-popvar-rem}
\begin{aligned}
        \abs{\frac{\sum_{i=2}^{5}\frE_i}{\sigma\norm{\cP_{\TT}(\bcI)}_\tF\sqrt{d^*/n} }} & \le C C_2 \kappa_0  \frac{\norm{\bcI}_{\ell_1}}{\norm{\cP_{\TT}(\bcI)}_\tF } \sqrt{\frac{m^3 (2\mu)^{3m-1} (r^*)^3 \dmax^2 \log^2\dmax  }{d^* n}} \\
        & + C \frac{ C_2^2\sigma  \norm{\bcI}_{\ell_1}}{\lambda_{\min}  \norm{\cP_{\TT}(\bcI)}_\tF }\frac{\alpha_d^{\frac{1}{2}} m^{\frac{5}{2}} (2\mu)^{\frac{3}{2}m} (r^*)^{\frac{3}{2}} \dmax \log \dmax }{ \sqrt{n}}.
\end{aligned}
\end{equation}
Given the alignment condition in Assumption \ref{asm:alignment}, we have the following bound from $\frE_2$ to $\frE_5$:
\begin{equation*}
    \begin{aligned}
       \abs{\frac{\sum_{i=2}^{5}\frE_i}{\sigma\norm{\cP_{\TT}(\bcI)}_\tF\sqrt{d^*/n} }} &\le C_2 \kappa_0  \frac{\norm{\bcI}_{\ell_1}}{\alpha_I \norm{\bcI}_{\tF}} \sqrt{\frac{m^3 (2\mu)^{3m-1} (r^*)^3 \dmax \log^2\dmax  }{ n}} \\
       & + C \frac{ C_2^2\sigma  \norm{\bcI}_{\ell_1} }{\alpha_I \lambda_{\min}  \norm{\bcI}_\tF }\sqrt{\frac{\alpha_d m^5 (2\mu)^{3 m} (r^*)^{3 }\dmax d^*  \log^2 \dmax }{ n } }
    \end{aligned}
\end{equation*}
with probability at least $1- \dmax^{-2m}$. Combining this result with \eqref{eq:core-clt}, we can derive that
\begin{equation*}
\begin{aligned}
        &\max_{t\in\R}\abs{\bbP\left( \frac{\left\langle  \widehat{\bcT} - \bcT,\bcI  \right\rangle }{ \sigma\norm{\cP_{ \TT }(\bcI)}_\tF \sqrt{{d^*}/{n}} } \le t\right)- \Phi(t)}  \\
        &\le  C_2 \kappa_0  \frac{\norm{\bcI}_{\ell_1}}{\alpha_I \norm{\bcI}_{\tF}} \sqrt{\frac{m^3 (2\mu)^{3m-1} (r^*)^3 \dmax \log^2\dmax  }{ n}} + C \frac{ C_2^2\sigma  \norm{\bcI}_{\ell_1} }{\alpha_I \lambda_{\min}  \norm{\bcI}_\tF }\sqrt{\frac{\alpha_d m^5 (2\mu)^{3 m} (r^*)^{3 }\dmax d^*  \log^2 \dmax }{ n } }+C \dmax^{-2m}.
\end{aligned}
\end{equation*}

With the help of \eqref{eq:homovar-vardiff}, we can immediately get the 

\begin{equation*}
\begin{aligned}
        &\max_{t\in\R}\abs{\bbP\left( \widehat W_{\test}(\bcI)\le t\right)- \Phi(t)}  \\
        &\le  C_2 \kappa_0  \frac{\norm{\bcI}_{\ell_1}}{\alpha_I \norm{\bcI}_{\tF}} \sqrt{\frac{m^3 (2\mu)^{3m-1} (r^*)^3 \dmax \log^2\dmax  }{ n}} + C \frac{ C_2^2\sigma  \norm{\bcI}_{\ell_1} }{\alpha_I \lambda_{\min}  \norm{\bcI}_\tF }\sqrt{\frac{\alpha_d m^5 (2\mu)^{3 m} (r^*)^{3 }\dmax d^*  \log^2 \dmax }{ n } } \\
        &\quad \quad +  C\frac{\norm{\bcI}_{\ell_1}}{\norm{\bcI}_{\tF}}\frac{C_1  m^{2.5}\kappa_0 \sigma  }{\alpha_I \lambda_{\min}}\sqrt{\frac{(2\mu)^m r^*d^*\dmax \log^2 \dmax }{n}} \\
        &\le  C_2 \kappa_0  \frac{\norm{\bcI}_{\ell_1}}{\alpha_I \norm{\bcI}_{\tF}} \sqrt{\frac{m^3 (2\mu)^{3m-1} (r^*)^3 \dmax \log^2\dmax  }{ n}} + C \frac{ C_2^2\sigma \kappa_0  \norm{\bcI}_{\ell_1} }{\alpha_I \lambda_{\min}  \norm{\bcI}_\tF }\sqrt{\frac{\alpha_d m^5 (2\mu)^{3 m} (r^*)^{3 }\dmax d^*  \log^2 \dmax }{ n } }.
\end{aligned}
\end{equation*}
This finishes the proof.

\subsection{Proof of Theorem \ref{thm:clt-two-test}}
\begin{proof}
According to the proof of Theorem \ref{thm:lf-inference-empvar} in Section \ref{apx:sec:clt-indp-var}, we have 
\begin{equation*}
\begin{aligned}
        \widehat W_{\test}(\bcI) & = W_{\test}(\bcI)+ W_{\test}(\bcI)\left( \frac{\sigma \norm{\cP_{\widehat \TT }(\bcI)}_\tF}{\widehat\sigma \norm{\cP_{\widehat \TT }(\bcI)}_\tF}-1\right) \\
        & = \frac{\frE_1}{\sigma\norm{\cP_{\TT}(\bcI)}_\tF\sqrt{d^*/n} }  +\frac{\sum_{i=2}^5\frE_i}{\sigma\norm{\cP_{\TT}(\bcI)}_\tF\sqrt{d^*/n} } + W_{\test}(\bcI)\left( \frac{\sigma \norm{\cP_{\widehat \TT }(\bcI)}_\tF}{\widehat\sigma \norm{\cP_{\widehat \TT }(\bcI)}_\tF}-1\right) \\
        & := \frac{\frE_1(\bcI)}{\sigma\norm{\cP_{\TT}(\bcI)}_\tF\sqrt{d^*/n} }  + \Delta(\bcI)
\end{aligned}
\end{equation*}
Again, invoking \eqref{eq:frE2-5-homo}, \eqref{eq:homovar-vardiff}, and  \eqref{eq:bound-Wtest}, we have 

\begin{equation}\label{eq:delta-I}
   \abs{\Delta(\bcI)}\le C C_1 \sqrt{\frac{m \mu^{m} \rmin r^* \dmax \log^2 \dmax }{n}} + C\frac{\norm{\bcI}_{\ell_1}}{\norm{\bcI}_{\tF}}\frac{C_1^2  \kappa_0 \sigma  }{\alpha_I \lambda_{\min}}\sqrt{\frac{m^5 (2\mu)^{3m} (r^*)^3 d^*\dmax \log^2 \dmax }{\rmin^2 n}},
\end{equation}
with probability at least $1-\dmax^{-2m}$. Replacing $\bcI$ by $\bcI_1$ and $\bcI_2$ respectively, we have  
\begin{equation*}
    \begin{aligned}
          \left[\begin{array}{c}\widehat W_{\test}(\bcI_1) \\ \widehat W_{\test}(\bcI_2)\end{array} \right] =   \underbrace{\left[\begin{array}{c}\frac{\frE_1(\bcI_1)}{\sigma\norm{\cP_{\TT}(\bcI_1)}_\tF\sqrt{d^*/n} }  \\ \frac{\frE_1(\bcI_2)}{\sigma\norm{\cP_{\TT}(\bcI_2)}_\tF\sqrt{d^*/n} }  \end{array} \right]}_{\frE_{1}(\bcI_1,\bcI_2)} +   \underbrace{\left[\begin{array}{c}\Delta(\bcI_1) \\ \Delta(\bcI_2)\end{array} \right]}_{\Delta(\bcI_1,\bcI_2) }.
    \end{aligned}
\end{equation*}
As $\frE_1$ is given by $\frE_1(\bcI)= \frac{d^*}{n} \sum_{i=1}^{n} \xi_i \left\langle \bcX_i, \cP_{\TT}(\bcI) \right\rangle$, we know that $\frE_{1}(\bcI_1,\bcI_2)$ converge to the 2-dimensional Gaussian distribution by mutivariate Berry-Esseen theorem \citep{stein1972bound,raivc2019multivariate}, with the covariance given by:
\begin{equation*}
    \rho(\bcI_1,\bcI_2) = \E\frac{d^* \xi_i^2 \left\langle \bcX_i, \cP_{\TT}(\bcI_1) \right\rangle \left\langle \bcX_i, \cP_{\TT}(\bcI_2) \right\rangle}{\sigma^2 \norm{\cP_{\TT}(\bcI_1)}_\tF \norm{\cP_{\TT}(\bcI_2)}_\tF } = \frac{\left\langle\cP_{ \TT }(\bcI_1),\cP_{ \TT }(\bcI_2)\right\rangle}{\norm{\cP_{ \TT }(\bcI_1)}_{\tF} \norm{\cP_{ \TT }(\bcI_2)}_{\tF} }.
\end{equation*}
Moreover, we have $\norm{\Delta(\bcI_1,\bcI_2)}_2$ bounded by the rate in \eqref{eq:delta-I} with probability at least $1-\dmax^{-2m}$. Using the fact that the gradient of $\Phi_{\rho}$ is bounded: $\norm{\nabla \Phi_{\rho} (t_1,t_2)}\le C$, which gives the Lipschitz condition of $\Phi_{\rho} (t_1,t_2)$, we prove the theorem.
\end{proof}

\subsection{Proof of Least Square Estimator in Theorem \ref{thm:constrained-ls}}
\begin{proof}
    Consider minimizing the loss function 
\begin{equation*}
\begin{aligned}
   \widetilde{\bcT}= & \argmin_{\bcW=\bcD\times_{j=1}^m\bW_j} h_n(\bcW)=\frac{d^*}{n}\sum_{i=1}^n\left(Y_i-\left\langle\bcW,\bcX_i \right\rangle \right)^2 \\
    & \text{s.t.}\  \bW_j\in\O^{d_j\times r_j}, \Inco(\bW_j)\le \mu \
   \end{aligned}
\end{equation*} 
We now introduce some notations to better describe the problem. we let $\bcZ$ to be the full noise tensor. We define the tensor operator $\cQ_{{\bA}}(\cdot)$ as $\cQ_{\bA}(\bcW)= \bcW\times_{j=1}^{m}\bA_j $, and the sampling operator $\cP_{\Omega}(\cdot)$ as $\cP_{\Omega}(\bcW)=\sum_{i=1}^n\left\langle\bcW,\bcX_i\right\rangle\bcX_i$ (sometimes we ignore the parentheses of operators for brevity). Assume the optimial solution is $\widetilde{\bcT}=\widetilde{\bcC}\times_{j=1}^m\widetilde{\bU}_j$ . Notice that the optimization problem has no restriction on the core tensor $\bcD$, from which we have the gradient condition:
\begin{equation*}
    \nabla_{\bcD} h_n(\widetilde{\bcC},\widetilde{\bU}_1,\cdots,\widetilde{\bU}_m )=\cQ_{\widetilde{\bU}^\top}\cP_{\Omega}\left( \bcT + \bcZ -\widetilde{\bcT}\right) =0.
\end{equation*}
Thus, the core tensor follows:
\begin{equation*}
    \begin{aligned}
      \frac{d^*}{n}\cQ_{\widetilde{\bU}^\top}\cP_{\Omega} \cQ_{\widetilde{\bU}}\widetilde{\bcC}=   \frac{d^*}{n}\cQ_{\widetilde{\bU}^\top}\cP_{\Omega}\left( \bcT + \bcZ \right).
    \end{aligned}
\end{equation*}
To proceed to study the solution $\widetilde{\bcT}$, we fix the randomenss of $\widetilde{\bU}$: accroding to the metric entropy of the Grassmannian manifold in \cite{cai2013sparse}, for the space $\O_{\bW}=\prod_{j=1}^{m} \O^{d_j\times r_j}\bigcap\{\Inco(\bW_j)\le \mu \}$, the metric entropy can be bounded by $\cN(\O_{\bW},\tF,\varepsilon)\le (\frac{c}{\varepsilon})^{\sum_{j=1}^{m}d_j r_j}$ when the distance is measured by $\max_j\norm{\cP_{\bW_j}-\cP_{\bW_j'}}_{\tF}$. For any fixed $\{\breve{\bU}_j\}_{j=1}^m$ in the $\varepsilon$-net, with probability at least $1-\dmax^{-Cm^2\rmax\dmax}$, we have 
\begin{equation}\label{eq:cct-QPQ}
    \norm{ \frac{d^*}{n}\cQ_{\breve{\bU}^\top}\cP_{\Omega} \cQ_{\breve{\bU}}-\Id}\le C_2\sqrt{\frac{\mu^m m^2\rmax r^*\dmax\log \dmax}{n}}+C_3 \frac{\mu^m m^2\rmax r^*\dmax\log \dmax}{n},
\end{equation}
where $\Id$ means the identity operator. \eqref{eq:cct-QPQ} can be proved by checking the concentration of 
\begin{equation*}
   \norm{\frac{d^*}{n}\sum_{i=1}^{n} \otimes_{j=1}^m \breve{\bU}_j^\top \Vect(\bcX_i)\Vect(\bcX_i)^\top\otimes_{j=1}^m \breve{\bU}_j -\bI_{r^*}}_2
\end{equation*}
using matrix Bernstein inequality. We check the first order condition: 
\begin{equation*}
    \norm{\otimes_{j=1}^m \breve{\bU}_j^\top \Vect(\bcX_i)\Vect(\bcX_i)^\top\otimes_{j=1}^m \breve{\bU}_j}_2\le \frac{\mu^m r^*}{d^*},
\end{equation*}
and the second order condition:
\begin{equation*}
    \norm{\E \otimes_{j=1}^m \breve{\bU}_j^\top \Vect(\bcX_i)\Vect(\bcX_i)^\top\otimes_{j=1}^m \cP_{\breve{\bU}_j}\Vect(\bcX_i)\Vect(\bcX_i)^\top\breve{\bU}_j }_2\le \frac{1}{d^*}\frac{\mu^m r^*}{d^*}.
\end{equation*}
This leads to the bound in \eqref{eq:cct-QPQ}.

Thus, taking $\varepsilon=\dmax^{-Cm}$ for large $C$ we can show that 
\eqref{eq:cct-QPQ} holds uniformly for all $\{\bW_j\}_{j=1}^m\in\O_{\bW}$ with probability at least $1-\dmax^{-3m}$. \eqref{eq:cct-QPQ} indicates that $\frac{d^*}{n}\cQ_{\widetilde{\bU}^\top}\cP_{\Omega} \cQ_{\widetilde{\bU}} $ is invertible given the sample size $n\ge C \mu^m m^2\rmax r^*\dmax\log \dmax$. We then have
\begin{equation}\label{eq:core-repres}
\begin{aligned}
        \widetilde{\bcC} &=  \left(\frac{d^*}{n}\cQ_{\widetilde{\bU}^\top}\cP_{\Omega} \cQ_{\widetilde{\bU}}\right)^{-1} \frac{d^*}{n}\cQ_{\widetilde{\bU}^\top}\cP_{\Omega}\left( \bcT + \bcZ \right), \\
       \widetilde{\bcT} & = \cQ_{\widetilde{\bU}} \widetilde{\bcC} 
       =  \cQ_{\widetilde{\bU}} \left(\frac{d^*}{n}\cQ_{\widetilde{\bU}^\top}\cP_{\Omega} \cQ_{\widetilde{\bU}}\right)^{-1} \frac{d^*}{n}\cQ_{\widetilde{\bU}^\top}\cP_{\Omega}\left( \bcT + \bcZ \right).
\end{aligned}
\end{equation}
\eqref{eq:cct-QPQ} and \eqref{eq:core-repres} tell us that given any $\{\bW_j\}_{j=1}^m$, the corresponding optimal core tensor and the induced optimal solution are determined, and the induced optimal solution changes smoothly with respect to $\{\bW_j\}_{j=1}^m$. The $\varepsilon$-net argument implies that $ \widetilde{\bcT}$ can be very close to the optimal solution induced by a group of singular subspaces in the $\varepsilon$-net of $\O_{\bW}$ by taking $\varepsilon=\dmax^{-Cm}$ for large $C$. Therefore, it suffices to treat the proxy of $\{\widetilde{\bU}_j\}_{j=1}^m$ in the $\varepsilon$-net as $\{\widetilde{\bU}_j\}_{j=1}^m$ itself. We now denote such a proxy still as $\{\widetilde{\bU}_j\}_{j=1}^m$ for a better illustration. $\{\widetilde{\bU}_j\}_{j=1}^m$ is thus fixed in our following discussion.

Using the optimality of $\widetilde{\bcT}$, we have
\begin{equation*}
    \begin{aligned}
        h_n(\widetilde{\bcT})\le h_n({\bcT}),\\
    \end{aligned}
\end{equation*}
which implies 
\begin{equation}\label{eq:LS-opt-HR}
    \begin{aligned}
         \frac{d^*}{n}\norm{\cP_{\Omega}\left( \widetilde{\bcT}-\bcT \right)}_{\tF}^2 \le 2 \frac{d^*}{n}\left\langle\widetilde{\bcT}-\bcT,\cP_{\Omega}(\bcZ)  \right\rangle.\\
    \end{aligned}
\end{equation}
Our gist of proof is that, we first provide an upper bound of the RHS of \eqref{eq:LS-opt-HR}, and then show that the LHS of \eqref{eq:LS-opt-HR} can be connected to (lower bounded by) $\norm{ \widetilde{\bcT} -\bcT}_{\tF}$. To this end, we first tackle the RHS of \eqref{eq:LS-opt-HR}. Since $\{\widetilde{\bU}_j\}_{j=1}^m$ are incoherent subspaces, $\frac{d^*}{n}\cQ_{\widetilde{\bU}^\top}\cP_{\Omega}\left( \bcT + \bcZ \right)$ can be very close to the determined $\cQ_{\widetilde{\bU}^\top}\bcT$ as described in the following Lemma \ref{lemma:LS-cct-QT}:
\begin{Lemma}\label{lemma:LS-cct-QT}
With probability at least $1-\dmax^{-3m}$, we have
\begin{equation}\label{eq:LS-cct-QT}
    \max_j\norm{\cM_j\left( \frac{d^*}{n}\cQ_{\widetilde{\bU}^\top}\cP_{\Omega}\left( \bcT + \bcZ \right) - \cQ_{\widetilde{\bU}^\top} \bcT \right)}_2 \le C_2 (\sigma\vee\norm{\bcT}_{\linf})\sqrt{\frac{\mu^m m^2\rmax r^* d^* \dmax \log \dmax}{n}}
\end{equation}
uniformly for any $\{\widetilde{\bU}_j\}_{j=1}^m$ in a net of $\O_{\bW}$ with cardinality at most $\dmax^{Cm^2\rmax\dmax}$. 
\end{Lemma}
\eqref{eq:cct-QPQ} and \eqref{eq:LS-cct-QT} indicate that the $\Omega$-dependent $\widetilde{\bcT} $ can be very close to its $\Omega$-independent version, $\cQ_{\widetilde{\bU}}\cQ_{\widetilde{\bU}^\top}\bcT$, with error 
\begin{equation}\label{eq:LS-opt-QQT-linf}
    \norm{\widetilde{\bcT} -\cQ_{\widetilde{\bU}}\cQ_{\widetilde{\bU}^\top}\bcT}_\linf \le C_2 (\sigma\vee\norm{\bcT}_{\linf})\sqrt{\frac{\mu^{2m} m^2\rmax (r^*)^2 \dmax \log \dmax}{n}}.
\end{equation}
To get rid of the dependence on $\Omega$ and study the concentration, we instead focus on $\cQ_{\widetilde{\bU}}\cQ_{\widetilde{\bU}^\top}\bcT$ and controlling $\norm{ \cQ_{\widetilde{\bU}}\cQ_{\widetilde{\bU}^\top}\bcT - \bcT }_{\tF}$. The RHS of \eqref{eq:LS-opt-HR} can thus be handled by
\begin{equation}\label{eq:LS-UB-decomp}
    2 \frac{d^*}{n}\left\langle\widetilde{\bcT}-\bcT,\cP_{\Omega}(\bcZ)  \right\rangle = \underbrace{2 \frac{d^*}{n}\left\langle\widetilde{\bcT}-\cQ_{\widetilde{\bU}}\cQ_{\widetilde{\bU}^\top}\bcT,\cP_{\Omega}(\bcZ)  \right\rangle}_{\text{Term 1}} 
    + \underbrace{2 \frac{d^*}{n}\left\langle \cQ_{\widetilde{\bU}}\cQ_{\widetilde{\bU}^\top}\bcT - \bcT ,\cP_{\Omega}(\bcZ)  \right\rangle}_{\text{Term 2}} 
\end{equation}
Using \eqref{eq:LS-opt-QQT-linf}, we can upper bound the Term 1 in \eqref{eq:LS-UB-decomp} by
\begin{equation}\label{eq:LS-UB-term-1}
    2 \frac{d^*}{n}\left\langle\widetilde{\bcT}-\cQ_{\widetilde{\bU}}\cQ_{\widetilde{\bU}^\top}\bcT,\cP_{\Omega}(\bcZ)  \right\rangle \le C \sigma (\sigma\vee\norm{\bcT}_{\linf})\sqrt{\frac{\mu^{2m} m^4\rmax^2 (r^*)^2 \dmax^2 (d^*)^2 \log^2 \dmax}{n^2}}.
\end{equation}
Term 2 in \eqref{eq:LS-UB-decomp} can be controlled by basic  Bernstein inequality by noticing that $\norm{\cQ_{\widetilde{\bU}}\cQ_{\widetilde{\bU}^\top}\bcT - \bcT }_\linf\le 2\sqrt{\mu^{m}r^*}\norm{\bcT}_{\linf} $, which gives:
\begin{equation}\label{eq:LS-UB-term-2}
    2 \frac{d^*}{n}\left\langle \cQ_{\widetilde{\bU}}\cQ_{\widetilde{\bU}^\top}\bcT - \bcT ,\cP_{\Omega}(\bcZ)  \right\rangle \lesssim  \sigma\sqrt{\mu^{m}r^*}\norm{\bcT}_{\linf} \frac{m^2\rmax\dmax d^* \log\dmax }{n} + \norm{ \cQ_{\widetilde{\bU}}\cQ_{\widetilde{\bU}^\top}\bcT - \bcT }_{\tF}\sigma\sqrt{\frac{m^2\rmax\dmax d^*\log\dmax}{n}}.
\end{equation}
Similar as Lemma \ref{lemma:LS-cct-QT}, \eqref{eq:LS-UB-term-1} and \eqref{eq:LS-UB-term-2} holds uniformly for any $\{\widetilde{\bU}_j\}_{j=1}^m$ in the net of $\O_{\bW}$. 
We now handle the LHS of \eqref{eq:LS-opt-HR}. Since $(a+b)^2=a^2+b^2+2ab\ge\frac{1}{2}a^2-b^2$ because $2ab\le\frac{1}{2}a^2+2b^2$, we have
\begin{equation*}
\begin{aligned}
         \frac{d^*}{n}\norm{\cP_{\Omega}\left( \widetilde{\bcT}-\bcT \right)}_{\tF}^2 & \ge \frac{1}{2}\frac{d^*}{n}\norm{\cP_{\Omega}\left(\cQ_{\widetilde{\bU}}\cQ_{\widetilde{\bU}^\top}\bcT-\bcT \right)}_{\tF}^2 - \frac{d^*}{n}\norm{\cP_{\Omega}\left( \cQ_{\widetilde{\bU}}\cQ_{\widetilde{\bU}^\top}\bcT-\widetilde{\bcT} \right)}_{\tF}^2 \\
         &\ge \frac{1}{2}\frac{d^*}{n}\norm{\cP_{\Omega}\left(\cQ_{\widetilde{\bU}}\cQ_{\widetilde{\bU}^\top}\bcT-\bcT \right)}_{\tF}^2  - C (\sigma\vee\norm{\bcT}_{\linf})^2\frac{\mu^{2m} m^2\rmax (r^*)^2 d^*\dmax \log \dmax}{n}.
\end{aligned}
\end{equation*}
It remains to study the concentration of $\frac{d^*}{n}\norm{\cP_{\Omega}\left(\cQ_{\widetilde{\bU}}\cQ_{\widetilde{\bU}^\top}\bcT-\bcT \right)}_{\tF}^2$. We invoke the following Lemma \ref{lemma:LS-cct-QQT}:
\begin{Lemma}\label{lemma:LS-cct-QQT}
    With probability at least $1-\dmax^{-3m}$, we have
\begin{equation}\label{eq:LS-cct-QQT}
\begin{aligned}
         &\abs{\frac{d^*}{n}\norm{\cP_{\Omega}\left(\cQ_{\widetilde{\bU}}\cQ_{\widetilde{\bU}^\top}\bcT-\bcT \right)}_{\tF}^2 - \norm{\cQ_{\widetilde{\bU}}\cQ_{\widetilde{\bU}^\top}\bcT-\bcT}_{\tF}^2} \\
         &\le C\frac{{m^2\rmax\mu^{m}r^* d^* \dmax \log \dmax }\norm{\bcT}_{\linf}^2}{n} +C\norm{\bcT}_{\linf}\norm{\cQ_{\widetilde{\bU}}\cQ_{\widetilde{\bU}^\top}\bcT-\bcT}_{\tF} \sqrt{\frac{ m^2\rmax{\mu^{m}r^*} d^* \dmax\log \dmax  }{n}},
\end{aligned}
\end{equation}
uniformly for any $\{\widetilde{\bU}_j\}_{j=1}^m$ in a net of $\O_{\bW}$ with cardinality at most $\dmax^{Cm^2\rmax\dmax}$.
\end{Lemma}
Thus, the LHS of \eqref{eq:LS-opt-HR} can be controlled by 
\begin{equation}\label{eq:LS-LB}
    \begin{aligned}
         &\frac{d^*}{n}\norm{\cP_{\Omega}\left( \widetilde{\bcT}-\bcT \right)}_{\tF}^2  \ge \frac{1}{2}\frac{d^*}{n}\norm{\cP_{\Omega}\left(\cQ_{\widetilde{\bU}}\cQ_{\widetilde{\bU}^\top}\bcT-\bcT \right)}_{\tF}^2  - C (\sigma\vee\norm{\bcT}_{\linf})^2\frac{\mu^{2m} m^2\rmax (r^*)^2 d^*\dmax \log \dmax}{n} \\
         & \ge \frac{1}{2}\norm{\cQ_{\widetilde{\bU}}\cQ_{\widetilde{\bU}^\top}\bcT-\bcT}_{\tF}^2 -C\frac{{m^2\rmax\mu^{m}r^* d^*\dmax\log \dmax }\norm{\bcT}_{\linf}^2}{n} -C\norm{\bcT}_{\linf}\norm{\cQ_{\widetilde{\bU}}\cQ_{\widetilde{\bU}^\top}\bcT-\bcT}_{\tF} \sqrt{\frac{ m^2\rmax{\mu^{m}r^*}d^* \dmax\log \dmax  }{n}} \\
         & \quad \quad -C (\sigma\vee\norm{\bcT}_{\linf})^2\frac{\mu^{2m} m^2\rmax (r^*)^2 d^*\dmax \log \dmax}{n} 
\end{aligned}
\end{equation}
Combining the lower bound of the LHS of \eqref{eq:LS-opt-HR} in \eqref{eq:LS-LB}, and the RHS of \eqref{eq:LS-opt-HR} in \eqref{eq:LS-UB-term-1}, \eqref{eq:LS-UB-term-2}, we have:
\begin{equation*}
    \begin{aligned}
        &\norm{\cQ_{\widetilde{\bU}}\cQ_{\widetilde{\bU}^\top}\bcT-\bcT}_{\tF}^2 - C (\sigma\vee\norm{\bcT}_{\linf}) \norm{ \cQ_{\widetilde{\bU}}\cQ_{\widetilde{\bU}^\top}\bcT - \bcT }_{\tF}\sqrt{\frac{m^2\rmax{\mu^{m}r^*} d^*\dmax\log \dmax}{n}} \\
        & \le  C (\sigma\vee\norm{\bcT}_{\linf})^2\frac{\mu^{2m} m^2\rmax (r^*)^2 d^*\dmax \log \dmax}{n}.
    \end{aligned}
\end{equation*}
This implies that 
\begin{equation*}
    \begin{aligned}
        &\norm{\cQ_{\widetilde{\bU}}\cQ_{\widetilde{\bU}^\top}\bcT-\bcT}_{\tF} \le  C (\sigma\vee\norm{\bcT}_{\linf})\sqrt{\frac{\mu^{2m} m^2\rmax (r^*)^2 d^*\dmax \log \dmax}{n}}.
    \end{aligned}
\end{equation*}
Since the error of $\norm{\widetilde{\bcT} -\cQ_{\widetilde{\bU}}\cQ_{\widetilde{\bU}^\top}\bcT}_{\tF}$ can be further controlled by \eqref{eq:LS-opt-QQT-linf}, we have
\begin{equation*}
  \norm{\widetilde{\bcT}-\bcT}_{\tF} \le \norm{\cQ_{\widetilde{\bU}}\cQ_{\widetilde{\bU}^\top}\bcT-\bcT}_{\tF} + \norm{\widetilde{\bcT} -\cQ_{\widetilde{\bU}}\cQ_{\widetilde{\bU}^\top}\bcT}_{\tF} \le C (\sigma\vee\norm{\bcT}_{\linf})\sqrt{\frac{\mu^{2m} m^2\rmax (r^*)^2 d^*\dmax \log \dmax}{n}},
\end{equation*}
with probability at least $1-6\dmax^{-3m}$. This finishes the proof.
\end{proof}


\subsection{Proof of Lower Bound in Theorem \ref{thm:opt-uct-qtf}}
We now assume the noise $\xi_i\sim\cN(0,\sigma^2)$. We endow $\cM_{\br}$ with the Euclidean (Frobenius) metric. Let $\Bar{\bcT} $ be the estimator of $\bcT$ stated above. notice that the MSE can be lower bounded by:
\begin{equation*}
\begin{aligned}
        & \E\left(g_{\bcI}(\Bar{\bcT}) - g_{\bcI}({\bcT})\right)^2\ge  \E\left(g_{\bcI}(\Bar{\bcT}) - g_{\bcI}({\bcT})\right)^2 \idc\{\bar{\bcT}\in\bbB_{\opt}\} \\
        &= \E\left(g_{\bcI}\left(\Bar{\bcT}\idc\{\bar{\bcT}\in\bbB_{\opt}\} +{\bcT}\idc\{\bar{\bcT}\notin\bbB_{\opt}\}\right) - g_{\bcI}({\bcT})\right)^2 
\end{aligned}
\end{equation*}
Clearly, we have  $\E\left[\Bar{\bcT}\idc\{\bar{\bcT}\in\bbB_{\opt}\} +{\bcT}\idc\{\bar{\bcT}\notin\bbB_{\opt}\}\right] = (1+\varepsilon_{\opt})\bcT$, where $\abs{\varepsilon_{\opt}}\le \dmax^{-C_{\gap}\widetilde{m}}$ is exceedingly small. Thus, 
\begin{equation*}
    \widetilde{\bcT}_{\ubs} = \frac{\Bar{\bcT}\idc\{\bar{\bcT}\in\bbB_{\opt}\} +{\bcT}\idc\{\bar{\bcT}\notin\bbB_{\opt}\}}{1+ \varepsilon_{\opt}}
\end{equation*}
is an unbiased estimator of $\bcT$, with the MSE lower bound:
\begin{equation}\label{eq:mse-lb}
\begin{aligned}
        & \E\left(g_{\bcI}(\Bar{\bcT}) - g_{\bcI}({\bcT})\right)^2\ge\E\left(g_{\bcI}\left(\Bar{\bcT}\idc\{\bar{\bcT}\in\bbB_{\opt}\} +{\bcT}\idc\{\bar{\bcT}\notin\bbB_{\opt}\}\right) - g_{\bcI}({\bcT})\right)^2  \\
        & = (1+\varepsilon_{\opt} )^2\E (g_{\bcI}(\widetilde{\bcT}_{\ubs}) - g_{\bcI}(\bcT))^2 + \varepsilon_{\opt}^2 g_{\bcI}(\bcT)^2
\end{aligned}
\end{equation}
Therefore, it amounts to studying the Cramér–Rao lower bound of the estimator $g_{\bcI}(\widetilde{\bcT}_{\ubs})$. Denote the vectorized error as $\bar{\bdelta}=\Vect({\widetilde{\bcT}_{\ubs}-\bcT})$.  By intrinsic Cramér–Rao lower bounds on Riemannian manifolds in \cite{smith2005covariance}, \cite{boumal2013intrinsic}, \cite{boumal2014optimization}, and the generalization of inverse Fisher information \citep{gorman1990lower,kay1993fundamentals,shao2003mathematical},
we have the covariance matrix $\bC$ of $\bar{\bdelta}$ satisfies
\begin{equation}\label{eq:CR-cov-lb}
    \bC \succeq  \underbrace{\bF^{\dagger}}_{\text{Fisher information term}} - \underbrace{\frac{1}{3}(\bR_{\sfm}\bF^{\dagger} + \bF^{\dagger}\bR_{\sfm})}_{\text{small curvature term}},
\end{equation}
where $\bF$ is the Fisher information metric on the Riemannian manifold, and $\bR_{\sfm}$ is a remainder matrix representing the Riemannian curvature terms defined in \cite{smith2005covariance} which is generally of small order. For the noisy tensor completion problem, we have the log-likelihood function
\begin{equation*}
    l(\bcW) = \sum_{i=1}^n\frac{(Y_i-\langle\bcW,\bcX_i \rangle)^2}{2\sigma^2},
\end{equation*}
with the Fisher information on the Riemannian manifold \citep{smith2005covariance,boumal2013intrinsic} given by:
\begin{equation}\label{eq:CR-Fisher}
\begin{aligned}
       \bF &=\E\left[\Vect\left(\cP_\TT(\nabla_{\bcT} l(\bcT))\right)\cdot \Vect^\top\left(\cP_\TT(\nabla_{\bcT} l(\bcT))\right) \right] =\frac{n}{\sigma^2}\E\left[\Vect\left(\cP_\TT(\bcX_i) \right) \Vect\left(\cP_\TT(\bcX_i)\right)\right] \\
       & = \frac{n}{\sigma^2} \frac{1}{d^*}\underbrace{\left[\otimes_{j=1}^m \cP_{\bU_j}+\sum_{j=1}^m \Perm_j\left[\left(\otimes_{k \neq j} {\bU}_{ k}\right)\cP_{\bV_j}\left(\otimes_{k \neq j} {\bU}_{ k}^\top\right)\right] \otimes \cP_{{\bU}_{j}}^{\perp}  \Perm_j^\top \right]  }_{\text{denoted by }\bT_{\cM}} \sum_{l=1}^{d^*}\Vect(\bomega_l)\Vect(\bomega_l)^\top   \\
       &\noindent \hfill \qquad  \qquad \qquad \qquad \qquad \qquad \qquad \cdot \left[\otimes_{j=1}^m \cP_{\bU_j}+\sum_{j=1}^m \Perm_j\left[\left(\otimes_{k \neq j} {\bU}_{ k}\right)\cP_{\bV_j}\left(\otimes_{k \neq j} {\bU}_{ k}^\top\right)\right] \otimes \cP_{{\bU}_{j}}^{\perp}  \Perm_j^\top \right] \\
       & = \frac{n}{\sigma^2}\frac{1}{d^*} \bT_{\cM} \bI_{d^*}\bT_{\cM}  = \frac{n}{\sigma^2}\frac{1}{d^*} \bT_{\cM},
\end{aligned}
\end{equation}
where $\bT_{\cM}$ defines the vectorized projection operator $\cP_{\TT}$. Here $\Perm_j\in\R^{d^*}$ is a permutation matrix that reshapes a mode-$j$ vectorized tensor to its default mode-1 vectorization. By definition, we have $\bT_{\cM}^2=\bT_{\cM}$. The inverse Fisher information is thus $\bF^{\dagger}=\frac{\sigma^2 d^*}{n}\bT_{\cM}$. Therefore, our main term $\bF^{\dagger}$ already gives the lower bound of estimating $\langle \bcT, \bcI \rangle$, which is exactly $\Vect(\bcI)^\top \bF^{\dagger} \Vect(\bcI)= \frac{\sigma^2 d^*}{n} \norm{\cP_\TT (\bcI)}_{\tF}^2$.


We now have a closer look at the curvature of the manifold $\cM_{\br}$ near $\bcT$ to justify its vanishing. 
According to the sectional curvature of the Segre manifold (rank-1 tensor manifold) in Section 6.2.3  of \cite{swijsen2022tensor} (also, \cite{jacobsson2024approximating}), we have the sectional curvature at $\bcT$ bounded by $\abs{K_{\max}}\le \frac{1}{\lambda_{\min}^2}$.
The error can be controlled by:
\begin{equation*}
\begin{aligned}
        \norm{\Bar{\bdelta}}_2 & = \norm{\widetilde{\bcT}_{\ubs} -\bcT}_{\tF} \le \norm{\frac{\Bar{\bcT}}{1+ \varepsilon_{\opt}}-\bcT}_{\tF}\idc\{\bar{\bcT}\in\bbB_{\opt}\}  + \norm{\frac{{\bcT}}{1+ \varepsilon_{\opt}} - \bcT}_{\tF} \idc\{\bar{\bcT}\notin\bbB_{\opt}\} \\
        & \le 2\norm{\Bar{\bcT}-\bcT}_{\tF}\idc\{\bar{\bcT}\in\bbB_{\opt}\} + 4\varepsilon_{\opt}\lambda_{\max}  
        \le C\sigma\sqrt{\frac{d^*\dmax\log\dmax}{n}},
\end{aligned}
\end{equation*}
where we use the fact that $\widetilde{m} = \max\{m, \log(n \vee \frac{\lambda_{\max}}{\sigma}) / \log\dmax\}$ implies $\abs{\varepsilon_{\opt}}\le \dmax^{-C_{\gap}\widetilde{m}}\le \sigma/(\lambda_{\max}\operatorname{Poly}(\dmax n))$ with any fixed polynomial degree as long as $C_{\gap}$ is large. We then have the relationship $\norm{\Bar{\bdelta}}_2\ll \frac{1}{\sqrt{\abs{K_{\max}}}}$ since $\lambda_{\min}/\sigma \gg \sqrt{\frac{d^* \dmax \log \dmax}{n}}$.   Thus, according to \cite{smith2005covariance,boumal2014optimization}, we know that the curvature term in Cramér–Rao lower bound is negligible. Moreover, to have a refined study on the upper bound of such a term,  \cite{smith2005covariance,boumal2013intrinsic} show that the exact $\bR_{\sfm}$ can be well approximated by 
$$
(\bR_{\sfm})_{i j}=\E\left[\left\langle\bcR\left(\bar{\bdelta}, \bT_{\cM} \be_i\right) \bT_{\cM}\be_j, \bar{\bdelta}\right\rangle\right],
$$
where $\bcR$ indicates the Riemannian curvature tensor. For more introduction to the  Riemannian curvature tensor, see also, \cite{lee2006riemannian,boumal2014optimization}. Given the bounded $\abs{K_{\max}}$ around $\bcT$, it is clear that 
\begin{equation*}
    \abs{\langle\bcR(\bx, \by) \bz,\bdelta\rangle} \le \abs{K_{\max}}\cdot\abs{\langle\langle \by, \bz\rangle \bx-\langle \bx, \bz\rangle \by,\bdelta\rangle}
\end{equation*}
for any $\bx, \by,\bz,\bdelta$ by the sectional curvature. Thus, we have the bound on $(\bR_{\sfm})_{i j}$:
\begin{equation*}
    \abs{(\bR_{\sfm})_{i j}}\le 2\abs{K_{\max}} \norm{\bT_{\cM} \be_i}_2 \norm{\bT_{\cM} \be_j }_2 \E\norm{\Bar{\bdelta}}_2^2\le \frac{\sigma^2}{\lambda_{\min}^2 }\frac{r^*\mu^m\dmax}{d^*} \frac{d^*\dmax\log\dmax}{n},
\end{equation*}
which leads to:
\begin{gather}
     \abs{(\bR_{\sfm})_{i j}}\norm{\be_i\be_j^\top \bT_{\cM}}_{2}\le  \frac{C_m\sigma^2 (r^*)^{\frac{3}{2}}\mu^{\frac{3m}{2}}\dmax^{\frac{5}{2}}\log\dmax}{ \sqrt{d^*} \lambda_{\min}^2 n}, \ \forall i,j \in [d^*] \label{eq:CR-curv-Rm}
\end{gather}
Combining \eqref{eq:CR-cov-lb}, \eqref{eq:CR-Fisher}, and \eqref{eq:CR-curv-Rm}, it is clear that for any $g_{\bcI}(\bcT)=\langle\bcT,\bcI\rangle$, the variance of the estimation $g_{\bcI}(\widetilde{\bcT}_{\ubs})=\langle\widetilde{\bcT}_{\ubs},\bcI\rangle$ satisfies
\begin{equation*}
\begin{aligned}
      \Var(g_{\bcI}(\widetilde{\bcT}_{\ubs}))  & = \nabla^\top g_{\bcI}(\bcT)\cdot  \bC \cdot  \nabla g_{\bcI}(\bcT)  = \Vect(\bcI)^\top \cdot \bC \cdot \Vect(\bcI) \\
        & \ge  \Vect(\bcI)^\top \bF^{\dagger} \Vect(\bcI) -\frac{1}{3}\left(\Vect(\bcI)^\top\bF^{\dagger}\bR_{\sfm}\Vect(\bcI)+\Vect(\bcI)^\top\bR_{\sfm}\bF^{\dagger}\Vect(\bcI)\right) \\
        & \ge \frac{\sigma^2 d^*}{n} \norm{\cP_\TT (\bcI)}_{\tF}^2 - \frac{2\sigma^2 d^*}{3n} \norm{\cP_\TT (\bcI)}_{\tF} \norm{\bcI}_{\ell_1}\max_{i,j\in [d^*]} \abs{(\bR_{\sfm})_{i j}}\norm{\be_i\be_j^\top \bT_{\cM}}_{2}\cdot\sqrt{d^*} \\
        & \ge \left(1-\frac{C_m\sigma^2 (r^*)^{\frac{3}{2}}\mu^{\frac{3m}{2}}\dmax^{\frac{5}{2}}\log\dmax \norm{\bcI}_{\ell_1} }{ \lambda_{\min}^2 n \norm{\cP_\TT (\bcI)}_{\tF}}\right)\frac{\sigma^2 d^*}{n} \norm{\cP_\TT (\bcI)}_{\tF}^2.
\end{aligned}
\end{equation*}
Accordin to the alignment assumption \ref{asm:alignment}, we have 
\begin{equation*}
    \norm{ \cP_{\TT}(\bcI) }_\tF \ge \alpha_{I} \cdot \dmax^{\frac{1}{2}}\left(d^*\right)^{-\frac{1}{2}}\norm{ \bcI }_\tF.
\end{equation*}
Therefore, the $\Var(g_{\bcI}(\widetilde{\bcT}_{\ubs}))$ can be lower bounded by
\begin{equation}\label{eq:var-lb-exact}
     \Var(g_{\bcI}(\widetilde{\bcT}_{\ubs})) \ge \left(1-\frac{C_m\sigma^2 (r^*)^{\frac{3}{2}}\mu^{\frac{3m}{2}}\dmax^{2} (d^*)^{\frac{1}{2}}\log\dmax \norm{\bcI}_{\ell_1} }{ \alpha_{I}\lambda_{\min}^2 n \norm{\bcI}_{\tF} }\right)\frac{\sigma^2 d^*}{n} \norm{\cP_\TT (\bcI)}_{\tF}^2.
\end{equation}
Since $m\ge 2$, we always have $ (d^*)^{\frac{1}{2}} \ge \dmax$, thus  
$$\frac{\lambda_{\min}}{\sigma }\gg \sqrt{\frac{C_m (r^*)^{\frac{3}{2}}\mu^{\frac{3m}{2}}\dmax d^*\log\dmax \norm{\bcI}_{\ell_1} }{\alpha_I\norm{\bcI}_{\tF} n}}\ge \sqrt{\frac{C_m (r^*)^{\frac{3}{2}}\mu^{\frac{3m}{2}}\dmax^{2} (d^*)^{\frac{1}{2}}\log\dmax \norm{\bcI}_{\ell_1} }{\alpha_I\norm{\bcI}_{\tF} n}}.$$ 

Inserting \eqref{eq:var-lb-exact} into \eqref{eq:mse-lb},  given $\abs{\varepsilon_{\opt}}\le \sigma/(\lambda_{\max}\operatorname{Poly}(\dmax n))$  and  $\varepsilon_{\SNR}\to 0$, we have 
\begin{equation*}
    \begin{aligned}
          \E\left(g_{\bcI}(\Bar{\bcT}) - g_{\bcI}({\bcT})\right)^2&\ge (1+\varepsilon_{\opt} )^2\E (g_{\bcI}(\widetilde{\bcT}_{\ubs}) - g_{\bcI}(\bcT))^2 = (1+\varepsilon_{\opt} )^2 \Var(g_{\bcI}(\widetilde{\bcT}_{\ubs}))  \\
         & \ge (1+\varepsilon_{\opt} )^2\left(1-\frac{C_m\sigma^2 (r^*)^{\frac{3}{2}}\mu^{\frac{3m}{2}}\dmax^{2} (d^*)^{\frac{1}{2}}\log\dmax \norm{\bcI}_{\ell_1} }{ \alpha_{I}\lambda_{\min}^2 n \norm{\bcI}_{\tF} }\right)\frac{\sigma^2 d^*}{n} \norm{\cP_\TT (\bcI)}_{\tF}^2 \\
         & \ge (1-\varepsilon_{\SNR}^2)\frac{\sigma^2 d^*}{n} \norm{\cP_\TT (\bcI)}_{\tF}^2
    \end{aligned}
\end{equation*}
Thus, we prove the desired claim.

\subsection{Proof of Theorem \ref{thm:hetero-exp-clt}}
The proof immediately follows the proof of Theorem \ref{thm:lf-inference-popvar} by noticing that, when under the heteroskedastic and sub-exponential noises, the error bounds for $\{\frE_i\}_{i=2}^5$ developed in Lemma \ref{lemma:oracle-init-frE-2-5} still hold by replacing the $\sigma$ with $\sigmax$. The reason that this replacement is correct can be justified by the matrix Bernstein inequality in Lemma \ref{lemma:sub-exp-bern}. It can also be checked in the proof of Theorem \ref{thm:lf-inference-empvar} that we can replace the $\psi_2$ norm by $\psi_1$ norm in all the concentration inequalities using the Bernstein-type conditions. Then, we go back to the decomposition \eqref{eq:test-decomp} and check the CLT of $\frE_1$ under the heteroskedastic and sub-exponential noises, where $\frE_1= \frac{d^*}{n} \sum_{i=1}^{n} \xi_i \left\langle \bcX_i, \cP_{\TT}(\bcI) \right\rangle$ according to \eqref{eq:frE-1-true}. 

Using the Berry-Esseen theorem \citep{stein1972bound,raivc2019multivariate}, we compute the second and third moments of $ \frE_1$:
\begin{equation*}
    \begin{gathered}
        \E\abs{\frE_1}^2 =\frac{d^*}{n}\sum_{\bomega}\E \left[\xi_i^2 \left\langle \bomega, \cP_{\TT}(\bcI) \right\rangle^2\mid\bcX_i = \bomega\right]= \frac{d^*}{n}\sum_{\omega}\left[\cP_{\TT}(\bcI)\right]_{\omega}^2\left[\bcS\right]_{\omega}^2 = \frac{d^*}{n} \norm{\cP_{\TT }(\bcI)\odot\bcS  }_{\tF}^2 
    \end{gathered}
\end{equation*}
and
\begin{equation*}
    \begin{aligned}
         & \E\abs{\xi_i \left\langle \bcX_i, \cP_{\TT}(\bcI) \right\rangle}^3 = \frac{1}{d^*}\sum_{\omega}\left[\cP_{\TT}(\bcI)\right]_{\omega}^3\E\left[\xi_i^3\mid[\bcX_i]_{\omega}=1\right]\le \frac{C}{d^*}\sum_{\omega}\left[\cP_{\TT}(\bcI)\right]_{\omega}^2\sigmax^3\max\abs{\left\langle \bcX_i, \cP_{\TT}(\bcI) \right\rangle}   \\
         & \le C  \sigmax^3 \frac{\norm{\cP_{\TT}(\bcI)}_\tF^3}{d^*} \max \left( \sqrt{\norm{\bcX_i \times_{j=1}^m \cP_{\bU_{j} }}_{\tF}^2+ \sum_{j=1}^m \norm{\cM_j(\bcX_i)\times_{k\neq j}^m \cP_{\bU_{k}}}_{\tF}^2} \right) \\
         & \le C \sigmax^3 \frac{\norm{\cP_{\TT}(\bcI)}_\tF^3}{d^*} \sqrt{\frac{2m r^* \mu^{m-1} \dmax}{ \rmin d^*}} \\
         & \le C \kappa_{\sigma}^3 \sqrt{m r^* \mu^{m-1} \dmax} \frac{\norm{\cP_{\TT}(\bcI)\odot\bcS}_\tF^3}{\sqrt{\rmin }\left(d^*\right)^{\frac{3}{2}}},
    \end{aligned}
\end{equation*}
where we use \eqref{eq:variance-BE} in the second line, and  \eqref{eq:variance-BE-max} in the third line, and the fact that $\sigmax/[\bcS]_{\omega}\le \kappa_{\sigma}$ for the heterogeneous noises in the last line. Thus, we have the new asymptotic normality for heteroskedastic and sub-exponential noises:
\begin{equation}\label{eq:hetero-core-clt}
    \max_{t\in\R}\abs{\bbP\left(\frac{\frE_1}{ \norm{\cP_{\TT }(\bcI)\odot\bcS  }_{\tF} \sqrt{d^*/n} } \le t\right)- \Phi(t)} \le  C \kappa_{\sigma}^3\sqrt{\frac{m r^* \mu^{m-1} \dmax}{\rmin n}}
\end{equation}
Clearly, according to the lower bound of each element in $\bcS$ and alignment assumption, we have:
 \begin{equation}\label{eq:hetero-align}
     \norm{\cP_{\TT }(\bcI)\odot\bcS  }_{\tF}\ge  \sigmin\norm{\cP_{\TT }(\bcI)\odot \bcJ }_{\tF}\ge \alpha_I  \sigmin\cdot \dmax^{\frac{1}{2}}\left(d^*\right)^{-\frac{1}{2}}\norm{ \bcI }_\tF.
 \end{equation}
Here $\bcJ$ is the full-1 tensor. Therefore, combining remainder terms in Lemma \ref{lemma:oracle-init-frE-2-5} and \eqref{eq:hetero-align}, we have:
\begin{equation*}
    \begin{aligned}
       \abs{\frac{\sum_{i=2}^{5}\frE_i}{\norm{\cP_{\TT }(\bcI)\odot\bcS  }_{\tF} \sqrt{d^*/n} }}\le C C_1 \kappa_{\sigma}\sqrt{\frac{ m^3 (2\mu)^{m-1} r^* \dmax \log^2\dmax }{ \rmax \rmin n}} + C \frac{ C_1^2\kappa_{\sigma}\sigmax  \norm{\bcI}_{\ell_1} }{\alpha_I \lambda_{\min}  \norm{\bcI}_\tF }\sqrt{\frac{ m^4 (2\mu)^{3 m} (r^*)^{3 }\dmax d^*  \log^2 \dmax }{\rmin^2 n } }
    \end{aligned}
\end{equation*}
with probability at least $1- \dmax^{-2m}$. Combining this result with \eqref{eq:hetero-core-clt}, we can derive that
\begin{equation*}
\begin{aligned}
        &\max_{t\in\R}\abs{\bbP\left( \frac{\left\langle  \widehat{\bcT} - \bcT,\bcI  \right\rangle }{ \norm{\cP_{\TT }(\bcI)\odot\bcS  }_{\tF} \sqrt{{d^*}/{n}} } \le t\right)- \Phi(t)}  \\
        &\le C C_1 \kappa_{\sigma}^3\sqrt{\frac{ m^3 (2\mu)^{m-1} r^* \dmax \log^2\dmax }{ \rmax \rmin n}} + C \frac{ C_1^2\kappa_{\sigma}\sigmax  \norm{\bcI}_{\ell_1} }{\alpha_I \lambda_{\min}  \norm{\bcI}_\tF }\sqrt{\frac{ m^4 (2\mu)^{3 m} (r^*)^{3 }\dmax d^*  \log^2 \dmax }{\rmin^2 n } } +C \dmax^{-2m}.
\end{aligned}
\end{equation*}
This gives the asymptotic normality of $W^{\sfh}_{\test}(\bcI)$. We now proceed to derive the asymptotic normality of $\widehat W^{\sfh}_{\test}(\bcI)$ by computing the estimation error of variance: 
\begin{equation}\label{eq:hetero-var-decomp}
\begin{aligned}
        \abs{\widehat s^2(\bcI) - \norm{\cP_{\TT }(\bcI)\odot\bcS  }_{\tF}^2 }& \le \abs{\frac{d^*}{n}\sum_{i=1}^n\left[\left(Y_i-\left\langle \bcT_{\init},\bcX_i \right\rangle\right)\left\langle\cP_{\widehat \TT }(\bcI), \bcX_i\right\rangle \right]^2 - \norm{\cP_{\widehat\TT }(\bcI)\odot\bcS  }_{\tF}^2 } \\
    & + \abs{\norm{\cP_{\widehat \TT }(\bcI)\odot\bcS  }_{\tF}^2 - \norm{\cP_{\TT }(\bcI)\odot\bcS  }_{\tF}^2}
\end{aligned}
\end{equation}
We  consider the first term in \eqref{eq:hetero-var-decomp}:
\begin{equation}\label{eq:hetero-var-conv}
    \begin{aligned}
        &\abs{\frac{d^*}{n}\sum_{i=1}^n 
        \left[\left(Y_i-\left\langle \bcT_{\init},\bcX_i \right\rangle\right)\left\langle\cP_{\widehat \TT }(\bcI), \bcX_i\right\rangle \right]^2 - \norm{\cP_{\widehat\TT }(\bcI)\odot\bcS  }_{\tF}^2 }  \\
        &=  \abs{\frac{d^*}{n}\sum_{i=1}^n\left(\xi_i^2 + 2 \xi_i \left\langle \bcT^{\vartriangle},\bcX_i \right\rangle + \left\langle \bcT^{\vartriangle},\bcX_i \right\rangle^2 \right)\left\langle\cP_{\widehat \TT }(\bcI), \bcX_i\right\rangle^2 - \norm{\cP_{\widehat\TT }(\bcI)\odot\bcS  }_{\tF}^2 } \\
        & \le \abs{\frac{d^*}{n}\sum_{i=1}^n \xi_i^2\left\langle \cP_{\widehat \TT }(\bcI),\bcX_i \right\rangle^2- \norm{\cP_{\widehat\TT }(\bcI)\odot\bcS  }_{\tF}^2}  \\
        & + 2\abs{\frac{d^*}{n}\sum_{i=1}^n \xi_i \left\langle \bcT^{\vartriangle},\bcX_i \right\rangle\left\langle \cP_{\widehat \TT }(\bcI),\bcX_i \right\rangle^2 } + \abs{\frac{d^*}{n}\sum_{i=1}^n\left\langle \bcT^{\vartriangle},\bcX_i \right\rangle^2\left\langle \cP_{\widehat \TT }(\bcI),\bcX_i \right\rangle^2 }  \\
    \end{aligned}
\end{equation}
The first and term in \eqref{eq:hetero-var-conv} follows the sub-Weibull distribution sub-Weibull($\alpha$) with $\alpha=\frac{1}{2}$, and its concentration has been studies in  \cite{kuchibhotla2022moving}. To use its concentration property, we check the Orlicz norm of each i.i.d terms and their second-order moment:
\begin{equation*}
\begin{aligned}
        &\norm{\xi_i^2\left\langle \cP_{\widehat \TT }(\bcI),\bcX_i \right\rangle^2}_{\psi_{\frac{1}{2}}} \le C \sigmax^2\cdot \frac{m\mu^{m-1} r^*\dmax }{\rmin d^*} \norm{\cP_{\widehat \TT }(\bcI)}_{\tF}^2 \\
        &\E\left(\xi_i^2\left\langle \cP_{\widehat \TT }(\bcI),\bcX_i \right\rangle^2\right)^2 \le C \sigmax^4 \frac{1}{d^*} \frac{m\mu^{m-1} r^*\dmax }{\rmin d^*} \norm{\cP_{\widehat \TT }(\bcI)}_{\tF}^4.
\end{aligned} 
\end{equation*}
Thus, according to Theorem 3.4 of \cite{kuchibhotla2022moving}, we have
\begin{equation}\label{eq:hetero-var-sigma}
\begin{aligned}
        &\abs{\frac{d^*}{n}\sum_{i=1}^n \xi_i^2\left\langle \cP_{\widehat \TT }(\bcI),\bcX_i \right\rangle^2- \norm{\cP_{\widehat\TT }(\bcI)\odot\bcS  }_{\tF}^2}\\
    &\le C\sigmax^2 \norm{\cP_{\widehat \TT }(\bcI)}_{\tF}^2\sqrt{\frac{m^2\mu^{m-1} r^*\dmax\log\dmax }{\rmin n }} + C \sigmax^2\cdot \frac{m\mu^{m-1} r^*\dmax (m\log\dmax)^4 }{\rmin n} \norm{\cP_{\widehat \TT }(\bcI)}_{\tF}^2 \\
    & \le C \sigmax^2\norm{\cP_{\widehat\TT }(\bcI)  }_{\tF}^2\sqrt{\frac{m^5\mu^{m-1} r^*\dmax\log^4\dmax }{\rmin n }}.
\end{aligned}
\end{equation}
with probability at least $1-\dmax^{-3m}$. Here we use the sample size condition $n\ge C_{\gap}m^5\mu^{m-1} r^*\dmax\log^4\dmax  $.
The second and third terms in \eqref{eq:hetero-var-conv} can be controlled by
\begin{equation}\label{eq:hetero-var-rem}
    \begin{aligned}
        &  2\abs{\frac{d^*}{n}\sum_{i=1}^n \xi_i \left\langle \bcT^{\vartriangle},\bcX_i \right\rangle\left\langle \cP_{\widehat \TT }(\bcI),\bcX_i \right\rangle^2 } + \abs{\frac{d^*}{n}\sum_{i=1}^n\left\langle \bcT^{\vartriangle},\bcX_i \right\rangle^2\left\langle \cP_{\widehat \TT }(\bcI),\bcX_i \right\rangle^2 } \\
         &\le CC_1\sigmax^2\sqrt{\frac{m\mu^{m-1} r^*\dmax^2\log\dmax}{\rmin n^2}} \norm{\cP_{\widehat \TT }(\bcI)}_{\tF}^2 + C_1^2\sigmax^2\frac{\dmax\log\dmax}{n}\norm{\cP_{\widehat \TT }(\bcI)}_{\tF}^2 \\
         & \le CC_1^2\sigmax^2\sqrt{\frac{m\mu^{m-1} r^*\dmax^2\log\dmax}{\rmin n^2}} \norm{\cP_{\widehat \TT }(\bcI)}_{\tF}^2
    \end{aligned}
\end{equation}
with probability at least $1-\dmax^{-3m}$ by the Bernstein inequality because we have
\begin{equation*}
    \begin{aligned}
        &\norm{\xi_i \left\langle \bcT^{\vartriangle},\bcX_i \right\rangle\left\langle \cP_{\widehat \TT }(\bcI),\bcX_i \right\rangle^2}_{\psi_1}\le CC_1\sigmax^2\sqrt{\frac{\dmax\log\dmax}{n}}\cdot \frac{m\mu^{m-1} r^*\dmax }{\rmin d^*} \norm{\cP_{\widehat \TT }(\bcI)}_{\tF}^2 \\
        &\norm{\left\langle \bcT^{\vartriangle},\bcX_i \right\rangle^2\left\langle \cP_{\widehat \TT }(\bcI),\bcX_i \right\rangle^2}_{\psi_1}\le CC_1^2\sigmax^2 {\frac{\dmax\log\dmax}{n}}\cdot \frac{m\mu^{m-1} r^*\dmax }{\rmin d^*} \norm{\cP_{\widehat \TT }(\bcI)}_{\tF}^2 \\
        &\E\left(\xi_i \left\langle \bcT^{\vartriangle},\bcX_i \right\rangle\left\langle \cP_{\widehat \TT }(\bcI),\bcX_i \right\rangle^2\right)^2 \le CC_1^2\sigmax^4 {\frac{\dmax\log\dmax}{n}} \frac{m\mu^{m-1} r^*\dmax }{\rmin (d^*)^2} \norm{\cP_{\widehat \TT }(\bcI)}_{\tF}^4\\
        &\E\left(\left\langle \bcT^{\vartriangle},\bcX_i \right\rangle^2\left\langle \cP_{\widehat \TT }(\bcI),\bcX_i \right\rangle^2\right)^2 \le CC_1^2\sigmax^4 \left({\frac{\dmax\log\dmax}{n}}\right)^2 \frac{m\mu^{m-1} r^*\dmax }{\rmin (d^*)^2} \norm{\cP_{\widehat \TT }(\bcI)}_{\tF}^4
    \end{aligned}
\end{equation*}
We now consider the second term in \eqref{eq:hetero-var-decomp}:
\begin{equation*}
    \begin{aligned}
        &\abs{\norm{\cP_{\widehat \TT }(\bcI)\odot\bcS  }_{\tF}^2 - \norm{\cP_{\TT }(\bcI)\odot\bcS  }_{\tF}^2} \\
        &\le \abs{\norm{\cP_{\widehat \TT }(\bcI)\odot\bcS  }_{\tF} - \norm{\cP_{ \TT }(\bcI)\odot\bcS  }_{\tF}}\cdot\abs{\norm{\cP_{\widehat \TT }(\bcI)\odot\bcS  }_{\tF} + \norm{\cP_{ \TT }(\bcI)\odot\bcS  }_{\tF}}\\
        &\le 2 \norm{\cP_{ \TT }(\bcI)\odot\bcS  }_{\tF}^2\left(1+\frac{\norm{(\cP_{ \TT }(\bcI) -\cP_{ \widehat\TT }(\bcI))\odot\bcS  }_{\tF}}{\norm{\cP_{ \TT }(\bcI)\odot\bcS  }_{\tF}}\right) \frac{\norm{(\cP_{ \TT }(\bcI) -\cP_{ \widehat\TT }(\bcI))\odot\bcS  }_{\tF}}{\norm{\cP_{ \TT }(\bcI)\odot\bcS  }_{\tF}}
    \end{aligned}
\end{equation*}
Using \eqref{eq:PT-diff-var}, we can derive that 
 \begin{equation*}
        \frac{\norm{(\cP_{ \TT }(\bcI) -\cP_{ \widehat\TT }(\bcI))\odot\bcS  }_{\tF}}{\norm{\cP_{ \TT }(\bcI)\odot\bcS  }_{\tF}}\le \kappa_{\sigma}\frac{\norm{\cP_{\widehat \TT }(\bcI) - \cP_{\TT }(\bcI) }_\tF}{ \norm{\cP_{\TT }(\bcI) }_\tF } \le C\frac{\norm{\bcI}_{\ell_1}}{\norm{\bcI}_{\tF}}\frac{C_1  m^2\kappa_0\kappa_{\sigma} \sigma  }{\alpha_I \lambda_{\min}}\sqrt{\frac{(2\mu)^m r^*d^*\dmax \log \dmax }{n}}\le\frac{1}{8},
    \end{equation*}
which gives:
\begin{equation}\label{eq:hetero-PT-diff-var}
   \abs{\norm{\cP_{\widehat \TT }(\bcI)\odot\bcS  }_{\tF}^2 - \norm{\cP_{\TT }(\bcI)\odot\bcS  }_{\tF}^2} \le C \norm{\cP_{ \TT }(\bcI)\odot\bcS  }_{\tF}^2 \frac{\norm{\bcI}_{\ell_1}}{\norm{\bcI}_{\tF}}\frac{C_1  m^2\kappa_0\kappa_{\sigma} \sigma  }{\alpha_I \lambda_{\min}}\sqrt{\frac{(2\mu)^m r^*d^*\dmax \log \dmax }{n}}
\end{equation}
Combining \eqref{eq:hetero-var-decomp},\eqref{eq:hetero-var-conv}, \eqref{eq:hetero-var-sigma}, \eqref{eq:hetero-var-rem}, and \eqref{eq:hetero-PT-diff-var}, it is clear that
\begin{equation}\label{eq:hetero-var-diff-ratio}
\begin{aligned}
        \abs{\frac{\widehat s^2(\bcI)}{\norm{\cP_{\TT }(\bcI)\odot\bcS  }_{\tF}^2}-1}& \le C\frac{\norm{\bcI}_{\ell_1}}{\norm{\bcI}_{\tF}}\frac{C_1  m^2\kappa_0\kappa_{\sigma} \sigma  }{\alpha_I \lambda_{\min}}\sqrt{\frac{(2\mu)^m r^*d^*\dmax \log \dmax }{n}} + C C_1\kappa_{\sigma}^2\frac{\norm{\cP_{\widehat\TT }(\bcI)  }_{\tF}^2}{\norm{\cP_{\TT }(\bcI)  }_{\tF}^2}\sqrt{\frac{m^5\mu^{m-1} r^*\dmax\log^4\dmax }{\rmin n }} \\
        & \le C C_1\kappa_{\sigma}^2\sqrt{\frac{m^5\mu^{m-1} r^*\dmax\log^4\dmax }{\rmin n }}  + C\frac{\norm{\bcI}_{\ell_1}}{\norm{\bcI}_{\tF}}\frac{C_1  m^2\kappa_0\kappa_{\sigma} \sigma  }{\alpha_I \lambda_{\min}}\sqrt{\frac{(2\mu)^m r^*d^*\dmax \log \dmax }{n}} 
\end{aligned}
\end{equation}
Moreover, we have
\begin{equation*}
    \abs{\frac{\frE_1}{ \norm{\cP_{\TT }(\bcI)\odot\bcS  }_{\tF} \sqrt{d^*/n}}} \le C\kappa_{\sigma}\sqrt{m\log\dmax},
\end{equation*}
with probability at least $1-\dmax^{-3m}$ by Bernstein inequality. Plugging in this result, together with \eqref{eq:hetero-core-clt}  \eqref{eq:hetero-align}, \eqref{eq:hetero-var-diff-ratio} into the decomposition \eqref{eq:hatW-decomp} , we have
\begin{equation*}
\begin{aligned}
        &\max_{t\in\R}\abs{\bbP\left( \widehat W^{\sfh}_{\test}(\bcI)\le t\right)- \Phi(t)}  \\
        &\le   C C_1\kappa_{\sigma}^3\sqrt{\frac{m^6\mu^{m} r^*\dmax\log^5\dmax }{n }}  + C \frac{ C_1^2\kappa_{\sigma}^2\sigmax  \norm{\bcI}_{\ell_1} }{\alpha_I \lambda_{\min}  \norm{\bcI}_\tF }\sqrt{\frac{ m^5 (2\mu)^{3 m} (r^*)^{3 }\dmax d^*  \log^2 \dmax }{\rmin^2 n } }.
\end{aligned}
\end{equation*}
Thus, we finish the proof.
\subsection{Proof of Proposition \ref{prop:corr-ub}}
\begin{proof}
  according to the projection onto the tangent space in \eqref{eq:true-TT-proj} and \eqref{eq:mode-j-tangent}, we have

  \begin{equation*}
\begin{aligned}
      \abs{\rho(\bcI_1,\bcI_2)} &= \frac{\abs{\left\langle\cP_{ \TT }(\bcI_1),\cP_{ \TT }(\bcI_2)\right\rangle}}{\norm{\cP_{ \TT }(\bcI_1)}_{\tF} \norm{\cP_{ \TT }(\bcI_2)}_{\tF} } \\
      & \le \frac{d^*}{\dmax\alpha_I^2\norm{\bcI_1}_{\tF}\norm{\bcI_2}_{\tF} }\left( \abs{\left\langle\bcI_1 \cdot \left( \cP_{{\bU}_{1}},\dots,\cP_{{\bU}_{m}}\right), \bcI_2 \cdot \left( \cP_{{\bU}_{1}},\dots,\cP_{{\bU}_{m}}\right)\right\rangle} \right.\\
      & \quad \left.+ \sum_{j=1}^{m} \abs{\left\langle \bcC \cdot\left(\bU_1,\dots, \bW_j(\bcI_1), \dots, \bU_m\right), \bcC \cdot\left(\bU_1,\dots, \bW_j(\bcI_1), \dots, \bU_m\right)\right\rangle}\right)\\
      &\le \frac{d^*}{\dmax\alpha_I^2\norm{\bcI_1}_{\tF}\norm{\bcI_2}_{\tF} } \left( 2\frac{\mu^m r^* \norm{\bcI_1}_{\ell_1}\norm{\bcI_2}_{\ell_1}}{d^*} + \sum_{j=1}^{m}\abs{ \left\langle \cM_j(\bcI_1)^\top \cM_j(\bcI_2), \cP_{\bH_j}\right\rangle}\right) \\
      &\le \frac{2\mu^m r^* \norm{\bcI_1}_{\ell_1}\norm{\bcI_2}_{\ell_1}}{\dmax\alpha_I^2\norm{\bcI_1}_{\tF}\norm{\bcI_2}_{\tF} } + \frac{d^*\sum_{j=1}^{m}\abs{ \left\langle \cM_j(\bcI_1)^\top \cM_j(\bcI_2), \cP_{\bH_j}\right\rangle}}{\dmax\alpha_I^2\norm{\bcI_1}_{\tF}\norm{\bcI_2}_{\tF} }.
\end{aligned}
  \end{equation*}
  Here the first inequality is because of the alignment condition and the second inequality is because of the incoherent condition and \eqref{eq:mode-j-tangent}.
\end{proof}
\section{Useful Technical Tools}

\begin{Lemma}[Core estimation error]\label{lemma:diff-core-init} Suppose $\widehat{\bcT}_0$ is a rank-$\br$ estimation of $\bcT$ with error $\norm{\widehat{\bcT}_0 - \bcT}_{\tF}\le\delta\le 2^{-\frac{5}{2}}/m \cdot \lambda_{\min}  $. Then, There exists a decomposition of $\widehat{\bcT}_0$, denoted by  $\widehat{\bcT}_0 = \widehat{\bcC}_0\times_{j=1}^m\widehat{\bU}_{0,j}$, such that 
\begin{equation*}
    \norm{\widehat{\bU}_{0,j}-\bU_{j}}_2\le\norm{\widehat{\bU}_{0,j}-\bU_{j}}_{\tF}\le \frac{2^{\frac{3}{2}} \delta }{ \lambda_{\min}},
\end{equation*}
and correspondingly, 
\begin{equation*}
    \norm{\widehat{\bcC}_0-\bcC}_{\tF} \le 2^{5} \sqrt{\rmin } m \kappa_0 \delta.
\end{equation*}

\end{Lemma}
\begin{proof}
The first claim is a direct consequence of the generalized Davis–Kahan theorem (e.g., Theorem 2 in \cite{yu2015useful}) by simply selecting ${\widehat{\bU}_{0,j}} = \widehat{\bU}_{0,j} \widehat{\bO}_{0,j}$ after rotation. For the core estimation error, we notice that 
\begin{equation*}
\begin{aligned}
        &\norm{\widehat{\bcC}_0 - \bcC }_{\tF}  = \norm{\cM_{j'}(\widehat{\bcC}_0 - \bcC) }_{\tF} = \norm{ \widehat{\bU}_{0,j'}^\top \widehat\bT_{j'}\otimes_{k\neq j'}\widehat{\bU}_{0,k} - \bU_{j'}^\top \bT_{j'}\otimes_{k\neq j'}{\bU}_{k}   }_{\tF}  \\
        & \le \sqrt{2\rmin}\norm{ \widehat{\bU}_{0,j'}^\top \widehat\bT_{j'}\otimes_{k\neq j'}\widehat{\bU}_{0,k} - \bU_{j'}^\top \bT_{j'}\otimes_{k\neq j'}{\bU}_{k}   }_{2} \\
        &\le  \sqrt{2\rmin}\left(\norm{ \widehat{\bU}_{0,j'}^\top (\widehat\bT_{j'}-\bT_j)\otimes_{k\neq j'}\widehat{\bU}_{0,k}   }_{2}+ \norm{ \widehat{\bU}_{0,j'}^\top \bT_{j'}\otimes_{k\neq j'}\widehat{\bU}_{0,k} - \bU_{j'}^\top \bT_{j'}\otimes_{k\neq j'}{\bU}_{k} }_{2}\right) \\
        & \le \sqrt{2\rmin}\left( \delta+ \norm{ \widehat{\bU}_{0,j'}^\top \bT_{j'}\otimes_{k\neq j'}\widehat{\bU}_{0,k} - \bU_{j'}^\top \bT_{j'}\otimes_{k\neq j'}{\bU}_{k} }_{2} \right)
\end{aligned}
\end{equation*}
where $j'=\argmin_{j\in[m]}{r_j}$. By assuming that $2^{3/2}m\delta/\lambda_{\min}\le \frac{1}{2}$, the expansion of the error $\norm{\otimes_{k\neq j'}\widehat{\bU}_{k}  - \otimes_{k\neq j'}{\bU}_{k} }_2$ will lead to the bound:
\begin{equation*}
    \norm{\otimes_{k\neq j'}\widehat{\bU}_{k}  - \otimes_{k\neq j'}{\bU}_{k} }_2 \le 2^{5/2} m\delta/\lambda_{\min}.
\end{equation*}
Thus, we have
\begin{equation*}
    \norm{ \widehat{\bU}_{0,j'}^\top \bT_{j'}\otimes_{k\neq j'}\widehat{\bU}_{0,k} - \bU_{j'}^\top \bT_{j'}\otimes_{k\neq j'}{\bU}_{k} }_{2} \le 2^{7/2} m \kappa_0 \delta,
\end{equation*}
and consequently,
\begin{equation*}
    \norm{\widehat{\bcC}_0 - \bcC }_{\tF}  \le  \sqrt{2\rmin}(\delta + 2^{7/2} m \kappa_0 \delta )\le 2^{5} \sqrt{\rmin } m \kappa_0 \delta.
\end{equation*}

\end{proof}

\begin{Lemma}[$\ell_{2,\infty}$ bound of projection onto tangent space]\label{lemma:diff-tangent-proj}
    Under the initialization error in Lemma \ref{lemma:l2inf-init}, suppose we have the SNR
    \begin{equation*}
    \frac{\lambda_{\min} }{\sigma} \ge C_{\mathsf{gap} }\kappa_0 m \sqrt{\frac{d^* \dmax \log \dmax }{n} }.
\end{equation*}
    Then, for any canonical basis  tensor  $\bomega=\be_{1,k_1}\circ\cdots\circ\be_{m,k_m}$, the difference can be controlled by:
    \begin{equation*}
        \norm{\cP_{\widehat \TT }(\bomega) - \cP_{\TT }(\bomega) }_\tF \le \frac{C C_1  m^2\kappa_0 \sigma  }{\lambda_{\min}}\sqrt{\frac{(2\mu)^m r^*\dmax^2 \log \dmax }{n}}. 
    \end{equation*}
\end{Lemma}

\begin{proof}
Since the projection onto tangent space $\cP_{\TT }$ is given by
\begin{equation*}
\begin{aligned}
        \cP_{\TT}(\bcI) & = \bcI \cdot \left( \cP_{{\bU}_{1}},\dots,\cP_{{\bU}_{m}}\right) + \sum_{j=1}^{m} \bcC \cdot\left(\bU_1,\dots, \bW_j, \dots, \bU_m\right), \\
\end{aligned}
\end{equation*}
where, according to \eqref{eq:mode-j-tangent}, for each $\bW_j$ we have
\begin{equation*}
    \cM_j\left(\bcC \cdot\left(\bU_1,\dots, \bW_j, \dots, \bU_m\right) \right) = \cP_{\bU_j}^\perp \cM_j(\bcI) \left( \otimes_{k\neq j}\bU_{k} \right)\cP_{\bV_j} \left( \otimes_{k\neq j}\bU_{k}^\top \right)
\end{equation*}

\begin{equation}
     \cP_{\widehat \TT }(\bcI) = \bcI \cdot \left( \cP_{\widehat{\bU}_{1,0}},\dots,\cP_{\widehat{\bU}_{m,0}}\right) + \sum_{j=1}^{m} \widehat{\bcC}_0 \cdot\left(\widehat{\bU}_{1,0},\dots, \widehat{\bW}_{j}, \dots, \widehat{\bU}_{m,0}\right),
\end{equation}
we can decouple the error as 
    \begin{equation}\label{eq:tangent-diff-decomp}
      \begin{aligned}
& \norm{\cP_{\widehat \TT }(\bomega) - \cP_{\TT }(\bomega) }_\tF \le  \norm{\bomega\times_{j=1}^m \cP_{\bU_j} - \bomega\times_{j=1}^m \cP_{\widehat\bU_{j,0}} }_{\tF}   \\
 & + \sum_{j=1}^{m} \norm{\cP_{\widehat \bU_{j,0}}^\perp \cM_j(\bomega) \left( \otimes_{k\neq j}\widehat \bU_{k,0} \right)\cP_{\widehat  \bV_j} \left( \otimes_{k\neq j}\widehat \bU_{k,0}^\top \right)-\cP_{\bU_j}^\perp \cM_j(\bomega) \left( \otimes_{k\neq j}\bU_{k} \right)\cP_{\bV_j} \left( \otimes_{k\neq j}\bU_{k}^\top \right) }_{\tF}
      \end{aligned}
    \end{equation}
For the first term at the RHS of \eqref{eq:tangent-diff-decomp}, we have
\begin{equation}\label{eq:tangent-diff-recur}
    \begin{aligned}
         &\norm{\cP_{\widehat \TT }(\bomega) - \cP_{\TT }(\bomega) }_\tF =  \norm{\otimes_{j=1}^m \cP_{\bU_{j}}\be_{j,k_j} - \otimes_{j=1}^m \cP_{\widehat\bU_{j,0}}\be_{j,k_j}  }_2 \\
         &= \norm{\otimes_{1} (\cP_{\widehat\bU_{1,0}}-\cP_{\bU_{1}})\be_{1,k_1}\otimes_{j=2}^m \cP_{\bU_{j}}\be_{j,k_j}  }_2 + \norm{\otimes_{1} \cP_{\widehat\bU_{1,0}}\be_{1,k_1} \left(\otimes_{j=2}^m \cP_{\widehat\bU_{j,0}}\be_{j,k_j} -\otimes_{j=2}^m \cP_{\bU_{j}}\be_{j,k_j}\right) }_2 \\
         & = \norm{\otimes_{1} (\cP_{\widehat\bU_{1,0}}-\cP_{\bU_{1}})\be_{1,k_1}\otimes_{j=2}^m \cP_{\bU_{j}}\be_{j,k_j}  }_2 + \norm{\otimes_{1} \cP_{\widehat\bU_{1,0}}\be_{1,k_1} \otimes_{2} (\cP_{\widehat\bU_{2,0}}-\cP_{\bU_{2}})\be_{2,k_2} \left(\otimes_{j=3}^m  \cP_{\bU_{j}}\be_{j,k_j}\right) }_2  \\
         & +  \norm{\otimes_{1} \cP_{\widehat\bU_{1,0}}\be_{1,k_1} \otimes_{2} \cP_{\widehat\bU_{2,0}}\be_{2,k_2} \left(\otimes_{j=3}^m \cP_{\widehat\bU_{j,0}}\be_{j,k_j} -\otimes_{j=3}^m \cP_{\bU_{j}}\be_{j,k_j}\right) }_2 \\
         &\cdots \\
         & \le  \frac{8 m C_1\sigma  }{\lambda_{\min}}\sqrt{\frac{(2\mu)^m r^*}{d^*} }\sqrt{\frac{d^*\dmax \log \dmax }{n}},
    \end{aligned}
\end{equation}
where we use the fact in Lemma \ref{lemma:l2inf-init} that
    \begin{equation*}
   \norm{ \cP_{\widehat{\bU}_{j,0}} - \cP_{\bU_j} }_{2,\infty} \le \frac{8 C_1\sigma  }{\lambda_{\min}}\sqrt{\frac{\mu r_j}{d_j} }\sqrt{\frac{d^*\dmax \log \dmax }{n}}.
\end{equation*}
We now consider other terms at the RHS of \eqref{eq:tangent-diff-decomp}. To handle them, we need to fix the decomposition $\widehat{\bcT}_0 = \widehat{\bcC}_0\times_{j=1}^m\widehat{\bU}_{0,j}$. We invoke Lemma \ref{lemma:init-rotation} and let the decomposition be the one described in Lemma \ref{lemma:init-rotation}. This indicates the bound:
    \begin{equation*}
   \norm{ \widehat{\bU}_{j,0} - \bU_j }_{2,\infty} \le \frac{C C_1\sigma  }{\lambda_{\min}}\sqrt{\frac{\mu r_j}{d_j} }\sqrt{\frac{d^*\dmax \log \dmax }{n}}.
\end{equation*}
For each $j$, we let $\frG_{j,1}:= \cP_{\bU_j}^\perp \cM_j(\bomega) \left( \otimes_{k\neq j}\bU_{k} \right) $, and $\frG_{j,2}:=\cP_{\bV_j} \left( \otimes_{k\neq j}\bU_{k}^\top \right)$, with the corresponding estimated version be $\widehat\frG_{j,1}$ and $\widehat\frG_{j,2}$. Then, for each $j$, we have
\begin{equation*}
\begin{aligned}
        &\norm{\cP_{\widehat \bU_{j,0}}^\perp \cM_j(\bomega) \left( \otimes_{k\neq j}\widehat \bU_{k,0} \right)\cP_{\widehat  \bV_j} \left( \otimes_{k\neq j}\widehat \bU_{k,0}^\top \right)-\cP_{\bU_j}^\perp \cM_j(\bomega) \left( \otimes_{k\neq j}\bU_{k} \right)\cP_{\bV_j} \left( \otimes_{k\neq j}\bU_{k}^\top \right) }_{\tF} \\
        & \le\sqrt{2}\norm{\widehat\frG_{j,1} \widehat\frG_{j,2} -\frG_{j,1}\frG_{j,2} }_2\\
        & \le \sqrt{2}\left(\norm{(\widehat\frG_{j,1} -\frG_{j,1})\frG_{j,2} }_2 + \norm{\frG_{j,1}(\widehat\frG_{j,2} -\frG_{j,2}) }_2 + \norm{(\widehat\frG_{j,1} -\frG_{j,1})(\widehat\frG_{j,2} -\frG_{j,2}) }_2\right).
\end{aligned}
\end{equation*}
Here, $ \widehat\frG_{j,1} -\frG_{j,1}$ can be controlled using the same technique as \eqref{eq:tangent-diff-recur}, which gives
\begin{equation*}
    \norm{ \widehat\frG_{j,1} -\frG_{j,1}}_2 \le  \frac{C m C_1\sigma  }{\lambda_{\min}}\sqrt{\frac{(2\mu)^m r^*}{\rmin d^*} }\sqrt{\frac{d^*\dmax^2 \log \dmax }{n}}.
\end{equation*}
For the $\frG_{j,2}$, We need to study $\cP_{\widehat\bV_{j,0}}-\cP_{\bV_j}$. We notice that $\cM_j(\bcC)^\top = \bV_j \bLambda_j$. Moreover, according to Lemma \ref{lemma:diff-core-init}, we can control the error for estimating the core with $\norm{\widehat{\bcC}_0-\bcC}_{\tF} \le 2^{5} \sqrt{\rmin } m \kappa_0 C_1\sigma \sqrt{\frac{d^*\dmax\log\dmax }{n}}$. Using the generalized Davis–Kahan theorem, we have
\begin{equation*}
    \norm{\cP_{\widehat\bV_{j,0}}-\cP_{\bV_j}}_2\le C \sqrt{\rmin } m \kappa_0 \frac{  C_1 \sigma}{\lambda_{\min}} \sqrt{\frac{d^*\dmax\log\dmax }{n}}.
\end{equation*}
Thus, with the given SNR, we have $\norm{\widehat\frG_{j,2} -\frG_{j,2}}_2\le C \sqrt{\rmin } m \kappa_0 \frac{  C_1 \sigma}{\lambda_{\min}} \sqrt{\frac{d^*\dmax\log\dmax }{n}}$, where we invoke Lemma \ref{lemma:init-rotation} again on the control of $\norm{\otimes_{k\neq j}\widehat\bU_{k}^\top - \otimes_{k\neq j}\bU_{k}^\top}_2$.  Combining the bounds for $  \norm{ \widehat\frG_{j,1} -\frG_{j,1}}_2$ and $  \norm{ \widehat\frG_{j,2} -\frG_{j,2}}_2$, we conclude that
\begin{equation*}
    \norm{\widehat\frG_{j,1} \widehat\frG_{j,2} -\frG_{j,1}\frG_{j,2} }_2 \le \frac{C  m\kappa_0 C_1\sigma  }{\lambda_{\min}}\sqrt{\frac{(2\mu)^m r^*}{d^*} }\sqrt{\frac{d^*\dmax^2 \log \dmax }{n}}.
\end{equation*}
Together with \eqref{eq:tangent-diff-decomp}, \eqref{eq:tangent-diff-recur}, we can derive the desired bound.



\end{proof}

\begin{Lemma}[Matrix Bernstein inequality with sub-exponential noises]\label{lemma:sub-exp-bern}
Consider independent random self-adjoint matrix $\bX_i\in \R^{d\times d}$ with bounded norm: $\norm{\bX_i}_2\le R$. Suppose $\xi_i$ are another series of sub-exponential random variables where each $\xi_i$ only depends on $\bX_i$, and conditional on any $\bX_i$, we have $\norm{\xi_i|\bX_i}_{\psi_1}\le \sigmax$ and $\E [\xi_i|\bX_i] =0$. Denote $\sigma^2_{\bX}:=\left\|\sum_i^n \E\bX_i^2\right\|_2$.
Then, their exist a numerical constant $C>0$ such that the following inequality holds for all $t \geq 0$:
$$
\mathbb{P}\left\{\norm{\sum_i^n \xi_i\mathbf{X}_i}_2 \geq t\right\} \leq d \cdot \exp \left(\frac{-C t^2 }{\sigma^2_{\bX}\sigmax^2 +R\sigmax t}\right)
$$
    
\end{Lemma}
\begin{proof}
    This is a direct application of Theorem 6.2 in \cite{tropp2012user}. We can check that, according to the definition of the sub-exponential norm, each element in the series
    satisfies the Bernstein condition: 
    $$\E \abs{\xi_i}^k\bX^{k}_i = \E \left[\E[\abs{\xi_i}^k\bX^{k}_i|\bX_i]\right]  \preceq \frac{k!}{2} \cdot (C \sigmax R)^{k-2} \sigmax^2\E\bX_i^2 , \text{ for } k\ge 3, i\in[n].$$ 
    Summing up all the independent elements we have
    \begin{equation*}
        \E \sum_{i=1}^n\abs{\xi_i}^k\bX^{k}_i \preceq \frac{k!}{2} \cdot (C \sigmax R)^{k-2} \sigmax^2\sigma^2_{\bX},
    \end{equation*}
    which matches the condition of Theorem 6.2 in \cite{tropp2012user}. Thus we prove the claim.
\end{proof}

\section{Proofs of Technical Lemmas}

\subsection{Proof of Lemma \ref{lemma:cover-init}}\label{apx:sec:proof-metric-ety}
\begin{proof}
    According to Lemma \ref{lemma:l2inf-init}, by setting $C_1$ in  Lemma \ref{lemma:l2inf-init} as $C_1=c_0\cdot\sqrt{\frac{n}{\dmax \log \dmax}}/\sigma$, for any Tucker decomposition $\widehat{\bcT}_0 =\widetilde{\bcC}_{0}\times_{j=1}^m \widetilde{\bU}_{j,0}  $, we have the bound
        \begin{equation*}
   \norm{ \cP_{\widetilde{\bU}_{j,0}} - \cP_{\bU_j} }_{2,\infty} \le \frac{8 c_0  }{\lambda_{\min}}\sqrt{\frac{\mu r_j}{d_j} }\sqrt{d^*},
\end{equation*}
and also, 
    \begin{equation*}
   \norm{ \cP_{\widetilde{\bU}_{j,0}} - \cP_{\bU_j} }_{2} \le \frac{8 c_0  }{\lambda_{\min}}\sqrt{d^*}.
\end{equation*}
for any $\widehat{\bcT}_0\in \mathbb{T}_{\init}(c_0)$. However, such a decomposition is arbitrary. To fix the decomposition and ease the discussion, we introduce the following lemma:

\begin{Lemma}\label{lemma:init-rotation}
    For every $\widehat{\bcT}_0\in \mathbb{T}_{\init}(c_0)$, there always exists a decomposition $\left(\widehat{\bcC}_{0},\widehat{\bU}_{1,0},\dots, \widehat{\bU}_{m,0} \right)$ such that $\widehat{\bcT}_0 = \widehat{\bcC}_{0}\cdot \left(\widehat{\bU}_{1,0},\dots,\widehat{\bU}_{m,0} \right)$, and 
\begin{equation*}
\begin{gathered}
        \norm{\widehat{\bU}_{j,0}- \bU_j}_{2,\infty}\le  \frac{8 \sqrt{3} c_0  }{\lambda_{\min}}\sqrt{\frac{\mu r_j}{d_j} }\sqrt{d^*}, \  \norm{\widehat{\bU}_{j,0}- \bU_j}_{2} \le \frac{8 \sqrt{2} c_0  }{\lambda_{\min}}\sqrt{d^*}, \\
        \max_{j}\norm{\cM_j\left(\widehat{\bcC}_{0} -\bcC\right) }_2 \le 2^{m+4}\kappa_0\sqrt{d^*}c_0 .
\end{gathered}
\end{equation*}
\end{Lemma}
Now we just use the decomposition in Lemma \ref{lemma:init-rotation} as the unique Tucker decomposition of elements in $\mathbb{T}_{\init}(c_0)$. According to this decomposition, we define the subsets w.r.t. $\widehat{\bU}_{j,0}$ and $\widehat{\bcC}_{0}$:
\begin{equation*}
    \begin{gathered}
        \bbS(\bU_j)= \left\{\widehat{\bU}_{j,0}\in \O^{d_j\times r_j} :\  \norm{\widehat{\bU}_{j,0} - \bU_j}_{2,\infty}\le \frac{8 \sqrt{3} c_0  }{\lambda_{\min}}\sqrt{\frac{\mu r_j}{d_j} }\sqrt{d^*} \right\}\\
        \bbS(\bcC) = \left\{\widehat{\bcC}_{0}\in \R^{\br}: 
        \ \norm{\widehat{\bcC}_{0} - \bcC}_\tF\le 2^{m+4}\rmax \kappa_0\sqrt{d^*}c_0  \right\}.
    \end{gathered}
\end{equation*}
It is clear that for any $\widehat{\bcT}_0\in \mathbb{T}_{\init}(c_0)$, the corresponding $\left(\widehat{\bcC}_{0},\widehat{\bU}_{1,0},\dots, \widehat{\bU}_{m,0} \right)$ always fall within the subsets $ \bbS(\bcC) $  and $\{\bbS(\bU_j)\}_{j=1}^{m}$. Notice that, for any $\varepsilon >0$,  if we control the error in each subset  with 
\begin{equation*}
    \begin{gathered}
        \norm{\widehat{\bU}_{j,0} - \bU_j}_{2,\infty}\le \frac{\sqrt{d_{-j}}}{2^{m+1}\kappa_0\lambda_{\min}\sqrt{(2\mu)^{m-1}r_{-j} }} \varepsilon := \widetilde{c}_{\bU_j}\varepsilon \\
        \norm{\widehat{\bcC}_{0} - \bcC}_\tF\le \frac{\sqrt{d^*}}{2^{m+1}\sqrt{(2\mu)^m r^*} } \varepsilon := \widetilde{c}_{\bcC}\varepsilon,
    \end{gathered}
\end{equation*}
then we shall have the control of $\norm{\widehat{\bcT}_0- \bcT}_{\linf}$ with the bound
\begin{equation*}
    \begin{aligned}
        \abs{\left[\widehat{\bcT}_0- \bcT\right]_{\omega}} & = \abs{ \widehat{\bcC}_{0}\times_{j=1}^m\widehat{\bU}_{j,0}^\top\be_{j,i}-\bcC\times_{j=1}^m\bU_j^\top\be_{j,i} } \\
        &\le 2^m \sqrt{\frac{(2\mu)^{m-1}r_{-j}}{d_{-j}}}\lambda_{\max} \norm{\widehat{\bU}_{j,0} - \bU_j}_{2,\infty} +2^m\sqrt{\frac{(2\mu)^m r^*}{d^*}} \norm{\widehat{\bcC}_{0} - \bcC}_\tF  \\
        & \le \varepsilon.
    \end{aligned}
\end{equation*}
This indicates that the covering number on $\mathbb{T}_{\init}(c_0)$ can be upper bounded by the covering numbers on $ \bbS(\bcC) $  and $\{\bbS(\bU_j)\}_{j=1}^{m}$, with the relationship:
\begin{equation}\label{eq:cover-init}
     \cN(\mathbb{T}_{\init}(c_0),\norm{\cdot}_{\linf},\varepsilon) \le \prod_{j} \cN(\bbS(\bU_j),\norm{\cdot}_{2,\infty},\widetilde{c}_{\bU_j}\varepsilon) \cdot \cN(\bbS(\bcC),\norm{\cdot}_{\tF},\widetilde{c}_{\bcC}\varepsilon).
\end{equation}
For each $j$, we can relax the ball with $\ell_{2,\infty}$ radius to the ball $\ell_2$ radius on $d_j$ rows, which gives that $\bbB(\bU_j,\norm{\cdot}_{2,\infty},c)=\times_{i=1}^{d_j}\bbB(\bU_{j,i},\ell_2,c)$. This argument leads to the covering number bound:
\begin{equation}\label{eq:cover-Uj}
\cN(\bbS(\bU_j),\norm{\cdot}_{2,\infty},\widetilde{c}_{\bU_j}\varepsilon) \le \left(1+ 2^{m+5}\kappa_0\sqrt{(2\mu)^{m} r^*}\frac{c_0}{\varepsilon}\right)^{d_j r_j}.
\end{equation}
For the $\bbS(\bcC)$, we also have the covering number:
\begin{equation}\label{eq:cover-S}
    \cN(\bbS(\bcC),\norm{\cdot}_{\tF},\widetilde{c}_{\bcC}\varepsilon) \le \left(1+ 2^{2m+5}\kappa_0\sqrt{(2\mu)^m r^*\rmax^2} \frac{c_0}{\varepsilon}\right)^{r^*}
\end{equation}
Inserting \eqref{eq:cover-Uj}, \eqref{eq:cover-S} into \eqref{eq:cover-init}, we prove the desired argument.
\end{proof}

\subsection{Proof of Lemma \ref{lemma:init-rotation}}\label{apx:sec:proof-PU-to-U}
\begin{proof}

To show this, notice that, for any decomposition $\widehat{\bcT}_0 =\widetilde{\bcC}_{0}\times_{j=1}^m \widetilde{\bU}_{j,0}  $, we can always select $\widetilde{\bO}_j = \widetilde{Q}^{L}_{j}(\widetilde{Q}^{R}_{j})^\top$, where $(\widetilde{Q}^{L}, \widetilde{Q}^{R})=\SVD(\widetilde{\bU}_{j,0}^\top\bU_j )$ are the left and right singular subspaces of $\widetilde{\bU}_{j,0}^\top\bU_j$. By doing so, we can rotate the decomposition as $\widehat{\bU}_{j,0}=\widetilde{\bU}_{j,0} \widetilde{\bO}_j$, and $\widehat{\bcC}_{0} = \widetilde{\bcC}_{0} \times_{j=1}^m \widetilde{\bO}_j^\top$. It can be verified that 
\begin{equation*}
    \norm{\widehat{\bU}_{j,0}- \bU_j}_{2} = \norm{ \widetilde{\bU}_{j,0} \widetilde{\bO}_j - \bU_j}_2 \le \sqrt{2} \norm{ \cP_{\widetilde{\bU}_{j,0}} - \cP_{\bU_j} }_{2} \le \frac{8 \sqrt{2} c_0  }{\lambda_{\min}}\sqrt{d^*}.
\end{equation*}
by the C-S decomposition and the $\sin\Theta$ distance of singular spaces \citep{edelman1998geometry,horn2012matrix}. For the $\ell_{2,\infty}$-norm, we have 
\begin{equation}\label{eq:rotation-PU-U}
    \begin{aligned}
        & \norm{\be_{j,i}^\top\left(\widehat{\bU}_{j,0}- \bU_j\right)}_2^2 =\be_{j,i}^\top\left( \widetilde{\bU}_{j,0} \widetilde{\bO}_j- \bU_j\right)\cdot\left( \widetilde{\bU}_{j,0} \widetilde{\bO}_j- \bU_j\right)^\top\be_{j,i} \\
   = &  \be_{j,i}^\top\left( \cP_{\widetilde{\bU}_{j,0}} + \cP_{\bU_j} - \widetilde{\bU}_{j,0} \widetilde{\bO}_j \bU_j^\top - \bU_j\widetilde{\bO}_j^\top\widetilde{\bU}_{j,0}^\top\right)\be_{j,i}\\
      = &\norm{\be_{j,i}^\top\left(\cP_{\widetilde{\bU}_{j,0}}- \cP_{\bU_j}\right)}_2^2  +\be_{j,i}^\top\left(\widetilde{\bU}_{j,0}\widetilde{Q}^{L}_{j}\left(\cos \widetilde{\Theta}-I_{r_j}\right)(\widetilde{Q}^{R}_{j})^\top \bU_j^\top \right)\be_{j,i} \\
      \le & \frac{64 c_0^2  }{\lambda_{\min}^2}\frac{\mu r_j}{d_j}\cdot d^*+ \norm{\widetilde{\bU}_{j,0}}_{2,\infty} \norm{\bU_{j}}_{2,\infty} \cdot\norm{I_{r_j}- \cos \widetilde{\Theta}}_2 \\
      \le & \frac{64 c_0^2  }{\lambda_{\min}^2}\frac{\mu r_j}{d_j}\cdot d^* + \frac{4\mu r_j}{d_j}\norm{\sin \frac{\widetilde{\Theta}}{2}}_2^2 \le \frac{64 c_0^2  }{\lambda_{\min}^2}\frac{\mu r_j}{d_j}\cdot d^* + \frac{2\mu r_j}{d_j}\norm{ \cP_{\widetilde{\bU}_{j,0}} - \cP_{\bU_j}}_2^2\\
      \le &\frac{3\cdot 64 c_0^2  }{\lambda_{\min}^2}\frac{\mu r_j}{d_j} \cdot d^*,
    \end{aligned}
\end{equation}
where the $ \widetilde{\Theta}$ indicates the principal angles between $\widetilde{\bU}_{j,0}, \bU_j$. This gives the bound of $\norm{\widehat{\bU}_{j,0}- \bU_j}_{2,\infty}$. For the spectral norm of $\widehat{\bcC}_{0}- \bcC$ on each mode, we have $\widehat{\bcC}_{0}=\widehat{\bcT}_0 \times_{j=1}^m \widehat{\bU}_{j,0}^\top$, and further
\begin{equation*}
\begin{aligned}
        &\norm{\cM_j\left(\widehat{\bcC}_{0} -\bcC\right) }_2 =\norm{\cM_j\left(\widehat{\bcT}_0 \times_{j=1}^m \widehat{\bU}_{j,0}^\top -\bcT\times_{j=1}^m {\bU}_{j}^\top \right) }_2  \\
        = & \norm{\widehat{\bU}_{j,0}^\top\cM_j(\widehat{\bcT}_0 )\left(\otimes_{k\neq j} \widehat{\bU}_{k,0}^\top\right) 
    - 
    {\bU}_{j}^\top\cM_j({\bcT} )\left(\otimes_{k\neq j} {\bU}_{k}^\top\right)}_2 \\
    \le & 2^m\sqrt{d^*}c_0 + 2^m \frac{8 \sqrt{2} c_0  }{\lambda_{\min}}\sqrt{d^*} \cdot\lambda_{\max}
    \le 2^m(8 \sqrt{2}\kappa_0 +1)\sqrt{d^*}c_0 \le 2^{m+4}\kappa_0\sqrt{d^*}c_0. 
\end{aligned}
\end{equation*}
This ratio is $j$-free. Thus, we prove the claim.
\end{proof}


\subsection{Proof of Lemma \ref{lemma:l2inf-powerit}}\label{apx:sec:proof-U1-l2inf}
\begin{proof}
    Without loss of generality, we prove the statement under the case when $j=1$. The uniform bound can be easily obtained by combining the single bound for each $j\in[m]$. To prove the bound, we follow the idea in the proof of Lemma \ref{lemma:l2inf-init} and notice the following facts:
    \begin{equation}\label{eq:powerit-delta11-spec}
\begin{aligned}
        & \max\left\{ \norm{\cP_{\bU_1}^\perp{\Delta}_{1,1}\bU_1\bLambda_1^{-2}}_2,\norm{\bLambda_1^{-2}\bU_1^\top{\Delta}_{1,1}\cP_{\bU_1}^\perp}_2 \right\} \le \frac{ \norm{\left(\bE_{1,\rn} + \bE_{1,\init}\right)\left(\otimes_{j\neq 1} \widehat{\mathbf{U}}_{j,0} \right)}_2}{\lambda_{\min}} \\
        & \quad \quad \quad \quad + \frac{\norm{\left(\bE_{1,\rn} + \bE_{1,\init}\right)\left(\otimes_{j\neq 1} \widehat{\mathbf{U}}_{j,0} \right)}_2^2}{\lambda_{\min}^2} \le C C_1 \frac{\sigma}{\lambda_{\min}} \sqrt{\frac{ r^* m (2\mu)^{m-1} d_1  d^*\log \dmax }{ \rmin r_1 n}}  \\
        &\norm{\bLambda_1^{-2}\bU_1^\top{\Delta}_{1,1}\bU_1\bLambda_1^{-2}}_2 
         \le C C_1\frac{\sigma}{\lambda_{\min}} \sqrt{\frac{ m (2\mu)^{m-1} r^* d_1  d^*\log \dmax }{\rmin r_1 n}} \cdot \frac{1}{\lambda_{\min}^2} \\ 
         &\norm{\cP_{\bU_1}^\perp{\Delta}_{1,1}\cP_{\bU_1}^\perp}_2 
          \le C_1\frac{\sigma}{\lambda_{\min}} \sqrt{\frac{ m (2\mu)^{m-1} r^* d_1  d^*\log \dmax }{\rmin r_1 n}} \cdot \lambda_{\min}^2,
\end{aligned}
\end{equation}
which hold with probability at least $1-2 \dmax^{-3m}$. Here, we use the bound \eqref{eq:E-init-Ber} and \eqref{eq:E-rn-Ber}. This can be viewed as the variant of \eqref{eq:init-delta10-spec}.

We first consider the spectral norm bound: from \eqref{eq:U1-expansion}, \eqref{eq:U1-expansion-S}, and \eqref{eq:powerit-delta11-spec}, we know that for each $k$, the individual term in $\cS_{\widehat{\bU}_{1,1}}^{(k)}$ is given by
\begin{equation*}
    \norm{\mathfrak{P}_1^{-s_1} {\Delta}_{1,1}  \cdots {\Delta}_{1,1} \mathfrak{P}_1^{-s_{k+1}}}_2 \le \left(\frac{C C_1\sigma  }{\lambda_{\min}}\sqrt{\frac{ m (2\mu)^{m-1}r^* d_1  d^*\log \dmax }{\rmin r_1 n}}\right)^{k}.
\end{equation*}
Summing up all the possible series with $\mathbf{s}: s_1+\cdots+s_{k+1}=k$, for $k$ from $1$ to $\infty$ we have the spectral norm bound:
\begin{equation*}
    \norm{\cP_{\widehat{\bU}_{j,1}} - \cP_{\bU_j}}_{2} \le \frac{C C_1\sigma  }{\lambda_{\min}}\sqrt{\frac{ m (2\mu)^{m-1} r^*d_1  d^*\log \dmax }{\rmin r_1 n}}.
\end{equation*}
Then, we turn to focus on the $\ell_{2,\infty}$ bound. To this end, we need the high-probability bound of $\norm{\cP_{\bU_1}^\perp{\Delta}_{1,1}\cP_{\bU_1}^\perp}_{2,\infty}$ and $\norm{\cP_{\bU_1}^\perp{\Delta}_{1,1}\bU_1\bLambda_1^{-2}}_{2,\infty}$ just like \eqref{eq:init-delta10-2inf}, so that we can leverage the argument in \eqref{eq:l2inf-decomp1} and \eqref{eq:l2inf-decomp2} to control $ \norm{\cP_{\widehat{\bU}_{j,1}} - \cP_{\bU_j}}_{2,\infty}$. Noticing that 
\begin{equation}\label{eq:powerit-U-delta11-U}
  \begin{aligned}
        \norm{\cP_{\bU_1}^\perp{\Delta}_{1,1}\cP_{\bU_1}^\perp}_{2,\infty}  \le & \norm{\cP_{\bU_1}\mathfrak{B}_3\cP_{\bU_1}^\perp}_{2,\infty} + \norm{\mathfrak{B}_3\cP_{\bU_1}^\perp}_{2,\infty}\le \left(\norm{\cP_{\bU_1}\left(\bE_{1,\rn} + \bE_{1,\init}\right)\left(\otimes_{j\neq 1} \widehat{\mathbf{U}}_{j,0} \right)}_{2,\infty} \right. \\
   & \left. + \norm{\left(\bE_{1,\rn} + \bE_{1,\init}\right)\left(\otimes_{j\neq 1} \widehat{\mathbf{U}}_{j,0} \right)}_{2,\infty}\right)\cdot \norm{\left(\bE_{1,\rn} + \bE_{1,\init}\right)\left(\otimes_{j\neq 1} \widehat{\mathbf{U}}_{j,0} \right)}_2,
  \end{aligned}
\end{equation}
and also 
\begin{equation}\label{eq:powerit-U-delta11-lam}
    \begin{aligned}
    & \norm{\cP_{\bU_1}^\perp{\Delta}_{1,1}\bU_1\bLambda_1^{-2}}_{2,\infty}  \le 
    \norm{\cP_{\bU_1}^\top\left(\bE_{1,\rn} + \bE_{1,\init}\right)\left(\otimes_{j\neq 1} \widehat{\mathbf{U}}_{j,0} \right)\bT_1^\top \cP_{\bU_1}\bLambda_1^{-2} }_{2,\infty} \\
    & +\norm{\cP_{\bU_1}^\top\left(\bE_{1,\rn} + \bE_{1,\init}\right)\left(\otimes_{j\neq 1} \widehat{\mathbf{U}}_{j,0} \right)}_{2,\infty}\norm{\left(\bE_{1,\rn} + \bE_{1,\init}\right)\left(\otimes_{j\neq 1} \widehat{\mathbf{U}}_{j,0} \right) }_{2}\cdot \lambda_{\min}^{-2} \\
    & \le \lambda_{\min}^{-1}\norm{\cP_{\bU_1}^\top\left(\bE_{1,\rn} + \bE_{1,\init}\right)\left(\otimes_{j\neq 1} \widehat{\mathbf{U}}_{j,0} \right)}_{2,\infty}\left( 1+  \norm{\left(\bE_{1,\rn} + \bE_{1,\init}\right)\left(\otimes_{j\neq 1} \widehat{\mathbf{U}}_{j,0} \right) }_{2}\cdot \lambda_{\min}^{-1}  \right) \\
    & \le 2 \lambda_{\min}^{-1}\left(\norm{\cP_{\bU_1}\left(\bE_{1,\rn} + \bE_{1,\init}\right)\left(\otimes_{j\neq 1} \widehat{\mathbf{U}}_{j,0} \right)}_{2,\infty} + \norm{\left(\bE_{1,\rn} + \bE_{1,\init}\right)\left(\otimes_{j\neq 1} \widehat{\mathbf{U}}_{j,0} \right)}_{2,\infty}\right),
    \end{aligned}
\end{equation}
we can turn to focus on the term $\norm{\left(\bE_{1,\rn} + \bE_{1,\init}\right)\left(\otimes_{j\neq 1} \widehat{\mathbf{U}}_{j,0} \right)}_{2,\infty}$. The detailed bound of this term is given in the following Lemma \ref{lemma:E-init-rn-2inf}.
\begin{Lemma}\label{lemma:E-init-rn-2inf} Under the same assumptions as Theorem \ref{thm:lf-inference-popvar}, when the initialization is independent of $\left\{(\bcX_i,\xi_i)\right\}_{i=1}^n$,  we have the following inequalities hold with probability at least $1-2m \dmax^{-3m}$:
\begin{equation*}
    \begin{gathered}
        \norm{\bE_{1,\init}\left(\otimes_{j\neq 1} \widehat{\mathbf{U}}_{j,0} \right)}_{2,\infty} \le C C_1 \sigma \sqrt{\frac{ m (2\mu)^{m-1} r_{-1} \dmax d^*\log^2 \dmax }{n^2}} \\
        \norm{\bE_{1,\rn}\left(\otimes_{j\neq 1} \widehat{\mathbf{U}}_{j,0} \right)}_{2,\infty} \le C \sigma \sqrt{\frac{ m (2\mu)^{m-1} r_{-1} d^*\log \dmax }{n}}\le \sqrt{\frac{\mu r_1 }{d_1}} \cdot C \sigma \sqrt{\frac{ m (2\mu)^{m-2} r_{-1} d_1 d^*\log \dmax }{n}}.\\
    \end{gathered}
\end{equation*}
\end{Lemma}
Applying Lemma \ref{lemma:E-init-rn-2inf} to the aforementioned inequalities \eqref{eq:powerit-U-delta11-U} and \eqref{eq:powerit-U-delta11-lam} with the $\ell_2$ bound in \eqref{eq:E-init-Ber} and \eqref{eq:E-rn-Ber}, we have
\begin{equation}\label{eq:powerit-delta11-2inf}
\begin{aligned}
&\norm{\cP_{\bU_1}^\perp{\Delta}_{1,1}\cP_{\bU_1}^\perp}_{2,\infty}  \le  C C_1 \sqrt{\frac{\mu r_1}{d_1}}\frac{\sigma}{\lambda_{\min}} \sqrt{\frac{ m (2\mu)^{m-1} r_{-1} d^* d_1 \log \dmax }{ n}}\cdot \lambda_{\min}^2 \\
&\norm{\cP_{\bU_1}^\perp{\Delta}_{1,1}\bU_1\bLambda_1^{-2}}_{2,\infty}  \\
& \le 2 \lambda_{\min}^{-1}\left(\norm{\cP_{\bU_1}\left(\bE_{1,\rn} + \bE_{1,\init}\right)\left(\otimes_{j\neq 1} \widehat{\mathbf{U}}_{j,0} \right)}_{2,\infty} + \norm{\left(\bE_{1,\rn} + \bE_{1,\init}\right)\left(\otimes_{j\neq 1} \widehat{\mathbf{U}}_{j,0} \right)}_{2,\infty}\right) \\
& \le C C_1 \sqrt{\frac{\mu r_1}{d_1} }\frac{\sigma}{\lambda_{\min}} \sqrt{\frac{m (2\mu)^{m-1} r_{-1} d^* d_1 \log \dmax }{ n}}.
\end{aligned}  
\end{equation}
Examining the series in \eqref{eq:U1-expansion-S} in the cases when $s_1=0$ or $s_1\ge 1$ just like \eqref{eq:l2inf-decomp1} and \eqref{eq:l2inf-decomp2}, we conclude that 
\begin{equation*}
\begin{aligned}
        &\norm{\be_{k_1}^{\top}\mathfrak{P}_1^{-s_1} {\Delta}_{1,1}  \cdots {\Delta}_{1,1} \mathfrak{P}_1^{-s_{k+1}}}_2 \\
        &\le \sqrt{\frac{\mu r_1}{d_1} } \frac{C C_1 \sigma}{\lambda_{\min}} \sqrt{\frac{m (2\mu)^{m-1} r_{-1} d^* d_1 \log \dmax }{ n}} \left(\frac{C C_1\sigma}{\lambda_{\min}} \sqrt{\frac{ m (2\mu)^{m-1} d_1  d^*\log \dmax }{n}} \right)^{k-1}.
\end{aligned}
\end{equation*}
Thus, we know that $\norm{\be_{k_1}^\top \cS_{\widehat{\bU}_{1,1}}^{(k)}}_2$ is controlled by
\begin{equation*}
\begin{aligned}
        &\norm{\be_{k_1}^\top \cS_{\widehat{\bU}_{1,1}}^{(k)} }_2 \le \sum_{\mathbf{s}: s_1+\cdots+s_{k+1}=k} \sqrt{\frac{\mu r_1}{d_1} } \frac{C C_1 \sigma}{\lambda_{\min}} \sqrt{\frac{m (2\mu)^{m-1} r_{-1} d^* d_1 \log \dmax }{ n}} \left(\frac{C C_1\sigma}{\lambda_{\min}} \sqrt{\frac{ m (2\mu)^{m-1} d_1  d^*\log \dmax }{n}} \right)^{k-1}\\
    &\le \sqrt{\frac{\mu r_1}{d_1} } \frac{4C C_1 \sigma}{\lambda_{\min}} \sqrt{\frac{m (2\mu)^{m-1} r_{-1} d^* d_1 \log \dmax }{ n}} \left(\frac{4 C C_1\sigma}{\lambda_{\min}} \sqrt{\frac{ m (2\mu)^{m-1} d_1  d^*\log \dmax }{n}} \right)^{k-1}
\end{aligned}
\end{equation*}
Given the SNR condition, we can sum up all the series for $k$ from $1$ to $\infty$ and can derive the following error bound
\begin{equation*}
    \begin{aligned}
    \norm{\be_{k_1}^{\top}\left( \cP_{\widehat{\bU}_{1,1}} - \cP_{\bU_1}\right)}_2 \le \sum_{k\ge 1} \norm{\be_{k_1}^\top \cS_{\widehat{\bU}_{1,1}}^{(k)} }_2 \le  \sqrt{\frac{\mu r_1}{d_1} } \frac{C C_1 \sigma}{\lambda_{\min}} \sqrt{\frac{m (2\mu)^{m-1} r_{-1} d^* d_1 \log \dmax }{ n}} ,
    \end{aligned}
\end{equation*}
for all $k_1\in[d_1]$ with probability at least $1-3m\dmax^{-3m}$. Extending $j=1$ to other modes, we can prove the statement. 

With the proof of $\ell_{2,\infty}$ bound above, the $\linf$ bound can be easily derived following the same fashion. We take the first-order term for example. In $\cS_{\widehat{\bU}_{1,1}}^{(1)}$  we have
\begin{equation*}
\begin{aligned}
       \norm{\cS_{\widehat{\bU}_{1,1}}^{(1)}}_{\linf} & = \norm{\bU_1\bLambda_1^{-2}\bU_1^\top {\Delta}_{1,1}\cP_{\bU_1}^\perp + \cP_{\bU_1}^\perp {\Delta}_{1,1}\bU_1\bLambda_1^{-2}\bU_1^\top}_{\linf} \le 2\norm{\bU_1}_{2,\infty} \norm{\cP_{\bU_1}^\perp {\Delta}_{1,1}\bU_1 \bLambda_1^{-2} }_{2,\infty}  \\
       & \le C C_1 \frac{\mu r_1}{d_1} \frac{\sigma}{\lambda_{\min}} \sqrt{\frac{m (2\mu)^{m-1} r_{-1} d^* d_1 \log \dmax }{ n}}.
\end{aligned}
\end{equation*}
For $\cS_{\widehat{\bU}_{1,1}}^{(k)}$ with larger $k$, two sides of $\linf$ norm error term can be controlled by $\norm{\bU_1}_{2,\infty}$, or $\norm{\cP_{\bU_1}^\perp {\Delta}_{1,1}\bU_1 \bLambda_1^{-2} }_{2,\infty}$, or $\norm{\cP_{\bU_1}^\perp{\Delta}_{1,1}\cP_{\bU_1}^\perp}_{2,\infty}$. According to \eqref{eq:powerit-delta11-2inf}, all of these quantities can be correspondingly controlled. Thus, we have
\begin{equation*}
    \norm{\cP_{\widehat{\bU}_{1,1}} - \cP_{\bU_1}}_{\linf}\le \frac{C C_1\sigma  }{\lambda_{\min}}\sqrt{\frac{ m (2\mu)^{m-1} r_{-1} d_1  d^*\log \dmax }{n}}\cdot\frac{\mu r_1}{d_1}
\end{equation*}
with probability at least $1-3m\dmax^{-3m}$ by summing up all the terms. Extending this argument from $j=1$ to all the cases, we prove the desired statement.
\end{proof}

\subsection{Proof of Lemma \ref{lemma:oracle-init-frE-2-5}}\label{apx:sec:proof-all-rem-ind}
\begin{proof} We now consider these $4$ terms separately.

\noindent \textbf{(i) $\frE_2$ in} \eqref{eq:test-decomp}.
Since $\frE_2=\left\langle\bcE_{\init}\times_{j=1}^m \cP_{\bU_{j} },  \bcI \right\rangle = \left\langle\bcE_{\init}\times_{j=1}^m \cP_{\bU_{j} },  \bcI \times_{j=1}^m \cP_{\bU_{j} } \right\rangle $, we know that 

\begin{equation*}
    \abs{\frE_2} \le \norm{\bcE_{\init}\times_{j=1}^m \cP_{\bU_{j} }}_\tF  \norm{\bcI \times_{j=1}^m \cP_{\bU_{j}}}_\tF  \le \sqrt{\rmin} \norm{ \cP_{\bU_{j'}} \bE_{j',\init}\otimes_{k\neq j' } \cP_{\bU_{k}} }_2 \norm{\cP_{\TT}(\bcI)}_\tF,
\end{equation*}
where $j'$ is the mode such that $r_{j'}=\rmin$. Similar as  \eqref{eq:E-init-Ber-condition} and \eqref{eq:E-init-Ber}, we can check that

\begin{equation}\label{eq:E-init-Ber-condition-full}
    \begin{gathered}
        \norm{\left\langle \bT^{\vartriangle}_{j'}, \cM_{j'}(\bcX_i)\right\rangle \cP_{{\mathbf{U}}_{j'} }\cM_{j'}(\bcX_i)\left(\otimes_{j\neq {j'}} {\mathbf{U}}_{j} \right)}_2\le C_1 \sigma \sqrt{\frac{\dmax \log \dmax }{n}}\sqrt{\frac{\mu^{m} r^*}{d^* }} \\
        \norm{\E \left\langle \bT^{\vartriangle}_{j'}, \cM_{j'}(\bcX_i)\right\rangle^2  \cP_{{\mathbf{U}}_{j'} } \cM_{j'}(\bcX_i)\left(\otimes_{j\neq _{j'}} \cP_{{\mathbf{U}}_{j}} \right)\cM_{j'}(\bcX_i)^\top  \cP_{{\mathbf{U}}_{j'} } }_2 \le C_1^2 \sigma^2 \frac{\dmax \log \dmax }{n} \frac{\mu^{m} r^* } {d_{}^*}  \\
        \norm{\E \left\langle \bT^{\vartriangle}_{j'}, \cM_{j'}(\bcX_i)\right\rangle^2 \left(\otimes_{j\neq {j'}} \cP_{{\mathbf{U}}_{j}} \right)\cM_{j'}(\bcX_i)^\top  \cP_{{\mathbf{U}}_{j'} } \cM_{j'}(\bcX_i)\left(\otimes_{j\neq {j'}} \cP_{{\mathbf{U}}_{j}} \right) }_2 \le C_1^2 \sigma^2 \frac{\dmax \log \dmax }{n} \frac{\mu \rmin}{d^*},
    \end{gathered}
\end{equation}
which gives the bound that, with probability at least $1-\dmax^{-3m}$, $\norm{ \cP_{\bU_{j'}} \bE_{j',\init}\otimes_{k\neq j' } \cP_{\bU_{k}} }_2$ can be controlled by
\begin{equation}\label{eq:E-init-Ber-full}
\begin{aligned}
        \norm{ \cP_{\bU_{j'}} \bE_{j',\init}\otimes_{k\neq j' } \cP_{\bU_{k}} }_2 & \le C C_1 \sigma \sqrt{\frac{\dmax \log \dmax }{n}}\sqrt{\frac{\mu^{m} r^*}{d^* }} \cdot \frac{m d^*\log d_1}{n} + CC_1 \sigma \sqrt{\frac{\dmax \log \dmax }{n}}\sqrt{\frac{\mu^{m} r^* m d^*\log \dmax }{ n }}  \\
        & \le CC_1 \sigma \sqrt{\frac{\dmax \log \dmax }{n}}\sqrt{\frac{m \mu^{m} r^*  d^* \log \dmax }{ n }},
\end{aligned}
\end{equation}
given that $n\ge m \rmax \dmax \log\dmax$. \eqref{eq:E-init-Ber-full} implies the desired bound on $\frE_2$:
\begin{equation*}
    \abs{\frE_2} \le \sqrt{\rmin} \norm{ \cP_{\bU_{j'}} \bE_{j',\init}\otimes_{k\neq j' } \cP_{\bU_{k}} }_2 \norm{\cP_{\TT}(\bcI)}_\tF \le CC_1 \sigma \sqrt{\frac{\dmax \log \dmax }{n}}\sqrt{\frac{m \mu^{m} \rmin r^*  d^* \log \dmax }{ n }} \norm{\cP_{\TT}(\bcI)}_\tF.
\end{equation*}

    
\noindent \textbf{(ii) $\frE_3$ in} \eqref{eq:test-decomp}. $\frE_3$ is defined by $\frE_3 = \sum_{j=1}^m \left\langle \bcT\times_{k\neq j}\cP_{\bU_{k} } \times_{j}\cS_{\widehat{\bU}_{j,1}}^{(\rem)}, \bcI \right\rangle$. We write the structure of $\cS_{\widehat{\bU}_{j,1}}^{(\rem)}$ in detail following \eqref{eq:U1-order1-expansion}:
\begin{equation*}
    \begin{aligned}
        \cS_{\widehat{\bU}_{j,1}}^{(\rem)} = &  \cP_{\bU_j}^\perp\bE_{j,\rn} \left( \otimes_{k\neq j}\cP_{\widehat{\mathbf{U}}_{k,0}} -\otimes_{k\neq j}\cP_{{\mathbf{U}}_{k}} \right)\left( \otimes_{k\neq j}\bU_{k} \right)\bV_j \bLambda_j^{-1}\bU_j^\top + \cP_{\bU_j}^\perp \mathfrak{B}_{j,3} \bU_j\bLambda_j^{-2}\bU_j^\top \\
        & +\cP_{\bU_j}^\perp\left( \bE_{j,\init}\left( \otimes_{k\neq j}\cP_{\widehat{\mathbf{U}}_{k,0}} \right)\bT_{j}^\top \right)\bU_j\bLambda_j^{-2}\bU_j^\top\\
        & + \left(\cP_{\bU_j}^\perp\bE_{j,\rn} \left( \otimes_{k\neq j}\cP_{\widehat{\mathbf{U}}_{k,0}} -\otimes_{k\neq j}\cP_{{\mathbf{U}}_{k}} \right)\left( \otimes_{k\neq j}\bU_{k} \right)\bV_j \bLambda_j^{-1}\bU_j^\top+ \cP_{\bU_j}^\perp \mathfrak{B}_{j,3} \bU_j\bLambda_j^{-2}\bU_j^\top\right)^\top\\
       & + \left(\cP_{\bU_j}^\perp\left( \bE_{j,\init}\left( \otimes_{k\neq j}\cP_{\widehat{\mathbf{U}}_{k,0}} \right)\bT_{j}^\top \right)\bU_j\bLambda_j^{-2}\bU_j^\top \right)^\top
    \end{aligned}
\end{equation*}
where term $\mathfrak{B}_{j,3}$ is defined similar as $\mathfrak{B}_{3}$ in \eqref{eq:svd-U1-decomp}, but is for the decomposition of $\widehat{\bU}_{j,1}$ on each mode-$j$.  Notice that, since we can write each term in $\frE_3$ as 
\begin{equation}\label{eq:frE-3-decomp}
 \begin{aligned}
       & \left\langle \bcT\times_{k\neq j}\cP_{\bU_{k} } \times_{j}\cS_{\widehat{\bU}_{j,1}}^{(\rem)}, \bcI \right\rangle = \left\langle  \cS_{\widehat{\bU}_{j,1}}^{(\rem)}  \bU_j \bLambda_j \bV_j^\top \otimes_{k\neq j}\bU_k^\top, \cM_j(\bcI)\otimes_{k\neq j}\cP_{{\mathbf{U}}_{k}}  \right\rangle  \\
       & = \left\langle\cP_{\bU_j}^\perp\bE_{j,\rn} \left( \otimes_{k\neq j}\cP_{\widehat{\mathbf{U}}_{k,0}} -\otimes_{k\neq j}\cP_{{\mathbf{U}}_{k}} \right)\left( \otimes_{k\neq j}\bU_{k} \right)\cP_{\bV_j} \left(\otimes_{k\neq j}\bU_k^\top\right) + \cP_{\bU_j}^\perp \mathfrak{B}_{j,3} \bU_j\bLambda_j^{-1}\bV_j^\top  , \cM_j(\bcI)\otimes_{k\neq j}\cP_{{\mathbf{U}}_{k}} \right\rangle \\
       & = \underbrace{\left\langle \cP_{\bU_j}^\perp\bE_{j,\rn} \left( \otimes_{k\neq j}\cP_{\widehat{\mathbf{U}}_{k,0}} -\otimes_{k\neq j}\cP_{{\mathbf{U}}_{k}} \right)\left( \otimes_{k\neq j}\bU_{k} \right)\cP_{\bV_j}\left(\otimes_{k\neq j}\bU_k^\top\right) , \cM_j(\bcI)\otimes_{k\neq j}\cP_{{\mathbf{U}}_{k}}\right\rangle}_{\frF_1}  \\
       & \quad + \underbrace{\left\langle \cP_{\bU_j}^\perp \mathfrak{B}_{j,3} \bU_j\bLambda_j^{-1}\bV_j^\top\left(\otimes_{k\neq j}\bU_k^\top\right)  , \cM_j(\bcI)\otimes_{k\neq j}\cP_{{\mathbf{U}}_{k}}\right\rangle }_{\frF_2}\\
       & \quad +\underbrace{\left\langle \cP_{\bU_j}^\perp \bE_{j,\init}\left( \otimes_{k\neq j}\cP_{\widehat{\mathbf{U}}_{k,0}} \right) \left( \otimes_{k\neq j}\bU_{k} \right)\cP_{\bV_j}\left(\otimes_{k\neq j}\bU_k^\top\right) , \cM_j(\bcI)\otimes_{k\neq j}\cP_{{\mathbf{U}}_{k}}\right\rangle }_{\frF_3}
 \end{aligned}
\end{equation}
we only need to care about $\frF_1$, $\frF_2$, and $\frF_3$ in \eqref{eq:frE-3-decomp}. $\frF_1$ is handled by an entry-wise argument:
\begin{equation}\label{eq:frF-1}
    \begin{aligned}
        \abs{\frF_1} & = \abs{\left\langle \cP_{\bU_j}^\perp\bE_{j,\rn} \left( \otimes_{k\neq j}\cP_{\widehat{\mathbf{U}}_{k,0}} -\otimes_{k\neq j}\cP_{{\mathbf{U}}_{k}} \right)\left( \otimes_{k\neq j}\bU_{k} \right)\cP_{\bV_j}\left(\otimes_{k\neq j}\bU_k^\top\right) , \sum_{\omega} \left[\bcI\right]_{\omega} \cM_j(\bomega)\otimes_{k\neq j}\cP_{{\mathbf{U}}_{k}}\right\rangle  } \\
        & \le \norm{\bcI}_{\ell_1}\sqrt{\frac{\mu^{m-1} r_{-j}}{d_{-j}}}\norm{ \cP_{\bU_j}^\perp\bE_{j,\rn} \left( \otimes_{k\neq j}\cP_{\widehat{\mathbf{U}}_{k,0}} -\otimes_{k\neq j}\cP_{{\mathbf{U}}_{k}} \right)\left( \otimes_{k\neq j}\bU_{k} \right)\cP_{\bV_j} }_{2,\infty} \\
        & \le \norm{\bcI}_{\ell_1}\sqrt{\frac{\mu^{m-1} r_{-j}}{d_{-j}}}\left( \sqrt{\frac{\mu r_j}{d_j}} \norm{\bE_{j,\rn} \left( \otimes_{k\neq j}\cP_{\widehat{\mathbf{U}}_{k,0}} -\otimes_{k\neq j}\cP_{{\mathbf{U}}_{k}} \right) }_{2} \right.\\ 
        & \left. \quad + \norm{\bE_{j,\rn} \left( \otimes_{k\neq j}\cP_{\widehat{\mathbf{U}}_{k,0}} -\otimes_{k\neq j}\cP_{{\mathbf{U}}_{k}} \right) }_{2,\infty} \right).
    \end{aligned}
\end{equation}
We now give the error bounds of key terms in \eqref{eq:frF-1}:
\begin{Lemma}\label{lemma:E-rn-Udiff}
    Under the same assumptions as Theorem \ref{thm:lf-inference-popvar}, when the initialization is independent of $\left\{(\bcX_i,\xi_i)\right\}_{i=1}^n$,  we have the following inequalities hold with probability at least $1-2m \dmax^{-3m}$:
    \begin{equation*}
        \begin{aligned}
            \norm{\bE_{j,\rn} \left( \otimes_{k\neq j}\cP_{\widehat{\mathbf{U}}_{k,0}} -\otimes_{k\neq j}\cP_{{\mathbf{U}}_{k}} \right)}_2 \le  C \sigma \frac{ C_1\sigma  }{\lambda_{\min}}\sqrt{\frac{d^*\dmax \log \dmax }{n}} \sqrt{\frac{ m^3 (2\mu)^{m-1} r^* d_j  d^*\log \dmax }{ \rmin r_j n}},  \\
            \norm{\bE_{j,\rn} \left( \otimes_{k\neq j}\cP_{\widehat{\mathbf{U}}_{k,0}} -\otimes_{k\neq j}\cP_{{\mathbf{U}}_{k}} \right)}_{2,\infty} \le C \sigma \frac{ C_1\sigma  }{\lambda_{\min}}\sqrt{\frac{d^*\dmax \log \dmax }{n}} \sqrt{\frac{ m^3 (2\mu)^{m-1} r_{-j}  d^*\log \dmax }{n}}.
        \end{aligned}
    \end{equation*}
\end{Lemma}
Applying Lemma \ref{lemma:E-rn-Udiff} to \eqref{eq:frF-1}, we can control $\frF_1$ by 

\begin{equation*}
    \abs{\frF_1} \le C \sigma  \norm{\bcI}_{\ell_1}\frac{ C_1\sigma  }{\lambda_{\min}}\sqrt{\frac{d^*\dmax \log \dmax }{n}} \sqrt{\frac{ m^3 (2\mu)^{2(m-1)} r_{-j} r^*  d_j\log \dmax }{\rmin n}}.
\end{equation*}

$\frF_2$ in \eqref{eq:frE-3-decomp} can also be handled by the same entry-wise argument:

\begin{equation}\label{eq:frF-2}
    \begin{aligned}
        \abs{\frF_2} & = \abs{\left\langle \cP_{\bU_j}^\perp \mathfrak{B}_{j,3} \bU_j\bLambda_j^{-1}\bV_j^\top  , \sum_{\omega} \left[\bcI\right]_{\omega} \cM_j(\bomega)\otimes_{k\neq j}\cP_{{\mathbf{U}}_{k}}\right\rangle  } \\
       & \le \norm{\bcI}_{\ell_1}\sqrt{\frac{\mu^{m-1} r_{-j}}{d_{-j}}}\norm{\cP_{\bU_j}^\perp \mathfrak{B}_{j,3} \bU_j\bLambda_j^{-1}\bV_j^\top }_{2,\infty} \\
       & \le \norm{\bcI}_{\ell_1}\sqrt{\frac{\mu^{m-1} r_{-j}}{d_{-j}}} \left( \sqrt{\frac{\mu r_j}{d_j}}\norm{\mathfrak{B}_{j,3} \bU_j\bLambda_j^{-1}\bV_j^\top }_{2} +\norm{\mathfrak{B}_{j,3} \bU_j\bLambda_j^{-1}\bV_j^\top }_{2,\infty}\right)
    \end{aligned}
\end{equation}
Since $\mathfrak{B}_{j,3}$ is defined similar as $\mathfrak{B}_3$ in \eqref{eq:svd-U1-decomp}, the studies of $\mathfrak{B}_3$ in Lemma \ref{lemma:E-init-rn-2inf} can be directly applied to $\mathfrak{B}_{j,3}$. We thus have the inequalities:
\begin{equation*}
    \begin{gathered}
        \norm{\bE_{j,\init}\left(\otimes_{k\neq j} \widehat{\mathbf{U}}_{k,0} \right)}_{2,\infty} \le C C_1 \sigma \sqrt{\frac{ m (2\mu)^{m-1} r_{-j} \dmax d^*\log^2 \dmax }{n^2}} \\
        \norm{\bE_{j,\rn}\left(\otimes_{k\neq j} \widehat{\mathbf{U}}_{k,0} \right)}_{2,\infty} \le C \sigma \sqrt{\frac{ m (2\mu)^{m-1} r_{-j} d^*\log \dmax }{n}}\le \sqrt{\frac{\mu r_j }{d_j}} \cdot C \sigma \sqrt{\frac{ m (2\mu)^{m-2} r_{-j} d_j d^*\log \dmax }{n}}.\\
    \end{gathered}
\end{equation*}
with probability at least $1- 2m^2\dmax^{-3m}$ for any $j\in [m]$. This leads to the bound of $\mathfrak{B}_{j,3}$:
\begin{equation}\label{eq:B-j3}
    \begin{gathered}
        \norm{\mathfrak{B}_{j,3}}_2 \le C C_1^2\sigma^2 \frac{ m (2\mu)^{m-1} r^* d_j  d^*\log \dmax }{\rmin r_j n} \\
        \norm{\mathfrak{B}_{j,3}}_{2,\infty} \le C C_1^2\sigma^2 \sqrt{\frac{ m (2\mu)^{m-1} r^* d_j  d^*\log \dmax }{\rmin r_j n}}  \sqrt{\frac{\mu r_j }{d_j}} \cdot\sqrt{\frac{ m (2\mu)^{m-2} r_{-j} d_j d^*\log \dmax }{n}}
    \end{gathered}
\end{equation}
Plugging in \eqref{eq:B-j3} to \eqref{eq:frF-2}, we can control $\frF_2$ at the level:
\begin{equation*}
    \abs{\frF_2} \le C \frac{C_1^2 \sigma^2}{\lambda_{\min}} \norm{\bcI}_{\ell_1}\sqrt{\frac{\mu^{m} r^* }{d^* }}   \frac{ \sqrt{r^* r_{-j}} m (2\mu)^{m-1} d_j  d^*\log \dmax }{\sqrt{\rmin} n}.
\end{equation*}
$\frF_3$ in \eqref{eq:frE-3-decomp} can be handled by the concentration inequality. Notice that in each mode $j$ of $\frF_3$, we have 
\begin{equation}\label{eq:frF-3-bern}
    \begin{aligned}
       & \norm{\left\langle \cP_{\bU_j}^\perp \langle\bcT^{\vartriangle},\bcX_i \rangle \cM_j(\bcX_i)\left( \otimes_{k\neq j}\cP_{\widehat{\mathbf{U}}_{k,0}} \right) \left( \otimes_{k\neq j}\bU_{k} \right)\cP_{\bV_j}\left(\otimes_{k\neq j}\bU_k^\top\right)  , \cP_{\bU_j}^\perp \cM_j(\bcI) \otimes_{k\neq j}\cP_{{\mathbf{U}}_{k}}\right\rangle  }_{\psi_2} \\
       & \le C_1\sigma \sqrt{\frac{\dmax\log\dmax}{n}}\cdot\sqrt{\frac{\mu^{m-1}r_{-j}}{d_{-j}}}  \norm{\cP_{\bU_j}^\perp \cM_j(\bcI) \left( \otimes_{k\neq j}\bU_{k} \right)\cP_{\bV_j}\left(\otimes_{k\neq j}\bU_k^\top\right)  }_{\tF} \\
       & \abs{\left\langle \cP_{\bU_j}^\perp \langle\bcT^{\vartriangle},\bcX_i \rangle \cM_j(\bcX_i)\left( \otimes_{k\neq j}\cP_{\widehat{\mathbf{U}}_{k,0}} \right) \left( \otimes_{k\neq j}\bU_{k} \right)\cP_{\bV_j}\left(\otimes_{k\neq j}\bU_k^\top\right)  , \cP_{\bU_j}^\perp \cM_j(\bcI) \otimes_{k\neq j}\cP_{{\mathbf{U}}_{k}}\right\rangle}^2 \\
       & \le C_1^2\sigma^2\frac{\dmax\log\dmax}{n} \cdot \frac{1}{d^*}\norm{\cP_{\bU_j}^\perp \cM_j(\bcI) \left( \otimes_{k\neq j}\bU_{k} \right)\cP_{\bV_j}\left(\otimes_{k\neq j}\bU_k^\top\right)}_{\tF}^2.
    \end{aligned}
\end{equation}
Thus, it is clear that with probability at least $1-d^{-3m}$, we can control $\frF_3$ by:
\begin{equation}\label{eq:frF-3} 
    \abs{\frF_3} \le C C_1 \sigma \sqrt{\frac{ m (2\mu)^{m-1} r^* \dmax d^* \log^2\dmax }{ \rmax r_j n^2}} \norm{\cP_{\bU_j}^\perp \cM_j(\bcI) \left( \otimes_{k\neq j}\bU_{k} \right)\cP_{\bV_j}\left(\otimes_{k\neq j}\bU_k^\top\right)}_{\tF}
\end{equation}
Combining the bound of $\abs{\frF_1}$ with $\abs{\frF_2}$, $\abs{\frF_3}$,  we have
\begin{equation*}
    \abs{\frE_3}\le C \frac{C_1^2 \sigma^2}{\lambda_{\min}} \norm{\bcI}_{\ell_1} \frac{ m^2 (2\mu)^{\frac{3}{2}m-1} (r^*)^{\frac{3}{2}}\dmax \sqrt{d^*} \log \dmax }{\rmin n}+ C C_1 \sigma \sqrt{\frac{ m^2 (2\mu)^{m-1} r^* \dmax d^* \log^2\dmax }{ \rmax \rmin n^2}}\norm{\cP_{\TT}(\bcI)}_{\tF},
\end{equation*}
which holds with probability at least $1- 2m^2\dmax^{-3m}$.

\noindent \textbf{(iii) $\frE_4$ in} \eqref{eq:test-decomp}. An entry-wise argument like \eqref{eq:frF-1} gives the following control of $\frE_4$:
\begin{equation}\label{eq:frE-4}
\begin{aligned}
         \abs{\frE_4} &= \abs{\sum_{j=1}^m \left\langle \bcT\times_{k\neq j}\cP_{\bU_{k} } \times_{j}\left(\sum_{l\ge 2}\cS_{\widehat{\bU}_{j,1}}^{(l)}\right), \bcI \right\rangle} \\
         & \le \sum_{j=1}^m \abs{\left\langle \left(\sum_{l\ge 2}\cS_{\widehat{\bU}_{j,1}}^{(l)}\right)\bT_j \left(\otimes_{k\neq j}\cP_{{\mathbf{U}}_{k}} \right), \sum_{\omega} \left[\bcI\right]_{\omega} \cM_j(\bomega)\otimes_{k\neq j}\cP_{{\mathbf{U}}_{k}}\right\rangle} \\
         & \le \sum_{j=1}^m \norm{\bcI}_{\ell_1}\sqrt{\frac{\mu^{m-1} r_{-j}}{d_{-j}}}\norm{\left(\sum_{l\ge 2}\cS_{\widehat{\bU}_{j,1}}^{(l)}\right)\bT_j}_{2,\infty}
\end{aligned}
\end{equation}
Since the series $\left\{\cS_{\widehat{\bU}_{j,1}}^{(l)} \right\}_{l\ge 1}$ have been carefully discussed in the proof of Lemma \ref{lemma:l2inf-powerit} (see Section \ref{apx:sec:proof-U1-l2inf} for more details), we can derive the following bound:
\begin{equation*}
    \norm{\left(\sum_{l\ge 2}\cS_{\widehat{\bU}_{j,1}}^{(l)}\right)\bT_j}_{2,\infty} \le \frac{C C_1^2\sigma^2  }{\lambda_{\min}}\frac{ m (2\mu)^{m-1} r_{-j} d_j  d^*\log \dmax }{n} \cdot\sqrt{\frac{\mu r_j}{d_j}},
\end{equation*}
with probability at least $1-2m^2 \dmax^{-3m}$. With the help of Section \ref{apx:sec:proof-U1-l2inf}, this rate can be easily justified by noticing that all the non-zero terms of $\sum_{l\ge 2}\cS_{\widehat{\bU}_{j,1}}^{(l)}$ when multiplied by $\bT_j$ must end with $\bU_j\bLambda_j^{-2s}\bU_j^\top$. Plugging in this result to \eqref{eq:frE-4}, we have
\begin{equation*}
    \abs{\frE_4}\le C\norm{\bcI}_{\ell_1}\sqrt{\frac{\mu^{m} r^* }{d^* }} \frac{ C_1^2\sigma^2  }{\lambda_{\min}}\frac{ m^2 (2\mu)^{m-1} r^* \dmax  d^*\log \dmax }{\rmin n}
\end{equation*}

\noindent \textbf{(iv) $\frE_5$ in} \eqref{eq:test-decomp}. $\frE_5$ is defined by
\begin{equation*}
    \frE_5 = \left\langle \bcE_{\rem},\bcI  \right\rangle
\end{equation*}
To prove the vanishing of $\bcE_{\rem}$, we need a close look at each term in $\bcE_{\rem}$. According to \eqref{eq:hatT-decomp-Erem}, $\bcE_{\rem}$ is the higher order remainder containing at most $2^{m+1}$ single terms, where each single term has at least 2 differences as the factors. For each term, we only need to consider 2 cases: (1)  it has no $\left(\widehat{\bcT}_\ubs- \bcT\right)$ as a factor; (2) it has $\left(\widehat{\bcT}_\ubs- \bcT\right)$ as one of the differences in factors. We first check case (1). If $\left(\widehat{\bcT}_\ubs- \bcT\right)$ is not in the term, then such a term has at least 2 $\left( \cP_{\widehat{\bU}_{j,1}} - \cP_{\bU_j}\right)$ for different $j$s as factors. In this case, we can write such a term as
\begin{equation*}
    \begin{aligned}
             \bcT\times_{j\in \cS } \left( \cP_{\widehat{\bU}_{j,1}} - \cP_{\bU_j}\right)\times_{k\in \cS^c } \cP_{{\bU}_{k}},
    \end{aligned}
\end{equation*}
where $\cS$ is the index set with $\abs{\cS}\ge 2$. Consequently, the inner product between this term and $\bcI$ can be controlled by
\begin{equation}\label{eq:Erem-noTdiff}
    \begin{aligned}
        & \abs{\left\langle\bcT\times_{j\in \cS } \left( \cP_{\widehat{\bU}_{j,1}} - \cP_{\bU_j}\right)\times_{k\in \cS^c } \cP_{{\bU}_{k}}, \bcI \right\rangle} =\abs{\left\langle\bcT\times_{j\in \cS } \left( \cP_{\widehat{\bU}_{j,1}} - \cP_{\bU_j}\right)\times_{k\in \cS^c } \cP_{{\bU}_{k}}, \sum_{\omega} \left[\bcI\right]_{\omega}\bomega \right\rangle} \\
        & \le \norm{\bcI}_{\ell_1} \max_{\omega}\abs{\left\langle\bcT\times_{j\in \cS } \left( \cP_{\widehat{\bU}_{j,1}} - \cP_{\bU_j}\right)\times_{k\in \cS^c } \cP_{{\bU}_{k}},\bomega \right\rangle}\\
        &\le \norm{\bcI}_{\ell_1}\norm{\left( \cP_{\widehat{\bU}_{j,1}} - \cP_{\bU_j}\right)\bT_j }_{2,\infty}\prod_{k\in \cS,k\neq j}\norm{\cP_{\widehat{\bU}_{j,1}} - \cP_{\bU_j}}_{2,\infty} \prod_{k\notin \cS} \sqrt{\frac{\mu r_k}{d_k} }
    \end{aligned}
\end{equation}
Applying the results in the proof of Lemma \ref{lemma:l2inf-powerit} to \eqref{eq:Erem-noTdiff}, we have
\begin{equation}\label{eq:Erem-noTdiff-res}
    \begin{aligned}
        \abs{\left\langle\bcT\times_{j\in \cS } \left( \cP_{\widehat{\bU}_{j,1}} - \cP_{\bU_j}\right)\times_{k\in \cS^c } \cP_{{\bU}_{k}}, \bcI \right\rangle} \le  \frac{C C_1^2\sigma^2  }{\lambda_{\min}}\norm{\bcI}_{\ell_1}\frac{ m (2\mu)^{m-1} r_{-j} \dmax  d^*\log \dmax }{n} \cdot\sqrt{\frac{\mu^m r^*}{d^*}}.
    \end{aligned}
\end{equation}
We now focus on the case (2) where the term has $\left(\widehat{\bcT}_\ubs- \bcT\right)$ as one of the differences in factors. In this case, the term can be written as 
\begin{equation*}
                 \left(\widehat{\bcT}_\ubs- \bcT\right)\times_{j\in \cS } \left( \cP_{\widehat{\bU}_{j,1}} - \cP_{\bU_j}\right)\times_{k\in \cS^c } \cP_{{\bU}_{k}}=  \left(\bcE_{\rn} + \bcE_{\init}\right)\times_{j\in \cS } \left( \cP_{\widehat{\bU}_{j,1}} - \cP_{\bU_j}\right)\times_{k\in \cS^c } \cP_{{\bU}_{k}},
\end{equation*}
where the index set  $\cS$ is with $\abs{\cS}\ge 1$. Notice that in this case, we are unable to use the concentration of $\bcE_{\rn}$ and $\bcE_{\init}$ derived under the independent initialization to handle the problem, because the $\bcE_{\rn}$ and $\bcE_{\init}$ depend on the error of singular subspaces $\cP_{\widehat{\bU}_{j,1}} - \cP_{\bU_j}$, where $\cP_{\widehat{\bU}_{j,1}}$ comes from the data themselves. However, Lemma \ref{lemma:l2inf-powerit} indicates that 
\begin{equation}\label{eq:l2-inf-Uj1}
     \norm{\cP_{\widehat{\bU}_{j,1}} - \cP_{\bU_j}}_{2,\infty } \le  \frac{C C_1\sigma  }{\lambda_{\min}}\sqrt{\frac{ m (2\mu)^{m-1} r_{-j} d_j  d^*\log \dmax }{n}}\cdot \sqrt{\frac{\mu r_j}{d_j}} 
\end{equation}
for all $j$ with probability at least $1-3m^2 \dmax^{-3m}$. Moreover, it shows that under statistical optimal SNR condition,  $\widehat{\bU}_{j,1}$ is $2\mu$-incoherent. To study the $\varepsilon$-net of $\widehat{\bU}_{j,1}$ where the difference is measured by $\cP_{\widehat{\bU}_{j,1}} - \cP_{\bU_j}$, we do the rotation on each $\widehat{\bU}_{j,1}$ as stated in Section \ref{apx:sec:proof-PU-to-U} such that $\norm{\widehat{\bU}_{j,1}-{\bU}_{j} }_2$ is minimized. Then, under Lemma \ref{lemma:l2inf-powerit}, according to \eqref{eq:rotation-PU-U}, we can derive that
\begin{equation}
\begin{aligned}
\norm{\widehat{\bU}_{j,1} - \bU_j}_{2,\infty } & \le \sqrt{\norm{\be_{j,i}^\top\left(\cP_{\widetilde{\bU}_{j,0}}- \cP_{\bU_j}\right)}_2^2+ \frac{2\mu r_j}{d_j}\norm{ \cP_{\widetilde{\bU}_{j,0}} - \cP_{\bU_j}}_2^2} \\
    &\le  \frac{C C_1\sigma  }{\lambda_{\min}}\sqrt{\frac{ m (2\mu)^{m-1} r_{-j} d_j  d^*\log \dmax }{n}}\cdot \sqrt{\frac{\mu r_j}{d_j}} 
\end{aligned}
\end{equation}

 For each $j$, it is clear that when $\widehat{\bU}_{j,1}$ is $2\mu$-incoherent, we have
\begin{equation}\label{eq:linf-to-l2inf}
    \norm{\cP_{\widehat{\bU}_{j,1}} - \cP_{\bU_j}}_{\linf }  \le 6\sqrt{\frac{\mu r_j}{d_j}} \norm{\widehat{\bU}_{j,1} - \bU_j}_{2,\infty }.
\end{equation}
We now construct a $\varepsilon$-net using the statement similar to \eqref{eq:cover-Uj} under the case when  Lemma \ref{lemma:l2inf-powerit} holds. Define the set of $\widehat{\bU}_{j,1}$ under the results of Lemma \ref{lemma:l2inf-powerit} after the rotation as $\bbS^1(\bU_j)$, then the net 
\begin{equation*}
    \cN(\bbS^1(\bU_j),\norm{\cP_{\widehat{\bU}_{j,1}} - \cP_{\bU_j}}_{\linf }, \varepsilon) \le \cN(\bbS^2(\bU_j), 6\sqrt{\frac{\mu r_j}{d_j}} \norm{\cdot}_{2,\infty },\frac{1}{2}\varepsilon ),
\end{equation*}
where we define
\begin{equation*}
     \bbS^2(\bU_j)= \left\{\widehat{\bU}_{j,1}\in \O^{d_j\times r_j} :\  \norm{\widehat{\bU}_{j,1} - \bU_j}_{2,\infty }\le \frac{C C_1\sigma  }{\lambda_{\min}}\sqrt{\frac{ m (2\mu)^{m-1} r_{-j} d_j  d^*\log \dmax }{n}}\cdot \sqrt{\frac{\mu r_j}{d_j}}  \right\}.
\end{equation*}
Clearly, $\bbS^1(\bU_j)\subseteq \bbS^2(\bU_j)$.
By relaxing the ball with $\ell_{2,\infty}$ norm like Section \ref{apx:sec:proof-metric-ety}, we have the covering number bound:
\begin{equation}\label{eq:metric-ety-Uj1}
     \cN(\bbS^1(\bU_j),\norm{\cP_{\widehat{\bU}_{j,1}} - \cP_{\bU_j}}_{\linf } ,\varepsilon)\le \left(1+\frac{C C_1\sigma }{ \varepsilon  \lambda_{\min}}\sqrt{\frac{ m (2\mu)^{m-1} r_{-j} d_j  d^*\log \dmax }{n}}\right)^{r_j d_j }.
\end{equation}
Taking the $\varepsilon$-net of $\prod_{j}\bbS^1(\bU_j)$ with measurement $\max_{j}\norm{\cP_{\widehat{\bU}_{j,1}} - \cP_{\bU_j}}_{\linf }$ and $\varepsilon=\dmax^{-Cm}$, the size of such a net can be controlled by $\log \abs{\cN} \le C m \rmax \dmax \log \dmax$. Then, for all the elements in the net,
\begin{equation}\label{eq:Erem-Tdiff-ind-conc}
    \begin{aligned}
        &\abs{\left\langle \left(\bcE_{\rn} + \bcE_{\init}\right)\times_{j\in \cS } \left(  \cP_{\widehat\bU_{j,\net}} - \cP_{\bU_{j} }\right)\times_{k\in \cS^c } \cP_{{\bU}_{k}} , \bcI \right\rangle} \\
        & \le \frac{\sigma d^*\sqrt{m\rmax\dmax } }{\sqrt{n}} \norm{\bcI}_{\ell_1} \frac{C C_1\sigma }{  \lambda_{\min}}\sqrt{\frac{ m (2\mu)^{m-1} r^* \dmax  d^*\log^2 \dmax }{\rmin n}} \cdot\frac{\mu^m r^* }{d^*} \\
        & \le \sigma\norm{\bcI}_{\ell_1} \frac{C C_1\sigma  }{\lambda_{\min}}\sqrt{\frac{ m^2 (2\mu)^{3m-1} r^* \dmax^2  d^*\log^2 \dmax }{n^2}},
    \end{aligned}
\end{equation}
holds uniformly with probability at least $1-\dmax^{-3m}$ using the sub-Gaussian norm of each i.i.d term in $\bcE_{\rn} + \bcE_{\init}$, and the entry-wise argument with the $\linf$ norm bound in Lemma \ref{lemma:l2inf-powerit}.

For the difference term between the element in the net and our objective, we can just use the size bound to control the error. That is to say, we have
\begin{equation}\label{eq:Erem-Tdiff}
\begin{aligned}
      &\abs{ \left\langle \left( \left\langle\bcT^{\vartriangle} ,\bcX_i \right\rangle\bcX_i 
       +\xi_i\bcX_i\right) \times_{j\in \cS } \left( \cP_{\widehat{\bU}_{j,1}} - \cP_{\widehat \bU_{j,\net} }\right)\times_{k\in \cS^c } \cP_{{\bU}_{k}},\bcI\right\rangle} \\
       &\le \left(\abs{\xi_i}+ \abs{\left\langle\bcT^{\vartriangle} ,\bcX_i \right\rangle} \right) \norm{\bcI}_{\ell_1} \frac{C C_1\sigma  }{\lambda_{\min}}\sqrt{\frac{ m (2\mu)^{m-1} r^* \dmax  d^*\log \dmax }{ \dmax^{Cm} n}} \cdot\frac{\mu^m r^* }{d^*},
\end{aligned}
\end{equation}
uniformly for all $i\in[n]$. \eqref{eq:Erem-Tdiff} indicates that 
\begin{equation*}
    \begin{aligned}
        &\abs{\left\langle \left(\bcE_{\rn} + \bcE_{\init}\right)\times_{j\in \cS } \left( \cP_{\widehat{\bU}_{j,1}} - \cP_{\widehat \bU_{j,\net} }\right)\times_{k\in \cS^c } \cP_{{\bU}_{k}} , \bcI \right\rangle} \\
        & \le \frac{d^*}{n} \sum_{i=1}^n \left(\abs{\xi_i}+ \abs{\left\langle\bcT^{\vartriangle} ,\bcX_i \right\rangle} \right) \norm{\bcI}_{\ell_1} \frac{C C_1\sigma  }{\lambda_{\min}}\sqrt{\frac{ m (2\mu)^{m-1} r^* \dmax  d^*\log \dmax }{\dmax^{Cm} n}} \cdot\frac{\mu^m r^* }{d^*} \\
        & \le \sigma\norm{\bcI}_{\ell_1} \frac{C C_1\sigma  }{\lambda_{\min}}\sqrt{\frac{ m^2 (2\mu)^{3m-1} (r^*)^3 \dmax  d^*\log^2 \dmax }{n^2}}
    \end{aligned}
\end{equation*}
 with probability at least $1-\dmax^{-3m}$ using the sub-Gaussian norm. Therefore, we can conclude that
\begin{equation}\label{eq:Erem-Tdiff-res}
    \begin{aligned}
        \abs{\left\langle \left(\bcE_{\rn} + \bcE_{\init}\right)\times_{j\in \cS } \left( \cP_{\widehat{\bU}_{j,1}} - \cP_{\bU_j}\right)\times_{k\in \cS^c } \cP_{{\bU}_{k}} , \bcI \right\rangle} \le C \norm{\bcI}_{\ell_1} \frac{C_1\sigma^2  }{\lambda_{\min}}\sqrt{\frac{ m^2 (2\mu)^{3m-1} r^* \dmax^2  d^*\log^2 \dmax }{n^2}},
    \end{aligned}
\end{equation}
with probability at least $1-(3m^2+2)\dmax^{-3m}$. Combining \eqref{eq:Erem-noTdiff-res} with \eqref{eq:Erem-Tdiff-res}, we have the bound on $\frE_5$:
\begin{equation*}
    \abs{\frE_5} = \abs{\left\langle \bcE_{\rem},\bcI  \right\rangle} \le  \frac{C C_1^2\sigma^2  }{\lambda_{\min}}\norm{\bcI}_{\ell_1}\frac{ m (2\mu)^{\frac{3}{2}m} (r^*)^{\frac{3}{2}} \dmax  \sqrt{d^*}\log \dmax }{\rmin n},
\end{equation*}
with probability at least $1-2^{m+1}(3m^2+2)\dmax^{-3m}$. Combining all the high probability bounds, we know that the desired claim holds with probability at least $1-\dmax^{-2m}$.
\end{proof}


\subsection{Proof of Lemma \ref{lemma:E-init-rn-2inf}}\label{apx:sec:proof-Einit-rn}
\begin{proof}
    This is equivalent to controlling $\norm{\be_{1,k}^\top \bE_{1,\rn}\left(\otimes_{j\neq 1} \widehat{\mathbf{U}}_{j,0} \right) }_{2}$ and $\norm{\be_{1,k}^\top\bE_{1,\init}\left(\otimes_{j\neq 1} \widehat{\mathbf{U}}_{j,0} \right) }_{2}$ for any $\be_{1,k}$ with $k\in[d_1]$ uniformly. We take $\norm{\be_{1,k}^\top\bE_{1,\init}\left(\otimes_{j\neq 1} \widehat{\mathbf{U}}_{j,0} \right) }_{2}$ for example, and the other term can be handled similarly. The bound of this term can also be given by the matrix Bernstein inequality like \eqref{eq:E-init-Ber-condition}. Notice that, if $\be_{1,k}$ is considered compared with \eqref{eq:E-init-Ber-condition}, the first-order bound remains unchanged but the second-order bound will be  
    \begin{equation}\label{eq:E-init-Ber-2inf-condition}
    \begin{gathered}
        \norm{\E \left\langle \bT^{\vartriangle}_1, \cM_1(\bcX_i)\right\rangle^2 \be_{1,k}^\top\cM_1(\bcX_i)\left(\otimes_{j\neq 1} \cP_{\widehat{\mathbf{U}}_{j,0}} \right)\cM_1(\bcX_i)^\top\be_{1,k} }_2 \le C_1^2 \sigma^2 \frac{\dmax \log \dmax }{n} \frac{(2\mu)^{m-1} r_{-1}}{d_{}^*}  \\
        \norm{\E \left\langle \bT^{\vartriangle}_1, \cM_1(\bcX_i)\right\rangle^2 \left(\otimes_{j\neq 1} \cP_{\widehat{\mathbf{U}}_{j,0}} \right)\cM_1(\bcX_i)^\top\be_{1,k} \be_{1,k}^\top\cM_1(\bcX_i)\left(\otimes_{j\neq 1} \cP_{\widehat{\mathbf{U}}_{j,0}} \right) }_2 \le C_1^2 \sigma^2 \frac{\dmax \log \dmax }{n} \frac{1}{d^*}.
    \end{gathered}
\end{equation}
This leads to the following inequality
\begin{equation}\label{eq:E-init-Ber-2inf}
    \begin{aligned}
        \norm{\be_{1,k}^\top\bE_{1,\init}\left(\otimes_{j\neq 1} \widehat{\mathbf{U}}_{j,0} \right)}_2 &\le C C_1 \sigma \left(\frac{d^*}{n} \sqrt{\frac{\dmax \log \dmax }{n}}\sqrt{\frac{\mu^{m-1} r_{-1}}{d_{-1}}}\cdot m\log \dmax  +  \sqrt{\frac{ m (2\mu)^{m-1} r_{-1} \dmax d^*\log^2 \dmax }{n^2}} \right) \\
        & \le C C_1 \sigma \sqrt{\frac{ m (2\mu)^{m-1} r_{-1} \dmax d^*\log^2 \dmax }{n^2}},
    \end{aligned}
\end{equation}
with probability at least $1-\dmax^{-3 m}$, given that $n\ge m\rmax\dmax \log \dmax$. Using the same technique and checking the conditions similar as \eqref{eq:E-rn-Ber-condition}, we have 
\begin{equation}\label{eq:E-rn-Ber-2inf}
    \begin{aligned}
        \norm{\be_{1,k}^\top\bE_{1,\rn}\left(\otimes_{j\neq 1} \widehat{\mathbf{U}}_{j,0} \right)}_2 &\le C \sigma \left(\frac{d^*}{n}\sqrt{\frac{\mu^{m-1} r_{-1}}{d_{-1}}}\cdot m\log \dmax  +  \sqrt{\frac{ m (2\mu)^{m-1} r_{-1}  d^*\log \dmax }{n}} \right) \\
        & \le C \sigma \sqrt{\frac{ m (2\mu)^{m-1} r_{-1} d^*\log \dmax }{n}},
    \end{aligned}
\end{equation}
with probability at least $1-\dmax^{-3 m}$. Combining all the probability bounds, we have the desired uniform bound.
\end{proof}

    

\subsection{Proof of Lemma \ref{lemma:E-rn-Udiff}}\label{apx:sec:proof-E-rn-Udiff}
\begin{proof} Under the independence initialization, it is clear from Lemma \ref{lemma:l2inf-init} that 
    \begin{equation*}
   \norm{ \cP_{\widehat{\bU}_{j,0}} - \cP_{\bU_j} }_{2,\infty} \le \frac{8 C_1\sigma  }{\lambda_{\min}}\sqrt{\frac{\mu r_j}{d_j} }\sqrt{\frac{d^*\dmax \log \dmax }{n}}, \quad \norm{ \cP_{\widehat{\bU}_{j,0}} - \cP_{\bU_j} }_{2} \le \frac{8 C_1 \sigma  }{\lambda_{\min}}\sqrt{\frac{d^*\dmax \log \dmax }{n}}.
\end{equation*}
We now check the first-order and second-order conditions:
\begin{equation}\label{eq:E-rn-Udiff-Ber-condition}
    \begin{gathered}
       \norm{ \norm{\xi_i \cM_j(\bcX_i)\left( \otimes_{k\neq j}\cP_{\widehat{\mathbf{U}}_{k,0}} -\otimes_{k\neq j}\cP_{{\mathbf{U}}_{k}} \right) }_2 }_{\psi_2} \le C \sigma\frac{ C_1\sigma  }{\lambda_{\min}}\sqrt{\frac{d^*\dmax \log \dmax }{n}}\cdot \sqrt{\frac{\mu^{m-1} r_{-j}}{d_{-j}}}  \\
       \norm{\E \xi_i^2\cM_j(\bcX_i)\left( \otimes_{k\neq j}\cP_{\widehat{\mathbf{U}}_{k,0}} -\otimes_{k\neq j}\cP_{{\mathbf{U}}_{k}} \right)^2\cM_j(\bcX_i)^\top   }_2 \le \sigma^2 \frac{(2\mu)^{m-1} r_{-j}}{d_{}^*}\left(\frac{8 C_1 \sigma  }{\lambda_{\min}}\sqrt{\frac{d^*\dmax \log \dmax }{n}}\right)^2   \\
       \norm{\E \xi_i^2\left( \otimes_{k\neq j}\cP_{\widehat{\mathbf{U}}_{k,0}} -\otimes_{k\neq j}\cP_{{\mathbf{U}}_{k}} \right)\cM_j(\bcX_i)^\top\cM_j(\bcX_i)\left( \otimes_{k\neq j}\cP_{\widehat{\mathbf{U}}_{k,0}} -\otimes_{k\neq j}\cP_{{\mathbf{U}}_{k}} \right)}_2 \le  \sigma^2  \frac{1}{d_{-j}} \left(\frac{8 C_1 \sigma  }{\lambda_{\min}}\sqrt{\frac{d^*\dmax \log \dmax }{n}}\right)^2.
    \end{gathered}
\end{equation}
Using the matrix Bernstein inequality, we have 
\begin{equation}\label{eq:E-rn-Udiff-Ber}
    \begin{aligned}
        \norm{\bE_{j,\rn}\left( \otimes_{k\neq j}\cP_{\widehat{\mathbf{U}}_{k,0}} -\otimes_{k\neq j}\cP_{{\mathbf{U}}_{k}} \right)}_2 &\le C \sigma \frac{ C_1\sigma  }{\lambda_{\min}}\sqrt{\frac{d^*\dmax \log \dmax }{n}} \left(\frac{d^*}{n}\sqrt{\frac{\mu^{m-1} r_{-j}}{d_{-j}}}\cdot m\log \dmax  +  \sqrt{\frac{ m (2\mu)^{m-1} d_j  d^*\log \dmax }{n}} \right) \\
        & \le C \sigma \frac{ C_1\sigma  }{\lambda_{\min}}\sqrt{\frac{d^*\dmax \log \dmax }{n}} \sqrt{\frac{ m (2\mu)^{m-1} r^* d_j  d^*\log \dmax }{\rmin r_jn}}.
    \end{aligned}
\end{equation}
For the $\ell_{2,\infty}$ error bound, we only need to re-do the analysis like the proof of Lemma \ref{lemma:E-init-rn-2inf} by checking the conditions of $\norm{\be_{j,k}^\top \bE_{j,\rn} \left( \otimes_{k\neq j}\cP_{\widehat{\mathbf{U}}_{k,0}} -\otimes_{k\neq j}\cP_{{\mathbf{U}}_{k}} \right) }_{2}$, which gives us the following error rate:
        \begin{equation}\label{eq:E-rn-Udiff-l2inf-Ber}
        \begin{aligned}
             \norm{\bE_{j,\rn} \left( \otimes_{k\neq j}\cP_{\widehat{\mathbf{U}}_{k,0}} -\otimes_{k\neq j}\cP_{{\mathbf{U}}_{k}} \right)}_{2,\infty} & \le C \sigma \frac{ C_1\sigma  }{\lambda_{\min}}\sqrt{\frac{d^*\dmax \log \dmax }{n}} \left(\frac{d^*}{n}\sqrt{\frac{\mu^{m-1} r_{-j}}{d_{-j}}}\cdot m\log \dmax  +  \sqrt{\frac{ m (2\mu)^{m-1} r_{-j}  d^*\log \dmax }{n}} \right) \\
        & \le C \sigma \frac{ C_1\sigma  }{\lambda_{\min}}\sqrt{\frac{d^*\dmax \log \dmax }{n}} \sqrt{\frac{ m (2\mu)^{m-1} r_{-j}  d^*\log \dmax }{n}}.
        \end{aligned}
    \end{equation}
\eqref{eq:E-rn-Udiff-Ber} and \eqref{eq:E-rn-Udiff-l2inf-Ber} hold with probability at least $1-\dmax^{-3m}$ for one $j$. Therefore,  such two rates hold for all $j\in [m]$ with probability at least $1-2m\dmax^{-3m}$.

\end{proof}

\subsection{Proof of Lemma \ref{lemma:l2inf-powerit-dep}}\label{apx:sec:proof-l2inf-powerit-dep}
We prove this statement by modifying the proof of Lemma \ref{lemma:l2inf-powerit} under independent initialization, which is shown in Section \ref{apx:sec:proof-U1-l2inf}. We focus on $j=1$. Such a proof relies on the $\ell_{2}$ and $\ell_{2,\infty}$ norm bound of the terms related to $\bE_{1,\rn}$ and $\bE_{1,\init}$, which we have presented in Lemma \ref{lemma:Einit-Ern-dep-l2} and Lemma \ref{lemma:E-init-rn-2inf-dep}.

\begin{equation*}
    \begin{aligned}
        \norm{ \bE_{1,\rn}\left(\otimes_{j\neq 1} \widehat{\mathbf{U}}_{j,0} \right) }_2 &  \le C \sigma\left(1+\frac{C_1\sigma}{\lambda_{\min}}\sqrt{\frac{m^2 d^*\dmax\log\dmax}{n}} + \frac{ C_1 \sigma^2  }{\lambda_{\min}^2}\frac{d^*\dmax^{\frac{3}{2}} \log \dmax }{n} \right) \sqrt{\frac{  m \mu^{m-1} r^* d_1  d^*\log \dmax }{\rmin r_1 n}},
    \end{aligned}
\end{equation*}

\begin{equation*}
    \norm{\bE_{1,\init}\left(\otimes_{j\neq 1} \widehat{\mathbf{U}}_{j,0} \right) }_2 \le C C_1 \sigma \sqrt{\frac{ m (2\mu)^{m-1} \rmax r^* \dmax d^*\log \dmax }{r_1 n}\cdot \frac{d_1\dmax \log \dmax}{n}\wedge 1}.
\end{equation*}

\begin{equation*}
    \begin{gathered}
        \norm{\bE_{1,\init}\left(\otimes_{j\neq 1} \widehat{\mathbf{U}}_{j,0} \right)}_{2,\infty} \le C C_1 \sigma \sqrt{\frac{1}{d_1}} \sqrt{\frac{ m (2\mu)^{m-1}\rmax r^* \dmax d^*\log \dmax }{r_1 n}\cdot \frac{d_1 \dmax  \log \dmax}{n}\wedge 1}\\
        \norm{\bE_{1,\rn}\left(\otimes_{j\neq 1} \widehat{\mathbf{U}}_{j,0} \right)}_{2,\infty} \le C \sigma \sqrt{\frac{ m (2\mu)^{m-1} r_{-1} d^*\log \dmax }{n}} + C \sigma \frac{ C_1\sigma }{\lambda_{\min}}\sqrt{\frac{m^2 d^*\dmax \log \dmax }{n}} \sqrt{\frac{ m (2\mu)^{m-1} \rmax r_{-1} \dmax d^*\log \dmax }{ n}}.\\
    \end{gathered}
\end{equation*}

Applying Lemma \ref{lemma:E-init-rn-2inf-dep} to the series expansions of $\cP_{\widehat{\bU}_{1,1}} - \cP_{\bU_1}$ in Section \ref{apx:sec:proof-U1-l2inf} (e.g., inequalities \eqref{eq:powerit-U-delta11-U} and \eqref{eq:powerit-U-delta11-lam}, etc) with the new $\ell_2$ norm bounds of $\norm{\bE_{1,\init}\left(\otimes_{j\neq 1} \widehat{\mathbf{U}}_{j,0} \right)}_{2} $ and $\norm{\bE_{1,\rn}\left(\otimes_{j\neq 1} \widehat{\mathbf{U}}_{j,0} \right)}_{2} $ in \eqref{eq:E-rn-Ber-dep} and \eqref{eq:dep-E-init-uniform}, we can derive that, when the SNR satisfies 
\begin{equation*}
\begin{aligned}
       \frac{\lambda_{\min}}{\sigma} &\ge C_{\gap}\kappa_0 \left(C_1 \sqrt{\frac{  m (2\mu)^{m} \rmax (r^*) d^*\dmax^{\frac{3}{2}} \log \dmax }{\rmin n}}\right),
\end{aligned}
\end{equation*}
the series are guaranteed to converge with the rate:

\begin{equation*}
\begin{aligned}
        \norm{\cP_{\widehat{\bU}_{1,1}} - \cP_{\bU_1}}_{2} &\le   \frac{C C_1\sigma}{\lambda_{\min}}  \left(1+\frac{\sigma}{\lambda_{\min}}\sqrt{\frac{ m^2 d^*\dmax\log\dmax}{n}} + \frac{\sigma^2  }{\lambda_{\min}^2}\frac{d^*\dmax^{\frac{3}{2}} \log \dmax }{n} \right) \sqrt{\frac{  m (2\mu)^{m-1}\rmax r^* d_1  d^*\log \dmax }{ r_1 n}} \\
        & \le \frac{C C_1 \sigma}{\lambda_{\min} }\sqrt{\frac{  m (2\mu)^{m-1}\rmax r^* d_1  d^*\log \dmax }{ r_1 n}},
\end{aligned}
\end{equation*}
with probability at least $1-m 2^{m+2}\dmax^{-3m}$.
and the corresponding $\ell_{2,\infty}$ norm bound:
\begin{equation*}
\begin{aligned}
        &\norm{\cP_{\widehat{\bU}_{1,1}} - \cP_{\bU_1}}_{2,\infty}  \le \frac{C}{\lambda_{\min} }\left( \sqrt{\frac{\mu r_1 }{d_1}}\norm{(\bE_{1,\rn}+\bE_{1,\init})\left(\otimes_{j\neq 1} \widehat{\mathbf{U}}_{j,0} \right)}_2 + \norm{(\bE_{1,\rn}+\bE_{1,\init})\left(\otimes_{j\neq 1} \widehat{\mathbf{U}}_{j,0} \right)}_{2,\infty}\right) \\
        & \le \frac{CC_1}{\lambda_{\min} } \sqrt{\frac{\mu r_1 }{d_1}} \left( \sigma\sqrt{\frac{  m (2\mu)^{m-1}\rmax r^* d_1  d^*\log \dmax }{ r_1 n}} + \sigma\frac{ \sigma }{\lambda_{\min}}\sqrt{\frac{m^2 d^*\dmax \log \dmax }{n}} \sqrt{\frac{ m (2\mu)^{m-1} \rmax r_{-1} d_1\dmax d^*\log \dmax }{ n}} \right)\\
        & \le CC_1\sqrt{\frac{\mu r_1 }{d_1}} \left(\frac{\sigma}{\lambda_{\min} }\sqrt{\frac{  m (2\mu)^{m-1}\rmax r^* d_1  d^*\log \dmax }{ r_1 n}} + \frac{\sigma^2}{\lambda_{\min}^2 }\frac{  m^{\frac{3}{2}} (2\mu)^{\frac{1}{2}m}\rmax^{\frac{1}{2}} (r^*)^{\frac{1}{2}} d^*\dmax^{\frac{3}{2}} \log \dmax }{ \sqrt{r_1} n} \right)
\end{aligned}
\end{equation*}
with the same probability. The $\ell_{\infty}$ norm bound instantly follows.




\subsection{Proof of Lemma \ref{lemma:dep-init-frE-2-5}}\label{apx:sec:proof-lemma:dep-init-frE-2-5}
\begin{proof}
We follow the steps in Section \ref{apx:sec:proof-all-rem-ind} to handle

\noindent \textbf{(i) $\frE_2$ in} \eqref{eq:test-decomp}.
Since $\frE_2=\left\langle\bcE_{\init}\times_{j=1}^m \cP_{\bU_{j} },  \bcI \right\rangle = \left\langle\bcE_{\init}\times_{j=1}^m \cP_{\bU_{j} },  \bcI \times_{j=1}^m \cP_{\bU_{j} } \right\rangle $, we know that 

\begin{equation*}
    \abs{\frE_2} \le \norm{\bcE_{\init}\times_{j=1}^m \cP_{\bU_{j} }}_\tF  \norm{\bcI \times_{j=1}^m \cP_{\bU_{j}}}_\tF  \le \sqrt{\rmin} \norm{ \cP_{\bU_{j'}} \bE_{j',\init}\otimes_{k\neq j' } \cP_{\bU_{k}} }_2 \norm{\cP_{\TT}(\bcI)}_{\tF}
\end{equation*}

An analysis similar to \eqref{eq:dep-E-init-uniform} will give the following bound 
\begin{equation}\label{eq:dep-E-init-full-uniform}
    \norm{ \cP_{\bU_{j'}} \bE_{j',\init}\otimes_{k\neq j' } \cP_{\bU_{k}} }_2 \le  C C_1 \sigma \sqrt{\frac{ m (2\mu)^{m} \rmax r^* \dmax d^*\log \dmax }{ n}\cdot \frac{\dmax \log \dmax}{n}\wedge \frac{1}{\dmin}},
\end{equation}
which is at the order of $\sqrt{\mu r_{j'}/d_{j'}}$ times \eqref{eq:dep-E-init-uniform}, and it holds with probability at least $1-2\dmax^{-3m}$. Therefore, we can bound $\frE_2$  by

\begin{equation*}
    \abs{\frE_2} \le C C_1 \sigma \norm{\cP_{\TT}(\bcI)}_{\tF}\sqrt{\frac{ m (2\mu)^{m} \rmax^2 r^* d^* \dmax \log \dmax }{ n}\cdot \frac{\dmax \log \dmax}{n} \wedge \frac{1}{\dmin}}.
\end{equation*}
On the other hand, we can control the $\frE_2$ by:
\begin{equation*}
    \abs{\frE_2} \le \norm{\bcE_{\init}\times_{j=1}^m \cP_{\bU_{j} }}_\tF  \norm{\bcI \times_{j=1}^m \cP_{\bU_{j}}}_\tF  \le \sqrt{\rmin} \norm{ \cP_{\bU_{j'}} \bE_{j',\init}\otimes_{k\neq j' } \cP_{\bU_{k}} }_2 \norm{\bcI}_{\ell_1} \sqrt{\frac{\mu^m r^* }{d^*}},
\end{equation*}
which leads to the following bound:
\begin{equation*}
    \abs{\frE_2} \le C C_1 \sigma \norm{\bcI}_{\ell_1} \sqrt{\frac{ m (2\mu)^{2m} \rmax^2 (r^*)^2 \dmax \log \dmax }{ n}\cdot \frac{\dmax \log \dmax}{n} \wedge \frac{1}{\dmin}}.
\end{equation*}

\noindent \textbf{(ii) $\frE_3$ in} \eqref{eq:test-decomp}.
We need to revisit $\frF_1$,  $\frF_2$, and  $\frF_3$ in \eqref{eq:frE-3-decomp}. To this end, we need the corresponding bound under dependent initialization:
\begin{Lemma}\label{lemma:E-rn-Udiff-dep}
    Under the same assumptions as Theorem \ref{thm:clt-dep-init}, when the initialization depends on the data $\left\{(\bcX_i,\xi_i)\right\}_{i=1}^n$,   the following inequalities hold with probability at least $1-m^2 2^{m+2}\dmax^{-3m}$:
    \begin{equation*}
        \begin{aligned}
               &\norm{\bE_{j,\rn} \left( \otimes_{k\neq j}\cP_{\widehat{\mathbf{U}}_{k,0}} -\otimes_{k\neq j}\cP_{{\mathbf{U}}_{k}} \right)}_2   \le CC_1 \sigma\left(\frac{\sigma}{\lambda_{\min}}\sqrt{\frac{m^2 d^*\dmax\log\dmax}{n}} + \frac{C_1 \sigma^2  }{\lambda_{\min}^2}\frac{d^*\dmax^{\frac{3}{2}} \log \dmax }{n} \right) \sqrt{\frac{  m \mu^{m-1} r^* d_j  d^*\log \dmax }{\rmin r_j n}} \\
       &\norm{\bE_{j,\rn} \left( \otimes_{k\neq j}\cP_{\widehat{\mathbf{U}}_{k,0}} -\otimes_{k\neq j}\cP_{{\mathbf{U}}_{k}} \right)}_{2,\infty} \le C \sigma \frac{ C_1\sigma  }{\lambda_{\min}}\sqrt{\frac{m^2 d^*\dmax \log \dmax }{ n}} \sqrt{\frac{ m \mu^{m-1}  \rmax  r^* \dmax d^*\log \dmax }{r_j n}}.\\
        \end{aligned}
    \end{equation*}
\end{Lemma}
Lemma \ref{lemma:E-rn-Udiff-dep} can be directly extracted from the proof of Lemma \ref{lemma:Einit-Ern-dep-l2} and \ref{lemma:E-init-rn-2inf-dep}, particularly in the bound of first-order and higher-order terms. Thus, we omit the proof here.


We can then plug in Lemma \ref{lemma:E-init-rn-2inf-dep}, \ref{lemma:E-rn-Udiff-dep} to $\frF_1$ and $\frF_2$, $\frF_3$ in \eqref{eq:frE-3-decomp}. For  $\frF_1$ we can immediately derive the following error rate:
\begin{equation}\label{eq:frF-1-dep}
        \begin{aligned}
        \abs{\frF_1} 
         & = \abs{\left\langle \cP_{\bU_j}^\perp\bE_{j,\rn} \left( \otimes_{k\neq j}\cP_{\widehat{\mathbf{U}}_{k,0}} -\otimes_{k\neq j}\cP_{{\mathbf{U}}_{k}} \right)\left( \otimes_{k\neq j}\bU_{k} \right)\cP_{\bV_j}\left(\otimes_{k\neq j}\bU_k^\top\right) , \sum_{\omega} \left[\bcI\right]_{\omega} \cM_j(\bomega)\otimes_{k\neq j}\cP_{{\mathbf{U}}_{k}}\right\rangle  } \\
       & \le \norm{ \cP_{\bU_j}^\perp\cM_j(\bcI)\otimes_{k\neq j}\cP_{{\mathbf{U}}_{k}} }_{\tF}\sqrt{2r_{-j}} \norm{\bE_{j,\rn} \left( \otimes_{k\neq j}\cP_{\widehat{\mathbf{U}}_{k,0}} -\otimes_{k\neq j}\cP_{{\mathbf{U}}_{k}} \right)}_2 \\
        &\le  C \norm{ \cP_{\bU_j}^\perp\cM_j(\bcI)\otimes_{k\neq j}\cP_{{\mathbf{U}}_{k}} }_{\tF} \cdot C_1 \sigma\left(\frac{\sigma}{\lambda_{\min}}\sqrt{\frac{m^2 d^*\dmax\log\dmax}{n}} + \frac{C_1 \sigma^2  }{\lambda_{\min}^2}\frac{d^*\dmax^{\frac{3}{2}} \log \dmax }{n} \right) \sqrt{\frac{  m \mu^{m-1} (r^*)^2 d_j  d^*\log \dmax }{\rmin r_j^2 n}} \\
        &\le  C C_1\norm{ \cP_{\bU_j}^\perp\cM_j(\bcI)\otimes_{k\neq j}\cP_{{\mathbf{U}}_{k}} }_{\tF} \cdot \frac{\sigma^2}{\lambda_{\min}}\sqrt{\frac{m^2 d^*\dmax\log\dmax}{n}}\sqrt{\frac{  m \mu^{m-1} (r^*)^2 d_j  d^*\log \dmax }{\rmin r_j^2 n}}
    \end{aligned} 
\end{equation}
For  $\frF_2$, \eqref{eq:frF-2} suggests the following bound:
\begin{equation}\label{eq:frF-2-dep-ub}
     \abs{\frF_2} \le \norm{\bcI}_{\ell_1}\sqrt{\frac{\mu^{m-1} r_{-j}}{d_{-j}}} \left( \sqrt{\frac{\mu r_j}{d_j}}\norm{\mathfrak{B}_{j,3} \bU_j\bLambda_j^{-1}\bV_j^\top }_{2} +\norm{\mathfrak{B}_{j,3} \bU_j\bLambda_j^{-1}\bV_j^\top }_{2,\infty}\right).
\end{equation}
In this upper bound, 
we notice that $\mathfrak{B}_{j,3} \bU_j$ is composed of  
\begin{equation*}
    \mathfrak{B}_{j,3} \bU_j = \left(\bE_{j,\init}+\bE_{j,\rn} \right) \left( \otimes_{k\neq j}\cP_{\widehat{\mathbf{U}}_{k,0}} \right)\left(\bE_{j,\init}+\bE_{j,\rn} \right)^\top\bU_j,
\end{equation*}
where $ \mathfrak{B}_{j,3}$ follows the definition of $ \mathfrak{B}_{3}$ in \eqref{eq:svd-U1-decomp}.

Together with \eqref{eq:frF-2-dep-ub} and Lemma \ref{lemma:Einit-Ern-dep-l2}, \ref{lemma:E-init-rn-2inf-dep}, we can control $ \abs{\frF_2}$ by:
\begin{equation}\label{eq:frF-2-dep}
\begin{aligned}
         \abs{\frF_2} & \le \norm{\bcI}_{\ell_1}\sqrt{\frac{\mu^{m-1} r_{-j}}{d_{-j}}} \left( \sqrt{\frac{\mu r_j}{d_j}}\norm{\mathfrak{B}_{j,3} \bU_j\bLambda_j^{-1}\bV_j^\top }_{2} +\norm{\mathfrak{B}_{j,3} \bU_j\bLambda_j^{-1}\bV_j^\top }_{2,\infty}\right)\\
         & \le CC_1 \norm{\bcI}_{\ell_1}\sqrt{\frac{\mu^{m-1} r_{-j}}{d_{-j}}} \sqrt{\frac{\mu r_j }{d_j}} \left(\frac{\sigma}{\lambda_{\min} }\sqrt{\frac{  m (2\mu)^{m-1}\rmax r^* d_j d^*\log \dmax }{ r_j n}} + \left(\frac{\sigma}{\lambda_{\min} }\sqrt{\frac{  m^{\frac{3}{2}} (2\mu)^{m-1}\rmax r^* d^*\dmax^{\frac{3}{2}} \log \dmax }{ r_j n}}\right)^2 \right) \\
         &\quad  \quad \cdot \sigma \sqrt{\frac{  m (2\mu)^{m-1}\rmax r^* d_j  d^*\log \dmax }{ r_j n}} \\
           & \le CC_1 \sigma^2\norm{\bcI}_{\ell_1}\sqrt{\frac{  m (2\mu)^{2m-1}\rmax (r^*)^2 \dmax \log \dmax }{ \rmin n}}  \\ 
           & \quad \cdot \left(\frac{\sigma}{\lambda_{\min} }\sqrt{\frac{  m (2\mu)^{m-1}\rmax r^* \dmax d^*\log \dmax }{ \rmin n}} + \frac{\sigma^2}{\lambda_{\min}^2 }\frac{  m^{\frac{3}{2}} (2\mu)^{m-1}\rmax r^* d^*\dmax^{\frac{3}{2}} \log \dmax }{ \rmin n} \right) 
\end{aligned}
\end{equation}
For  $\frF_3$ in \eqref{eq:frE-3-decomp}, we follow the steps in \eqref{eq:frF-3-bern} and \eqref{eq:frF-3} and use the $\varepsilon$-net to control it by repeating the procedures in the proof of Lemma \ref{lemma:Einit-Ern-dep-l2}. This will lead to the following bound:
\begin{equation}\label{eq:frF-3-dep} 
    \abs{\frF_3} \le C C_1 \sigma \sqrt{\frac{ m^2 (2\mu)^{m-1} r^* \dmax^2 d^* \log^2\dmax }{  r_j n^2}} \norm{\cP_{\bU_j}^\perp \cM_j(\bcI) \left( \otimes_{k\neq j}\bU_{k} \right)\cP_{\bV_j}\left(\otimes_{k\neq j}\bU_k^\top\right)}_{\tF},
\end{equation}
which holds with probability at least $1-\dmax^{-3m}$.

Combining \eqref{eq:frF-1-dep}, \eqref{eq:frF-2-dep}, and \eqref{eq:frF-3-dep}, we have the bound on $\frE_3$:
\begin{equation*}
    \begin{aligned}
       \abs{\frE_3} & \le C C_1\norm{ \cP_{\TT}(\bcI) }_{\tF} \cdot \frac{\sigma^2}{\lambda_{\min}}\sqrt{\frac{m^2 d^*\dmax\log\dmax}{n}}\sqrt{\frac{  m^2 \mu^{m-1} (r^*)^2 \dmax  d^*\log \dmax }{\rmin^3 n}} \\
       & +C C_1 \norm{\bcI}_{\ell_1}\sigma \sqrt{\frac{ m^3 (2\mu)^{2m} (r^*)^2 \dmax \log \dmax }{\rmin^2 n}}\left(\frac{\sigma}{\lambda_{\min} }\sqrt{\frac{  m (2\mu)^{m-1}\rmax r^* \dmax d^*\log \dmax }{ \rmin n}} + \frac{\sigma^2}{\lambda_{\min}^2 }\frac{  m^{\frac{3}{2}} (2\mu)^{m-1}\rmax r^* d^*\dmax^{\frac{3}{2}} \log \dmax }{ \rmin n}\right) \\
       & +  C C_1\norm{ \cP_{\TT}(\bcI) }_{\tF} \cdot \sigma \sqrt{\frac{ m^3 (2\mu)^{m-1} r^* \dmax^2 d^* \log^2\dmax }{  r_j n^2}}.
    \end{aligned}
\end{equation*}

\noindent \textbf{(iii) $\frE_4$ in} \eqref{eq:test-decomp}. With the help of \eqref{eq:frE-4}, we have
\begin{equation}
\begin{aligned}
         \abs{\frE_4}\le \sum_{j=1}^m \norm{\bcI}_{\ell_1}\sqrt{\frac{\mu^{m-1} r_{-j}}{d_{-j}}}\norm{\left(\sum_{l\ge 2}\cS_{\widehat{\bU}_{j,1}}^{(l)}\right)\bT_j}_{2,\infty},
\end{aligned}
\end{equation}
where the expressions of these series $\left\{\cS_{\widehat{\bU}_{j,1}}^{(l)} \right\}_{l\ge 1}$ have been given in Section \ref{apx:sec:proof-U1-l2inf}. Under the dependent initialization, an analysis of the expressions combined with Lemma \ref{lemma:Einit-Ern-dep-l2}, \ref{lemma:E-init-rn-2inf-dep} will give the following bound:
\begin{equation*}
\begin{aligned}
        \norm{\left(\sum_{l\ge 2}\cS_{\widehat{\bU}_{j,1}}^{(l)}\right)\bT_j}_{2,\infty} & \le CC_1\sigma \sqrt{\frac{\mu r_j }{d_j}}\sqrt{\frac{  m (2\mu)^{m-1}\rmax r^* d_j  d^*\log \dmax }{ r_j n}} \\
        &\quad \cdot\left(\frac{\sigma}{\lambda_{\min} }\sqrt{\frac{  m (2\mu)^{m-1}\rmax r^* d_j d^*\log \dmax }{ r_j n}} + \frac{\sigma^2}{\lambda_{\min}^2 }\frac{  m^{\frac{3}{2}} (2\mu)^{\frac{1}{2}m}\rmax^{\frac{1}{2}} (r^*)^{\frac{1}{2}} d^*\dmax^{\frac{3}{2}} \log \dmax }{ \sqrt{r_1} n} \right) 
\end{aligned}
\end{equation*}
with probability at least $1-2m^2 \dmax^{-3m}$ for all $j\in[m]$. Therefore, we can bound $\abs{\frE_4}$ by 
\begin{equation*}
\begin{aligned}
         \abs{\frE_4}\le & C C_1\sigma \norm{\bcI}_{\ell_1}\sqrt{\frac{  m^3 (2\mu)^{2m-1}\rmax (r^*)^2 \dmax  \log \dmax }{ \rmin n}} \\
        &\quad \cdot\left(\frac{\sigma}{\lambda_{\min} }\sqrt{\frac{  m (2\mu)^{m-1}\rmax r^* d_j d^*\log \dmax }{ \rmin n}} + \frac{\sigma^2}{\lambda_{\min}^2 }\frac{  m^{\frac{3}{2}} (2\mu)^{\frac{1}{2}m}\rmax^{\frac{1}{2}} (r^*)^{\frac{1}{2}} d^*\dmax^{\frac{3}{2}} \log \dmax }{ \sqrt{\rmin} n} \right) 
\end{aligned}
\end{equation*}

\noindent \textbf{(iv) $\frE_5$ in} \eqref{eq:test-decomp}. We modify the proof in Section \ref{apx:sec:proof-all-rem-ind} to control $\frE_5$. To this end, we need to control the two types of errors defined in \textbf{(iv)} of Section \ref{apx:sec:proof-all-rem-ind}. For the first type, we consider the case $\abs{\cS}=2$ for brevity:

\begin{equation}\label{eq:Erem-noTdiff-res-dep}
    \begin{aligned}
       & \abs{\left\langle\bcT\times_{j\in \cS } \left( \cP_{\widehat{\bU}_{j,1}} - \cP_{\bU_j}\right)\times_{k\in \cS^c } \cP_{{\bU}_{k}}, \bcI \right\rangle} \le  C C_1^2 \norm{\bcI}_{\ell_1}  \cdot\sqrt{\frac{\mu^m r^*}{d^*}} \\
        & \cdot \frac{\sigma^2 }{\lambda_{\min} }\frac{  m (2\mu)^{m-1}\rmax r^* \dmax  d^*\log \dmax }{ \rmin n} \left(1+ \frac{C_1 \sigma }{\lambda_{\min} }\sqrt{\frac{m d^* \dmax^2 \log\dmax}{n}} \right)^2 \\
        & \le   C C_1^2 \norm{\bcI}_{\ell_1} \frac{\sigma^2 }{\lambda_{\min} }\frac{  m (2\mu)^{\frac{3}{2}m-1}\rmax (r^*)^{\frac{3}{2}} \dmax  \sqrt{d^*} \log \dmax }{ \rmin n} \left( 1 + \frac{C_1^2 \sigma^2 }{\lambda_{\min}^2 }\frac{m d^* \dmax^2 \log\dmax}{n} \right).
    \end{aligned}
\end{equation}
Moreover, following the steps in 
\eqref{eq:metric-ety-Uj1}, \eqref{eq:Erem-Tdiff}, \eqref{eq:Erem-Tdiff-ind-conc}, \eqref{eq:Erem-Tdiff-res}, and substituting the $\ell_{2,\infty}$, $\ell_{\infty}$ bound of $\cP_{\widehat{\bU}_{j,1}} - \cP_{\bU_j}$ with the new order presented in Lemma \eqref{lemma:l2inf-powerit-dep}, we can conclude that the error of the second type can be controlled by:
\begin{equation}\label{eq:Erem-Tdiff-res-dep}
    \begin{aligned}
        & \abs{\left\langle \left(\bcE_{\rn} + \bcE_{\init}\right)\times_{j\in \cS } \left( \cP_{\widehat{\bU}_{j,1}} - \cP_{\bU_j}\right)\times_{k\in \cS^c } \cP_{{\bU}_{k}} , \bcI \right\rangle} \\
        & \le C \norm{\bcI}_{\ell_1} \frac{C_1\sigma^2  }{\lambda_{\min}}\sqrt{\frac{ m^2 (2\mu)^{3m-1} r^* \dmax^2  d^*\log^2 \dmax }{n^2}}\left(1+ \frac{C_1 \sigma }{\lambda_{\min} }\sqrt{\frac{m d^* \dmax^2 \log\dmax}{n}} \right).
    \end{aligned}
\end{equation}
Here the term related to $\bcE_{\rn}$ can be handled by \eqref{eq:Erem-Tdiff-ind-conc}, and the term  related to $\bcE_{\init}$ can be simply controlled by the size bound for $\abs{\cS}=1$:
\begin{equation*}
\begin{aligned}
        & \abs{\langle \bcT^{\vartriangle}, \bcX_i\rangle\left\langle \left(\bcX_i\right)\times_{j\in \cS } \left( \cP_{\widehat{\bU}_{j,1}} - \cP_{\bU_j}\right)\times_{k\in \cS^c } \cP_{{\bU}_{k}} , \bomega \right\rangle }\le C C_1\sigma\sqrt{\frac{\dmax\log\dmax}{n}} \prod_{j\in\cS}\norm{\cP_{\widehat{\bU}_{j,1}} - \cP_{\bU_j}}_{\ell_{\infty}}\prod_{j\in\cS^c}\frac{\mu r_j }{d_j} \\
    & \le CC_1\frac{\mu^{m} r^* }{d^*} \cdot \sigma\sqrt{\frac{\dmax\log\dmax}{n}} \left(\frac{\sigma}{\lambda_{\min} }\sqrt{\frac{  m (2\mu)^{m-1}\rmax r^* d_j  d^*\log \dmax }{ r_j n}} + \frac{C_1\sigma^2}{\lambda_{\min}^2 }\frac{  m^{\frac{3}{2}} (2\mu)^{\frac{1}{2}m}\rmax^{\frac{1}{2}} (r^*)^{\frac{1}{2}} d^*\dmax^{\frac{3}{2}} \log \dmax }{ \sqrt{r_j} n}\right).
\end{aligned}
\end{equation*}
Putting \eqref{eq:Erem-noTdiff-res-dep} and \eqref{eq:Erem-Tdiff-res-dep} together, we can prove the desired bound.
\end{proof}

\subsection{Proof of Lemma \ref{lemma:LS-cct-QT}}
For the unfolding on any mode $j$, it is clear that given a fixed  $\{\widetilde{\bU}_j\}_{j=1}^m$, we have
\begin{equation*}
    \frac{d^*}{n}\cM_j\left( \cQ_{\widetilde{\bU}^\top}\cP_{\Omega}\left( \bcT + \bcZ \right) \right) =  \frac{d^*}{n}\sum_{i=1}^n (\xi_i+\langle  \bcT,\bcX_i \rangle)\widetilde{\bU}_j^\top\cM_j(\bcX_i) \otimes_{k\neq j} \widetilde{\bU}_k. 
\end{equation*}
According to the matrix Bernstein inequality, we need to check the first-order condition:
\begin{equation}\label{eq:LS-cct-QT-1}
    \norm{(\xi_i+\langle  \bcT,\bcX_i \rangle)\widetilde{\bU}_j^\top\cM_j(\bcX_i) \otimes_{k\neq j} \widetilde{\bU}_k }_{\psi_2}\le C(\sigma\vee\norm{\bcT}_{\linf})\sqrt{\mu^m r^*}/\sqrt{d^*}
\end{equation}
and the second order condition:
\begin{equation}\label{eq:LS-cct-QT-2}
    \norm{\E\left[(\xi_i+\langle  \bcT,\bcX_i \rangle)\widetilde{\bU}_j^\top\cM_j(\bcX_i) \otimes_{k\neq j} \widetilde{\bU}_k\right]^2 }_2 \le C(\sigma^2\vee\norm{\bcT}_{\linf}^2)\frac{\mu^m r^*}{d^*}.
\end{equation}
Combining \eqref{eq:LS-cct-QT-1} and \eqref{eq:LS-cct-QT-2}, we can derive that 
\begin{equation*}
    \norm{\frac{d^*}{n}\cM_j\left( \cQ_{\widetilde{\bU}^\top}\cP_{\Omega}\left( \bcT + \bcZ \right) \right) - \cM_j(\cQ_{\widetilde{\bU}^\top} \bcT)}_2 \le C (\sigma\vee\norm{\bcT}_{\linf})\left(\frac{\sqrt{\mu^m r^* d^*} m^2\rmax\dmax\log \dmax }{n} + \sqrt{\frac{\mu^m m^2\rmax r^* d^* \dmax \log \dmax}{n}}\right).
\end{equation*}
Since $n\ge C m^2\rmax\dmax\log \dmax$, we prove the desired bound.
\subsection{Proof of Lemma \ref{lemma:LS-cct-QQT}}
We use the basic  Bernstein inequality to prove the claim. Notice that
\begin{equation*}
    \frac{d^*}{n}\norm{\cP_{\Omega}\left(\cQ_{\widetilde{\bU}}\cQ_{\widetilde{\bU}^\top}\bcT-\bcT \right)}_{\tF}^2 = \frac{d^*}{n}\sum_{i=1}^n \langle\cQ_{\widetilde{\bU}}\cQ_{\widetilde{\bU}^\top}\bcT-\bcT ,\bcX_i \rangle^2,
\end{equation*}
with first-order condition: 
$\abs{\langle\cQ_{\widetilde{\bU}}\cQ_{\widetilde{\bU}^\top}\bcT-\bcT ,\bcX_i \rangle}^2\le \norm{\cQ_{\widetilde{\bU}}\cQ_{\widetilde{\bU}^\top}\bcT - \bcT }_\linf^2\le 4{\mu^{m}r^*}\norm{\bcT}_{\linf}^2 $,
and second-order condition:
\begin{equation*}
\E\langle\cQ_{\widetilde{\bU}}\cQ_{\widetilde{\bU}^\top}\bcT-\bcT ,\bcX_i \rangle^4\le 4{\mu^{m}r^*}\norm{\bcT}_{\linf}^2 \E\langle\cQ_{\widetilde{\bU}}\cQ_{\widetilde{\bU}^\top}\bcT-\bcT ,\bcX_i \rangle^2 \le \frac{ 4{\mu^{m}r^*}\norm{\bcT}_{\linf}^2}{d^*}\norm{\cQ_{\widetilde{\bU}}\cQ_{\widetilde{\bU}^\top}\bcT-\bcT}_{\tF}^2.
\end{equation*}
We thus have
\begin{equation*}
\begin{aligned}
         &\abs{\frac{d^*}{n}\norm{\cP_{\Omega}\left(\cQ_{\widetilde{\bU}}\cQ_{\widetilde{\bU}^\top}\bcT-\bcT \right)}_{\tF}^2 - \norm{\cQ_{\widetilde{\bU}}\cQ_{\widetilde{\bU}^\top}\bcT-\bcT}_{\tF}^2} \\
         &\le C\frac{{m^2\rmax\mu^{m}r^* \dmax d^* \log \dmax }\norm{\bcT}_{\linf}^2}{n} +\norm{\bcT}_{\linf}\norm{\cQ_{\widetilde{\bU}}\cQ_{\widetilde{\bU}^\top}\bcT-\bcT}_{\tF} \sqrt{\frac{ m^2\rmax{\mu^{m}r^*} d^*\dmax\log \dmax  }{n}},
\end{aligned}
\end{equation*}
uniformly for any $\{\widetilde{\bU}_j\}_{j=1}^m$ in a net of $\O_{\bW}$ with cardinality at most $\dmax^{Cm^2\rmax\dmax}$, which proves the desired claim.
\end{document}